\documentclass{amsart}
\usepackage{amssymb,amsmath,amsfonts,amsthm,epsfig,amscd}
\usepackage{stmaryrd}
\usepackage[all,cmtip,poly]{xy}
\usepackage{svn}
\usepackage{mathrsfs}

\newtheorem{thm}[subsubsection]{Theorem}
\newtheorem{lemma}[subsubsection]{Lemma}

\newtheorem{cor}[subsubsection]{Corollary}

\newtheorem{prop}[subsubsection]{Proposition}

\newtheorem{atheorem}[subsection]{Theorem}

\newtheorem{adefinition}[subsection]{Definition}

\newtheorem{aprop}[subsection]{Proposition}

\newtheorem{defn}[subsubsection]{Definition}

\theoremstyle{remark}
\newtheorem{remark}[subsubsection]{Remark}

\newtheorem{aremark}[subsection]{Remark}

\newtheorem{example}[subsubsection]{Example}

\numberwithin{equation}{subsection}
\def\nummultline{\addtocounter{subsubsubsection}{1}\begin{multline}}
\def\anumequation{\addtocounter{subsection}{1}\begin{equation}}

\newif\iffinalrun
\iffinalrun
\else
\fi

\iffinalrun
  \newcommand{\need}[1]{}
  \newcommand{\mar}[1]{}
\else
  \newcommand{\need}[1]{{\tiny *** #1}}
  \newcommand{\mar}[1]{\marginpar{\raggedright\tiny #1}}
\fi

\newcommand{\p}{\frakp}
\newcommand{\q}{\frakq}

\renewcommand{\bf}{\ensuremath{\mathbf{f}}}

\newcommand{\cA}{{\mathcal A}}

\newcommand{\cC}{{\mathcal C}}
\newcommand{\cD}{{\mathcal D}}
\newcommand{\cE}{{\mathcal E}}

\newcommand{\cG}{{\mathcal G}}
\newcommand{\cH}{{\mathcal H}}

\newcommand{\cK}{{\mathcal K}}

\newcommand{\cM}{{\mathcal M}}

\newcommand{\cO}{{\mathcal O}}
\newcommand{\cP}{{\mathcal P}}

\newcommand{\cR}{{\mathcal R}}
\newcommand{\cS}{{\mathcal S}}

\newcommand{\cV}{{\mathcal V}}
\newcommand{\cW}{{\mathcal W}}
\newcommand{\cX}{{\mathcal X}}
\newcommand{\cY}{{\mathcal Y}}
\newcommand{\cZ}{{\mathcal Z}}

\newcommand{\frakp}{\mathfrak{p}}
\newcommand{\frakq}{\mathfrak{q}}

\DeclareMathOperator{\Gal}{Gal}
\DeclareMathOperator{\GL}{GL}

\DeclareMathOperator{\Hom}{Hom}

\DeclareMathOperator{\Spf}{Spf}

\newcommand{\dR}{\mathrm{dR}}

\newcommand{\ur}{\mathrm{ur}}

\newcommand{\et}{\mathrm{\acute{e}t}}

\newcommand{\toisom}{\buildrel\sim\over\to}

\newcommand{\Fl}{\mathscr{F}\!\ell}

\newcommand{\GrBdR}{\mathrm{Gr}^{B^+_{\mathrm{dR}}}}

\newcommand{\Spl}{\mathrm{Spl}}

\providecommand{\MR}[1]{}
\renewcommand{\MR}[1]{}

\begin{document}
\title[The generic part of the cohomology of Shimura varieties]{On the generic part of the cohomology of compact unitary Shimura varieties}

\author[A. Caraiani]{Ana Caraiani}\email{caraiani@princeton.edu}
\address{Department of Mathematics, Princeton University, Fine Hall, Washington Rd., Princeton, NJ 08544, USA}
\thanks{ A.C. was partially supported by the NSF Postdoctoral Fellowship DMS-1204465 and NSF Grant DMS-1501064.}

\author[P. Scholze]{Peter Scholze}\email{scholze@math.uni-bonn.de}
\address{Mathematisches Institut der Universit\"at Bonn, Endenicher Allee 60, 53115 Bonn, Germany}

\begin{abstract} The goal of this paper is to show that the cohomology of compact unitary Shimura varieties is concentrated in the middle degree and torsion-free, after localizing at a maximal ideal of the Hecke algebra satisfying a suitable genericity assumption. Along the way, we establish various foundational results on the geometry of the Hodge-Tate period map. In particular, we compare the fibres of the Hodge-Tate period map with Igusa varieties.
\end{abstract}

\maketitle

\tableofcontents

\section{Introduction}\label{sec:introduction}
Let $G/\mathbb{Q}$ be a reductive group. The real group $G(\mathbb R)$ acts on its associated symmetric domain $X = G(\mathbb R)/K_\infty$, where $K_\infty\subset G(\mathbb R)$ is a maximal compact subgroup. For any congruence subgroup $\Gamma\subset G(\mathbb Q)$, one can form the locally symmetric space
\[
X_\Gamma = \Gamma\backslash X\ .
\]
We assume that $X_\Gamma$ is compact, and that $\Gamma$ is torsion-free. Then Matsushima's formula, \cite{Matsushima}, expresses the cohomology groups $H^i(X_\Gamma,\mathbb C)$ with complex coefficients in terms of automorphic forms $\pi$ on $G$, and the $(\mathfrak g,K_\infty)$-cohomology of their archimedean component $\pi_\infty$.\footnote{In the non-compact case, this is still true, and a theorem of Franke, \cite{Franke}.} A computation of $(\mathfrak g,K_\infty)$-cohomology then shows that the part of cohomology to which tempered representations contribute is concentrated in the middle range $q_0\leq i\leq q_0+l_0$, cf.~\cite[Theorem III.5.1]{borelwallach}; here $l_0=\mathrm{rk}\ G - \mathrm{rk}\ K_\infty$, and $q_0=\frac{1}{2}(\dim X-l_0)$.

In particular, if $l_0=0$, then tempered representations occur only in the middle degree $q_0$. This happens when the $X_\Gamma$ are complex algebraic varieties, e.g.~when $G$ gives rise to a Shimura variety.

The motivating question of this paper is to establish a similar result for the cohomology groups $H^i(X_\Gamma,\mathbb F_\ell)$ with torsion coefficients. In this context, it is difficult to formulate the analogue of the temperedness condition, which is an analytic one. We learnt the following formulation from M.~Emerton. Recall that for any system $\mathfrak m$ of Hecke eigenvalues appearing in $H^i(X_\Gamma,\mathbb F_\ell)$, one expects to have a mod $\ell$ Galois representation $\rho_{\mathfrak{m}}$ (with values in the Langlands dual group). One may then put the condition that $\rho_{\mathfrak{m}}$ is irreducible, and ask whether this implies that $q_0\leq i\leq q_0+l_0$. In particular, a result of this type for $G=GL_n$ (where $l_0>0$) is important for automorphy lifting theorems in the non-self dual case as in work of Calegari-Geraghty, \cite[Conjecture B]{calegari-geraghty}.

In the present paper, we deal with this question in the case where $X_\Gamma$ is a Shimura variety (so that $l_0=0$). More precisely, we will consider the case where $G$ is an anisotropic unitary similitude group of dimension $n$, for some CM field $F$ with totally real subfield $F^+\subset F$. We assume that $F$ contains an imaginary-quadratic field. Assume moreover that $G$ is associated with a division algebra over $F$, i.e., it is one of Kottwitz' simple Shimura varieties, \cite{Kottwitz-lambda}.\footnote{We also allow the complementary case where $G$ is quasisplit at all finite places, under a small extra assumption (cf.~Section~\ref{setup}), so that our main result also covers cases where nontrivial endoscopy occurs.} Our main theorem is the following.

\begin{atheorem}\label{main theorem intro} Let $\mathfrak m$ be a system of Hecke eigenvalues appearing in $H^i(X_\Gamma,\mathbb F_\ell)$. Then there is an associated Galois representation
\[
\rho_{\mathfrak{m}}: \mathrm{Gal}(\overline{F}/F)\to GL_n(\bar{\mathbb F}_\ell)\ .
\]

Assume that there is a rational prime $p$ such that $F$ is completely decomposed above $p$, and
\[
\rho_{\mathfrak{m}}\mathrm{\ is\ unramified\ and\ decomposed\ generic}
\]
at all places of $F$ above $p$. Then $i=q_0$ is the middle degree.
\end{atheorem}

\begin{aremark} The first part of the theorem can be deduced from \cite{scholze}, but we give a different proof in this paper. We will make use of the Hodge-Tate period map again, but this time in a $p$-adic context with $p\neq \ell$ (whereas~\cite{scholze} worked in the situation $p=\ell$). We note that this should make it possible to understand the behaviour of $\rho_{\mathfrak{m}}$ at places above $\ell$.
\end{aremark}

\begin{aremark} It is a formal consequence that the $\mathbb Z_\ell$-cohomology localized at $\mathfrak m$ is concentrated in degree $q_0$, and torsion-free, if the conclusion of the theorem holds true.
\end{aremark}

\begin{aremark} The condition that $\rho_{\mathfrak{m}}$ is decomposed generic is defined below. It follows from a suitable ``big image'' assumption. However, note that if $\rho_{\mathfrak{m}}$ is a generic sum of characters, there will still be a prime $p$ as in the theorem, so that our result also applies to many reducible representations.
\end{aremark}

\begin{aremark} We prove the result under a slightly weaker assumption depending on the precise signature of $G$. In particular, if the signature of $G$ is $(0,n)$ at all except for one infinite place, e.g.~in the Harris-Taylor case, we only need the existence of one finite prime $v$ of $F$ at which $\rho_{\mathfrak m}$ is unramified and decomposed generic.
\end{aremark}

\begin{aremark} In the Harris-Taylor case (i.e., $G$ is of signature $(1,n-1)$ at one infinite place, and $(0,n)$ at the other places), there has been previous work on this question, notably by Shin, \cite{shin-torsion}, restricting attention to the cohomology that is supercuspidal modulo $\ell$ at some finite prime $p$, by Emerton and Gee, \cite{emerton-gee}, making suitable assumptions on $\rho_{\mathfrak m}$ at $\ell$-adic places, and by Boyer, \cite{boyer}, under a condition very closely related to our condition.
\end{aremark}

\begin{aremark} Lan and Suh, \cite{lan-suh}, prove that if the level is hyperspecial at $\ell$ and one takes cohomology with coefficients in the local system $\mathcal{L}_\xi$ corresponding to a suitably generic algebraic representation $\xi$ of $G$, then the whole $\ell$-adic cohomology groups $H^i(X_\Gamma,\mathcal{L}_\xi)$ for $i\neq q_0$ vanish. This behaviour cannot be expected in our situation, as at least all even cohomology groups $H^{2i}(X_\Gamma,\mathbb F_\ell)$ are nonzero, so it is necessary to localize at some maximal ideal of the Hecke algebra.
\end{aremark}

\begin{remark} An argument involving the Hochschild-Serre spectral sequence and Poincar\'e duality shows that the theorem also holds when $\mathbb F_\ell$ is replaced by a non-trivial (Hecke-equivariant) coefficient system.
\end{remark}

\begin{aremark} Let $F$ be a CM field and $\Pi$ be a conjugate self-dual regular algebraic cuspidal automorphic representation of $GL_n(\mathbb{A}_F)$. Then $\Pi$ will be obtained by base change from an automorphic representation $\pi$ on a unitary group, which contributes to the cohomology of a compact unitary Shimura variety (see, for example,~\cite{harris-taylor, shin-galois, caraiani}). In this situation, $\pi$ contributes only to the middle degree cohomology, and the proof relies on genericity rather than temperedness. In fact, concentration in middle degree is proved simultaneously with the Ramanujan-Petersson conjecture (at finite places) for $\Pi$ as above, using the template of~\cite{harris-taylor} rather than appealing to~\cite{borelwallach}. These results rely on the fact that the local components of cuspidal automorphic representations of $GL_n$ are \emph{generic}, and follow by combining the classification of unitary generic representations of $GL_n$ due to Tadic (and the bounds of Jacquet-Shalika) with the Weil conjectures. While temperedness is an analytic condition, genericity can be formulated modulo $\ell$.
\end{aremark}

Let us define the critical notion of being decomposed generic.

\begin{adefinition} Let $L$ be a $p$-adic field with residue field $\mathbb F_q$, $\ell\neq p$. An unramified representation
\[
\overline{\rho}: \mathrm{Gal}(\overline{L}/L)\to GL_n(\bar{\mathbb F}_\ell)
\]
is \emph{decomposed generic} if the eigenvalues $\{\lambda_1,\ldots,\lambda_n\}$ of $\overline{\rho}(\mathrm{Frob})$, where $\mathrm{Frob}\in \mathrm{Gal}(\overline{L}/L)$ is an arithmetic Frobenius, satisfy $\lambda_i/\lambda_j\not\in \{1,q\}$ for all $i\neq j$.
\end{adefinition}

The main consequence of this definition is that any characteristic $0$ lift of $\overline{\rho}$ is a direct sum of characters (i.e., ``\emph{decomposed}''), and the associated representation of $GL_n(L)$ under the local Langlands correspondence is a \emph{generic}\footnote{Recall that a generic representation is one which admits a Whittaker model - see, for example Section 2.3 of~\cite{kudla}.} principal series representation, cf.~Lemma~\ref{generic lifts}. 

\textbf{The rough idea.} Let us now explain the idea of our proof. In very rough terms, the idea is to work at a fixed prime $p\neq \ell$, and look at the projection from the Shimura variety $S$, which is a moduli space of abelian varieties (with extra structures), to the corresponding moduli space $M$ of $p$-divisible groups (with extra structures),
\[
\pi: S\to M\ .\footnote{This idea is also behind \cite{scholzeLLC}, and was also mentioned to one of us (P.S.) by R.~Kottwitz.}
\]
One could then analyze the cohomology of the Shimura variety in terms of a Leray spectral sequence. Note that the fibres of $\pi$ should be a moduli space of abelian varieties with a trivialization of their $p$-divisible group, which are essentially the Igusa varieties of~\cite{mantovan}, cf.~also~\cite{harris-taylor}. This means that one can compute the fibres of $R\pi_\ast \mathbb Z_\ell$ in terms of the cohomology of Igusa varieties. The alternating sum of the $\bar{\mathbb Q}_\ell$-cohomology groups has been analyzed in depth by Sug Woo Shin, \cite{shin-igusa}, \cite{shin-stable}.

An important property of the situation is that the Hecke operators away from $p$ act trivially on $M$, so the passage to the localization at $\mathfrak m$ can already be done \emph{on the sheaf} $R\pi_\ast \mathbb F_\ell$. This makes it possible to use geometry on $M$. More specifically, in the actual setup considered below, (the localization at $\mathfrak m$ of) $R\pi_\ast \mathbb F_\ell$ will turn out to be perverse (up to shift), and thus is concentrated in one degree on the largest stratum where it is nonzero. In that case, (the localization at $\mathfrak m$ of) $R\pi_\ast \mathbb Z_\ell$ will be concentrated in one degree \emph{and flat}. Thus, not much information is lost by passing to the alternating sum of the $\bar{\mathbb Q}_\ell$-cohomology groups. Specifically, we will use this argument inductively to show that $(R\pi_\ast \mathbb Z_\ell)_{\mathfrak m}$ is trivial on all strata except the $0$-dimensional stratum, which will then give the desired bound.

Unfortunately, the moduli space $M$ of $p$-divisible groups does not really exist, or at the very least has horrible properties. This makes it hard to execute this strategy in a naive way. In April 2011, \cite{scholze-letter-dat}, one of us realized (in the Harris-Taylor case) that there should be a Hodge-Tate period map, which would make a good substitute for $\pi$.\footnote{We learnt that L.~Fargues had also been aware of the Hodge-Tate period map in some form.} The idea here is that if $C/\mathbb Q_p$ is a complete algebraically closed nonarchimedean field with ring of integers $\cO_C$, then by \cite[Theorem B]{scholzeweinstein}, $p$-divisible groups over $\cO_C$ are classified by pairs $(T,W)$, where $T$ is a finite free $\mathbb Z_p$-module, and $W\subset T\otimes_{\mathbb Z_p} C$ is a subvectorspace, the \emph{Hodge-Tate filtration}. In particular, $p$-divisible groups with a trivialization of their Tate module are classified by a Grassmannian, at least on $(C,\cO_C)$-valued points. Now, even if the moduli space of $p$-divisible groups is not a nice object, one can replace it by this Grassmannian, which is manifestly a nice object. It turns out that with this modification, the argument outlined above works.

\textbf{The precise ideas.} Let us now be more precise. We work ad\`elically, so for any compact open subgroup $K\subset G(\mathbb A_f)$, we have the Shimura variety $S_K$, which is a quasiprojective scheme over the reflex field $E$. For the moment, we allow an arbitrary Shimura variety. Recall that these are associated with Shimura data, which consist of a reductive group $G/\mathbb Q$ and a $G(\mathbb R)$-conjugacy class $X$ of homomomorphisms $h: \mathrm{Res}_{\mathbb C/\mathbb R} \mathbb G_m\to G_{\mathbb R}$, subject to the usual axioms. Then
\[
S_K(\mathbb C) = G(\mathbb Q)\backslash (X\times G(\mathbb A_f)/K)\ .
\]

Associated with any $h$, one has a minuscule cocharacter $\mu = \mu_h: \mathbb G_m\to G_{\mathbb C}$. The reflex field $E\subset \mathbb C$ is the field of definition of the conjugacy class of $\mu$. With any cocharacter $\mu$, one can associate two opposite parabolics $P_\mu$ and $P_\mu^{\mathrm{std}}$, and there are two corresponding flag varieties $\mathrm{Fl}_{G,\mu}$ and $\mathrm{Fl}_{G,\mu}^{\mathrm{std}}$ over $E$, parametrizing parabolic subgroups in the given conjugacy class. The association $h\mapsto \mu_h\mapsto P_{\mu_h}^{\mathrm{std}}$ defines the (holomorphic) Borel embedding $X\hookrightarrow \mathrm{Fl}_{G,\mu}^{\mathrm{std}}(\mathbb C)$. There is also an antiholomorphic embedding $X\hookrightarrow \mathrm{Fl}_{G,\mu}(\mathbb C)$ defined using $P_{\mu_h}$.

Fix any prime $p$, and $\p|p$ a place of the reflex field $E$. Denote by $\cS_K$ the rigid-analytic variety, or rather the adic space, corresponding to $S_K\otimes_E E_\p$, and similarly for $\Fl_{G,\mu}$. Our first main result refines the theory behind the Hodge-Tate period map from~\cite{scholze}, which can be regarded as a $p$-adic version of the (antiholomorphic) Borel embedding.

\begin{atheorem} Assume that the Shimura datum is of Hodge type. Then for any sufficiently small compact open subgroup $K^p\subset G(\mathbb A_f^p)$, there is a perfectoid space $\cS_{K^p}$ over $E_\p$ such that
\[
\cS_{K^p}\sim \varprojlim_{K_p} \cS_{K_pK^p}\ .
\]
Moreover, there is a Hodge-Tate period map
\[
\pi_{HT}: \cS_{K^p}\to \Fl_{G,\mu}\ ,
\]
which agrees with the Hodge-Tate period map constructed in~\cite{scholze} for the Siegel case, and is functorial in the Shimura datum.
\end{atheorem}

Moreover, we prove a result saying that all semisimple automorphic vector bundles come via pullback along $\pi_{HT}$.

The idea here is to chase Hodge tensors through all constructions, which is possible by using Deligne's results that they are absolute Hodge, \cite{delignehodge}, (and also satisfy a compatibility under the $p$-adic comparison isomorphism,~\cite{blasius}), and the results on relative $p$-adic Hodge theory of \cite{scholzerigid}. The details appear in Section~\ref{refining HT}. As stated above, one should think of $\Fl_{G,\mu}$ as a (substitute for the) moduli space of $p$-divisible groups with extra structure and trivialized Tate module.

Next, we want to identify the fibres of $\pi_{HT}$ with Igusa varieties. First, we have to define a natural stratification on $\Fl_{G,\mu}$, which correspondends under $\pi_{HT}$ to the Newton stratification (pulled back from the special fibre through the specialization map). Recall that the Newton strata are parametrized by the finite subset $B(G,\mu^{-1})\subset B(G)$ of Kottwitz' set $B(G)$ of isocrystals with $G$-structure.

\begin{atheorem} Let $G$ be a reductive group over $\mathbb Q_p$, and $\mu$ a conjugacy class of minuscule cocharacters. There is a natural decomposition
\[
\Fl_{G,\mu} = \bigsqcup_{b\in B(G,\mu^{-1})} \Fl_{G,\mu}^b
\]
into locally closed subsets $\Fl_{G,\mu}^b$. The union
\[
\bigsqcup_{b\preceq b^\prime} \Fl_{G,\mu}^{b^\prime}
\]
is closed for all $b\in B(G,\mu^{-1})$; in particular, $\Fl_{G,\mu}^b$ is open when $b$ is the basic element of $B(G,\mu^{-1})$.
\end{atheorem}

Thus, the closure relations are exactly the opposite of the closure relations of the Newton stratification on the Shimura variety;\footnote{We note that we do not prove that the closure of a stratum is a union of strata, so the term ``closure relations'' is meant in a loose sense.} this change of closure relations is related to a subtle behaviour of $\pi_{HT}$ on certain higher-rank points of the adic space.

To give an idea of what the stratification looks like, we recall the example of the modular curve. In that case, the flag variety is just $\mathbb P^1$. The whole ordinary locus of the modular curve is contracted to $\mathbb P^1(\mathbb Q_p)$, and the Hodge-Tate period map just measures the position of the canonical subgroup on this locus. The supersingular locus is mapped onto Drinfeld's upper half-plane $\Omega^2 = \mathbb P^1\setminus \mathbb P^1(\mathbb Q_p)$ in a way best understood using the isomorphism between the Lubin-Tate and Drinfeld towers. Thus, in this case the relevant stratification of $\mathbb P^1$ is simply the stratification into $\mathbb P^1(\mathbb Q_p)$ and $\Omega^2$. We caution the reader that in general, the strata $\Fl_{G,\mu}^b$ are quite amorphous, and it happens that some nonempty strata have no classical points. The reason is that if $b$ is basic, $\Fl_{G,\mu}^b$ agrees with the \emph{admissible} locus in the sense of~\cite{rapoport-zink}, which does not admit a nice description, but whose classical points agree with the explicit \emph{weakly admissible} locus. If $G$ is a non-split inner form of $GL_5$ and $\mu$ corresponds to $(1,1,0,0,0)$, one can verify that all classical points of $\Fl_{G,\mu}$ are contained in the basic locus, while there are many other nonempty strata.

The proof of this theorem relies on certain recent advances in $p$-adic Hodge theory. First, to define the stratification on points, we make use of the classification of $G$-bundles on the Fargues-Fontaine curve; by a recent result of L.~Fargues, \cite{farguesGbun}, they are classified up to isomorphism by $B(G)$. Here, we construct a $G$-bundle on the Fargues-Fontaine curve by starting with the trivial $G$-bundle and modifying it at the infinite point of the Fargues-Fontaine curve. To construct the modification, we have to relate the flag variety $\Fl_{G,\mu}$ to a Schubert cell in a $B_{\mathrm{dR}}^+$-affine Grassmannian as studied in~\cite{scholzelectures}; however, for our applications, the theory of~\cite{scholzelectures} is not necessary.

Finally, to check the closure relations, we use recent results of Kedlaya and Liu, \cite{kedlayaliu}, on the semicontinuity of the Newton polygon for families of $\varphi$-modules over the Robba ring.

Now we can relate the fibres of $\pi_{HT}$ to Igusa varieties. From now on, we assume that the Shimura variety is of PEL type (of type A or C), and compact, with good reduction at $p$. Pick any $b\in B(G,\mu^{-1})$. Corresponding to $b$, we can find a $p$-divisible group $\mathbb X_b$ over $\bar{\mathbb F}_p$ equipped with certain extra endomorphism and polarization structures. We consider the following kind of Igusa varieties.

\begin{aprop} There is a perfect scheme $\mathrm{Ig}^b$ over $\bar{\mathbb F}_p$ which parametrizes abelian varieties $A$ with extra structures, equipped with an isomorphism $\rho: A[p^\infty]\cong \mathbb X_b$.

One can identify $\mathrm{Ig}^b$ with the perfection of the tower $\mathscr{I}^b_{\mathrm{Mant}} = \varprojlim_m \mathscr{I}^b_{\mathrm{Mant},m}$ of Igusa varieties constructed by Mantovan,~\cite{mantovan}.
\end{aprop}

In particular, the \'etale cohomology of $\mathrm{Ig}^b$ agrees with the \'etale cohomology of Igusa varieties.

Let us also mention the following proposition. Here, $\mathscr{S}_K^b\subset \mathscr{S}_K\otimes \bar{\mathbb F}_p$, $b\in B(G,\mu^{-1})$ denotes a Newton stratum of the natural integral model $\mathscr{S}_K$ of the Shimura variety $S_K$ at hyperspecial level.

\begin{aprop} Fix a geometric base point $\bar{x}\in \mathscr{S}_K^b$. There is a natural map
\[
\pi_1^{\mathrm{pro\acute{e}t}}(\mathscr{S}_K^b,\bar{x})\to J_b(\mathbb Q_p)\ ,
\]
corresponding to a $J_b(\mathbb Q_p)$-torsor over $\mathscr{S}_K^b$ which above any geometric point pa\-ra\-me\-tri\-zes quasi-isogenies between $A[p^\infty]$ and $\mathbb X_b$ respecting the extra structures.
\end{aprop}

\begin{aremark} Here, $\pi_1^{\mathrm{pro\acute{e}t}}$ is the pro-\'etale fundamental group introduced in~\cite{bhattscholze}. For normal schemes, it agrees with the usual profinite \'etale fundamental group of SGA1. However, Newton strata are usually not normal, and in fact the homomorphism to $J_b(\mathbb Q_p)$ often has noncompact image. For example, if $b$ is basic, then the image is a discrete cocompact subgroup of $J_b(\mathbb Q_p)$, related to the $p$-adic uniformization of the basic locus as in~\cite{rapoport-zink}. Thus, the formalism of $\pi_1^{\mathrm{pro\acute{e}t}}$ is crucial for this statement.

Restricted to the leaf $\mathscr{C}_b\subset \mathscr{S}_K^b$ (the set of points where $A[p^\infty]\cong \mathbb X_b$), the map $\pi_1(\mathscr{C}_b,\bar{x})\to J_b(\mathbb Q_p)$ takes values in a compact open subgroup of $J_b(\mathbb Q_p)$, and then corresponds to the tower of finite \'etale covers $\mathscr{I}^b_{\mathrm{Mant},m}\to \mathscr{C}_b$ considered by Mantovan.
\end{aremark}

There is a close relation between the fibres of $\pi_{HT}$ over points in $\Fl_{G,\mu}^b$ and the perfect schemes $\mathrm{Ig}^b$; note however that the former are of characteristic $0$ while the latter are of characteristic $p$. Roughly, one is the canonical lift of the other, except for issues of higher rank points. In any case, one gets the following cohomological consequence.

\begin{atheorem} Let $\overline{x}$ be any geometric point of $\Fl_{G,\mu}^b\subset \Fl_{G,\mu}$. For any $\ell\neq p$, there is an isomorphism
\[
(R\pi_{HT\ast} \mathbb Z/\ell^n\mathbb Z)_{\overline{x}}\cong R\Gamma(\mathrm{Ig}^b,\mathbb Z/\ell^n\mathbb Z)
\]
compatible with the Hecke action of $G(\mathbb A_f^p)$.
\end{atheorem}

We recall that the alternating sum of the $\bar{\mathbb Q}_\ell$-cohomology of Igusa varieties has been computed by Sug Woo Shin, \cite{shin-igusa}, \cite{shin-stable}. His results are presented in Section~\ref{Igusa varieties} and combined with the (twisted) trace formula.

The final ingredient necessary for the argument as outlined above is that $R\pi_{HT\ast} \mathbb F_\ell$ is perverse. Obviously, $R\pi_{HT\ast} \mathbb F_\ell$ should be constructible with respect to the stratification
\[
\Fl_{G,\mu} = \bigsqcup_{b\in B(G,\mu^{-1})} \Fl_{G,\mu}^b\ .
\]
However, as the strata are amorphous, it is technically difficult to define a notion of perverse sheaf in this setup. We content ourselves here with proving just what is necessary for us to conclude. Specifically, we will prove that the $K_p$-invariants of the nearby cycles of $R\pi_{HT\ast} \mathbb F_\ell$ are perverse, for any formal model of $\Fl_{G,\mu}$ and sufficiently small compact open subgroup $K_p\subset G(\mathbb Q_p)$. Choosing these formal models correctly will then make it possible to deduce that the cohomology is concentrated in one degree on the largest stratum where it is nonzero.

\begin{aremark} Heuristically, the reason that $R\pi_{HT\ast} \mathbb F_\ell$ is perverse is that $\pi_{HT}$ is simultaneously affine and partially proper (i.e., satisfies the valuative criterion of properness). In classical algebraic geometry, this would mean that $\pi_{HT}$ is finite, and pushforward along finite morphisms preserves perversity. In general, partially proper implies that $R\pi_{HT\ast} = R\pi_{HT!}$, so assuming that there is a Verdier duality which exchanges these two functors, one has to prove only one of the two support inequalities defining a perverse sheaf. This inequality is precisely Artin's bound on the cohomological dimension of affine morphisms.
\end{aremark}

\begin{aremark} The fact that the closure relations are reversed on the flag variety is critical to our strategy. Namely, our assumption on $\rho_{\mathfrak m}$ ensures that the cohomology should be ``maximally ordinary'', and this makes it reasonable to hope that everything comes from the $\mu$-ordinary locus. In our setup, the $\mu$-ordinary locus inside the flag variety is the closed stratum, and $0$-dimensional. In the naive moduli space of $p$-divisible groups, the $\mu$-ordinary locus would be open and dense (cf.~\cite{wedhorn}), and the inductive argument outlined above would stop at the first step.
\end{aremark}

\begin{aremark} Recently, L.~Fargues, \cite{fargues-geom-langlands}, has conjectured that to any local $L$-parameter, there is a corresponding perverse sheaf on the stack $\mathrm{Bun}_G$ of $G$-bundles on the Fargues-Fontaine curve, thus realizing the local Langlands correspondence as a geometric Langlands correspondence on the Fargues-Fontaine curve. We conjecture that the perverse sheaves $R\pi_{HT\ast} \bar{\mathbb Q}_\ell$ on $\Fl_{G,\mu}$ are related to these conjectural perverse sheaves on $\mathrm{Bun}_G$ via pullback along the natural map $\Fl_{G,\mu}\to \mathrm{Bun}_G$, by some form of local-global compatibility. In the Harris-Taylor case, one can be more explicit, and this was the subject of~\cite{scholze-letter-dat}.
\end{aremark}

\textbf{Acknowledgments.} First, we wish to thank J.-F.~Dat for many discussions on the ``geometrization'' of the results of \cite{scholzeLLC} using perverse sheaves on the moduli space of $p$-divisible groups. The rough strategy of this argument was first explained during a workshop on Barbados in May 2014, and we want to thank the organizers for the chance to present these ideas there. Moreover, we want to thank F.~Calegari, L.~Fargues, K.~Kedlaya and S. W.~Shin for many helpful discussions, and C.-L.~Chai for sending us a preliminary version of his work with F.~Oort on the ``internal Hom $p$-divisible group''. Part of this work was completed while both authors attended the special program on ``New geometric methods in number theory'' at MSRI in Fall 2014; we thank the institute and the organizers for the excellent working atmosphere. During that time, P.~Scholze held a Chancellor's Professorship at UC Berkeley. This work was done while P.~Scholze was a Clay Research Fellow.

\textbf{Notation and conventions.} A nonarchimedean field $K$ is a topological field whose topology is induced by a continuous rank $1$ valuation (which is necessarily uniquely determined, up to equivalence). We denote by $\cO_K\subset K$ the subring of powerbounded elements, which is the set of element of absolute value $\leq 1$ under the rank $1$ valuation. If, in the context of adic spaces, $K$ is equipped with a higher rank valuation, we denote by $K^+\subset \cO_K$ the open and bounded valuation subring of elements which are $\leq 1$ for this higher rank valuation.

We have tried our best to make our signs internally consistent, although the reader may often feel the presence of unnecessarily many minus signs. As regards slopes, we observe the following. We use covariant Dieudonn\'e theory. Usually, this sends $\mathbb Q_p/\mathbb Z_p$ to $(\mathbb Z_p,F=p)$ and $\mu_{p^\infty}$ to $(\mathbb Z_p,F=1)$; this is, however, not compatible with passage to higher tensors. The underlying reason is that in the duality between covariant and contravariant Dieudonn\'e theory, there is an extra Tate twist; for this reason, we divide the usual Frobenius by $p$, which gets rid of this Tate twist. Thus, the covariant Dieudonn\'e module for $\mu_{p^\infty}$ is $(\mathbb Z_p, F=p^{-1})$ in our setup, and one sees that the Frobenius operator does not preserve the lattice; in general, the associated Dieudonn\'e module will have slopes in $[-1,0]$. However, in the passage from isocrystals to vector bundles on the Fargues-Fontaine curve, the isocrystal $(\mathbb Q_p,F=p^{-1})$ is sent to the ample line bundle $\cO(1)$, so the slope changes sign once more, and in the end the usual slope of a $p$-divisible group agrees with the slope of the associated vector bundle on the Fargues-Fontaine curve. We feel that any confusion about signs on this part of the story is inherent to the mathematics involved.

As regards cocharacters (and associated filtrations), we have adopted what we think is the standard definition of the cocharacter $\mu = \mu_h$ corresponding to a Shimura datum $\{h\}$; for example, in the case of the modular curve, $\mu(t) = \mathrm{diag}(t,1)$ as a map $\mathbb G_m\to GL_2$. This has the advantage of being ``positive'', but the disadvantage that virtually everywhere we have to consider $\mu^{-1}$ instead; e.g., with this normalization, it is the set $B(G,\mu^{-1})$ which parametrizes the Newton strata. We feel that on this side of the story, it might be a good idea to exchange $\mu$ by $\mu^{-1}$, but we have stuck with the standard choice.
\newpage

\section{Refining the Hodge-Tate period map}\label{refining HT}
In this section, we work with a general Shimura variety of Hodge type and we prove that the Hodge-Tate period map from the corresponding perfectoid Shimura variety factors through the expected flag variety.

\subsection{Recollections on the Hodge-Tate period map}\label{recollections for Hodge type}
Let $(G,X)$ be a Shimura datum, where $X$ is a $G(\mathbb{R})$-conjugacy class of homomorphisms \[h:\mathrm{Res}_{\mathbb{C}/\mathbb{R}}\mathbb{G}_m\to G_{\mathbb{R}}.\] Recall that $(G,X)$ is a Shimura datum if it satisfies the following three conditions:
\begin{enumerate} 
\item Let $\mathfrak{g}$ denote the Lie algebra of $G(\mathbb{R})$. For any choice of $h\in X$, its composition with the adjoint action of $G(\mathbb{R})$ on $\mathfrak{g}$ determines a Hodge structure of type $(-1,1), (0,0), (1,-1)$ on $\mathfrak{g}$;\footnote{Here, an action of $\mathbb{C}^\ast$ on a $\mathbb{C}$-vector space is said to be of type $\{(p_i,q_i)\}$ if the vector space decomposes as a direct sum of subspaces, on which the action is through the cocharacters $z\mapsto z^{-p_i} \overline{z}^{-q_i}$.}
\item $h(i)$ is a Cartan involution of $G^{\mathrm{ad}}(\mathbb{R})$;
\item $G^\mathrm{ad}$ has no factor defined over $\mathbb{Q}$ whose real points form a compact group.
\end{enumerate}
The second condition implies that the stabilizer of any $h$ is compact modulo its center.

A choice of cocharacter $h$ determines, via base change to $\mathbb{C}$ and restriction to the first $\mathbb{G}_m$ factor, a Hodge cocharacter $\mu:\mathbb{G}_m\to G_{\mathbb{C}}$. This allows us to define two opposite parabolic subgroups:
\[
P_\mu^{\mathrm{std}}:=\{g\in G | \lim_{t\to \infty}\mathrm{ad}(\mu(t))g\ \mathrm{exists}\},\ \mathrm{ and}
\]
\[
P_\mu:=\{g\in G | \lim_{t\to 0}\mathrm{ad}(\mu(t))g\ \mathrm{exists}\}.
\]
The Hodge cocharacter $\mu$ defines a filtration on the category $\mathrm{Rep}_{\mathbb{C}}(G)$ of finite-dimensional representations of $G$ on $\mathbb{C}$-vector spaces. Indeed, the action of $\mathbb{G}_m$ on $\mathrm{Rep}_{\mathbb{C}}(G)$ via $\mu$ induces a grading on $\mathrm{Rep}_{\mathbb{C}}(G)$ and we take $\mathrm{Fil}^\bullet(\mu)$ to be the descending filtration on $\mathrm{Rep}_{\mathbb{C}}(G)$ associated with this grading. Concretely, $\mathrm{Fil}^p(\mu)$ is the direct sum of all subspaces of type $(p^\prime,q^\prime)$ with $p^\prime\geq p$. The parabolic $P_\mu^{\mathrm{std}}$ can equivalently be defined as the subgroup of $G$ stabilizing $\mathrm{Fil}^\bullet(\mu)$. The opposite parabolic $P_\mu$ can be defined as the stabilizer of the opposite, ascending filtration $\mathrm{Fil}_\bullet(\mu)$, where $\mathrm{Fil}_p(\mu)$ is the direct sum of all subspaces of type $(p^\prime,q^\prime)$ with $p^\prime\leq p$. Both conjugacy classes of parabolics are defined over the reflex field $E$ of the Shimura datum, which is the minimal field of definition of the conjugacy class $\{\mu\}$. Note that \[M_\mu:=\mathrm{Cent}_G(\mu)\] is the Levi component of both parabolics.

The two parabolics determine two flag varieties $\mathrm{Fl}_{G,\mu}^\mathrm{std}$ and  $\mathrm{Fl}_{G,\mu}$ over $E$ parametrizing parabolics in the given conjugacy class. The choice of a base point $h$ allows us to identify  $\mathrm{Fl}_{G,\mu}^\mathrm{std}(\mathbb C) \simeq G(\mathbb{C})/P^\mathrm{std}_\mu(\mathbb{C})$. There is an embedding \[\beta: X\hookrightarrow \mathrm{Fl}_{G,\mu}^\mathrm{std}(\mathbb C),\] called the Borel embedding, defined by $h\mapsto \mathrm{Fil}^\bullet(\mu_{h})$. It is easy to see that the Borel embedding is holomorphic. (There is also an embedding \[X\hookrightarrow \mathrm{Fl}_{G,\mu}(\mathbb C),\] which is antiholomorphic, defined in the natural way from the opposite filtration $\mathrm{Fil}_\bullet(\mu)$.)

Let $K\subset G(\mathbb{A}_f)$ be a compact open subgroup. Let
\[
S_K(\mathbb C):= G(\mathbb{Q})\backslash (X\times G(\mathbb{A}_f)/K)\ .
\]
When $K$ is neat (so, when $K$ is small enough), $S_K(\mathbb{C})$ has the structure of an algebraic variety over $\mathbb{C}$ (by a theorem of Baily-Borel) and has a model $S_K$ over the reflex field $E$~\cite{milne}.

\begin{example}\label{Siegel moduli space} Let $g\geq 1$ and let
\[
(V,\psi) = (\mathbb Q^{2g},\psi((a_i),(b_i)) = \sum_{i=1}^g (a_i b_{g+i} - a_{g+i} b_i))
\]
be the split symplectic space of dimension $2g$ over $\mathbb{Q}$. Let $\tilde G:=GSp(V,\psi)$. The hermitian symmetric domain $\tilde X$ is the Siegel double space. Fix the self-dual lattice $\Lambda=\mathbb{Z}^{2g}$ in $V$. For every $h\in \tilde X$, the Hodge structure induced by $\mu_h$ on $V$ has type $(-1,0),(0,-1)$ and $V^{(-1,0)}/\Lambda$ is an abelian variety over $\mathbb{C}$ of dimension $g$. 

For $\tilde K\subset \tilde G(\mathbb{A}_f)$ a neat compact open subgroup, the corresponding Shimura variety $\tilde S_{\tilde K}$ is the moduli space of principally polarized $g$-dimensional abelian varieties with level-$\tilde K$-structure. It has a model over the reflex field $\mathbb{Q}$. It carries a universal abelian variety $\cA$ and a natural ample line bundle $\omega$ given by the determinant of the sheaf of invariant differentials on $\cA$. The flag variety $\mathrm{Fl}_{\tilde{G},\tilde{\mu}}$ parametrizes totally isotropic subspaces $W\subset V$.
\end{example}

We say that a Shimura datum is of \emph{Hodge type} if it admits a closed embedding $(G,X)\hookrightarrow (\tilde G,\tilde X)$, for some choice of Siegel data $(\tilde G,\tilde X)$. A consequence of this is that the associated Shimura variety $S_K$ (for some neat level $K$) carries a universal abelian variety, which is the restriction of the universal abelian variety over $\tilde S_{\tilde K}$. One can regard $S_K$ as a moduli space for abelian varieties equipped with certain Hodge tensors, cf.~below.

Let $(G,X)$ be a Shimura datum of Hodge type and let $(\tilde{G},\tilde{X})$ be a choice of Siegel data, for which there exists an embedding $(G,X)\hookrightarrow (\tilde{G},\tilde{X})$. Fixing such an embedding gives rise to closed embeddings $\mathrm{Fl}_{G,\mu}^{(\mathrm{std})}\hookrightarrow \mathrm{Fl}_{\tilde{G},\tilde{\mu}}^{(\mathrm{std})}$. By \cite[Proposition 1.15]{DeligneVarSh}, there exists some compact open subgroup $\tilde{K}\subset\tilde{G}(\mathbb{A}_f)$ with $K=\tilde{K}\cap G(\mathbb A_f)$ such that there is a closed embedding of the corresponding Shimura varieties over $E$, \[S_K\hookrightarrow \tilde{S}_{\tilde{K}}\otimes_{\mathbb{Q}}E.\]

Let $p$ be a prime number. We will consider compact open subgroups of the forms $K=K^p\times K_p\subset G(\mathbb{A}_f^p)\times G(\mathbb{Q}_p)$, where $K^p$ and $K_p$ are compact open. Fix a place $\mathfrak p$ of $E$ above $p$. Let $\Fl_{G,\mu}$ be the adic space associated with $\mathrm{Fl}_{G,\mu}\otimes_E E_{\mathfrak p}$. The following is part of Theorem IV.1.1 of~\cite{scholze}.\footnote{The setup is slightly different, but the proof works verbatim.}
\begin{thm}\label{htperiodmap}
\begin{enumerate}
\item For any sufficiently small tame level $K^p\subset G(\mathbb{A}_f^p)$, there exists a perfectoid space $\mathcal{S}_{K^p}$ over $E_{\mathfrak p}$, such that \[\mathcal{S}_{K^p}\sim \varprojlim_{K_p} (S_{K^pK_p}\otimes_E E_{\mathfrak p})^\mathrm{ad}.\]
\item There exists a $G(\mathbb{Q}_p)$-equivariant Hodge-Tate period map \[\pi_{HT}:\mathcal{S}_{K^p}\to \Fl_{\tilde{G},\tilde{\mu}}.\]
\item The map $\pi_{HT}$ is equivariant with respect to the natural Hecke action of $G(\mathbb{A}_f^p)$ on the inverse system of $\mathcal{S}_{K^p}$ and the trivial action of $G(\mathbb{A}_f^p)$ on $\Fl_{\tilde{G},\tilde{\mu}}$.
\end{enumerate}
\end{thm}

Recall that the Hodge-Tate period map~\cite{scholze, scholzeweinstein} has the following description on points: for $A/C$ an abelian variety of dimension $g$, the Tate module of $A$ admits the Hodge-Tate decomposition: \[0\to (\mathrm{Lie}\ A)(1)\to T_pA\otimes _{\mathbb{Z}_p}C\to (\mathrm{Lie}\ A^\vee)^\vee\to 0.\] A point $x\in \cS_{G, K^p}(C,C^+)$ corresponding to $A/C$ together with a symplectic isomorphism $T_pA \toisom \mathbb{Z}_p^{2g}$ (and extra structures) is mapped to the point $\pi_{HT}(x)\in\Fl_{\tilde G,\tilde{\mu}}(C,C^+)$ corresponding to the Hodge-Tate filtration $\mathrm{Lie}(A)\subset C^{2g}$.

We note that one can think of the Hodge-Tate period map as a $p$-adic analogue of the Borel embedding. The goal of this section is to prove the following theorem. 

\begin{thm}\label{refinedht}
\begin{enumerate}
\item The Hodge-Tate period map for $\mathcal{S}_{K^p}$ factors through $\Fl_{G,\mu}$ and the resulting map \[\pi_{HT}:\mathcal{S}_{K^p}\to\Fl_{G,\mu}\] is independent of the choice of embedding of Shimura data. 
\item Fix some $\mu$ in the given conjugacy class, defined over a finite extension of $E$. The tensor functor from $\mathrm{Rep}\ M_\mu$ to $G(\mathbb{Q}_p)$-equivariant vector bundles on $\mathcal{S}_{K^p}$ given as the composition
\[\begin{aligned}
f_p: \mathrm{Rep}\ M_\mu\hookrightarrow \mathrm{Rep}\ P_\mu&\longrightarrow \{G(\mathbb{Q}_p)\mathrm{-equivariant\ vector\ bundles\ on\ }\Fl_{G,\mu}\}\\
&\buildrel{\pi_{HT}^\ast}\over\longrightarrow\{G(\mathbb{Q}_p)\mathrm{-equivariant\ vector\ bundles\ on\ }\mathcal{S}_{K^p}\}
\end{aligned}\]
is isomorphic to the tensor functor
\[\begin{aligned}
f_\infty: \mathrm{Rep}\ M_\mu\hookrightarrow \mathrm{Rep}\ P_\mu^\mathrm{std}&\longrightarrow \{\mathrm{automorphic\ vector\ bundles\ on\ }S_{K}\}\\
&\longrightarrow \{G(\mathbb{Q}_p)\mathrm{-equivariant\ vector\ bundles\ on\ }\mathcal{S}_{K^p}\}\ .
\end{aligned}\]
The isomorphism is independent of the choice of Siegel embedding, and equivariant for the Hecke action of $G(\mathbb A_f^p)$.
\end{enumerate}
\end{thm}

\begin{remark} One may avoid choosing $\mu$ by replacing $\mathrm{Rep}\ M_\mu$ with the category of $G$-equivariant vector bundles on the space of cocharacters in the conjugacy class of $\mu$. Note that after fixing any $\mu$, this space identifies with $G/M_\mu$, and so $G$-equivariant vector bundles are identified with representations of $M_\mu$. We leave it as an exercise to the reader to reformulate the theorem and its proof in this more canonical language.
\end{remark}

Let us first recall how the tensor functor $f_\infty$ is defined: any representation $\xi$ of $M_\mu$ determines a representation of $P_\mu^\mathrm{std}$ by making the unipotent radical act trivially. Now, starting with a representation of $P_\mu^\mathrm{std}$, we can define an automorphic vector bundle on $S_{K}$ as in Section III of ~\cite{milne}, provided that the level $K$ is sufficiently small: first, there is an equivalence of categories \[\xi \mapsto \cW(\xi)\] between $\mathrm{Rep}_{\mathbb{C}}(P_\mu^\mathrm{std})$ and the category of $G_\mathbb{C}$-equivariant vector bundles on $\mathrm{Fl}^\mathrm{std}_{G,\mu}$ (the functor in one direction is taking the stalk of the vector bundle above the point corresponding to $\mu$). Then restriction along the image of the Borel embedding gives a $G(\mathbb{R})$-equivariant vector bundle on $X$. Passing to the double quotient defining the Shimura variety \[S_K(\mathbb{C})=G(\mathbb{Q})\setminus (X\times G(\mathbb{A}_f)/K)\] over $\mathbb{C}$ defines the automorphic vector bundle \[\cV(\xi):= G(\mathbb{Q})\setminus(\cW(\xi) \times G(\mathbb{A}_f)/K).\] The automorphic vector bundles $\cV(\xi)$ are algebraic and, when the representation $\xi$ is defined over a finite extension $E^\prime$ of $E$, $\cV(\xi)$ is also defined over $E^\prime$.

\begin{remark}
Proving that the automorphic vector bundles descend to the reflex field makes use of an intermediate algebraic object between $S_{K}$ and $\mathrm{Fl}_{G,\mu}^\mathrm{std}$, called the \emph{standard principal bundle} (see Section IV of~\cite{milne}), which is a $G$-torsor over $S_{K}$. See the proof of Lemma~\ref{spb} for more details.  
\end{remark}

In particular, $f_p$ is defined in an analogous way to $f_\infty$, except that it uses the Hodge-Tate period map in place of the Borel embedding. The appearance of the opposite parabolic $P_\mu$ in this picture forces one to look only at representations inflated from the common Levi $M_\mu$.

\subsection{The $p$-adic-de Rham comparison isomorphism}\label{hodge-tate filtration} For an abelian variety over $C$, its image under the Hodge-Tate period map is determined by the Hodge-Tate filtration on $H^1_\et(A,\mathbb{Q}_p)\otimes_{\mathbb{Q}_p}C$. The Hodge-Tate period map as a map of adic spaces $\cS_{K^p}\to \Fl_{\tilde G,\tilde{\mu}}$ is defined via a relative version of the Hodge-Tate filtration, which is a filtration on the local system given by the $p$-adic \'etale cohomology of the universal abelian variety over $\cS_{K}$, tensored with the completed structure sheaf of the base. In fact, the Hodge-Tate filtration is defined more generally: see Section 3 of~\cite{scholzesurvey} for a construction of the Hodge-Tate filtration for a proper smooth rigid-analytic variety over a geometric point.

As we will need to work with higher tensors in our analysis of Hodge type Shimura varieties, our goal in this section is to give a construction of the relative Hodge-Tate filtration in the case of a proper smooth morphism $\pi: X \to S$ of smooth adic spaces over $\mathrm{Spa}(K,\cO_K)$, where $K$ is a complete discretely valued field of characteristic $0$ with perfect residue field $k$ of characteristic $p$. This will be done in a way that also clarifies its relationship to the relative $p$-adic-de Rham comparison isomorphism. 

The following sheaves on $X_{\mathrm{pro\acute{e}t}}$ are defined in~\cite{scholzerigid}: the completed structure sheaf $\hat{\cO}_X$, the tilted completed structure sheaf $\hat{\cO}_{X^\flat}$, the relative period sheaves $\mathbb{B}^+_{\mathrm{dR},X}$ and $\mathbb{B}_{\mathrm{dR},X}$ as well as the structural de Rham sheaves $\cO\mathbb{B}^+_{\mathrm{dR},X}$ and  $\cO\mathbb{B}_{\mathrm{dR},X}$. We recall some of their definitions: the tilted integral structure sheaf $\hat{\cO}^+_{X^\flat}$ is the (inverse) perfection of $\hat{\cO}^+_X/p$ (i.e., the inverse limit of $\hat{\cO}^+_X/p$ with respect to the Frobenius morphism).

\begin{defn}\label{derham period sheaf}\begin{enumerate} 
\item The relative period sheaf $\mathbb{B}^+_{\mathrm{dR},X}$ is the completion of $W(\hat{\cO}^+_{X^\flat})[1/p]$ along the kernel of the natural map $\theta: W(\hat{\cO}^+_{X^\flat})[1/p] \to \hat{\cO}_{X}$. 
\item The relative period sheaf $\mathbb{B}_{\mathrm{dR},X}$ is $\mathbb{B}^+_{\mathrm{dR},X}[\xi^{-1}]$, where $\xi$ is any element that generates the kernel of $\theta$. 
\end{enumerate}
\end{defn}

\noindent Lemma 6.3 of~\cite{scholzerigid} shows that $\xi$ exists pro\'etale locally on $X$, is not a zero divisor and is unique up to a unit. Therefore, the sheaf $\mathbb{B}_{\mathrm{dR},X}$ is well-defined. When $X=\mathrm{Spa}(C,\cO_C)$, we recover Fontaine's period ring $B_{\mathrm{dR},C}$ from this construction. By construction, the relative period sheaf $\mathbb{B}_{\mathrm{dR},X}$ is equipped with a natural filtration
$\mathrm{Fil}^i \mathbb{B}_{\mathrm{dR},X}= \xi^i \mathbb{B}^+_{\mathrm{dR},X}$, with $\mathrm{gr}^0 \mathbb{B}_{\mathrm{dR},X} = \hat{\cO}_X$.

Recall that $k$ is the residue field of $K$. Then $\cO_X\otimes_{W(k)} W(\hat{\cO}_{X^\flat})$ also admits a map $\theta: \cO_X\otimes_{W(k)}W(\hat{\cO}_{X^\flat}) \to \hat{\cO}_X$. Then $\cO\mathbb{B}^{+}_{\mathrm{dR}}$ is defined as the completion of $\cO_X\otimes_{W(k)}W(\hat{\cO}_{X^\flat})$ along $\ker \theta$ and $\cO\mathbb{B}_\mathrm{dR}:=\cO\mathbb{B}^+_\mathrm{dR}[\xi^{-1}]$ as above. The structural de Rham sheaves $\cO\mathbb{B}^{(+)}_{\mathrm{dR}}$ are equipped with filtrations and connections \[\nabla: \cO\mathbb{B}^{(+)}_{\mathrm{dR},X}\to \cO\mathbb{B}^{(+)}_{\mathrm{dR},X}\otimes_{\cO_X}\Omega^1_X.\] We have an identification $(\cO\mathbb{B}^{(+)}_{\mathrm{dR}})^{\nabla=0}= \mathbb{B}^{(+)}_{\mathrm{dR}}$.

We now recall the relative $p$-adic-de Rham comparison isomorphism for a proper smooth morphism $\pi: X\to S$ of smooth adic spaces over $K$.

\begin{thm}[{\cite[Theorem 8.8]{scholzerigid}}] Assume that $R^i \pi_\ast \mathbb F_p$ is locally free on $S_{\mathrm{pro\acute{e}t}}$ for all $i\geq 0$.\footnote{This condition is verified if $\pi$ is algebraizable, and has been announced in general by Gabber. Another proof will appear in a forthcoming version of \cite{scholzelectures}; the idea is to use (the new version of) pro-\'etale descent to reduce to the case where $S$ is w-strictly local, in which case one can redo the finiteness argument over a geometric point. With $\mathbb Q_p$-coefficients, it has also been announced by Kedlaya-Liu.} Then, for all $i\geq 0$, $R^i \pi_\ast \hat{\mathbb Z}_p$ is de Rham in the sense of \cite[Definition 7.5]{scholzerigid}, with associated filtered module with integrable connection given by $R^i \pi_{\mathrm{dR}\ast} \cO_X$ (with its Hodge filtration, and Gauss-Manin connection). In particular, there is an isomorphism
\[
R^i\pi_{*}\hat{\mathbb Z}_{p,X}\otimes_{\hat{\mathbb{Z}}_{p,S}}\cO\mathbb{B}_{\mathrm{dR},S}\simeq R^i\pi_{\mathrm{dR}\ast}\cO_X\otimes_{\cO_S}\cO\mathbb{B}_{\mathrm{dR},S}
\]
of sheaves on $S_{\mathrm{pro\acute{e}t}}$, compatible with filtrations and connections.
\end{thm}

Moreover, we need to recall the two different $\mathbb{B}_{\mathrm{dR}}^+$-local systems associated with $R^i \pi_\ast \hat{\mathbb{Z}}_p$. The first one, which is closely related to \'etale cohomology, is given by
\[
\mathbb{M} = R^i \pi_\ast \hat{\mathbb{Z}}_{p,X}\otimes_{\hat{\mathbb{Z}}_{p,S}}\mathbb{B}_{\mathrm{dR},S}^+ \cong R^i \pi_\ast \mathbb{B}_{\mathrm{dR},X}^+\ .
\]
The other one, which is closely related to de Rham cohomology, is given by
\[
\mathbb{M}_0 = (R^i \pi_{\mathrm{dR}\ast}\cO_X\otimes_{\cO_S}\cO\mathbb{B}_{\mathrm{dR},S}^+)^{\nabla = 0}\ .
\]
Note that the definition of $\mathbb{M}_0$ did not make use of the Hodge filtration. The relation between these two lattices is given by the following proposition, which reformulates the condition of being associated.

\begin{prop}[{\cite[Proposition 7.9]{scholzerigid}}]\label{TwoBdRLattices} There is a canonical isomorphism
\[
\mathbb{M}\otimes_{\mathbb{B}_{\mathrm{dR},S}^+} \mathbb{B}_{\mathrm{dR},S}\cong \mathbb{M}_0\otimes_{\mathbb{B}_{\mathrm{dR},S}^+} \mathbb{B}_{\mathrm{dR},S}\ .
\]
Moreover, for any $j\in \mathbb Z$, one has an identification
\[\begin{aligned}
(\mathbb{M}\cap \mathrm{Fil}^j \mathbb{M}_0) / (\mathbb{M}\cap \mathrm{Fil}^{j+1} \mathbb{M}_0) &= (\mathrm{Fil}^{-j} R^i \pi_{\mathrm{dR}\ast}\cO_X)\otimes_{\cO_S} \hat{\cO}_S(j)\\
\subset \mathrm{gr}^j \mathbb{M}_0 &= R^i \pi_{\mathrm{dR}\ast}\cO_X\otimes_{\cO_S} \hat{\cO}_S(j)\ .
\end{aligned}\]
In particular, $\mathbb{M}_0\subset \mathbb{M}$.
\end{prop}

In particular, we get an ascending filtration on
\[
\mathrm{gr}^0 \mathbb{M} = R^i \pi_\ast \hat{\mathbb Z}_{p,X}\otimes_{\hat{\mathbb Z}_{p,S}} \hat{\cO}_S
\]
given by
\[
\mathrm{Fil}_{-j}(R^i \pi_\ast \hat{\mathbb Z}_{p,X}\otimes_{\hat{\mathbb Z}_{p,S}} \hat{\cO}_S) = (\mathbb{M}\cap \mathrm{Fil}^j \mathbb{M}_0) / (\mathrm{Fil}^1 \mathbb{M}\cap \mathrm{Fil}^j \mathbb{M}_0)\ .
\]
Here, $\mathrm{Fil}_{-1}=0$, and $\mathrm{Fil}_i$ is everything. We call this filtration the \emph{relative Hodge-Tate filtration}.

\begin{cor}\label{gradedHodgeTate} For all $j\geq 0$, there are canonical isomorphisms
\[
\mathrm{gr}_j(R^i \pi_\ast \hat{\mathbb Z}_{p,X}\otimes_{\hat{\mathbb Z}_{p,S}} \hat{\cO}_S)\cong (\mathrm{gr}^j R^i \pi_{\mathrm{dR}\ast}\cO_X)\otimes_{\cO_S} \hat{\cO}_S(-j)\ .
\]
\end{cor}

\begin{proof} This is immediate from Proposition \ref{TwoBdRLattices} by passing to gradeds.
\end{proof}

In particular, one sees that
\[
\mathrm{Fil}_0 (R^i \pi_\ast \hat{\mathbb Z}_{p,X}\otimes_{\hat{\mathbb Z}_{p,S}} \hat{\cO}_S) = R^i \pi_\ast \cO_X\otimes_{\cO_S} \hat{\cO}_S\ .
\]
This map can be identified.

\begin{prop}\label{firstgradedHodgeTate} The first filtration step $\mathrm{Fil}_0$ of the relative Hodge-Tate filtration is given by the natural map
\[
R^i \pi_\ast \cO_X\otimes_{\cO_S} \hat{\cO}_S\to R^i \pi_\ast \hat{\cO}_X\cong R^i \pi_\ast {\hat{\mathbb Z}_{p,X}} \otimes_{\hat{\mathbb Z}_{p,S}} \hat{\cO}_S,
\]
which is injective.
\end{prop}

We note that in \cite{scholze}, only the first step of the Hodge-Tate filtration was used (for $i=1$), and it was defined as the natural map \[R^i \pi_\ast \cO_X\otimes_{\cO_S} \hat{\cO}_S\to R^i \pi_\ast \hat{\cO}_X.\]

\begin{proof} We have to identify the image of $\mathbb{M}_0\to \mathrm{gr}^0 \mathbb{M}$. This can be done after $\otimes_{\mathbb{B}_{\mathrm{dR},S}^+} \cO\mathbb{B}_{\mathrm{dR},S}^+$, as this operation preserves $\mathrm{gr}^0$. Now note that
\[
\mathbb{M}_0\otimes_{\mathbb{B}_{\mathrm{dR},S}^+} \cO\mathbb{B}_{\mathrm{dR},S}^+ = R^i \pi_{\mathrm{dR}\ast} \cO_X\otimes_{\cO_S} \cO\mathbb{B}_{\mathrm{dR},S}^+\ ,
\]
and
\[
\mathbb{M}\otimes_{\mathbb{B}_{\mathrm{dR},S}^+} \cO\mathbb{B}_{\mathrm{dR},S}^+ = R^i \pi_{\mathrm{dR}\ast} \cO\mathbb{B}_{\mathrm{dR},X}^+\ ,
\]
by the relative Poincar\'e lemma. The map $\mathbb{M}_0\to \mathbb{M}$ is induced by the natural inclusion $\cO_X\to \cO\mathbb{B}_{\mathrm{dR},X}^+$, which commutes with the natural connections.

Passing to $\mathrm{gr}^0$ on the side of $\mathbb{M}$ replaces the relative de Rham complex of $\cO\mathbb{B}_{\mathrm{dR},X}^+$ with just $\hat{\cO}_X$, as the differentials sit in positive degrees. We note that the composite $\cO_X\to \cO\mathbb{B}_{\mathrm{dR},X}^+\to \hat{\cO}_X$ is the natural inclusion, as
\[
\hat{\cO}_X = \mathrm{gr}^0 \cO\mathbb{B}_{\mathrm{dR},X}^+ = (\cO_X\otimes_{W(k)} W(\hat{\cO}_{X^\flat}^+)) / (\ker \theta)\ ,
\]
using the map $\theta: \cO_X\otimes_{W(k)} W(\hat{\cO}_{X^\flat}^+)\to \hat{\cO}_X$, which is $\cO_X$-linear. It follows that the map
\[
\mathbb{M}_0\otimes_{\mathbb{B}_{\mathrm{dR},S}^+} \cO\mathbb{B}_{\mathrm{dR},S}^+\to \mathrm{gr}^0 \mathbb{M}
\]
agrees with the map
\[
R^i f_{\mathrm{dR}\ast} \cO_X\otimes_{\cO_S} \cO\mathbb{B}_{\mathrm{dR},S}^+\to R^i f_\ast \hat{\cO}_X
\]
which projects $R^i f_{\mathrm{dR}\ast} \cO_X\to R^i f_\ast \cO_X\to R^i f_\ast \hat{\cO}_X$, and then extends $\cO\mathbb{B}_{\mathrm{dR},S}^+$-linearly. Thus, its image is given by the image of $R^i f_\ast \cO_X\otimes_{\cO_S} \hat{\cO}_S\to R^i f_\ast \hat{\cO}_X$. By the identification of the graded pieces of the relative Hodge-Tate filtration, this map has to be injective, giving the result.
\end{proof}

\subsection{Hodge cycles and torsors}\label{comparisons}

Let \[(G,X)\hookrightarrow (\tilde G,\tilde X)\] be an embedding of Shimura data, as in the previous section, where $\tilde G=GSp(V,\psi)$. Let \[V^{\otimes}:=\bigoplus_{r,s\in \mathbb{N}} V^{\otimes r}\otimes (V^\vee)^{\otimes s}.\] By Proposition 3.1 of~\cite{delignehodge}, the subgroup $G$ of $\tilde G$ is the pointwise stabilizer of a finite collection of tensors $(s_\alpha)\subset V^\otimes$. 

As above, the embedding of Shimura data determines an embedding of Shimura varieties defined over $E$: \[S_K\hookrightarrow \tilde S_{\tilde K}\otimes_{\mathbb{Q}}E.\] Let $\cA$ be the abelian scheme over $S_K$ obtained by pulling back the universal abelian scheme over the Siegel moduli space. Let $\pi:\cA\to S_K$ be the projection. The first relative Betti homology of $\cA$, i.e. the dual of $R^1\pi^\mathrm{an}_*\mathbb{Q}$, defines a local system of $\mathbb{Q}$-vector spaces $\cV_B$ on $S_K(\mathbb{C})$. Since the Betti cohomology of an abelian variety parametrized by $X\times G(\mathbb A_f)/K$ gets identified with $V$, $\cV_B$ can be identified with the local system of $\mathbb{Q}$-vector spaces over $S_K(\mathbb{C})$ given by the $G(\mathbb{Q})$-representation $V$ and the $G(\mathbb Q)$-torsor
\[
X\times G(\mathbb A_f)/K\to G(\mathbb Q)\backslash (X\times G(\mathbb A_f)/K) = S_K(\mathbb{C})\ .
\]
Corresponding to the $G(\mathbb{Q})$-invariant tensors $(s_{\alpha})$, we get global sections $(s_{\alpha,B})\subset \cV^\otimes_B$. Moreover, these are Hodge tensors for the Hodge structure on Betti homology, since they are $G$-invariant, and in particular invariant under the action of any $h\in X$.

\begin{lemma} The $G(\mathbb Q)$-torsor
\[
X\times G(\mathbb A_f)/K\to G(\mathbb Q)\backslash (X\times G(\mathbb A_f)/K) = S_K(\mathbb{C})
\]
can be identified with the $G(\mathbb Q)$-torsor sending any $U\subset S_K(\mathbb{C})$ to
\[
\{\beta: V\times U\cong \cV_B|_U\mid \beta(s_\alpha) = s_{\alpha,B}\}\ .
\]
\end{lemma}

\begin{proof} This follows from the fact that $G\subset \GL(V)$ is the closed subgroup which is the stabilizer of the $s_\alpha$.
\end{proof}

Now assume that $(G,X)\hookrightarrow (\tilde G^\prime, \tilde X^\prime)$ is a second symplectic embedding, where $\tilde G^\prime = GSp(V^\prime,\psi^\prime)$. Like for any representation of $G$, there is a $G$-invariant idempotent $e\in V^\otimes$ such that $V^\prime = e V^\otimes$. Using $e$, any $G$-invariant tensor $s_\alpha^\prime\in (V^\prime)^\otimes$ can be transferred to a $G$-invariant tensor in $V^\otimes$. Moreover, one also has an identification
\[
\cV_B^\prime = e \cV_B^\otimes\ ,
\]
compatibly with their natural Hodge structures. We will generally assume that $e$ belongs to the family $s_\alpha$, by adjoining it if necessary.

Let $\cV_{\mathrm{dR}}:=(R^1\pi_{\mathrm{dR}\ast}\cO_{\cA})^\vee$ be the first relative de Rham homology of $\cA$. This is a vector bundle over $S_K$ equipped with an integrable connection $\nabla$. The base change to $\mathbb{C}$ can be defined directly: We have to specify an analytic vector bundle $\cV^{\mathrm{an}}_{\mathrm{dR},\mathbb{C}}$ over $S_K(\mathbb{C})$, which corresponds to the algebraic vector bundle $\cV_{\mathrm{dR},\mathbb{C}}$. (Here, we make use of the equivalence of categories between algebraic vector bundles equipped with a flat connection with regular singular points and analytic vector bundles equipped with a flat connection~\cite{delignevb}.) Then the relative de Rham comparison isomorphism over $\mathbb{C}$ gives rise to an isomorphism
\[
\cV^{\mathrm{an}}_{\mathrm{dR},\mathbb{C}}\cong \cV_B\otimes_{\mathbb Q} \cO_{S_K(\mathbb{C})}\ ,
\]
compatible with the connection.

In particular, the global sections $(s_{\alpha,B})\subset \cV_B^\otimes$ give rise to horizontal global sections $(s_{\alpha,\mathrm{dR}})\subset (\cV^{\mathrm{an}}_{\mathrm{dR},\mathbb{C}})^\otimes$, which are necessarily algebraic, i.e.
\[
(s_{\alpha,\mathrm{dR}})\subset \cV_{\mathrm{dR},\mathbb{C}}^\otimes\ .
\]
The following lemma appears in work of Kisin~\cite{kisinintegral}, based on Deligne's result that Hodge cycles on abelian varieties are absolute Hodge, \cite{delignehodge}.

\begin{lemma}\label{reflex field} The tensors $s_{\alpha, \mathrm{dR}}$ in $\cV^\otimes_{\mathrm{dR},\mathbb{C}}$ are defined over $E$.
\end{lemma}

\begin{proof} We sketch Kisin's proof here. We work with each connected component of $S_K$ individually. Let $x$ be the generic point of one such component, with function field $\kappa$ (containing $E$) and choose a complex embedding of its algebraic closure $\bar \kappa \hookrightarrow \mathbb{C}$. Let $\cA_x$ be the corresponding abelian variety over $\kappa$. Let $s_{\alpha,B,x}$ be the fiber of $s_{\alpha,B}$ over $x$. Let $s_{\alpha,\mathrm{dR},x}\in H^1_{\mathrm{dR}}(\cA_{x})^\otimes\otimes_{\kappa}\mathbb{C}$ be the image of $s_{\alpha,B,x}$ under the de Rham comparison isomorphism (this is also the fiber of $s_{\alpha,\mathrm{dR}}$ over $x$.) Let $s_{\alpha,p,x}\in H^1_\et(\cA_{x,\bar{\kappa}},\mathbb{Q}_p)^\otimes$ be the image of $s_{\alpha,B,x}$ under the comparison between Betti and $p$-adic \'etale cohomology. 

Note that by definition $(s_{\alpha,\mathrm{dR},x},s_{\alpha,p,x})$ is a Hodge cycle. By Deligne~\cite{delignehodge}, it is an absolute Hodge cycle. This means that $s_{\alpha,\mathrm{dR},x}$ is defined over $\bar \kappa$ and it remains to show that the action of $\mathrm{Gal}(\bar \kappa/\kappa)$ on it is trivial. For this, it is enough to check that the $\mathrm{Gal}(\bar \kappa/\kappa)$-action on $s_{\alpha,p,x}$ is trivial, since a Hodge cycle is determined by either its de Rham or \'etale component.

For this latter statement, consider the $\tilde K_p$-torsor over Siegel moduli space given by $\varprojlim_{\tilde K'_p}\tilde S_{\tilde K^p \tilde K'_p}$, where $\tilde K'_p$ runs over open compact subgroups of $\tilde K_p$. Fixing a $\bar \kappa$-point $\tilde{x}$ of this tower above $x$, the $\mathrm{Gal}(\bar \kappa/\kappa)$-action on $H^1_\et(\cA_{x,\bar{\kappa}},\mathbb{Q}_p)$ is induced by the map $\mathrm{Gal}(\bar \kappa/\kappa)\to \tilde K_p$ describing the action on $\tilde{x}$.  There is an analogous $K_p$-torsor over $S_K$ defined by $\varprojlim_{K'_p}S_{K^pK'_p}$. This fits into a commutative, $K_p$-equivariant diagram \[\xymatrix{S_{K^p}\ar[d]\ar[r]&\tilde S_{\tilde K^p}\ar[d]\\S_{K}\ar[r]&\tilde S_{\tilde K}}.\] Taking for $\tilde{x}$ a lift to $S_{K^p}$, we see that the action of $\mathrm{Gal}(\bar \kappa/\kappa)$ on $H^1_\et(\cA_{x,\bar{\kappa}},\mathbb{Q}_p)$ factors through a map \[\mathrm{Gal}(\bar \kappa/\kappa)\to  K_p\subset G(\mathbb{Q}_p).\] Since the tensors $s_{\alpha,p,x}$ are $G(\mathbb{Q}_p)$-invariant, the Galois action on these tensors is trivial as well.
\end{proof}

\begin{remark}\label{reflexfieldsecondemb} If $(G,X)\hookrightarrow (\tilde G^\prime,\tilde X^\prime)$ is a second symplectic embedding with $\tilde G^\prime = GSp(V^\prime,\psi^\prime)$, and $e \in V^\otimes$ is an idempotent with $V^\prime = e V^\otimes$, as above, then applying Lemma \ref{reflex field} to the embedding $G\hookrightarrow GSp(V\oplus V^\prime,\psi\oplus \psi^\prime)$, one sees that the isomorphism
\[
\cV_{\mathrm{dR},\mathbb{C}}^\prime\cong e_{\mathrm{dR}} \cV_{\mathrm{dR},\mathbb{C}}^\otimes
\]
is defined over $E$.
\end{remark}

There is also a $\mathbb{Q}_p$-local system $\cV_p$ over $S_K$ defined by restricting to $S_K$ the first relative $p$-adic \'etale homology of the family $\cA$. There are families of Hodge tensors $(s_{\alpha,p})\subset \cV^\otimes_p$ coming from the comparison between Betti and $p$-adic \'etale homology (over $\mathbb{C}$). By the argument in Lemma~\ref{reflex field}, the $s_{\alpha,p}$ are also defined over the reflex field $E$.

Choose a cocharacter $\mu$ in the conjugacy class $X$, which is defined over some finite extension $E'/E$. We will base change everything to $E'$ from now on, but drop $E'$ from our notation. Recall that $P^{\mathrm{std}}_\mu$ can be identified with the parabolic subgroup of $G$ which stabilizes the descending filtration induced by $\mu$ on a faithful representation $V$ of $G$. We can define a $P^{\mathrm{std}}_\mu$-torsor $\cP_{\mathrm{dR}}$ over $S_{K}$ as the torsor of frames on the vector bundle $\cV_{\mathrm{dR}}$ which respect the Hodge filtration. More precisely, for any $U\subset S_{K}$, we have:
\[
\cP_{\mathrm dR}(U) = \{\beta: \cV_{\mathrm{dR}}|_U \toisom V\otimes_{\mathbb Q} \cO_U \mid \beta(s_{\alpha,\mathrm{dR}})=s_\alpha\otimes 1,\beta(\mathrm{Fil}^\bullet)=\mathrm{Fil}_\mu^\bullet\}\ ,
\]
where $\mathrm{Fil}^\bullet$ on $\cV_{\mathrm{dR}}$ is the Hodge filtration and $\mathrm{Fil}_\mu^\bullet$ on $V$ is the descending filtration defined by $\mu$. The existence of one such isomorphism $\beta$ follows from the fact that the comparison between Betti and de Rham cohomology respects the Hodge filtrations and matches the Hodge cycles $s_\alpha$ with $s_{\alpha,\mathrm{dR}}$. 

\begin{lemma}\label{de Rham torsor} The $P^{\mathrm{std}}_\mu$-torsor $\cP_{\mathrm{dR}}$ over $S_K$ is independent of the choice of symplectic embedding $G\hookrightarrow GSp(V,\psi)$.
\end{lemma}
\begin{proof} Considering a second symplectic embedding $G\hookrightarrow GSp(V',\psi')$, there is a $G$-invariant idempotent $e\in V^\otimes$ such that $V^\prime = e V^\otimes$. This determines a Hodge tensor $e_B$ in $\cV^\otimes_B$, and by Lemma~\ref{reflex field} a tensor $e_{\mathrm{dR}}$ in $\cV^\otimes_{\mathrm{dR}}$. This defines an isomorphism of vector bundles $\cV'_{\mathrm{dR}}\simeq e_{\mathrm{dR}} \cV^\otimes_{\mathrm{dR}}$ by Remark~\ref{reflexfieldsecondemb}, which respects all the Hodge tensors $s_{\alpha,\mathrm{dR}}$ and which respects the Hodge filtration on the two vector bundles (because $e_{\mathrm{dR}}$ is a Hodge tensor). This gives a map of $P^{\mathrm{std}}_\mu$-torsors $\cP_{\mathrm{dR}}\to \cP'_{\mathrm{dR}}$ and any such map is an isomorphism.
\end{proof}

From the above $P^{\mathrm{std}}_\mu$-torsor $\cP_{\mathrm{dR}}$ and from the projection $P_\mu \twoheadrightarrow M_\mu$, we get an $M_\mu$-torsor $\cM_\mathrm{dR}$ over $S_K$ via pushout:
\[
\cM_{\mathrm{dR}}=\cP_{\mathrm{dR}}\times_{P^{\mathrm{std}}_\mu} M_\mu\ .
\]
Since $\cP_{\mathrm{dR}}$ is independent of the choice of symplectic embedding, so is $\cM_{\mathrm{dR}}$. This $M_\mu$-torsor corresponds to trivializing the graded pieces of the Hodge filtration on $\cV_{\mathrm{dR}}$ individually. By the Tannakian formalism, $\cM_{\mathrm{dR}}$ is equivalent to a functor from finite-dimensional representations of the Levi subgroup $M_\mu$ to vector bundles on $S_K$.

\begin{lemma}\label{spb} The $M_\mu$-torsor $\cM_{\mathrm{dR}}$ encodes the tensor functor
\[
f_\infty: \mathrm{Rep}\ M_\mu \to  \{\mathrm{automorphic\ vector\ bundles\ on\ }S_{K}\}
\]
in the statement of Theorem~\ref{refinedht}.
\end{lemma}

\begin{proof} By construction, the tensor functor corresponding to $\cM_{\mathrm{dR}}$ factors through the inflation map $\mathrm{Rep}\ M_\mu \to \mathrm{Rep}\ P^{\mathrm{std}}_\mu$.

It remains to see that the functor corresponding to $P^\mathrm{std}_{\mu}$ maps a representation of $P^{\mathrm{std}}_\mu$ to the associated automorphic vector bundle on $S_K$. This is essentially the definition of automorphic vector bundles, as given by~\cite{milne}. For this, note that $\cP_{\mathrm{dR}}$ and the map $P^{\mathrm{std}}_\mu \to G$ define by pushout a $G$-torsor $\cG_{\mathrm{dR}}$ over $S_K$, which parametrizes frames of $\cV_{\mathrm{dR}}$ respecting the Hodge tensors $s_{\alpha,\mathrm{dR}}$ (but not necessarily respecting the Hodge filtration). This is what Milne calls the standard principal bundle. Since it was constructed from a $P^\mathrm{std}_\mu$-torsor, $\cG_{\mathrm{dR}}$ is equipped with a canonical map to the flag variety $\mathrm{Fl}^\mathrm{std}_{G,\mu}\simeq G/P_\mu$. We have a diagram \[\xymatrix{S_K&\cG_{\mathrm{dR}}\ar[l]\ar[r]&\mathrm{Fl}^{\mathrm{std}}_{G,\mu}}.\] Proposition 3.5 of~\cite{milne} proves that automorphic vector bundles are obtained by pullback from $\mathrm{Fl}_{G,\mu}$ to $\cG_{\mathrm{dR}}$ followed by descent to $S_K$.  We note that Theorems 4.1 and 4.3 of~\cite{milne} show that the diagram is algebraic and has a model over the reflex field $E$. 
\end{proof}

We now work with the local system $\cV_p$ determined by the relative $p$-adic \'etale cohomology of $\cA$. This is a local system of $\mathbb{Q}_p$-vector spaces over $S_K$. After pulling it back to the adic space $\cS_K$, we can think of it as a locally free $\hat{\mathbb{Q}}_p$-module on $(\cS_K)_{\mathrm{pro\acute{e}t}}$. 

Regard $P_\mu$ as a group object in the pro-\'etale site of $\cS_K$ by sending $U$ to $P_\mu(\hat{\cO}_{\cS_K}(U))$; we emphasize that we are using the completed structure sheaf in this definition. We can now define a $P_\mu$-quasitorsor $\mathscr{P}_p$ on the pro-\'etale site of $\cS_K$ from the Hodge-Tate filtration on $\cV_p\otimes_{\hat{\mathbb{Q}}_p}\hat{\cO}_{S_K}$ as follows. For $U$ in $(\cS_K)_{\mathrm{pro\acute{e}t}}$, set
\[
\mathscr{P}_p(U)=\{\beta: \cV_p\otimes_{\hat{\mathbb Q}_p} \hat{\cO}_{\cS_K}|_U\toisom V\otimes_{\mathbb Q} \hat{\cO}_{\cS_K}|_U \mid \beta(s_{\alpha, p}\otimes 1)=s_\alpha\otimes 1, \beta(\mathrm{Fil}_\bullet)=\mathrm{Fil}_\bullet(\mu)\}\ ,
\]
where $\mathrm{Fil}_\bullet$ on $\cV_p\otimes \hat{\cO}_{S_K}$ is the relative Hodge-Tate filtration and $\mathrm{Fil}_\bullet(\mu)$ is the ascending filtration determined by $\mu$ on $V$.

\begin{lemma} The object $\mathscr{P}_p$ over $\cS_K$ is a $P_\mu$-torsor.
\end{lemma}

\begin{proof} Similarly to $\mathscr{P}_p$, one can define a $G$-quasitorsor $\mathscr{G}_p$ over the pro-\'etale site of $\cS_K$, by removing the condition on filtrations. The latter is the pushout of a $G(\hat{\mathbb Q}_p)$-torsor on the pro-\'etale site of $\cS_K$ given by looking at isomorphisms between $\cV_p$ and $V\otimes_{\mathbb Q_p} \hat{\mathbb Q}_p$ respecting all tensors. This is a torsor, since, for example, it admits a global section over the perfectoid Shimura variety $\cS_{K^p}$. In order to prove that $\mathscr{P}_p$ is a torsor, we note that the type of the Hodge-Tate filtration on $\cV_p\otimes_{\hat{\mathbb{Q}}_p}\hat{\cO}_{\cS_K}$ is a discrete invariant, so it is constant on each connected component of $\cS_K$. Therefore, it suffices to check the statement above classical points.

Thus, let $x\in \cS_K(L,\cO_L)$ be a point defined over a finite extension $L$ of $E_{\mathfrak p}$ with completed algebraic closure $C$. We may pick a point of $\cM_{\mathrm{dR}}(C)$ above $x$, which amounts to trivializing all Hodge cohomology groups (compatibly with the tensors). Then the Hodge-Tate decomposition reads
\[
\cV_{p,x}\otimes_{\mathbb Q_p} C\cong \bigoplus_j V_j\otimes C(-j)\cong V\otimes C\ ,
\]
where $V=\bigoplus_j V_j$ is the weight decomposition according to the action of $\mu$, and we are using any fixed choice of $p$-power roots of unity in $C$ in the second isomorphism. Under this isomorphism, the Hodge-Tate filtration on the left-hand side is taken to $\mathrm{Fil}_\bullet(\mu)$, as desired.

The fact that $s_{\alpha,p}$ can be identified with $s_{\alpha}$ under the Hodge-Tate isomorphism is proved in~\cite{blasius}.
\end{proof}

As before, this torsor is independent of the choice of symplectic embedding.

\begin{lemma}\label{p-adic} The $P_\mu$-torsor $\mathscr{P}_p$ is independent of the choice of symplectic embedding. 
\end{lemma}
\begin{proof} This uses the same idea as the proof of Lemma~\ref{de Rham torsor}. Let $(V,\psi)$ be a symplectic embedding of $G$, which defines the $P_\mu$-torsor $\mathscr{P}_p$. For another symplectic embedding $G\hookrightarrow GSp(V',\psi')$, we define a $P_\mu$-torsor $\mathscr{P}'_p$ analogously. We can relate the two symplectic embeddings given by $(V,\psi)$ and $(V',\psi')$ via a $G$-invariant idempotent $e\in V^\otimes$, with $p$-adic realization $e_p\in \cV^\otimes_p$. The tensor $e_p$ defines an isomorphism of vector bundles \[\cV'_{p}\otimes \hat{\cO}_{\cS_K}\simeq e_{p}(\cV^\otimes_{p}\otimes \hat{\cO}_{\cS_K}),\] which matches the tensors $s'_{\alpha,p}\in \cV'^\otimes_p$ with tensors in $\cV^\otimes_p$.

Moreover, $e_p$ respects the Hodge-Tate filtration on the two vector bundles. Indeed, $e_p$ is the image of $e_{\mathrm{dR}}$ under the $p$-adic-de Rham comparison isomorphism. At points of $S_K$ corresponding to abelian varieties defined over number fields, this follows from~\cite{blasius}. Since both $e_p$ and $e_{\mathrm{dR}}$ are horizontal sections, the result extends over all of $\cS_{K}$ after checking it at such a point in every connected component of $\cS_K$. The definition of the relative Hodge-Tate filtration in terms of the $p$-adic-de Rham comparison isomorphisms then ensures that $e_p$ respects the Hodge-Tate filtration, and the isomorphism induced by $e_p$ gives a map of $P_\mu$-torsors $\mathscr{P}_{p}\to \mathscr{P}^\prime_{p}$, which has to be an isomorphism.
\end{proof}

The $P_\mu$-torsor $\mathscr{P}_p$ defines a $G$-torsor $\mathscr{G}_p$ by inflation along the map $P_\mu\to G$. For any $U\in (\cS_K)_\mathrm{pro\acute{e}t}$,
\[
\mathscr{G}_p(U)=\{\beta:\cV_{p}\otimes_{\hat{\mathbb Q}_p} \hat{\cO}_{\cS_K}|_U\toisom V\otimes_{\mathbb Q} \hat{\cO}_{\cS_K}|_U \mid \beta(s_{\alpha, p}\otimes 1)=s_\alpha\otimes 1\}\ .
\]
The perfectoid Shimura variety $\cS_{K^p}$ can be regarded as a $K_p$-torsor in $(\cS_K)_{\mathrm{pro\acute{e}t}}$. From the moduli description of $\cS_{K^p}$, we see that $\mathscr{G}_p(\cS_{K^p})$ has a canonical section, given by the trivialization of the $p$-adic Tate module of the universal abelian variety $\cA$ over $\cS_{K^p}$, which by definition respects the tensors $(s_{\alpha,p})$.

The map $P_\mu\twoheadrightarrow M_\mu$ defines an $M_\mu$-torsor $\mathscr{M}_{p}$ by pushout. This can be described as a sheaf on $(\cS_K)_{\mathrm{pro\acute{e}t}}$ as follows:
\[
\mathscr{M}_p(U)=\{\beta: \mathrm{gr}_\bullet(\cV_p\otimes \hat{\cO}_{\cS_K})|_U\toisom \mathrm{gr}_\bullet(\mu)(V\otimes_{\mathbb Q_p} \hat{\cO}_{\cS_K})|_U \mid \beta(s_{\alpha, p}\otimes 1)=s_\alpha\otimes 1\}\ .
\]

As in the complex case, the existence of $\mathscr{P}_p$ determines a map $\mathscr{G}_p\to \Fl_{G,\mu}$, which is independent of the choice of symplectic embedding $G\hookrightarrow GSp(V,\psi)$ by Lemma~\ref{p-adic}. Here, we abuse notation by writing $\Fl_{G,\mu}$ for the sheaf on $(\cS_K)_{\mathrm{pro\acute{e}t}}$ sending $U$ to $\Fl_{G,\mu}(U)$. This and the given section of $\mathscr{G}_p(\cS_{K^p})$ define an element of $\Fl_{G,\mu}(\cS_{K^p})$, i.e. a map of adic spaces
\[
\pi_{HT}: \cS_{K^p}\to \Fl_{G,\mu}\ .
\]
By functoriality of this construction (for $G$ and for $\tilde G:=GSp(V,\psi)$), we have the commutative diagram of adic spaces \[\xymatrix{\cS_{K^p}\ar[d]\ar[r]&\tilde{\cS}_{\tilde K^p}\ar[d]\\ \Fl_{G,\mu}\ar[r]&\Fl_{\tilde G,\tilde{\mu}}.}\] Therefore, the Hodge-Tate period map for $\cS_{K^p}$ factors through this canonical map $\cS_{K^p}\to \Fl_{G,\mu}$. This proves the first part of Theorem~\ref{refinedht}. 

The second part of Theorem~\ref{refinedht} will follow from the next lemma and from the comparison isomorphism between de Rham and $p$-adic \'etale cohomology. 

\begin{lemma}~\label{tannetale}The $M_\mu$-torsor $\mathscr{M}_p$ encodes the tensor functor
\[
f_p: \mathrm{Rep}\ M_\mu \to \{G(\mathbb{Q}_p)-\mathrm{equivariant\ vector\ bundles\ on\ }\cS_{K^p}\}
\]
in the statement of Theorem~\ref{refinedht}.
\end{lemma}

\begin{proof} This is immediate from the definitions.
\end{proof}

We now compare the two $M_\mu$-torsors, $\cM_{\mathrm{dR}}$ and $\mathscr{M}_p$. For this, we first consider a $P^{\mathrm{std}}_\mu$-torsor $\mathscr{P}_{\mathrm{dR}}$ over $\cS_{K}$, which will be the sheaf on $(\cS_{K})_{\mathrm{pro\acute{e}t}}$ defined by \[\mathscr{P}_{\mathrm{dR}}(U)=\{\beta:\cV_{\dR}\otimes_{\cO_{\cS_K}} \hat{\cO}_{\cS_K}|_U\toisom V\otimes_{\mathbb Q} \hat{\cO}_{\cS_K}|_U \mid \beta(s_{\alpha, \mathrm{dR}}\otimes 1)=s_\alpha\otimes 1, \beta(\mathrm{Fil}^\bullet)=\mathrm{Fil}^\bullet(\mu)\},\]  where $\mathrm{Fil}^\bullet$ is the Hodge-de Rham filtration on $\cV_{\dR}$. It is easy to see from the definitions that $\mathscr{P}_{\mathrm{dR}}$ is the pullback of $\cP_{\mathrm{dR}}$ from $S_K$ (ringed with $\cO_{S_K}$) to $(\cS_{K})_{\mathrm{pro\acute{e}t}}$ (ringed with $\hat{\cO}_{\cS_K}$). We can define $\mathscr{M}_{\mathrm{dR}}$ by pushout. This is also a sheaf on $(\cS_K)_{\mathrm{pro\acute{e}t}}$, parametrizing isomorphisms 
\[
\mathrm{gr}^\bullet(\cV_{\mathrm{dR}}\otimes \hat{\cO}_{\cS_K})\toisom \mathrm{gr}^\bullet(\mu)(V\otimes \hat{\cO}_{\cS_K})
\]
which map the tensors $s_{\alpha,\mathrm{dR}}$ to $s_\alpha$. Again, $\mathscr{M}_{\mathrm{dR}}$ is the pullback of $\cM_{\mathrm{dR}}$ from $S_K$ to $(\cS_K)_{\mathrm{pro\acute{e}t}}$. 
 
\begin{prop}\label{isomorphism of torsors} There is a canonical isomorphism $\mathscr{M}_{\mathrm{dR}}\cong \mathscr{M}_{p}$ of $M_\mu$-torsors on $(\cS_K)_{\mathrm{pro\acute{e}t}}$, independent of the choice of symplectic embedding.
\end{prop}

\begin{proof} The determinant representation $GSp(V,\psi)\to \mathbb{G}_m$ gives rise geometrically to the Tate motive, and is independent of the choice of symplectic embedding. Using this, both torsors map to the torsor of trivializations $\hat{\cO}_{\cS_K}(1)\cong \hat{\cO}_{\cS_K}$. Now, for any $j\in \mathbb Z$, there is the isomorphism
\[
\mathrm{gr}^j(\cV_{\mathrm{dR}}\otimes_{\cO_{\cS_K}}\hat{\cO}_{\cS_K})\toisom \mathrm{gr}_j(\cV_p\otimes_{\hat{\mathbb{Q}}_p}\hat{\cO}_{\cS_K})(j)
\]
coming from the relative $p$-adic-de Rham comparison isomorphism, Corollary \ref{gradedHodgeTate}. One gets a similar comparison for $\cV_{\mathrm{dR}}^\otimes$ and $\cV_p^\otimes$, and we know by~\cite{blasius} that all tensors $s_{\alpha,\mathrm{dR}}$ resp. $s_{\alpha,p}$ are matched at points defined over number fields, and thus globally.

Using these isomorphisms as well as the trivialization $\hat{\cO}_{\cS_K}(1)\cong \hat{\cO}_{\cS_K}$, one writes down the isomorphism $\mathscr{M}_{\mathrm{dR}}\cong \mathscr{M}_p$. To check that it is independent of the choice of symplectic embedding, one argues as before.
\end{proof}

As mentioned above, Proposition~\ref{isomorphism of torsors} implies the second part of Theorem~\ref{refinedht}, once we use the Tannakian formalism in Lemmas~\ref{spb} and~\ref{tannetale} to reinterpret $\mathscr{M}_{\mathrm{dR}}$ and $\mathscr{M}_p$ as tensor functors \[\mathrm{Rep}\ M_\mu\to \{G(\mathbb{Q}_p)-\mathrm{equivariant\ vector\ bundles\ on\ }\cS_{K^p}\}.\]
\newpage

\section{The Newton stratification on the flag variety}

We start with some motivation. Assume that the Shimura varieties $\cS_{K}$ are of Hodge type and that $K=K^pK_p\subset G(\mathbb{A}_f)$ is a compact open subgroup such that $K_p$ is hyperspecial. This means that $G$ extends to a reductive group over $\mathbb{Z}_p$ and that $K_p=G(\mathbb{Z}_p)$. Then (at least if $p>2$) the Shimura variety $\cS_{K}$ admits an integral model $\mathscr{S}_K$ by~\cite{kisinintegral}. Moreover, as in Section 1.4 of~\cite{kisin}, we can define a Newton stratification on the special fiber of $\mathscr{S}_{K}$, in terms of the Kottwitz set $B(G, \mu^{-1})$ (whose definition we recall below). Pulling this stratification back along the continuous specialization map, we get a stratification on $\cS_K$, which in turn can be pulled back to the perfectoid Shimura variety to get a Newton stratification $\cS_{K^p}=\bigsqcup_{b\in B(G,\mu^{-1})} \cS^b_{K^p}$. There is a unique closed stratum, corresponding to the basic locus and a unique open stratum, corresponding to the $\mu$-ordinary locus. 

Our goal in this section is to define a stratification on the flag variety  \[\Fl_{G,\mu}=\bigsqcup_{b\in B(G, \mu^{-1})} \Fl^b_{G,\mu},\] such that the following properties are satisfied:
\begin{enumerate}
\item On points of rank one, \[\cS^b_{K^p}=\pi_{HT}^{-1}(\Fl_{G,\mu}^b)\ .\]
\item All $\Fl_{G,\mu}^b$ are locally closed subspaces of the adic space $\Fl_{G,\mu}$, in the topological sense.
\item The basic stratum is open, and the $\mu$-ordinary stratum is closed.
\end{enumerate}

We will define this stratification independently of the one on the Shimura variety, using relative versions of the Fargues-Fontaine curve~\cite{farguesfontaine} and a classification result for vector bundles with $G$-structure over this curve, due to Fargues, \cite{farguesGbun}. We will reinterpret vector bundles over the curve as $\varphi$-modules over the Robba ring, \`a la Kedlaya-Liu~\cite{kedlayaliu}, and use their results to conclude that the strata we define are locally closed. In Section~\ref{a product formula}, we will see that this is compatible with the stratification pulled back from the special fiber, in the sense described above, for compact Shimura varieties of PEL type.

Throughout this section, our notation will be purely local, so fix a prime $p$ and a connected reductive group $G$ over $\mathbb Q_p$. Moreover, we fix a conjugacy class of cocharacters $\mu: \mathbb G_m\to G_{\overline{\mathbb Q}_p}$, defined over the reflex field $E/\mathbb Q_p$. Often, we will assume that $\mu$ is minuscule, meaning that in the induced action on the Lie algebra of $G$, only the weights $-1$, $0$ and $1$ appear. However, for the moment, $\mu$ is allowed to be arbitrary.

\subsection{Background on isocrystals with $G$-structure}\label{isocrystals} We recall here the definition of the sets $B(G)$ and $B(G, \mu)$, originally due to Kottwitz~\cite{kottwitz}. We start with $B(G)$. Let $L:=W(\bar{\mathbb{F}}_p)[1/p]$. Let $\sigma$ be the automorphism of $L$ induced by the $p$th power Frobenius on $\bar{\mathbb{F}}_p$. There is an action of $G(L)$ on itself by $\sigma$-conjugation, defined by $g\mapsto hg\sigma(h)^{-1}$ for $g,h\in G(L)$. Then $B(G)$ is defined to be the set of $\sigma$-conjugacy classes of elements $b\in G(L)$. (We note that instead of working with $\bar{\mathbb{F}}_p$ here, we could work with any algebraically closed field of characteristic $p$, as Kottwitz shows that the definition is independent of this choice.)

One can reinterpret this definition in terms of isocrystals with $G$-structure. Recall the following definition.

\begin{defn} An isocrystal over $\bar{\mathbb{F}}_p$ is a pair $(V,\phi)$ consisting of a finite-dimensional $L$-vector space and a $\sigma$-linear automorphism $\phi$ of $V$. The height of an isocrystal $(V,\phi)$ is the dimension of $V$ over $L$.

An isocrystal with $G$-structure is an exact tensor functor \[\mathrm{Rep}_{\mathbb{Q}_p}G\to \{\mathrm{Isocrystals}/\bar{\mathbb{F}}_p\}.\]
\end{defn}

\noindent For $G=GL_{n}/\mathbb{Q}_p$, the set $B(G)$ is in bijection with the set of isomorphism classes of isocrystals of height $n$ over $\bar{\mathbb{F}}_p$ via $b\mapsto (L^n,b\sigma)$. For general $G$, this extends to a bijection between $B(G)$ and isomorphism classes of isocrystals with $G$-structure.

The Dieudonn\'e-Manin classification shows that $B(GL_n)$ is in bijection with a corresponding set of Newton polygons, via the slope decomposition of the isocrystals. More precisely, any isocrystal $(V,\phi)$ over $\bar{\mathbb{F}}_p$ is isomorphic to a unique isocrystal of the form
\[
V \cong \bigoplus_{\lambda=s/r\in \mathbb Q} V_\lambda^{\oplus n_\lambda}\ ,
\]
where $\lambda=s/r$ runs through rational numbers written in primitive form with $r>0$, the $n_\lambda$ are nonnegative integers, almost all zero, and
\[
V_\lambda = (L^r,\left(\begin{array}{cccc} & 1 &  &  \\  &  & \ddots &  \\  &  &  & 1 \\ p^s &  &  &  \end{array}\right)\sigma)\ .
\]
The subspaces $V_\lambda^{\oplus n_\lambda}\subset V$ are uniquely determined, and referred to as the subspace of slope $\lambda$.

For a general reductive group $G$, an element $b\in B(G)$ is determined by a version of the Newton polygon, and an additional finite datum encoded in the Kottwitz invariant. In the following, fix a splitting of $G_{\bar{\mathbb Q}_p}$ and in particular a maximal torus $T\subset G_{\bar{\mathbb Q}_p}$, and let $X_*(G) := X_*(T)$ be the corresponding cocharacter lattice, which comes with a dominant chamber.

Let us first recall the Newton map
\[
\nu: B(G)\to (X_*(G)\otimes\mathbb{Q})_{\mathrm{dom}}^\Gamma\ .
\]
Here, $\Gamma:=\mathrm{Gal}(\bar{\mathbb{Q}}_p/\mathbb{Q}_p)$ is the absolute Galois group of $\mathbb Q_p$, and $(X_*(G)\otimes\mathbb{Q})_{\mathrm{dom}}$ is the set of dominant rational cocharacters. If we let $\mathbb{D}$ be the (pro-)algebraic torus with character group $\mathbb{Q}$, the latter set can be identified with the set of conjugacy classes of $\mathrm{Hom}(\mathbb{D}_{\bar{\mathbb{Q}}_p},G_{\bar{\mathbb{Q}}_p})$, on which $\Gamma$ acts naturally.

To construct the Newton map, Kottwitz assigns to any $b\in G(L)$ a slope homomorphism $\nu_b\in \mathrm{Hom}(\mathbb{D}_L,G_L)$. In the case of $G=GL_n$, this gives the slope decomposition of the corresponding isocrystal; in general, it is defined by the Tannakian formalism. Changing $b$ by a $\sigma$-conjugate does not change the conjugacy class of $\nu_b$, and (thus) this conjugacy class is invariant under $\sigma$.

However, the Newton map is not, in general, injective. In fact, $\nu_b$ is trivial if and only if $b$ is in the image of the natural injection $H^1(\mathbb Q_p,G)\hookrightarrow B(G)$. Here, one can identify the Galois cohomology group $H^1(\mathbb Q_p,G)$ with the isomorphism classes of exact tensor functors
\[
\mathrm{Rep}_{\mathbb Q_p} G\to \{\mathbb Q_p\mathrm{-vector\ spaces}\}\ .
\]
Such tensor functors embed fully faithfully into the category of isocrystals with $G$-structure, via sending a $\mathbb Q_p$-vector space $W$ to $W\otimes_{\mathbb Q_p} L$ with the induced Frobenius from $L$.

For this reason, Kottwitz also constructs a map
\[
\kappa: B(G)\to \pi_1(G_{\bar{\mathbb Q}_p})_{\Gamma}\ .
\]
For $G=GL_n$, this map is defined by $b\mapsto \kappa(b)=\mathrm{val}_p(\det b)\in \mathbb{Z}$. In general, there is a unique natural transformation $B(\ )\to \pi_1(\ )_{\Gamma}$ of set-valued functors on the category of connected reductive groups over $\mathbb{Q}_p$ with this property. (Kottwitz defines his map in terms of the center of the Langlands dual group. See Section 1.13 of~\cite{rapoport-richartz} for more on the definition using the algebraic fundamental group.) Again, we abbreviate $\pi_1(G) = \pi_1(G_{\bar{\mathbb Q}_p})$. Moreover, according to Theorem 1.15 of~\cite{rapoport-richartz}, the natural transformation $B(\ )\to  \pi_1(\ )_{\Gamma}$ fits into a commutative diagram
\[
\xymatrix{B(G)\ar[r]\ar[d]& (X_*(G)\otimes\mathbb{Q})^{\Gamma}\ar[d]\\
\pi_1(G)_{\Gamma}\ar[r] & \pi_1(G)^{\Gamma}\otimes\mathbb{Q}},
\]
where the lower horizontal arrow is given by averaging over all Galois conjugates. Then Kottwitz proves that
\[
(\nu,\kappa): B(G)\to (X_*(G)\otimes\mathbb{Q})_{\mathrm{dom}}^\Gamma\times \pi_1(G)_{\Gamma}
\]
is injective.

The set $(X_*(G)\otimes \mathbb{Q})^{\Gamma}_{\mathrm{dom}}$ admits a partial ordering. Under this ordering, we say that $\nu\preceq \nu'$ if $\nu'-\nu$ is a non-negative $\mathbb{Q}$-linear combination of positive coroots. This defines a partial ordering on $B(G)$, where we say $b\preceq b'$ if $\nu_b\preceq \nu_{b'}$ and $\kappa(b)=\kappa(b')$.     

Now, recall that we have fixed a conjugacy class of cocharacters $\mu: \mathbb{G}_m\to G_{\bar{\mathbb Q}_p}$. The set of conjugacy classes of cocharacters of $G_{\mathbb{\bar Q}_p}$ is in bijection with the set $X_*(G)_{\mathrm{dom}}$. There is a natural map $X_*(G)_{\mathrm{dom}}\to (X_*(G)\otimes \mathbb{Q})^{\Gamma}_{\mathrm{dom}}$ given by averaging over all Galois conjugates:
\[\bar \mu = \frac{1}{[E^\prime:\mathbb{Q}_p]}\sum_{\gamma\in \mathrm{Gal}(E^\prime/\mathbb{Q}_p)}\gamma(\mu)\]
for $E^\prime$ large enough. Let $\mu^\flat$ be the image of $\mu$ in $\pi_1(G)_{\Gamma}$. 

\begin{defn}The subset $B(G,\mu)\subset B(G)$ of $\mu$-\emph{admissible} elements is the subset of elements $b$ for which $\nu_b\preceq \bar \mu$ and $\kappa(b)=\mu^\flat$.
\end{defn}

In fact, we will really be interested in $B(G,\mu^{-1})$, where $\mu^{-1}$ denotes a dominant representative of the inverse of $\mu$.

\subsection{The Fargues-Fontaine curve}\label{the curve}

The goal of this subsection is to define the (adic) Fargues-Fontaine curve, and discuss some of its properties. For this, we start with some background on the curve as in~\cite{farguesfontaine} and~\cite{scholzelectures}, and then compare with constructions of Kedlaya and Liu~\cite{kedlayaliu}. 

Let $F$ be a complete algebraically closed nonarchimedean field of characteristic $p$, e.g. $F=\widehat{\overline{\mathbb F_p((t))}}$. Let $\cO_F\subset F$ be its ring of integers, i.e. the subring of powerbounded elements. Fix $\varpi\in F$ with $0<|\varpi|<1$; different choices will give rise to the same objects. First, we define the Fargues-Fontaine curve as an adic space. Let
\[
\cY_{(0,\infty)}=\mathrm{Spa}(W(\cO_F),W(\cO_F))\setminus\left(p[\varpi]=0\right)\ ,
\]
where $W(\cO_F)$ is endowed with the $(p,[\varpi])$-adic topology. As in \cite{scholzelectures}, this space admits a natural continuous map
\[
\alpha: \cY_{(0,\infty)}\to (0,\infty)\ ,
\]
sending any point $x\in \cY_{(0,\infty)}$ to
\[
\alpha(x)=\frac{\log |[\varpi](\tilde x)|}{\log|p(\tilde x)|}\in (0,\infty)\ ,
\]
where $\tilde x$ is the maximal generalization of $x$, which corresponds to a continuous rank-$1$-valuation on $W(\cO_F)$ taking nonzero values on $[\varpi]$ and $p$. For any interval $I\subset (0,\infty)$, we let $\cY_I\subset \cY_{(0,\infty)}$ be the interior of $\alpha^{-1}(I)$. In the following proposition, we use some terminology from~\cite{scholzeweinstein}.

\begin{prop}\label{Yhonest} For any closed interval $I=[s,r]\subset (0,\infty)$ with $r,s\in \mathbb Q$, the space
\[
\cY_I = \mathrm{Spa}(\cR^{[s,r]}_F,\cR^{[s,r],+}_F)
\]
is a sheafy affinoid adic space, where $\cR^{[s,r],+}_F$ is the $p$-adic completion of the integral closure of
\[
W(\cO_F)\left[\frac{p}{[\varpi^{1/r}]},\frac{[\varpi^{1/s}]}{p}\right]
\]
inside $W(\cO_F)\left[\frac{p}{[\varpi^{1/r}]},\frac{[\varpi^{1/s}]}{p}\right][1/p]$, and $\cR^{[s,r]}_F = \cR^{[s,r],+}_F[1/p]$. More precisely, $\cR^{[s,r]}_F$ is preperfectoid in the sense that $\cR^{[s,r]}_F\hat{\otimes}_{\mathbb Q_p} K$ is a perfectoid $K$-algebra for any perfectoid field $K/\mathbb Q_p$.

In particular, $\cY_{(0,\infty)}$ is an honest adic space. 
\end{prop}

\begin{proof} The identification
\[
\cY_I = \mathrm{Spa}(\cR^{[s,r]}_F,\cR^{[s,r],+}_F)
\]
follows from the definitions. By~\cite[Theorem 3.7.4]{kedlayaliu}, it is enough to show that $\cR^{[s,r]}_F$ is preperfectoid, for which cf.~\cite[Theorem 5.3.9]{kedlayaliu}. One can also argue as follows. Let $K/\mathbb Q_p$ be any perfectoid field. We can consider the auxiliary space $\cZ = \mathrm{Spa}(W(\cO_F)[1/p],W(\cO_F))$, where we endow $W(\cO_F)$ with the $p$-adic topology. As on $\cY_{(0,\infty)}$, $p$ is topologically nilpotent, one gets a map $\cY_{(0,\infty)}\to \cZ$, which is an open embedding, and one can thus consider $\cY_I$ as a rational subset of $\cZ$. As the base change of $\cZ$ to $K$ is perfectoid, or more precisely $W(\cO_F)\hat{\otimes}_{\mathbb Z_p} K$ is a perfectoid $K$-algebra, and the property of being a perfectoid $K$-algebra passes to rational subsets, one finds that also $\cR^{[s,r]}_F\hat{\otimes}_{\mathbb Q_p} K$ is a perfectoid $K$-algebra.
\end{proof}

The space $\cY_{(0,\infty)}$ has an action of $\varphi$, defined by taking the lift of the Frobenius on $\cO_F$. This $\varphi$-action is properly discontinuous, as can be seen by observing that $\alpha$ is equivariant with respect to the $\varphi$-action if one lets $\varphi$ act through multiplication by $p$ on $(0,\infty)$. Therefore, the following definition is sensible.

\begin{defn}\label{fargues-fontaine curve}
The adic Fargues-Fontaine curve is given by $\cX_F =\cY_{(0,\infty)}/\varphi^\mathbb{Z}$. 
\end{defn}

After defining the scheme version of the curve, we will discuss more precisely in which sense this is a curve.

Often, we will be in the situation where we start with a complete algebraically closed nonarchimedean field $C$ over $\mathbb Q_p$, and take $F=C^\flat$, the tilt of $C$. In that case, there is a natural map $\theta: W(\cO_F)\to \cO_C$, which induces a natural $(C,\cO_C)$-point of $\cY_{(0,\infty)}$, and thus of $\cX_F$, which we denote by $\infty\in \cX_F(C,\cO_C)$. In fact, $\infty$ is a closed point of $\cX_F$ with residue field $C$. We will denote the inclusion
\[
i_\infty: \mathrm{Spa}(C,\cO_C)\to \cX_F\ .
\]
The completed local ring of $\cX_F$ at $\infty$ can be identified with the ring of periods $B^+_{\mathrm{dR},C}$, which is the $\ker \theta$-adic completion of $W(\cO_F)[1/p]$, cf. also Definition~\ref{derham period sheaf}. Note that $B^+_{\mathrm{dR},C}$ is a complete discrete valuation ring, as expected for the completed local ring of a curve.

There is a close relationship between vector bundles on $\cX_F$ and isocrystals. Recall that $L$ was defined as $W(\bar{\mathbb F}_p)[1/p]$. A choice of an embedding $\bar{\mathbb F}_p\to \cO_F$ gives a structure map $\cY_{(0,\infty)}\to \mathrm{Spa}(L,\cO_L)$. If $(V,\varphi_V)$ is an isocrystal, one can thus pull it back to a vector bundle on $\cY_{(0,\infty)}$ with a $\varphi$-linear automorphism; by descent, this gives a vector bundle on $\cX_F$. We denote the resulting functor by $V\mapsto \cE(V)$.

\begin{thm}[\cite{farguesfontaine}]\label{isocrystals to vector bundles} The above composition of functors induces a bijection between isomorphism classes of isocrystals, and isomorphism classes of vector bundles on $\cX_F$. 
\end{thm}

\begin{remark} In fact, Fargues-Fontaine prove this result for the scheme version of their curve, which we introduce below. However, by a GAGA result proved in~\cite{kedlayaliu} and~\cite{fargues}, this is equivalent to the stated result for the adic curve.

It is important to note that this functor from isocrystals to vector bundles is not an equivalence of categories; there are nonzero maps between vector bundles of different slope, in general.
\end{remark}

To define a scheme version of the curve, we define a natural line bundle $\cO_{\cX_F}(1)$ on $\cX_F$, which we regard as ample.

\begin{defn}\label{line bundles} For any $d\in \mathbb Z$, let $\cO_{\cX_F}(d)$ be the line bundle corresponding to the isocrystal $(L,p^{-d}\sigma)$.
\end{defn}

\begin{remark} This construction induces a map $\mathbb Z\to \mathrm{Pic}\ \cX_F$. It follows from Theorem~\ref{isocrystals to vector bundles} that this is an isomorphism. Using this identification, one can define the degree of any vector bundle on $\cX_F$ by looking at the determinant line bundle. This gives rise to a notion of slopes of vector bundles, and a Harder-Narasimhan filtration. We warn the reader that if an isocrystal $V$ is sent to the vector bundle $\cE(V)$, then the slopes of $V$ and $\cE(V)$ differ by a sign.
\end{remark}

Now we define a scheme
\[
X_F = \mathrm{Proj}\left(\oplus_{d\geq 0}H^0\left(\cX_F,\cO_{\cX_F}(d)\right)\right).
\]
There is a natural map of locally ringed topological spaces $\cX_F\to X_F$. In particular, there is a natural functor from vector bundles on $X_F$ to vector bundles on $\cX_F$. This functor is an equivalence of categories, cf.~\cite{kedlayaliu} and~\cite{fargues}. The following theorem summarizes some of the properties of $X_F$.

\begin{thm}[\cite{farguesfontaine}] The scheme $X_F$ is a regular, noetherian scheme of Krull dimension $1$ with field of constants $\mathbb Q_p$. All residue fields of $X_F$ at closed points are algebraically closed complete extensions $C$ of $\mathbb Q_p$ with $C^\flat\cong F$. For any closed point $x\in X_F$, $X_F\setminus \{x\}$ is the spectrum of a principal ideal domain.
\end{thm}

Fargues, \cite{farguesGbun}, has recently extended the classification of vector bundles to a classification of $G$-bundles for any reductive group $G$ over $\mathbb Q_p$. As it is technically easiest for us to do so, we define $G$-bundles on $X_F$ (or $\cX_F$) using the Tannakian perspective.

\begin{defn} A $G$-bundle on $X_F$ (or $\cX_F$) is an exact tensor functor
\[
\mathrm{Rep}_{\mathbb Q_p} G\to \mathrm{Bun}_{X_F}\cong \mathrm{Bun}_{\cX_F}\ .
\]
\end{defn}

Using the functor from isocrystals over $\bar{\mathbb F}_p$ to vector bundles on the Fargues-Fontaine curve, we get a natural functor from isocrystals with $G$-structure to $G$-bundles on $X_F$. We denote this functor by $b\mapsto \cE_b$.

\begin{thm}[\cite{farguesGbun}]\label{classification of G-bundles} The functor from isocrystals with $G$-structure to $G$-bundles on $X_F$ induces a bijection on isomorphism classes.
\end{thm}

\noindent In other words, any $G$-bundle on $X_F$ is isomorphic to $\cE_b$ for a unique $b\in B(G)$.

Next, we discuss the relationship between vector bundles on the Fargues-Fontaine curve and $\varphi$-modules over the Robba ring. The Robba ring is the ring of functions defined on a small unspecified annulus $\cY_{(0,r)}$:

\begin{defn} The Robba ring is the direct limit \[\tilde \cR_F=\varinjlim_r H^0(\cY_{(0,r]}, \cO_{\cY_{(0,r]}})\ .\] 
\end{defn}

One can make this more explicit, cf.~\cite[Definition 4.2.2]{kedlayaliu}. The space of global sections $\tilde \cR_F^r=H^0(\cY_{(0,r]},\cO_{\cY})$ can be identified with the inverse limit of the Banach algebras $\tilde \cR^{[s,r]}_F$ as $s$ runs over $(0,r]$, and thus acquires a structure of Fr\'echet algebra. Let
\[
W(\cO_F)\left\langle \frac{p}{[\varpi]^{1/r}}\right\rangle=\left\{\sum_{n\geq 0}[c_n]p^n  \mid c_n\in \varpi^{-n/r} \cO_F, c_n \varpi^{n/r}\to 0\right\}\ .
\]
Then $\tilde \cR_F^r$ can also be described as the Fr\'echet completion of
\[
W(\cO_F)\left\langle \frac{p}{[\varpi]^{1/r}}\right\rangle \left[\frac{1}{p}\right]=\left\{\sum_{n>-\infty}[c_n]p^n \mid c_n\in F, c_n \varpi^{n/r}\to 0\right\}
\]
along the norms $\max_n\{|c_n \varpi^{n/s}|\}$ for $s\in(0,r]$. When $r'<r$, there is a natural inclusion map $\tilde \cR_F^r\hookrightarrow \tilde \cR_F^{r'}$ coming from restriction of global sections. The $\varphi$-action on $\cY_{(0,\infty)}$ sends $\cY_{[s,r]}$ isomorphically to $\cY_{[ps,pr]}$ and $\cY_{(0,r]}$ isomorphically to $\cY_{(0,pr]}$. Therefore, $\varphi$ induces isomorphisms $\tilde \cR^{[s,r]}_F\toisom \tilde \cR^{[s/p,r/p]}_F$ and $\tilde \cR^r_F\toisom \tilde \cR^{r/p}_F$, and thus an automorphism of $\tilde \cR_F$.

We note that the Robba ring is the ring of functions defined on some small punctured disc of unspecified radius around the point $\mathrm{Spa}(F,\cO_F)$ of $\mathrm{Spa}(W(\cO_F),W(\cO_F))$.

\begin{defn} A \emph{$\varphi$-module over $\tilde \cR_F$} is a finite projective $\tilde \cR_F$-module $M$ equipped with a $\varphi$-linear automorphism.
\end{defn}

\begin{remark} As $\tilde \cR_F$ is a B\'ezout domain, cf.~\cite[Lemma 4.2.6]{kedlayaliu}, any $\varphi$-module $M$ is finite free as $\tilde \cR_F$-module.
\end{remark}

\begin{thm}[{\cite[Theorem 6.3.12]{kedlayaliu}}]\label{Robba ring equivalence} 
There is an equivalence of categories \[\left\{\mathrm{Vector\ bundles\ on}\ \cX_F\right\}\simeq\left\{\varphi\mathrm{-modules\ over}\ \tilde \cR_F\right\}.\] 
\end{thm}

The proof is based on the observation that any $\varphi$-module over $\tilde \cR_F$ is defined over $\tilde \cR_F^r$ for $r$ small enough. This can be turned into a $\varphi$-module over $\cY_{(0,r]}$, and then be spread to a $\varphi$-module over all of $\cY_{(0,\infty)}$ via pullback under Frobenius. By descent, this gives a vector bundle over $\cX_F$.

\subsection{The relative Fargues-Fontaine curve}\label{relative curve}

In this subsection, we extend the constructions to the relative setting. Here, our basic input will be a perfectoid affinoid algebra $(R,R^+)$ of characteristic $p$.\footnote{We will not fix a perfectoid base field inside $R$, although one can always find one.} Let $\varpi$ be a pseudouniformizer of $R$. Define
\[
\cY_{(0,\infty)}(R, R^+)=\mathrm{Spa}(W(R^+),W(R^+))\setminus\left(p[\varpi]=0\right)\ .
\]
Many constructions carry over to this relative situation. In particular, there is still a continuous map
\[
\alpha: \cY_{(0,\infty)}(R,R^+)\to (0,\infty)
\]
defined in the same way. Again, we let $\cY_I(R,R^+)\subset \cY_{(0,\infty)}(R,R^+)$ denote the interior of the preimage $\alpha^{-1}(I)$, for any interval $I\subset (0,\infty)$. Proposition~\ref{Yhonest} extends to the relative setting.

\begin{prop}\label{relYhonest} For any closed interval $I=[s,r]\subset (0,\infty)$ with $r,s\in \mathbb Q$, the space
\[
\cY_I(R,R^+) = \mathrm{Spa}(\cR^{[s,r]}_R,\cR^{[s,r],+}_{R,R^+})
\]
is a sheafy affinoid adic space, where $\cR^{[s,r],+}_{R,R^+}$ is the $p$-adic completion of the integral closure of
\[
W(R^+)\left[\frac{p}{[\varpi^{1/r}]},\frac{[\varpi^{1/s}]}{p}\right]
\]
inside $W(R^+)\left[\frac{p}{[\varpi^{1/r}]},\frac{[\varpi^{1/s}]}{p}\right][1/p]$, and $\cR^{[s,r]}_R = \cR^{[s,r],+}_{R,R^+}[1/p]$.\footnote{One can check that $\cR^{[s,r]}_R$ depends only on $R$, and not on $R^+$.} More precisely, $\cR^{[s,r]}_R$ is preperfectoid in the sense that $\cR^{[s,r]}_R\hat{\otimes}_{\mathbb Q_p} K$ is a perfectoid $K$-algebra for any perfectoid field $K/\mathbb Q_p$.

In particular, $\cY_{(0,\infty)}(R,R^+) = \bigcup_I \cY_I(R,R^+)$ is an honest adic space. 
\end{prop}

\begin{proof} The same arguments as for Proposition \ref{Yhonest} apply.
\end{proof}

Again, there is a totally discontinuous action $\varphi$ of Frobenius.

\begin{defn} The relative Fargues-Fontaine curve $\cX(R,R^+)$ is the quotient $\cY_{(0,\infty)}(R,R^+)/\varphi^\mathbb{Z}$.
\end{defn}

As before, there is a line bundle $\cO_{\cX(R,R^+)}(d)$ for any $d\in \mathbb Z$, and one can form the scheme
\[
X(R) = \mathrm{Proj}\left(\oplus_{d\geq 0}H^0\left(\cX(R,R^+),\cO_{\cX(R,R^+)}(d)\right)\right).\footnote{As notation suggests, this does not depend on $R^+$.}
\]
This comes with a map of locally ringed topological spaces $\cX(R,R^+)\to X(R)$, and one has a relative GAGA result.

\begin{thm}[{\cite[Theorem 8.7.7]{kedlayaliu}}] The pullback functor from vector bundles on $X(R)$ to vector bundles on $\cX(R,R^+)$ is an equivalence of categories.
\end{thm}

Moreover, we can define $\tilde \cR^{r}_R$ as the inverse limit of the Banach algebras $\tilde \cR^{[s,r]}_R$ as $s$ runs over $(0,r]$ and the \emph{relative Robba ring} $\tilde \cR_R$ as the direct limit of the Fr\'echet algebras $\tilde \cR^{r}_R$ over $r>0$. Again, a $\varphi$-module over $\tilde \cR_R$ is a finite projective $\tilde \cR_R$-module $M$ equipped with a $\varphi$-linear automorphism.

\begin{thm}[{\cite[Theorems 6.3.12, 8.7.7]{kedlayaliu}}]\label{relative Robba ring equivalence} There is an equivalence of categories
\[\left\{\mathrm{Vector\ bundles\ on}\ \cX(R,R^+)\right\}\simeq\left\{\varphi\mathrm{-modules\ over}\ \tilde \cR_R\right\}.\] 
\end{thm}

\subsection{The mixed characteristic affine Grassmannian}\label{affinegrass}

Our goal in this section is to construct an isomorphism between the flag variety $\Fl_{G,\mu}$ and the Schubert cell corresponding to $\mu$ in the $B_{\mathrm{dR}}^+$-Grassmannian for $G$, \emph{assuming that $\mu$ is minuscule}. This is an analogue of a classical statement about the usual affine Grassmannian.

Throughout this section, $G$ is a connected reductive group over $\mathbb Q_p$. First, we define the version of the affine Grassmannian that we will consider. Let $(R,R^+)$ a perfectoid affinoid algebra over $\mathbb Q_p$, in the sense of~\cite[Definition 3.6.1]{kedlayaliu}.\footnote{If $R$ contains a perfectoid field, this agrees with the definition of~\cite{scholzeperfectoid}, and this case would suffice for our discussion here.} One has the surjective map $\theta: W(R^{\flat +})\to R^+$, whose kernel is generated by a non-zerodivisor $\xi\in W(R^+)$. Then $\mathbb B_{\mathrm{dR},R}^+$ is defined as the $\xi$-adic completion of $W(R^{\flat +})[1/p]$, and $\mathbb B_{\mathrm{dR},R} = \mathbb B_{\mathrm{dR},R}^+[\xi^{-1}]$. We note that, as notation suggests, these rings are independent of the choice of $R^+$.

\begin{defn} Let $\GrBdR_G$ be the functor associating to any perfectoid affinoid $\mathbb Q_p$-algebra $(R,R^+)$ the set of $G$-torsors over $\mathrm{Spec}\ \mathbb B_{\mathrm{dR},R}^+$ trivialized over $\mathrm{Spec}\ \mathbb B_{\mathrm{dR},R}$, up to isomorphism.
\end{defn}

We refer to~\cite{scholzelectures} for a more thorough discussion of this object, in the case $G=GL_n$.

If $(R,R^+) = (K,K^+)$ where $K$ is a perfectoid field, then $\mathbb B_{\mathrm{dR},K}^+$ is a complete discrete valuation ring, abstractly isomorphic to $K[[\xi]]$. In that case, one sees that
\[
\GrBdR_G(K,K^+) = G(\mathbb B_{\mathrm{dR},K}) / G(\mathbb B_{\mathrm{dR},K}^+)\ .
\]
In particular, assume that $K=C$ is algebraically closed, and fix an embedding $\bar{\mathbb Q}_p\to C$. Then, using the Cartan decomposition
\[
G(\mathbb B_{\mathrm{dR},C}) = \bigsqcup_{\mu\in X_\ast(G)_{\mathrm{dom}}} G(\mathbb B_{\mathrm{dR},C}^+) \mu(\xi)^{-1} G(\mathbb B_{\mathrm{dR},C}^+)
\]
(where the induced embedding $\bar{\mathbb Q}_p\hookrightarrow \mathbb B_{\mathrm{dR},C}^+$ is used to define $\mu(\xi)$ for a cocharacter $\mu: \mathbb G_m\to G_{\bar{\mathbb Q}_p}$), we can associate an element of $\mu(x)\in X_\ast(G)_{\mathrm{dom}}$ to any point of $x\in \GrBdR_G(C,\cO_C)$. This is the decomposition into Schubert cells.\footnote{We have inserted a slightly nonstandard sign in $\mu(\xi)^{-1}$.}

Now, we fix a conjugacy class $\mu$ of cocharacters $\mathbb{G}_m\to G_{\bar{\mathbb Q}_p}$, defined over $E$. In the following, we assume that $R$ is an $E$-algebra. Any choice of representative $\mu: \mathbb{G}_m\to G_{\bar{\mathbb Q}_p}$ in this conjugacy class determines an ascending filtration $\mathrm{Fil}_\bullet(\mu)$ on $\mathrm{Rep}_{\bar{\mathbb Q}_p} G$, where $\mathrm{Fil}_m(\mu)$ is the direct sum of all subspaces where $\mu$ acts through weights $m^\prime\geq -m$.\footnote{One reason the minus sign appears here is for consistenty with the global definitions, where type $(p,q)$ refers to characters $z\mapsto z^{-p}\overline{z}^{-q}$.} Let $\Fl_{G,\mu}/E$ be the rigid-analytic flag variety parametrizing all such filtrations. The choice of $\mu$ identifies $\Fl_{G,\mu} = G/P_\mu$, where $P_\mu\subset G$ is the stabilizer of $\mathrm{Fil}_\bullet(\mu)$.

\begin{defn} Let $\GrBdR_{G,\mu}\subset \GrBdR_G\otimes_{\mathbb Q_p} E$ be the subfunctor sending a perfectoid affinoid $E$-algebra $(R,R^+)$ to the set of those $G$-torsors over $\mathrm{Spec}\ \mathbb B_{\mathrm{dR},R}^+$ trivialized over $\mathrm{Spec}\ \mathbb B_{\mathrm{dR},R}$ whose relative position $\mu(x)$ is given by $\mu$, for all $x\in \mathrm{Spa}(R,R^+)$.
\end{defn}

\begin{prop} There is a natural Bialynicki-Birula map
\[
\pi_{G,\mu}: \GrBdR_{G,\mu}\to \Fl_{G,\mu}\ ,
\]
where we regard $\Fl_{G,\mu}$ as a functor on perfectoid affinoid $E$-algebras.
\end{prop}

\begin{proof} By the Tannakian formalism, it is enough to prove this result in the case $G=GL_n$. In that case, write $\mu=(k_1,\ldots,k_n)$ as a tuple of $n$ integers, $k_1\geq k_2\geq \ldots \geq k_n$. The functor $\GrBdR_{GL_n}$ parametrizes $\mathbb B_{\mathrm{dR},R}^+$-lattices $\Lambda\subset \mathbb B_{\mathrm{dR},R}^n$, i.e. finite projective submodules such that $\Lambda[1/\xi] = \mathbb B_{\mathrm{dR},R}^n$. Any such lattice gives rise to a filtration on $R^n$ by setting
\[
\mathrm{Fil}_m R^n = ((\mathbb B_{\mathrm{dR},R}^+)^n\cap \xi^{-m} \Lambda) / ((\xi \mathbb B_{\mathrm{dR},R}^+)^n\cap \xi^{-m} \Lambda)\ .
\]
Using the fact that a finitely generated $R$-module $M$ for which $\dim_{C(x)} M\otimes_R C(x)$ is the same for all $x=\mathrm{Spa}(C(x),\cO_{C(x)})\to \mathrm{Spa}(R,R^+)$ is finite projective, cf.~\cite[Proposition 2.8.4]{kedlayaliu}, one verifies that $R^n/\mathrm{Fil}_m R^n$ is a finite projective $R$-module for any $m$.

Note that $\mathrm{Fil}_\bullet R^n$ is an increasing filtration, where the rank of $\mathrm{Fil}_m R^n$ is given by the largest $i$ such that $k_i\geq -m$. The same type of filtrations is parametrized by $\Fl_{G,\mu}$, as desired.
\end{proof}

\begin{lemma}\label{BialynickiBirulaPoints} Assume that $\mu$ is minuscule, and that $(R,R^+)=(K,K^+)$, where $K/E$ is a perfectoid field. Then
\[
\pi_{G,\mu}: \GrBdR_{G,\mu}(K,K^+)\to \Fl_{G,\mu}(K,K^+)
\]
is a bijection.
\end{lemma}

\begin{proof} Recall that $\mathbb B_{\mathrm{dR},K}^+$ is a complete discrete valuation ring with residue field $K$. By the Cohen structure theorem, we may choose an isomorphism $\mathbb B_{\mathrm{dR},K}^+\cong K[[\xi]]$. This identifies
\[
\GrBdR_{G,\mu}(K,K^+) = G(K((\xi))) / G(K[[\xi]])\ ,
\]
and the Bialynicki-Birula morphism becomes the Bialynicki-Birula morphism for the usual affine Grassmannian for $G/\mathbb Q_p$. This is known to be an isomorphism, cf.~e.g.~\cite[Lemme 6.2]{ngopolo}.
\end{proof}

\begin{thm}\label{BialynickiBirulaIsom} Assume that $\mu$ is minuscule. Then the Bialynicki-Birula morphism
\[
\pi_{G,\mu}: \GrBdR_{G,\mu}\to \Fl_{G,\mu}
\]
is an isomorphism.
\end{thm}

\begin{proof} In the proof, we will use the Tannakian formalism. This interprets $\GrBdR_G$ as the associations mapping any $V\in \mathrm{Rep}\ G$ to a lattice $\Lambda_V\subset V\otimes \mathbb B_{\mathrm{dR}}$, compatibly with tensor products and short exact sequences.

First, let us check injectivity of $\pi_{G,\mu}$. Thus, take two $(R,R^+)$-valued points $x,y\in \GrBdR_{G,\mu}(R,R^+)$ which are sent to the same point of $\Fl_{G,\mu}$. We have to show that the corresponding lattices $\Lambda_{V,x}$, $\Lambda_{V,y}$ agree for all $V\in \mathrm{Rep}\ G$. But at any point $z\in \mathrm{Spa}(R,R^+)$ with completed residue field $K(z)$, Lemma~\ref{BialynickiBirulaPoints} implies that
\[
\Lambda_{V,x}\otimes_{\mathbb B_{\mathrm{dR},R}^+} \mathbb B_{\mathrm{dR},K(z)}^+ = \Lambda_{V,y}\otimes_{\mathbb B_{\mathrm{dR},R}^+} \mathbb B_{\mathrm{dR},K(z)}^+\ .
\]
One concludes that $\Lambda_{V,x} = \Lambda_{V,y}$ by applying the following lemma to all elements of $\Lambda_{V,x}$, and $\Lambda_{V,y}$.

\begin{lemma}\label{CheckContainmentAtPoints} Let $\Lambda$ be a finite projective $\mathbb B_{\mathrm{dR},R}^+$-module, and $a\in \Lambda\otimes_{\mathbb B_{\mathrm{dR},R}^+} \mathbb B_{\mathrm{dR},R}$ any element. Assume that for all $z\in \mathrm{Spa}(R,R^+)$ with completed residue field $K(z)$, $a\in \Lambda\otimes_{\mathbb B_{\mathrm{dR},R}^+} \mathbb B_{\mathrm{dR},K(z)}^+$. Then $a\in \Lambda$.
\end{lemma}

\begin{proof} We may choose $m\geq 0$ minimal such that $a\in \xi^{-m} \Lambda$, and assume $m>0$ for contradiction. Then $a$ induces a nonzero element $\bar{a}$ of the finite projective $R$-module $\xi^{-m} \Lambda / \xi^{-m+1} \Lambda$. By assumption, the specialization of $\bar{a}$ to $K(z)$ vanishes for all $z\in \mathrm{Spa}(R,R^+)$. But an element of $R$ vanishing at all points of $\mathrm{Spa}(R,R^+)$ is trivial, as $R$ is reduced.
\end{proof}

Now, to prove surjectivity, we first observe that $\GrBdR_G$ is in fact a sheaf for the pro-\'etale topology used in~\cite{scholzerigid}.\footnote{It is also a sheaf for stronger topologies as used in~\cite{scholzelectures}, but we do not need this here.} More precisely, we allow covers $Y=\mathrm{Spa}(S,S^+)\to X=\mathrm{Spa}(R,R^+)$ which can be written as a composite $Y\to Y_0\to X$, where $Y\to Y_0$ is an inverse limit of finite \'etale surjective maps, and $Y_0\to X$ is \'etale. This pro-\'etale topology of perfectoid spaces is defined in~\cite[\S 9.2]{kedlayaliu}. The descent result we need is~\cite[Theorem 9.2.15]{kedlayaliu}. Indeed, using the Tannakian formalism, it is enough to prove that one can glue finite projective $\mathbb B_{\mathrm{dR},R}^+$-modules in the pro-\'etale topology. As $\mathbb B_{\mathrm{dR},R}^+$ is $\xi$-adically complete with $\xi$ a non-zerodivisor and $\mathbb B_{\mathrm{dR},R}^+/\xi = R$, a standard argument reduces us to gluing finite projective $R$-modules, which is precisely~\cite[Theorem 9.2.15]{kedlayaliu}.

Thus, we see that it is enough to construct, for any representation $V$ of $G$, a $\mathbb B_{\mathrm{dR}}^+$-local system $\mathbb{M}_V\subset V\otimes \mathbb{B}_{\mathrm{dR}}$ on the pro-\'etale site of $\Fl_{G,\mu}$, compatibly with tensor products and short exact sequences, which maps to the correct filtration under the Bialynicki-Birula morphism. Indeed, by pullback, this will induce a similar $\mathbb B_{\mathrm{dR}}^+$-local system on the pro-\'etale site of $\mathrm{Spa}(R,R^+)$ for any $(R,R^+)$-valued point of $\Fl_{G,\mu}$, which by the descent result above gives an $(R,R^+)$-point of $\GrBdR_{G,\mu}$.

Now note that any representation $V$ of $G$ gives rise to a filtered module with integrable connection $(V\otimes \cO_{\Fl_{G,\mu}},\mathrm{id}\otimes \nabla,\mathrm{Fil}_{-\bullet})$, where $\mathrm{Fil}_\bullet$ is the universal ascending filtration parametrized by $\Fl_{G,\mu}$ (so that $\mathrm{Fil}_{-\bullet}$ is a descending filtration). Because $\mu$ is minuscule, this filtered module with integrable connection satisfies Griffiths transversality (with the same proof as in the complex case, cf.~\cite[Proposition 1.1.14]{deligne-varietes}). Now~\cite[Proposition 7.9]{scholzerigid} constructs a corresponding $\mathbb B_{\mathrm{dR}}^+$-local system $\mathbb{M}_V\subset V\otimes \mathbb{B}_{\mathrm{dR}}$ on the pro-\'etale site of $\Fl_{G,\mu}$, and this construction is compatible with tensor products and short exact sequences. One verifies that the induced filtration is correct, finishing the proof.
\end{proof}

\subsection{Vector bundles over $\cX$ and the Newton stratification}\label{Newton stratification}

The goal of this subsection is to define the Newton stratification on $\Fl_{G,\mu}$, where $G/\mathbb Q_p$ is a reductive group, and $\mu$ is a conjugacy class of minuscule cocharacters, defined over the reflex field $E$. The idea is that, given a $(C,\cO_C)$-point of $\Fl_{G,\mu}\cong \GrBdR_{G,\mu}$, one can modify the trivial $G$-bundle over $\cX_{C^\flat}$ along $\infty$ to obtain a new $G$-bundle over $\cX_{C^\flat}$, and therefore (by Fargues' theorem) an element of $B(G)$.

Fix any perfectoid affinoid $(R,R^+)$ over $\mathbb Q_p$. We recall how to construct a vector bundle over $\cX(R^\flat,R^{\flat +})$ from a $\mathbb B_{\mathrm{dR},R}^+$-lattice in $\mathbb{B}_{\mathrm{dR},R}^n$. First note that, by GAGA for the curve, it is enough to define a vector bundle on a scheme version $X(R^\flat)$ of $\cX(R^\flat,R^{\flat +})$. Let $Z$ be the image of the canonical closed immersion
\[
i_\infty: \mathrm{Spec}\ R\to X(R^\flat)\ .
\]
Then $\mathrm{Spec}\ \mathbb{B}^+_{\mathrm{dR},R}$ is the completion of $X(R^\flat)$ along $Z$. Moreover, $\mathrm{Spec}\ \mathbb{B}_{\mathrm{dR},R}$ can be identified with the fiber product of $\mathrm{Spec}\ \mathbb{B}^+_{\mathrm{dR},R}$ and the complement of $Z$ over $X(R^\flat)$.

\begin{thm}[{\cite[Theorem 8.9.6]{kedlayaliu}}]\label{BeauvilleLaszloRelCurve} There is an equivalence between the category of vector bundles over $X(R^\flat)$ (or over $\cX(R^\flat,R^{\flat +})$) and the category of triples $(M_1,M_2,\iota)$, where $M_1$ is a vector bundle on $X(R^\flat)\setminus Z$, $M_2$ is a vector bundle over $\mathrm{Spec}\ \mathbb B_{\mathrm{dR},R}^+$, and $\iota$ is an isomorphism between $M_1|_{\mathrm{Spec}\ \mathbb B_{\mathrm{dR},R}}$ and $M_2|_{\mathrm{Spec}\ \mathbb B_{\mathrm{dR},R}}$. This equivalence is compatible with tensor products and short exact sequences.
\end{thm}

In particular, one gets a functor from $\mathbb B_{\mathrm{dR},R}^+$-lattices in $\mathbb B_{\mathrm{dR},R}^n$ by gluing it to the trivial rank $n$ vector bundle on $X(R^\flat)\setminus Z$.

\begin{cor}\label{ModifVectBund} For any perfectoid affinoid $\mathbb Q_p$-algebra $(R,R^+)$, there is a natural map
\[
\cE: \GrBdR_G(R,R^+)\to \{G\mathrm{-bundles\ over\ }\cX(R^\flat,R^{\flat +})\}\ .
\]
\end{cor}

\begin{proof} If $G=GL_n$, this follows from the discussion above. In general, it follows from the Tannakian formalism.
\end{proof}

In particular, consider the case where $(R,R^+) = (C,\cO_C)$, with $C/\mathbb Q_p$ complete and algebraically closed, and $\cO_C\subset C$ its ring of integers; moreover, fix an embedding $\bar{\mathbb Q}_p\hookrightarrow C$. Using Fargues' classification of $G$-bundles, Theorem~\ref{classification of G-bundles}, one gets a composite map
\[
b(\cdot): \GrBdR_G(C,\cO_C)\to B(G)\ :\ x\mapsto b(\cE(x))
\]
classifying the isomorphism class of the associated $G$-bundle $\cE(x)$. We will need to know the following compatibility between $\mu$ and $b$.

\begin{prop}\label{InBGmu} Let $G$ be any reductive group over $\mathbb Q_p$, and $\mu$ any conjugacy class of cocharacters (not necessarily minuscule). For any $x\in \GrBdR_{G,\mu}(C,\cO_C)$ with $b=b(\cE(x))$, one has $b\in B(G,\mu^{-1})$.
\end{prop}

\begin{proof} Unraveling the definition of $B(G,\mu^{-1})$, we have to prove two separate statements. The first statement is $\nu_b\preceq \overline{\mu^{-1}}$ as elements of $(X_\ast(G)\otimes \mathbb{Q})^\Gamma_{\mathrm{dom}}$. This reduces to the case of $G=GL_n$ by~\cite[Lemma 2.2]{rapoport-richartz}. In that case, the statement is the following.

\begin{lemma}\label{Newton above Hodge}
Let $\cE$ be a vector bundle of rank $n$ over $X_{C^\flat}$, together with a trivialization outside the point $\infty$. Its relative position from the trivial bundle on $X_{C^\flat}$ is measured by a cocharacter $\mu(\cE)$ of $GL_n$. Let $\nu_{\cE}\in (X_\ast(GL_n)\otimes \mathbb{Q})_{\mathrm{dom}}$ be the Newton polygon of $\cE$, with slopes $\{\lambda_i\mid \cE\cong \bigoplus_i \cO_{X_{C^\flat}}(\lambda_i)\}$. One has the inequality
\[
\nu_{\cE}\preceq \mu(\cE)\ ,
\]
i.e.~``The Newton polygon of $\cE$ lies above its Hodge polygon''.\footnote{We remind the reader that the correspondence between isocrystals and vector bundles on $X_{C^\flat}$ reverses slopes, so that this statement translates into $b(\cE)^{-1}\in B(GL_n,\mu(\cE))$, which is equivalent to $b(\cE)\in B(GL_n,\mu(\cE)^{-1})$.}
\end{lemma}

\begin{proof} We adapt the original argument in~\cite{katz}. By considering exterior powers of vector bundles, it suffices to check that 
\begin{enumerate} 
\item the Newton and Hodge slopes match for the top exterior power of $\cE$, and
\item the first slope of the Newton polygon of $\cE$ always lies above the first slope of the Hodge polygon of $\cE$.
\end{enumerate}

The fact that the Hodge and Newton slopes match in the case of line bundles on $\cX_{FF,C}$ is a direct verification: The modification $\cE$ is given by the lattice $\cE\otimes_{\cO_{X^\flat}} \mathbb B_{\mathrm{dR},C}^+ = \xi^{-d} \mathbb B_{\mathrm{dR},C}$ for a unique $d\in \mathbb{Z}$, and in fact $\mu(\cE) = d\in X_\ast(GL_1) = \mathbb Z$ in our normalization. The resulting line bundle is given by $\cO_{X^\flat}(d)$, which is of slope $d$, as desired.

For the second part, up to twisting, we may assume that the first slope of the Hodge polygon is $0$, in particular all Hodge slopes are nonnegative. This implies that
\[
(\mathbb{B}^+_{\mathrm{dR}, C})^n\subseteq \cE\otimes_{\cO_{X_{C^\flat}}} \mathbb B_{\mathrm{dR},C}^+\ .
\]
This, in turn, implies that the trivialization of $\cE$ away from $\infty$ extends to an injection $\cO_{X_{C^\flat}}^n\hookrightarrow \cE$. We have to show that all slopes of $\cE$ are nonnegative, so assume for contradiction that there is a quotient $\cE\to \cO_{X_{C^\flat}}(\lambda)$ with $\lambda<0$. This induces a nonzero map $\cO_{X_{C^\flat}}^n\to \cO_{X_{C^\flat}}(\lambda)$. On the other hand, there are no nonzero maps $\cO_{X_{C^\flat}}\to \cO_{X_{C^\flat}}(\lambda)$ by~\cite{farguesfontaine}.
\end{proof}

The other part of the condition $b\in B(G,\mu^{-1})$ concerns the Kottwitz map, and is given by the following lemma.

\begin{lemma}\label{Kottwitz map} The composition $\GrBdR_{G,\mu}(C,\cO_C)\to B(G)\buildrel{\kappa}\over\longrightarrow \pi_1(G)_\Gamma$ is constant, and equal to $-\mu^\flat$.
\end{lemma}

\begin{proof} We note that the map in question is functorial in $(G,\mu)$. We first reduce to the case where $G$ has simply connected derived group by making a central extension $\tilde{G}\to G$ (cf.~\cite[5.6]{kottwitz}); picking any lift $\tilde{\mu}$ of $\mu$, the resulting map
\[
\GrBdR_{\tilde{G},\tilde{\mu}}(C,\cO_C)\to \GrBdR_{G,\mu}(C,\cO_C)
\]
is surjective, as follows from the Cartan decomposition, so it is enough to prove the result for $(\tilde{G},\tilde{\mu})$.

Now if $G$ has simply connected derived group $G^{\mathrm{der}}$, then $T=G/G^{\mathrm{der}}$ is a torus for which $\pi_1(G)_\Gamma\to \pi_1(T)_\Gamma$ is an isomorphism; thus, we are reduced to the case of a torus.

If $G=T$ is a torus, we may find a surjection $\tilde{T}\to T$, where $\tilde{T}$ is a product of induced tori $\mathrm{Res}_{K/\mathbb Q_p} \mathbb{G}_m$. Arguing as before, we are reduced to the case of $\tilde{T}$, and then to the case $\tilde{T} = \mathrm{Res}_{K/\mathbb Q_p} \mathbb G_m$. In that case, $\pi_1(\tilde{T})_\Gamma = \mathbb Z$ (cf.~\cite[Lemma 2.2]{kottwitz}), which is torsion-free, so it is enough to identify the image in $\pi_1(\mathbb G_m) = \mathbb Z$ under the norm map $\mathrm{Norm}_{K/\mathbb Q_p} : \tilde{T}\to \mathbb G_m$. Finally, we are reduced to the case $G=\mathbb G_m$, which is part of Lemma~\ref{Newton above Hodge}.
\end{proof}
\end{proof}

Now fix a minuscule $\mu$ as above, defined over $E$. The inverse of the isomorphism $\pi_{G,\mu}$ in Theorem~\ref{BialynickiBirulaIsom} gives rise to a composition
\[
\cE: \Fl_{G,\mu}(R,R^+)\to \GrBdR_{G,\mu}(R,R^+)\to \{G\mathrm{-bundles\ over\ }\cX(R^\flat,R^{\flat +})\}\ .
\]

\begin{defn} The map
\[
|\Fl_{G,\mu}|\to B(G)
\]
sends any $(C,C^+)$-valued point $x\in \Fl_{G,\mu}(C,C^+)$, where $C$ is a complete algebraically closed extension of $E$ and $C^+\subset C$ is an open and bounded valuation subring, to the isomorphism class of the associated $G$-bundle $\cE(x)$, which by Theorem~\ref{classification of G-bundles} is given by an element of $B(G)$.

For any $b\in B(G)$, we let $\Fl_{G,\mu}^b\subset \Fl_{G,\mu}$ be the subset of all points with image $b$.
\end{defn}

One easily checks that this map is well-defined as a map on $|\Fl_{G,\mu}|$, i.e.~is independent of the choice of complete algebraically closed extension of the residue field at any point. We remark that by definition a higher rank point has the same image as its maximal, rank $1$, generalization, and therefore the map factors over the maximal hausdorff quotient of $|\Fl_{G,\mu}|$, which can be identified with the topological space $\Fl_{G,\mu}^{\mathrm{Berk}}$ underlying the corresponding Berkovich space.

\begin{prop}\label{map to B(G)}\begin{enumerate} \item The map $b(\cdot): |\Fl_{G,\mu}| \to B(G)$ is lower semicontinuous.
\item The image of the map $b(\cdot): |\Fl_{G,\mu}|\to B(G)$ is contained in the set of $\mu^{-1}$-admissible elements $B(G,\mu^{-1})$.
\end{enumerate}
\end{prop}

\begin{remark} In~\cite[Proposition A.9]{rapoportappendix}, based on the discussion here, it is proved that in fact the image of $|\Fl_{G,\mu}|\to B(G,\mu^{-1})$ is all of $B(G,\mu^{-1})$.
\end{remark}

\begin{proof} The second part follows from Proposition~\ref{InBGmu} above. For the first part, by the definition of the partial ordering on $B(G)$, and the fact that the Kottwitz map is constant by the second part, it remains to prove semicontinuity of the Newton map. We may pick an affinoid perfectoid space $\mathrm{Spa}(R,R^+)$ with a map to $\Fl_{G,\mu}$ which is a topological quotient map, by using a pro-\'etale cover.  It is then enough to show that the composite map $|\mathrm{Spa}(R,R^+)|\to |\Fl_{G,\mu}|\to B(G)$ is lower semicontinuous. But semicontinuity of the Newton map can be checked on representations of $G$ (cf.~\cite[Lemma 2.2]{rapoport-richartz}), so pick a representation of $G$. We get a corresponding vector bundle over $\cX(R^\flat,R^{\flat +})$. Now, the result follows from Theorem 7.4.5 of~\cite{kedlayaliu}, using Corollary~\ref{relative Robba ring equivalence}.
\end{proof}

\begin{cor}\label{locally closed strata} The strata $\Fl^b_{G,\mu}$ are locally closed in $\Fl_{G,\mu}$. More precisely, the stratum corresponding to the basic element is open in $\Fl_{G,\mu}$, and the strata \[\Fl^{\succeq b}_{G,\mu}:=\bigsqcup_{b\preceq b'}\Fl^{b'}_{G,\mu}\] are closed.    
\end{cor}

\begin{proof} This follows immediately from Proposition~\ref{map to B(G)}.
\end{proof}
\newpage

\section{The geometry of Newton strata and Igusa varieties}

In this section, we will return to the global setup, but will in addition assume that the Shimura datum $(G,X)$ is of \emph{PEL type}, and has good reduction at $p$. This means that they will admit smooth integral models which are moduli spaces of abelian varieties equipped with \emph{p}olarizations, \emph{e}ndomorphisms and \emph{l}evel structure. Our goal is to understand the fibers of the Hodge-Tate period map \[\pi_{HT}:\cS_{K^p}\to \Fl_{G,\mu}\] defined in Theorem~\ref{refinedht} in terms of the Igusa varieties introduced by Mantovan,~\cite{mantovan}.

We start with some preliminaries on $p$-divisible groups, which recall material from~\cite{scholzeweinstein} as well as a construction of Chai and Oort. We then express the Newton strata in $\cS_{K^p}$ in terms of Rapoport-Zink spaces and Igusa varieties, in the spirit of~\cite{mantovan}.

\subsection{Preliminaries on $p$-divisible groups}\label{$p$-divisible groups}

We recall the notions of Tate module and universal cover of a $p$-divisible group as used in~\cite{scholzeweinstein}, together with some of their properties. Let $\mathrm{Nilp}$ be the category of $\mathbb Z_p$-algebras on which $p$ is nilpotent. If $R$ is a $p$-adically complete $\mathbb Z_p$-algebra, let $\mathrm{Nilp}^{\mathrm{op}}_R$ be the opposite category to the category of $R$-algebras on which $p$ is nilpotent. A $p$-divisible group $\cG$ can be thought of as an fpqc sheaf on $\mathrm{Nilp}^{\mathrm{op}}_R$ sending an $R$-algebra $S$ to $\varinjlim \cG[p^n](S)$.

\begin{defn}\label{Tate module and universal cover}\begin{enumerate}\item The fpqc sheaf $T_p(\cG)(S)=\varprojlim_n\cG[p^n](S)$ on $\mathrm{Nilp}^{\mathrm{op}}_R$ is called the (integral) Tate module of $\cG$. 
\item The fpqc sheaf $\tilde \cG(S)=\varprojlim_{p:\cG\to \cG} \cG(S)$ on $\mathrm{Nilp}^{\mathrm{op}}_R$ is called the universal cover of $\cG$. 
\end{enumerate}
\end{defn}

\noindent We note that $T_p(\cG)$ is a sheaf of $\mathbb{Z}_p$-modules, while $\tilde \cG = T_p(\cG)[1/p]$ is a sheaf of $\mathbb{Q}_p$-vector spaces. We can canonically identify \[T_p\cG=\mathscr{H}om(\mathbb{Q}_p/\mathbb{Z}_p,\cG), \tilde \cG=\mathscr{H}om(\mathbb{Q}_p/\mathbb{Z}_p,\cG)[1/p].\]

\begin{prop}\label{properties of p-divisible groups}\begin{enumerate}\item If $\cG$ is connected, then it is representable by an affine formal scheme with finitely generated ideal of definition. If $\mathrm{Lie}\ \cG$ is free of dimension $r$ then \[\cG\simeq \mathrm{Spf}\ R[[x_1,\dots,x_r]].\]
\item If $\rho :\cG_1 \to \cG_2$ is an isogeny, then the induced morphism $\tilde \rho : \tilde\cG_1\to \tilde \cG_2$ is an isomorphism.
\item If $R$ is perfect of characteristic $p$, $\cG$ is connected and $\mathrm{Lie}\ \cG$ is free of dimension $r$ then \[\tilde \cG\simeq \mathrm{Spf}\ R[[x^{1/p^\infty}_1,\dots, x^{1/p^\infty}_r]].\]
\item If $R$ is perfect of characteristic $p$, $\cG$ is connected and $\mathrm{Lie}\ \cG$ is free of dimension $r$ then \[T_p\cG\simeq \mathrm{Spec}\ R[[x^{1/p^\infty}_1,\dots, x^{1/p^\infty}_r]]/(x_1,\dots,x_r).\]
\end{enumerate}
\end{prop}
\begin{proof} The first part is proved in~\cite{messing}. The remaining results are proved in~\cite{scholzeweinstein}: the second and third parts in Proposition 3.1.3 and the fourth part follows from the first part, the third part and the short exact sequence of sheaves on $\mathrm{Nilp}^{\mathrm{op}}_R$ given by \[0\to T_p\cG\to \tilde \cG\to \cG\to 0.\] (This short exact sequence is a restatement of Proposition 3.3.1 of~\cite{scholzeweinstein} in the case when $\cG$ is connected: the Tate module is the closed subfunctor of $\tilde \cG$ given by pullback along the natural map $\tilde \cG\to \cG$ - projection onto the last coordinate - from the zero section in $\cG$.) \end{proof}

The universal vector extension $E\cG$ of $\cG$ is a crystal on the nilpotent crystalline site of $R$ defined in~\cite{messing}. Its Lie algebra $\mathrm{Lie}\ E\cG$ can be made into a crystal on the crystalline site of $R$ by~\cite{bbm}, which we will denote by $\mathbb{M}(\cG)$. 

If $\cG$ is a $p$-divisible group over $\bar{\mathbb{F}}_p$, the Dieudonn\'e module $D(\cG)$ is obtained by evaluating the crystal $\mathbb{M}(\cG)$ on the PD thickening $W(\bar{\mathbb{F}}_p)\to \bar{\mathbb{F}}_p$. Then $D(\cG)[1/p]$ is an isocrystal over $L$, as defined in Section 4. Here, the Frobenius $\varphi_\cG$ on $D(\cG)[1/p]$ satisfies
\[
D(\cG)\subset \varphi_\cG(D(\cG))\subset p^{-1} D(\cG)\ ,
\]
and $p \varphi_\cG$ is the Frobenius usually considered.\footnote{If one uses the usual Frobenius on contravariant Dieudonn\'e theory, then our convention corresponds to defining covariant Dieudonn\'e theory as the literal dual of contravariant Dieudonn\'e theory, i.e.~without a Tate twist.} We will call a $p$-divisible group $\cG$ over $\bar{\mathbb{F}}_p$ \emph{isoclinic} if the corresponding isocrystal has only one slope. If this slope is given by $-\lambda$, we say that $\cG$ is isoclinic of slope $\lambda$, so that $\mu_{p^\infty}$ is isoclinic of slope $1$.

Given a $p$-divisible group $\cG$ over $\bar{\mathbb F}_p$, we can use the isocrystal $D(\cG)[1/p]$ to construct a vector bundle $\cE(\cG)$ over the Fargues-Fontaine curve $\cX_F$, for any complete algebraically closed nonarchimedean field $F\supset \bar{\mathbb F}_p$.

\begin{example} If $\cG = \mathbb Q_p/\mathbb Z_p$, then $D(\cG) = L$ with $\varphi_\cG = \sigma$, and $\cE(\cG) = \cO_{\cX_F}$. If $\cG = \mu_{p^\infty}$, then $D(\cG) = L$ with $\varphi_\cG = p^{-1}\sigma$, and $\cE(\cG) = \cO_{\cX_F}(1)$.
\end{example}

On the other hand, one can use the schematic version of the Fargues-Fontaine curve to build a vector bundle corresponding to a $p$-divisible group over $\cO_C/p$, where $C$ is any complete algebraically closed extension of $\mathbb Q_p$ with ring of integers $\cO_C/p$. Define $A_{\mathrm{cris}}$ to be the $p$-adic completion of the PD envelope of the surjection $W(\cO_C^\flat)\twoheadrightarrow \cO_C/p$ and $B^+_{\mathrm{cris}}:=A_{\mathrm{cris}}[1/p]$. If $\cG$ is a $p$-divisible group over the semiperfect ring $\cO_C/p$, then its Dieudonn\'e module is a finite projective $A_{\mathrm{cris}}$-module $M(\cG)$ obtained by evaluating $\mathbb{M}(\cG)$ on the PD thickening $A_{\mathrm{cris}}\to \cO_C/p$. Then $M(\cG)[1/p]$ is a $B^+_{\mathrm{cris}}$-module equipped with a Frobenius-semilinear map $\varphi_\cG$. Recall, cf.~\cite{farguesfontaine}, that the schematic Fargues-Fontaine curve can also be defined as
\[
X_{C^\flat}=\mathrm{Proj}\left(\oplus_{d\geq 0}\left(B^+_{\mathrm{cris}}\right)^{\varphi=p^d}\right).
\]
We associate to $\cG$ the vector bundle $E(\cG)$ on $X_{C^\flat}$ corresponding to the graded module \[\oplus_{d\geq 0}\left(M(\cG)[1/p]\right)^{\varphi=p^d}.\]
 
\begin{thm}\label{p-divisible groups and vector bundles}\begin{enumerate}
\item For any $p$-divisible group $\cG$ over $\cO_C/p$, there exists a $p$-divisible group $\cH$ over $\mathbb{\bar F}_p$ and a quasi-isogeny \[\rho : \cH\times_{\mathbb{\bar F}_p}\cO_C/p\to \cG\]
\item The functor $\cG\mapsto E(\cG)$ from $p$-divisible groups over $\cO_C/p$ up to isogeny to vector bundles on $X_{C^\flat}$ is fully faithful, with essential image the vector bundles whose slopes are all between $0$ and $1$.
\item Let $\cG$ be a $p$-divisible group over $\mathbb{\bar F}_p$. Then GAGA for the curve identifies $\cE(\cG)$ with $E(\cG)$. 
\end{enumerate}
\end{thm}
\begin{proof} The first two parts are Theorem 5.1.4 of~\cite{scholzeweinstein}. The last part is clear.
\end{proof}

We now specialize to $p$-divisible groups over a perfect field $k$. (Since every $p$-divisible group over $\cO_C/p$ is quasi-isogenous to one defined over $\bar {\mathbb{F}}_p$, if we are interested in understanding quasi-self-isogenies, it is enough to restrict to this case.) Let $\cG,\cG'$ be two isoclinic $p$-divisible groups over $k$. Our goal is to define an ``internal Hom'' $p$-divisible group $\cH_{\cG,\cG'}$ over $k$ satisfying the following two properties:
\begin{enumerate}
\item The Tate module $T_p(\cH_{\cG,\cG'})$ can be identified with the sheaf $\mathscr{H}om(\cG,\cG')$. 
\item The Dieudonn\'e module $D(\cH_{\cG,\cG'})[1/p]$ is equal to \[\Hom(D(\cG)[1/p], D(\cG')[1/p])^{\leq 0},\] where the latter denotes the internal homomorphism in Dieudonn\'e modules, and we are taking the slope $\leq 0$-part.
\end{enumerate}

In a talk of C.-L.~Chai at the Faltings conference 2014, we learnt that a $p$-divisible group satisfying these properties has been defined by Chai and Oort. We explain their construction below.  

We define $\cH_{\cG,\cG'}$ as an inductive system of finite group schemes. For each $n\geq 1$ consider the commutative group schemes of finite type over $k$ defined as \[\cH_n:=\mathscr{H}om(\cG[p^n],\cG'[p^n]).\]
For $m\geq n$, there are natural restriction maps
\[
r_{m,n}: \cH_m\to \cH_n
\]
which restrict a homomorphism $\cG[p^m]\to \cG^\prime[p^m]$ to $\cG[p^n]\subset \cG[p^m]$. The kernel $\mathrm{ker}\ r_{m,n}\subset \cH_m$ is a closed subgroup scheme. As we are working over a field, one can form the qoutient $\cH_n^{(m)} = \cH_m / \mathrm{ker}\ r_{m,n}$, which is a subgroup scheme of $\cH_n$. As $m$ increases, they form a descending chain.

\begin{lemma}\label{precise chai-oort} The subgroup scheme $\cH_n^{(m)}$ stabilizes for $m\gg 0$; let $\cH_n^\prime = \cH_n^{(m)}$ for $m$ sufficiently large. Then $\cH'_n$ is a finite group scheme over $k$.
\end{lemma}

\begin{proof} We may assume that $k$ is algebraically closed. First, we claim that $\cH_n^{(m)}$ is a finite group scheme for $m\gg 0$. It is enough to see that $\cH_n^{(m)}(k)$ is finite. By Dieudonn\'e theory, one sees that $\Hom(\cG,\cG^\prime)$ is a finite free $\mathbb Z_p$-module, independent of the algebraically closed field $k$. In particular, the image $\cH_n(k)_\infty\subset \cH_n(k)$ of
\[
\Hom(\cG,\cG^\prime)\to \cH_n(k)
\]
is finite, and independent of $k$. Now the sequence of $\cH_m\times_{\cH_n} (\cH_n\setminus \cH_n(k)_\infty)$ forms a cofiltered system of quasicompact schemes with affine transition maps and with empty inverse limit. It follows that one of the schemes is already empty, showing that the image of $\cH_m(k)\to \cH_n(k)$ agrees with the finite set $\cH_n(k)_\infty$.

Now, the $\cH_n^{(m)}$ form a decreasing sequence of finite group schemes over $k$. As such, they are eventually constant, e.g.~by looking at their order.
\end{proof}

We define $\iota_n: \cH_n\to \cH_{n+1}$ to be the map given by pre-composition with the multiplication by $p$ map $\cG[p^{n+1}] \to \cG[p^{n}]$ followed by composition with the inclusion $\cG'[p^n] \hookrightarrow \cG'[p^{n+1}]$.

\begin{lemma} The maps $\iota_n: \cH_n\to \cH_{n+1}$ send $\cH^\prime_n$ into $\cH^\prime_{n+1}$. The colimit
\[
\cH = \cH_{\cG,\cG^\prime} = \varinjlim_{\iota_n} \cH^\prime_n
\]
is a $p$-divisible group over $k$ with $\cH[p^n] = \cH^\prime_n$.
\end{lemma}

\begin{proof} From the commutation between $\iota_n$ and $r_{m,n}$, one infers that $\iota_n$ sends $\cH^\prime_n$ into $\cH^\prime_{n+1}$. First, we check that $\iota_n: \cH^\prime_n\to \cH^\prime_{n+1}$ is injective with image $\cH^\prime_{n+1}[p^n]$. Let $S$ be any $k$-scheme. If $f: \cG[p^n]_S\to \cG^\prime[p^n]_S$ induces the trivial map
\[
\cG[p^{n+1}]_S\buildrel p\over\longrightarrow \cG[p^n]_S\buildrel f\over\longrightarrow \cG^\prime[p^n]_S\hookrightarrow \cG^\prime[p^{n+1}]_S\ ,
\]
then $f=0$ as the first map is surjective, and the last injective; this proves injectivity of $\iota_n$. Now let $f: \cG[p^{n+1}]_S\to \cG^\prime[p^{n+1}]_S$ be a map killed by $p^n$, which for any $m\geq n$ lifts fppf locally to a map $f_m: \cG[p^{m+1}]_S\to \cG^\prime[p^{m+1}]_S$. It follows that $f$ factors uniquely as
\[
\cG[p^{n+1}]_S\buildrel p\over\longrightarrow \cG[p^n]_S\buildrel g\over\longrightarrow \cG^\prime[p^n]_S\hookrightarrow \cG^\prime[p^{n+1}]_S\ ,
\]
for some $g: \cG[p^n]\to \cG^\prime[p^n]$, as $f$ has image in the $p^n$-torsion, and kills $p^n \cG[p^{n+1}] = \cG[p]$. Similarly, any lift $f_m: \cG[p^{m+1}]_S\to \cG^\prime[p^{m+1}]_S$ of $f$ is killed by $p^m$, which implies that $f_m$ factors uniquely through a map $g_m: \cG[p^m]\to \cG^\prime[p^m]$, which necessarily lifts $g$. This shows that $\cH^\prime_n = \cH^\prime_{n+1}[p^n]$.

Moreover, we need to see that $p: \cH^\prime_{n+1}\to \cH^\prime_{n+1}$ has image $\cH^\prime_n$; by the above, it follows that the image is contained in $\cH^\prime_n$; the resulting map $\cH^\prime_{n+1}\to \cH^\prime_n$ is in fact the map $r_{n+1,n}$. By construction of the $\cH^\prime_n$, the map $r_{n+1,n}$ is indeed surjective, finishing the proof.
\end{proof}

\begin{lemma}\label{Tate module} The Tate module $T_p\cH_{\cG,\cG'}$ can be identified with the sheaf $\mathscr{H}om(\cG,\cG')$. 
\end{lemma}
\begin{proof} The Tate module  $T_p\cH_{\cG,\cG'}$ is the inverse limit of $\cH_{\cG,\cG'}[p^n] \simeq \cH'_n$ with respect to the $r_{n+1,n}$ maps. This, by definition is the same as the inverse limit of the projective system of $\cH_n$'s with respect to the $r_{n+1,n}$ maps, which is the sheaf $\mathscr{H}om(\cG,\cG')$. 
\end{proof}

\begin{lemma}\label{dieudonne module} The Dieudonn\'e module $D(\cH_{\cG,\cG'})[1/p]$ is equal to \[\Hom(D(\cG)[1/p], D(\cG')[1/p])^{\leq 0},\] where $\Hom(D(\cG)[1/p], D(\cG')[1/p])$ is the internal homomorphism in Dieudonn\'e modules, and we are taking the slope $\leq 0$-part. 
\end{lemma}
\begin{remark} Note that the statement only depends on $\cG$ and $\cG'$ up to quasi-isogeny. Chai and Oort prove Lemma~\ref{dieudonne module} by directly computing the relative Frobenius on $\cH_{\cG,\cG'}$ in terms of the relative Frobenius on conveniently chosen $\cG$ and $\cG'$. We give a different proof below. Also, Chai-Oort give an integral version of Lemma~\ref{dieudonne module}.
\end{remark}

\begin{proof} Let $\cH_D$ be a $p$-divisible group over $k$ with rational Dieudonn\'e module
\[
\Hom(D(\cG)[1/p], D(\cG')[1/p])^{\leq 0}\ .
\]
First, we construct a natural map
\[
\tilde{\cH}_D\to \tilde{\cH}_{\cG,\cG^\prime} = \mathscr{H}om(\cG,\cG^\prime)[1/p]\ .
\]
In order to construct such a map, it is enough to construct a functorial map on $R$-valued points, where $R$ is f-semiperfect in the sense of~\cite[Definition 4.1.2]{scholzeweinstein}, as $\tilde{\cH}_{\cG,\cG^\prime}$, like the universal cover of any $p$-divisible group, is represented by a formal scheme which is locally of the form $\mathrm{Spf}\ S$, where $S$ is an inverse limit of f-semiperfect rings.

Thus, let $R$ be f-semiperfect, with associated $B_{\mathrm{cris}}^+(R)$. Then by~\cite[Theorem A]{scholzeweinstein}, we have
\[\begin{aligned}
\tilde{\cH}_{\cG,\cG^\prime}(R) &= \Hom_R(\cG,\cG^\prime)[1/p] = \Hom_{B_{\mathrm{cris}}^+(R),\varphi}(D(\cG)\otimes B_{\mathrm{cris}}^+(R),D(\cG^\prime)\otimes B_{\mathrm{cris}}^+(R))\\
& = (\Hom(D(\cG)[1/p],D(\cG^\prime)[1/p])\otimes B_{\mathrm{cris}}^+(R))^{\varphi = 1}\ ,
\end{aligned}\]
and
\[\begin{aligned}
\tilde{\cH}_D(R) &= \Hom_R(\mathbb Q_p/\mathbb Z_p,\cH_D)[1/p] = (D(\cH_D)\otimes B_{\mathrm{cris}}^+(R))^{\varphi = 1} \\
&= (\Hom(D(\cG)[1/p],D(\cG^\prime)[1/p])^{\leq 0}\otimes B_{\mathrm{cris}}^+(R))^{\varphi = 1}\ .
\end{aligned}\]
Now the obvious inclusion
\[
\Hom(D(\cG)[1/p],D(\cG^\prime)[1/p])^{\leq 0}\subset \Hom(D(\cG)[1/p],D(\cG^\prime)[1/p])
\]
induces the desired map $\tilde{\cH}_D\to \tilde{\cH}_{\cG,\cG^\prime}$.

To check that this is an isomorphism, it suffices by the same argument to check on $R$-valued points, where $R$ is f-semiperfect. Thus, it remains to see that
\[
(\Hom(D(\cG)[1/p],D(\cG^\prime)[1/p])\otimes B_{\mathrm{cris}}^+(R))^{\varphi = 1} = (\Hom(D(\cG)[1/p],D(\cG^\prime)[1/p])^{\leq 0}\otimes B_{\mathrm{cris}}^+(R))^{\varphi = 1}\ .
\]
For this, it suffices to see that for any Dieudonn\'e module $D$ with only positive slopes,
\[
(D\otimes B_{\mathrm{cris}}^+(R))^{\varphi = 1} = 0\ .
\]
For this, using the Dieudonn\'e-Manin classification, we have to see that there are no elements $x\in A_{\mathrm{cris}}(R)$ with $p^a \varphi^b(x) = x$, where $a,b>0$ are positive integers. Note that $\varphi$ preserves the $p$-adically complete ring $A_{\mathrm{cris}}(R)$; on the other hand, the equation on $x$ implies $x=p^{ma} \varphi^{mb}(x)$ for any $m\geq 1$, so that $x$ is infinitely divisible by $p$, which implies $x=0$.
\end{proof}

\begin{cor}\label{the hom p-divisible group} Assume that $\cG$ and $\cG^\prime$ are isoclinic. \begin{enumerate}\item If the slope of $\cG$ is strictly greater than the slope of $\cG'$, then $\cH_{\cG,\cG'}$ vanishes. 
\item If the slopes of $\cG$ and $\cG'$ are equal, then $\cH_{\cG,\cG'}$ is an \'etale $p$-divisible group. 
\item If the slope of $\cG$ is strictly less than the slope of $\cG'$, then $\cH_{\cG,\cG'}$ is a connected $p$-divisible group. 
\end{enumerate}
\end{cor}

\begin{cor}\label{representability} If $\cG$ and $\cG^\prime$ are isoclinic and the slope of $\cG$ is strictly less than the slope of $\cG'$ and $\cH_{\cG,\cG'}$ has dimension $r$, then the sheaf $\mathscr{H}om(\cG,\cG')$ is representable by the scheme \[\mathrm{Spec}\ k[[x^{1/p^\infty}_1,\dots,x^{1/p^\infty}_r]]/(x_1,\dots,x_r).\] 
\end{cor}
\begin{proof} This follows from Proposition~\ref{properties of p-divisible groups} and Corollary~\ref{the hom p-divisible group}.
\end{proof}

\subsection{Rapoport-Zink spaces of PEL type}\label{section on rapoport-zink spaces}

In this section, we introduce the Rapoport-Zink spaces of PEL type that we will consider, and recall some of the results we will need. In close analogy to the EL case treated in~\cite{scholzeweinstein}, we define a local avatar of the Hodge-Tate period morphism, mapping the infinite-level Rapoport-Zink space to $\Fl_{G,\mu}$.

We first introduce PEL structures, as in~\cite{rapoport-zink}, with several simplifying assumptions that will be verified in the global case that we want to consider. Fix a finite-dimensional, semisimple algebra $B$ over $\mathbb{Q}_p$, endowed with an anti-involution $*$, and a finite left $B$-module $V$ equipped with an alternating bilinear form
\[
(\cdot,\cdot):V\otimes_{\mathbb{Q}_p}V\to \mathbb{Q}_p
\]
such that $(bv,w)=(v,b^*w)$ for all $v,w\in V$, $b\in B$. The data so far define an algebraic group $G$ over $\mathbb{Q}_p$, whose values over a $\mathbb{Q}_p$-algebra $R$ are
\[
G(R)=\{(g,c)\in GL_{B\otimes R}(V\otimes R)\times R^\times\mid (gv,gw) = c(v,w)\}\ .
\]
We refer to $c: G\to \mathbb{G}_m$ as the multiplier character of $G$. We make the general assumption that $G$ is connected, which amounts to excluding type D in the classification.

Moreover, we assume that the data are unramified. More precisely, we assume that $B$ is a product of matrix algebras over unramified extensions of $\mathbb Q_p$, and admits a $\ast$-stable maximal $\mathbb Z_p$-order $\cO_B\subset B$, which we fix. Moreover, we assume that there is an $\cO_B$-stable lattice $\Lambda \subset V$, which is self-dual under $(\cdot,\cdot)$; again, we fix such a lattice $\Lambda$. These data define a reductive group $G_{\mathbb Z_p}$ over $\mathbb Z_p$ via
\[
G(R) = \{(g,\lambda)\in GL_{\cO_B\otimes R}(\Lambda\otimes R)\times R^\times\mid (gv,gw) = \lambda (v,w)\}\ .
\]

Now also fix a conjugacy class of cocharacters $\mu:\mathbb{G}_m\to G_{\mathbb{\bar Q}_p}$ such that in the induced weight decomposition of $V_{\bar{\mathbb Q}_p}$, only weights $0$ and $1$ appear,
\[
V_{\bar{\mathbb Q}_p} = V_0\oplus V_1\ ,
\]
and $\lambda\circ \mu: \mathbb G_m\to \mathbb G_m$ is the identity morphism. This implies, in particular, that the subspaces $V_0$ and $V_1$ are totally isotropic. We let $E/\mathbb{Q}_p$ be the field of definition of $\mu$. Finally, we fix an element $b\in G(L)$, satisfying the compatibility $b\in B(G,\mu^{-1})$. Set $\breve E:=E\cdot L$.

Note that the condition $b\in B(G,\mu^{-1})$ together with the condition on the weights of $\mu$ on $V$ imply that the slopes of $b$ on $V$ are in $[-1,0]$. In particular, in our (nonstandard) normalization of the covariant Dieudonn\'e module, there is a $p$-divisible group $\mathbb X_b$ over $\bar{\mathbb F}_p$ whose rational Dieudonn\'e module is given by
\[
\left(V\otimes_{\mathbb{Q}_p}L, b(\mathrm{id}\otimes \sigma)\right)\ ;
\]
then $\mathbb X_b$ is uniquely determined up to isogeny, and its universal cover $\widetilde{\mathbb X}_b$ is uniquely determined. By functoriality, $\mathbb X_b$ is equipped with an action $\iota: B\to \mathrm{End}(\widetilde{\mathbb{X}}_b)$ and with a symmetric polarization (i.e. an anti-symmetric quasi-isogeny to its dual), with induced Rosati involution being compatible with $\ast$ on $B$.

Write $\cD=(B,*,V,(\cdot,\cdot),b, \mu)$ for the rational data and $\cD^\mathrm{int}= (\cO_B,*,\Lambda, (\cdot,\cdot),b,\mu)$ for the integral data.

\begin{defn} The Rapoport-Zink space $\mathfrak{M}_{\cD^\mathrm{int}}$ of PEL type associated to $\cD^\mathrm{int}$ is the functor on $\mathrm{Nilp}_{\cO_{\breve E}}^{\mathrm{op}}$ sending an $\cO_{\breve E_0}$-algebra $R$ to the set of isomorphism classes of pairs $(\cG,\rho)$, where $\cG$ is a $p$-divisible group over $R$ equipped with an action of $\cO_B$ and a principal polarization whose induced Rosati involution is compatible with $\ast$ on $\cO_B$, such that the $\cO_B$-action satisfies the \emph{determinant condition} (see 3.23 in~\cite{rapoport-zink} for a precise formulation), and
\[
\rho:\mathbb{X}_b\times_{\mathbb{\bar F}_p}R/p\to \cG\times_{R}R/p
\]
is a quasi-isogeny compatible with the $\cO_B$-action and the polarization, up to an automorphism of $\widetilde{\mu}_{p^\infty,R/p}$.
\end{defn}

The following combines Theorem 3.25 and \S 3.82 of~\cite{rapoport-zink}. In our unramified situation, and excluding type D, we may allow $p=2$.

\begin{thm} The functor $\mathfrak{M}_{\cD^{\mathrm{int}}}$ is representable by a formal scheme which locally admits a finitely generated ideal of definition. Moreover, $\mathfrak{M}_{\cD^{\mathrm{int}}}$ is formally smooth.
\end{thm}

We let $\cM_{\cD^\mathrm{int}}:=(\mathfrak{M}_{\cD^{\mathrm{int}}})^{\mathrm{ad}}_{\eta}$ be the adic generic fiber associated to the formal scheme (representing) $\mathfrak{M}_{\cD^\mathrm{int}}$. The adic generic fiber is taken in the sense of Section 2 of~\cite{scholzeweinstein}: Proposition 2.2.1 of loc. cit. gives a fully faithful functor \[\mathfrak{M}\mapsto \mathfrak{M}^{\mathrm{ad}}\] from formal schemes over $\cO_{\breve E}$ which locally admit a finitely generated ideal of definition to adic spaces over $\mathrm{Spa}(\cO_{\breve{E}},\cO_{\breve{E}})$, and \[\mathfrak{M}_\eta^{\mathrm{ad}}:=\mathfrak{M}^\mathrm{ad}\times_{\mathrm{Spa}(\cO_{\breve E}, \cO_{\breve E})}\mathrm{Spa}(\breve E, \cO_{\breve E}).\] 
Then $\cM_{\cD^{\mathrm{int}}}$ agrees with the adic space corresponding to the usual rigid-analytic generic fibre of $\mathfrak{M}_{\cD^{\mathrm{int}}}$.

For each $n\geq 1$, one can define a cover $\cM_{\cD^\mathrm{int},n}$ of $\cM_{\cD^\mathrm{int}}$ which parametrizes full level $n$ structures. More precisely, define the compact open subgroups \[K_0:=\{g\in G(\mathbb{Q}_p) \mid g\Lambda = \Lambda\}\] and \[K_n:=\{g\in K_0 \mid g\equiv 1\pmod{p^n} \}.\] Let $\cM_{\cD^\mathrm{int},n}$ be the functor on complete affinoid $(\breve E(\zeta_{p^n}),\cO_{\breve E(\zeta_{p^n})})$-algebras parametrizing $\cO_B$-linear maps \[\Lambda/p^n\to \cG[p^n]^\mathrm{ad}_{\eta}(R,R^+),\] which match the pairing $(\cdot,\cdot)$ on $\Lambda$ with the one induced by the polarization on $\cG[p^n]$. Here, note that the second pairing takes values in $\mu_{p^n}$, but using the fixed primitive $p^n$-th root of unity $\zeta_{p^n}\in E(\zeta_{p^n})$, we can identify $\mu_{p^n}\cong \mathbb Z/p^n$. Then by Lemma 5.33 of~\cite{rapoport-zink}, the $\cM_{\cD^\mathrm{int},n}$ are finite \'etale covers of $\cM_{\cD^\mathrm{int}}$.

We can also define an infinite-level version of these Rapoport-Zink spaces.

\begin{defn}\label{infinite rz spaces} Let $\cM_{\cD^\mathrm{int}, \infty}$ be the functor on complete affinoid $(\breve E(\zeta_{p^\infty}),\cO_{\breve E(\zeta_{p^\infty})})$-algebras sending $(R,R^+)$ to the set of triples $(\cG,\rho,\alpha)$, where $(\cG,\rho)\in \cM_{\cD^\mathrm{int}}(R,R^+)$ and \[\alpha:\Lambda\to T_p\cG^\mathrm{ad}_\eta(R,R^+)\] is a morphism of $\cO_B$-modules such that the following conditions are satisfied. 
\begin{enumerate}
\item The pairing $(\cdot,\cdot)$ on $\Lambda$ matches the pairing on $T_p\cG$ induced by the polarization. More precisely, the diagram
\[\xymatrix{
\Lambda\otimes_{\mathbb Z_p} \Lambda\ar[rrr]^{\alpha\otimes \alpha} \ar[d]_{(\cdot,\cdot)} &&& T_p\cG^{\mathrm{ad}}_\eta(R,R^+)\otimes_{\mathbb Z_p} T_p \cG^{\mathrm{ad}}_\eta(R,R^+)\ar[d] \\
\mathbb Z_p\ar[rrr]_{(1,\zeta_p,\zeta_{p^2},\ldots)} &&& (T_p\mu_{p^\infty})^{\mathrm{ad}}_\eta(R,R^+)
}\]
commutes, where the right vertical map is the pairing induced from the polarization, and the lower map is defined using the fixed $p$-power roots of unity in the base field $E(\zeta_{p^\infty})$.
\item The induced maps \[\Lambda\to T_p\cG^\mathrm{ad}_{\eta}(C,C^+),\] are isomorphisms, for all geometric points $\mathrm{Spa}(C,C^+)$ of $\mathrm{Spa}(R,R^+)$.
\end{enumerate}
\end{defn}

\noindent Recall that we have the quasi-logarithm map defined in Section 3 of~\cite{scholzeweinstein}, which induces a map of sheaves on complete affinoid $(\breve E,\cO_{\breve E})$-algebras $(R,R^+)$: \[\mathrm{qlog}_{\mathbb{X}_b}: (\tilde{\mathbb{X}}_b)^\mathrm{ad}_\eta(R,R^+) \to D(\mathbb{X}_b)[1/p]\otimes_L R.\] If $(R,R^+) = (C,C^+)$ is a geometric point, then the image of $T_p\cG^\mathrm{ad}_\eta(C,C^+)\otimes_{\mathbb Z_p} C$ under $\mathrm{qlog}_{\mathbb{X}_b}$ can be identified with $(\mathrm{Lie}\ \cG^\vee)^\vee\otimes C$.

The arguments in Section 6 of~\cite{scholzeweinstein} give the following theorem. (The case of Rapoport-Zink spaces of EL type is Theorem 6.5.4 of~\cite{scholzeweinstein}. We remark that \cite{scholzeweinstein} follows the conventions on $b$ and $\mu$ in \cite{rapoport-zink}, which differ from our conventions here.)

\begin{thm}\label{perfectoid rz spaces} The functor $\cM_{\cD^\mathrm{int},\infty}$ is representable by an adic space over $\mathrm{Spa}(\breve E(\zeta_{p^\infty}),\cO_{\breve E(\zeta_{p^\infty})})$. The space $\cM_{\cD^\mathrm{int},\infty}$ is preperfectoid, and \[\cM_{\cD^\mathrm{int},\infty}\sim \varprojlim_n \cM_{\cD^\mathrm{int},n}.\]

Moreover, there is the following alternate description of $\cM_{\cD^\mathrm{int},\infty}$, which depends only on the rational data $\cD$. The sheaf $\cM_{\cD^\mathrm{int},\infty}$ is the sheafification of the functor on complete affinoid $(\breve E(\zeta_{p^\infty},\cO_{\breve E(\zeta_{p^\infty})})$-algebras sending $(R,R^+)$ to the set of $B$-linear maps \[V\to (\tilde{\mathbb{X}}_b)^\mathrm{ad}_\eta(R,R^+)\]  which match the pairing $(\cdot,\cdot)$ on $V$ with the polarization on $(\widetilde{\mathbb{X}}_b)^\mathrm{ad}_{\eta}$ (up to the fixed choice of $p$-power roots of unity, as above) and which in addition satisfy:

\begin{enumerate}
\item The image of $V\otimes_{\mathbb{Q}_p} R$ in $D(\mathbb{X}_b)[1/p]\otimes_L R$ is totally isotropic under the pairing $(\cdot,\cdot)$ induced by the identification $D(\mathbb{X}_b)[1/p]\simeq V\otimes_{\mathbb{Q}_p} L$.
\item The quotient $W$ of $D(\mathbb{X}_b)[1/p]\otimes_L R$ by the image of $V\otimes_{\mathbb{Q}_p} R$ is a finite projective $R$-module, which locally on $R$ is isomorphic to $V_1\otimes R$ as a $B\otimes_{\mathbb{Q}_p}R$-module. 
\item For any point $\mathrm{Spa}(C,C^+)$ of $\mathrm{Spa}(R,R^+)$, the sequence
\[
0\to V\to (\widetilde{\mathbb{X}}_b)^\mathrm{ad}_\eta(C,C^+)\to W\otimes_R C\to 0
\]
is exact.
\end{enumerate}
\end{thm}

\begin{proof} To see that $\cM_{\cD^\mathrm{int}, \infty}$ is representable by a preperfectoid space, we will show that it is a closed subfunctor of the Rapoport-Zink space at infinite level $\cM_\infty$ for the $p$-divisible group $\mathbb{X}_b$, which is defined in Section 6.3 of~\cite{scholzeweinstein}. Recall that the space $\cM_\infty$ only keeps track of deformations of $\mathbb{X}_b$, without the $\cO_B$-action or the polarization. By abuse of notation, let us actually denote by $\cM_\infty$ the base change of this space to $\mathrm{Spa}(\breve{E}(\zeta_{p^\infty}),\cO_{\breve{E}(\zeta_{p^\infty})})$.

We claim that the natural (forgetful) map $\cM_{\cD^\mathrm{int},\infty} \hookrightarrow \cM_\infty$ is a closed embedding. We follow Theorem 3.25 of~\cite{rapoport-zink}. Let $\cG$ be the universal $p$-divisible group over $\cM_\infty$. The conditions that the $\cO_B$-action and the polarization lift to quasi-isogenies on $\cG$ depend only on preserving the Hodge filtration on $D(\cG)[1/p]$, by Grothendieck-Messing theory, so these are closed conditions. They correspond to restricting to a closed subset of the image of the Grothendieck-Messing period morphism. On the other hand, the condition that a quasi-isogeny be a genuine isogeny on the adic generic fiber is an open and closed condition. (This follows in the same way as Proposition 3.3.3 of~\cite{scholzeweinstein}, which is the special case of a quasi-isogeny from the $p$-divisible group $\mathbb{Q}_p/\mathbb{Z}_p$. In the general case, the key observation is that $\{e\}\hookrightarrow \mathrm{Hom}(\cG_1[p^n],\cG_2[p^n])^{\mathrm{ad}}_\eta$ is an open and closed embedding when $\{e\}$ corresponds to the trivial isogeny and $n\in \mathbb{Z}_{\geq 1}$.) Finally, the condition that the trivialization $\alpha$ of $(T_p\cG)^\mathrm{ad}_\eta$ be $\cO_B$-linear and respect the polarization is closed.

The first part of the theorem now follows from Theorem 6.3.4 of~\cite{scholzeweinstein}, which shows that $\cM_\infty$ is preperfectoid and Proposition 2.3.7 of loc. cit., which shows that a closed subspace of a preperfectoid space is preperfectoid. 

For the second part, let $\cM_{\cD, \infty}$ be the functor defined by the rational data. There is a natural map of functors $\cM_{\cD^\mathrm{int},\infty} \to \cM_{\cD,\infty}$: For $(R,R^+)$ a complete affinoid algebra over $(\breve E(\zeta_{p^\infty}), \cO_{\breve E(\zeta_{p^\infty})})$, let $(\cG,\rho,\alpha)\in \cM_{\cD^\mathrm{int},\infty} (R,R^+)$. The quasi-isogeny $\rho$ gives an identification $\widetilde{\mathbb{X}}_b\simeq \widetilde{\cG}$. The map from the rational Tate module of $\cG$ to its universal cover, precomposed with the trivialization $\alpha$ gives a map \[V\to (\widetilde{\mathbb{X}}_b)^\mathrm{ad}_\eta(R,R^+).\] By construction, this map will respect the polarization and the $B$-action. The first condition is satisfied because the image of $V\otimes_{\mathbb{Q}_p}R$ in $D(\mathbb{X}_b)[1/p]\otimes_L R$ can be identified with $(\mathrm{Lie}\ \cG^\vee)^\vee\otimes R$ (see the proof of Proposition 7.1.1 of~\cite{scholzeweinstein}). The compatibility between the pairing $(\cdot,\cdot)$ on $V$ and the polarization on $\cG$ imply that $(\mathrm{Lie}\ \cG^\vee)^\vee\otimes R$ is totally isotropic under $(\cdot,\cdot)$. The second condition is satisfied because $W$ can be identified with $\mathrm{Lie}\ \cG \otimes R$. The third condition follows from~\cite[Proposition 3.4.2 (v)]{scholzeweinstein}.

We also recall the functor $\cM'_\infty$ defined in Section 6.3 of~\cite{scholzeweinstein} (which again, we base change to $\mathrm{Spa}(\breve{E}(\zeta_{p^\infty}),\cO_{\breve{E}(\zeta_{p^\infty})})$): this parametrizes maps \[V\to (\widetilde{\mathbb{X}}_b)^\mathrm{ad}_\eta(R,R^+)\]  which satisfy: 
\begin{enumerate}
\item the quotient $W$ of $D(\mathbb{X}_b)[1/p]\otimes_L R$ by the image of $V\otimes_{\mathbb{Q}_p} R$ is a finite projective $R$-module, of the same rank as that of $V_1$. 
\item For any geometric point $\mathrm{Spa}(C,C^+)$ of $\mathrm{Spa}(R,R^+)$, the sequence
\[
0\to V\to (\widetilde{\mathbb{X}}_b)^\mathrm{ad}_\eta(C,C^+)\to W\otimes_R C\to 0
\]
is exact.
\end{enumerate}
Lemma 6.3.6 of~\cite{scholzeweinstein} shows that $\cM_\infty \toisom \cM'_\infty$ and we have a commutative diagram of adic spaces \[\xymatrix{\cM_{\cD^\mathrm{int}, \infty}\ar@{^{(}->}[d]\ar[r]&\cM_{\cD,\infty}\ar@{^{(}->}[d] \\ \cM_\infty\ar[r]^{\sim}&\cM'_\infty.}\] The bottom map is an isomorphism and the vertical maps are closed embeddings. 

It remains to see that the top map is surjective. For this, note that there is a $p$-divisible group $\cG$ over $\cM_{\cD,\infty}$, obtained by restriction from $\cM_\infty$. The integral Tate module $(T_p\cG)^\mathrm{ad}_\eta$ is identified with the lattice $\Lambda\subset V$, which is stable under $\cO_B$ and self-dual under $(\cdot,\cdot)$. The $p$-divisible group $\cG$ is equipped locally on $\cM_{\cD,\infty}$ with a quasi-isogeny on the special fiber to $\mathbb{X}_b$. The first two conditions on the image of $V\otimes_{\mathbb{Q}_p}R$ ensure that the $B$-action and the polarization on $D(\mathbb{X}_b)[1/p]\otimes_L R$ preserve the Hodge filtration of $\cG$, so that they define quasi-isogenies on $\cG$. The fact that these quasi-isogenies are genuine isogenies follows from the fact that they preserve the integral Tate module.
\end{proof}

\noindent From now on, we identify $\cM_{\cD^\mathrm{int},\infty} \simeq \cM_{\cD,\infty}$, so the moduli problem only depends on the rational data $\cD$. 

Recall that $\Fl_{G,\mu}$ is the flag variety over $\mathrm{Spa}(E,\cO_E)$ parametrizing filtrations on $\mathrm{Rep}\ G$ of the same type as the ascending filtration corresponding to the cocharacter $\mu$. On the faithful representation $V$ of $G$, $\mu$ induces the decomposition \[V_{\bar{\mathbb Q}_p}=V_0\oplus V_1\ ,\]
and the ascending filtration is given by \[\mathrm{Fil}_{-1,\mu}(V_{\bar{\mathbb Q}_p}):=V_1\ \mathrm{and}\ \mathrm{Fil}_{0,\mu}(V_{\bar{\mathbb Q}_p}):=V_{\bar{\mathbb Q}_p}.\] 

In the case we are considering, we can be more explicit: $\Fl_{G,\mu}$ parametrizes $B$-equivariant quotients $W'$ of $V\otimes_{\mathbb{Q}_p}R$ that are finite projective $R$-modules such that
\begin{enumerate}
\item the kernel of the map $V\otimes R\twoheadrightarrow W'$ is totally isotropic under $(\cdot,\cdot)$ and
\item locally on $R$, $W'$ is isomorphic to $V_0\otimes R$ as $B\otimes_{\mathbb{Q}_p}R$-modules.
\end{enumerate}

\begin{prop}There is a local Hodge-Tate period map \[\pi_{HT}:\cM_{\cD,\infty}\to \Fl_{G,\mu},\] sending an $(R,R^+)$-valued point of $\cM_{\cD,\infty}$ given by a map $V\to (\widetilde{\mathbb{X}}_b)^\mathrm{ad}_\eta(R,R^+)$ to the quotient of $V\otimes_{\mathbb{Q}_p}R$ given as the image of the map \[V\otimes_{\mathbb{Q}_p}R \to D(\mathbb{X}_b)[1/p]\otimes_L R.\] The local Hodge-Tate period map is $G(\mathbb{Q}_p)$-equivariant.
\end{prop} 

\begin{proof} This is proved in exactly the same way as Proposition 7.1.1 of~\cite{scholzeweinstein}.
\end{proof}

\noindent Recall that, by Theorem~\ref{locally closed strata}, we have a stratification of $\Fl_{G,\mu}$ by locally closed strata indexed by elements of $B(G,\mu^{-1})$ and that we have fixed an element $b\in B(G,\mu^{-1})$.

\begin{prop}\label{local Hodge-Tate period map} The local Hodge-Tate period map factors through \[\pi^b_{HT}: \cM_{\cD,\infty}\to \Fl^b_{G,\mu}.\] 
\end{prop}

\begin{proof}
It suffices to check this on $\mathrm{Spa}(C,\cO_C)$-valued points. Thus, we have a $p$-divisible group $\cG/\cO_C$ with extra structures, equipped with a quasi-isogeny $G\times_{\cO_C} \cO_C/p\to \mathbb{X}_b\times_{\mathbb{\bar F}_p}\cO_C/p$. Moreover, there is a trivialization $T_p\cG\otimes_{\mathbb Z_p} \mathbb Q_p = V$ compatible with all extra structures, and we have the Hodge-Tate filtration
\[
0\to \mathrm{Lie}\ \cG\otimes C(1)\to T_p \cG\otimes_{\mathbb Z_p} C\to (\mathrm{Lie}\ \cG^\vee)^\vee\otimes C\to 0\ ,
\]
where $\mathrm{Fil}_{-1} = \mathrm{Lie}\ \cG\otimes C(1)$, and $\mathrm{Fil}_0 = T_p \cG\otimes_{\mathbb Z_p} C$.

Let $\cE$ be the $G$-bundle on $X_{C^\flat}$ corresponding to the image of $\cG$ under $\pi_{HT}$ and the identification $\Fl_{G,\mu}\cong \GrBdR_{G,\mu}$. Let $\cE_V$ be the vector bundle on $X_{C^\flat}$ corresponding to $\cE$ and the faithful representation $V$; note that $\mu$ is still minuscule as cocharacter into $GL(V)$. Then $\cE_V$ is constructed from the $\mathbb B^+_{\mathrm{dR},C}$-lattice $\Xi$ in $V\otimes_{\mathbb{Q}_p}\mathbb B_{\mathrm{dR},C}$ inducing the above filtration on $V\otimes_{\mathbb{Q}_p}C$ under the Bialynicki-Birula map. Explicitly, if $\xi\in \mathbb B_{\mathrm{dR},C}^+$ is a generator of the maximal ideal, then the lattice $\Xi_x\subset V\otimes_{\mathbb{Q}_p}\mathbb B_{\mathrm{dR},C}$ satisfies
\[
V\otimes_{\mathbb{Q}_p}\mathbb B^+_{\mathrm{dR},C}\subset \Xi \subset V\otimes_{\mathbb{Q}_p}\xi^{-1} \mathbb B^+_{\mathrm{dR},C}
\]
and
\[
\Xi/(V\otimes_{\mathbb{Q}_p} \mathbb B_{\mathrm{dR},C}^+)=\mathrm{Lie}\ \cG\otimes C\ .
\]
Then $\cE_V$ is the modification of the trivial vector bundle $V\otimes_{\mathbb Q_p} \cO_{X_{C^\flat}}$ at the point $\infty$ by the lattice $\Xi$. 

In the case of a one-step filtration, one can construct the vector bundle $\cE_V$ directly: it is the unique vector bundle on $X_{C^\flat}$ which fits into the diagram of coherent sheaves
\[\xymatrix{
0\ar[r]&\cO_{X_{C^\flat}}\otimes_{\mathbb{Q}_p}V\ar@{=}[d]\ar[r]&\cE_V\ar@{^{(}->}[d]\ar[r]&i_{\infty*}(\mathrm{Lie}\ \cG\otimes C)\ar@{^{(}->}[d]\ar[r]&0\\
0\ar[r]&\cO_{X_{C^\flat}}\otimes_{\mathbb{Q}_p}V\ar[r]&\cO_{X_{FF,C}}(1)\otimes_{\mathbb{Q}_p}V\ar[r]&i_{\infty*}(V\otimes_{\mathbb{Q}_p}C(-1))\ar[r]&0.
}\] But then the proof of Proposition 5.1.6 of~\cite{scholzeweinstein} shows that $\cE_V$ is the vector bundle attached to the $p$-divisible group $\cG\times_{\cO_C}\cO_C/p$, which is quasi-isogenous to $\mathbb{X}_b\times_{\mathbb{\bar F}_p}\cO_C/p$. 

By unraveling the Tannakian formalism behind the construction of the $G$-bundle $\cE$ and keeping in mind the fact that $\mathbb{X}_b$ together with the $B$-action and polarization determine $b$, we see that $\cE\simeq \cE_b$ as $G$-bundles, as desired.
\end{proof}

\begin{remark} The same proof, without keeping track of the polarization, also works in the case of Rapoport-Zink spaces of EL type to show that the local Hodge-Tate period map defined in Proposition 7.1.1 of~\cite{scholzeweinstein} factors through $\Fl^b_{G,\mu}$. 
\end{remark}

\begin{remark}\label{local and global HT}
We have defined the Hodge-Tate filtration in Section~\ref{refining HT} in terms of the $p$-adic \'etale cohomology of a universal family of abelian varieties. If $A/\cO_C$ is an abelian variety and $\cG=A[p^\infty]$, then Proposition 4.15 of~\cite{scholzesurvey} shows that the Hodge-Tate filtration on $T_p\cG\otimes_{\mathbb{Z}_p}C$ is compatible with the filtration defined in Section~\ref{refining HT}, so the local and global Hodge-Tate period maps are compatible. 
\end{remark}

\begin{defn}\label{automorphism sheaf} Define the sheaf $\mathrm{Aut}_G(\widetilde{\mathbb{X}}_b)$ on $\mathrm{Nilp}^{\mathrm{op}}_{W(\bar{\mathbb F}_p)}$ by
\[
\mathrm{Aut}_G(\widetilde{\mathbb{X}}_b)(R)=\{\alpha\in \mathrm{Aut}_B(\widetilde{\mathbb{X}}_{b,R}), \beta\in \mathrm{Aut}(\widetilde{\mu}_{p^\infty,R})\mid \alpha\mathrm{\ respects\ the\ polarization\ up\ to\ }\beta\}\ .
\]
\end{defn}

\begin{lemma}\label{aut-representability} The sheaf $\mathrm{Aut}_G(\widetilde{\mathbb{X}}_b)$ is representable by a formal scheme over $\mathrm{Spf}\ W(\bar{\mathbb F}_p)$, locally of the form $\mathrm{Spf}\ W(R)$ for a perfect ring $R$.
\end{lemma}

\begin{proof} Forgetting all extra structures defines a closed embedding, so it is enough to show representability of $\mathrm{Aut}(\widetilde{\mathbb X})$ for any $p$-divisible group $\mathbb X$ over $\bar{\mathbb F}_p$. We may assume that $\mathbb{X}$ is \emph{completely slope divisible}, i.e.~that it is isomorphic to a direct sum of slope divisible isoclinic $p$-divisible groups $\mathbb{X}_i$, defined over a finite field, for $i=1,\dots, r$, with non-increasing slopes. Then $\mathrm{Aut}(\widetilde{\mathbb{X}})$ is a closed subfunctor of the product of two copies of $\mathscr{H}om(\mathbb{X}_i,\mathbb{X}_j)[1/p]$ over $i,j\in \{1,\dots,r\}$ with $i\geq j$, via sending an automorphism to the endomorphism, and its inverse. Each of the factors can be identified with the universal cover of the $p$-divisible group $\cH_{\mathbb{X}_i,\mathbb{X}_j}$. Therefore, each of the factors is representable by a formal scheme over $\mathrm{Spf}\ W(\bar{\mathbb F}_p)$, by Proposition~\ref{properties of p-divisible groups}.

For the final statement, it is enough to see that $\mathrm{Aut}_G(\widetilde{\mathbb X}_b)(R) = \mathrm{Aut}_G(\widetilde{\mathbb X}_b)(R/p)$, and that if $R$ is of characteristic $p$, then Frobenius induces a bijection of $\mathrm{Aut}_G(\widetilde{\mathbb X}_b)(R)$. Both statements follow from the similar properties of universal covers of $p$-divisible groups, for which see~\cite[Proposition 3.1.3]{scholzeweinstein}.
\end{proof}

In fact, one can give a more precise description of $\mathrm{Aut}_G(\widetilde{\mathbb{X}}_b)$. As usual, we denote by
\[
\rho\in X^\ast(G)_{\mathrm{dom}}
\]
the half-sum of the positive roots.

\begin{prop}\label{description quasi-aut} Let $\underline{J_b(\mathbb Q_p)}$ be the locally profinite set $J_b(\mathbb Q_p)$ made into a formal scheme over $W(\bar{\mathbb F}_p)$, i.e. the sections over $U\subset J_b(\mathbb Q_p)$ are continuous maps $U\to W(\bar{\mathbb F}_p)$. There is a natural map
\[
\mathrm{Aut}_G(\widetilde{\mathbb X}_b)\to \underline{J_b(\mathbb Q_p)}
\]
all of whose fibres are isomorphic to
\[
\mathrm{Spf} W(\bar{\mathbb F}_p)[[x_1^{1/p^\infty},\ldots,x_d^{1/p^\infty}]]\ ,
\]
where $d=\langle 2\rho,\nu_b\rangle$.
\end{prop}

\begin{remark} Let us illustrate this result in the case $\mathbb X_b = \mu_{p^\infty}\times \mathbb Q_p/\mathbb Z_p$, without extra structures. Then there are no maps $\mu_{p^\infty}\to \mathbb Q_p/\mathbb Z_p$, so $\mathrm{Aut}_G(\widetilde{\mathbb X}_b)$ has lower triangular form; more precisely,
\[
\mathrm{Aut}(\widetilde{\mathbb X}_b) = \left(\begin{array}{cc} \underline{\mathbb Q_p^\times} & 0 \\ \widetilde{\mu_{p^\infty}} & \underline{\mathbb Q_p^\times}\end{array}\right)\ .
\]
In this case, $J_b(\mathbb Q_p) = \mathbb Q_p^\times\times \mathbb Q_p^\times$, and the projection
\[
\mathrm{Aut}_G(\widetilde{\mathbb X}_b)\to \underline{J_b(\mathbb Q_p)}
\]
is given by the diagonal elements. The fibres are given by the unipotent part $\widetilde{\mu_{p^\infty}}\cong \mathrm{Spf}\ W(\bar{\mathbb F}_p)[[x^{1/p^\infty}]]$.
\end{remark}

\begin{proof} It is enough to prove the results for $\mathrm{Aut}_G(\widetilde{\mathbb{X}}_b)$ as a formal scheme over $\mathbb{\bar F}_p$, as all structures lift uniquely to $W(\bar{\mathbb F}_p)$ by rigidity of perfect rings. We first consider the case when $\mathbb{X}_b$ has an unramified EL structure. By standard Morita arguments, one can reduce to the case when the EL structure is given by $(F,\cO_F)$, with $F/\mathbb{Q}_p$ an unramified extension and $G=\mathrm{Res}_{F/\mathbb{Q}_p}GL_n$. If $(B,\cO_B)$ is an unramified PEL datum and $B=\prod_iB_i$ is its decomposition into simple factors, then $\mathbb{X}_b$ decomposes as $\prod_i\mathbb{X}_{b,i}$ and $\mathrm{Aut}_G(\widetilde{\mathbb{X}}_b)=\prod_i \mathrm{Aut}_{G_i}(\widetilde{\mathbb{X}}_{b,i})$. Similarly, when $B\simeq M_d(F)$ is simple, the equivalence of categories between $p$-divisible groups with $(B,\cO_B)$-EL structure and $p$-divisible groups with $(F,\cO_F)$-structure means that it suffices to compute $\mathrm{Aut}_F(\widetilde{\mathbb{X}}_b)$. See~\cite[Section 4.1]{hamacher} for more details on this reduction step. 

If $F=\mathbb{Q}_p$, then $G=GL_n$ and we are considering quasi-self-isogenies of $p$-divisible groups, without any extra compatibilities. Since $\mathbb{X}_b$ is completely slope divisible, we can write it as $\mathbb{X}_b=\oplus_{i=1}^r\mathbb{X}_i$, where the $\mathbb{X}_i$ are isoclinic $p$-divisible groups of strictly decreasing slopes $\lambda_i\in [0,1]$. Using Corollary~\ref{the hom p-divisible group} (1), we see that $\mathrm{Aut}(\widetilde{\mathbb X}_b)$ takes the lower triangular form
\[
\mathrm{Aut}(\widetilde{\mathbb X}_b) = \left(\begin{array}{cccc} \mathrm{Aut}(\widetilde{\mathbb X_1}) & & & \\ \widetilde{\cH}_{\mathbb X_2,\mathbb X_1} & \mathrm{Aut}(\widetilde{\mathbb X_2}) & & \\ \vdots & \vdots & \ddots & \\ \widetilde{\cH}_{\mathbb X_r,\mathbb X_1} & \widetilde{\cH}_{\mathbb X_r,\mathbb X_2} & \cdots & \mathrm{Aut}(\widetilde{\mathbb X_r}) \end{array}\right)\ .
\]
Moreover, Corollary~\ref{the hom p-divisible group} (2) implies that $\mathrm{Aut}(\widetilde{\mathbb X_i}) = \underline{\mathrm{Aut}(\widetilde{\mathbb X_i})(\bar{\mathbb F}_p)}$; as
\[
J_b(\mathbb Q_p) = \mathrm{Aut}(\widetilde{\mathbb X}_b)(\bar{\mathbb F}_p) = \prod_{i=1}^r \mathrm{Aut}(\widetilde{\mathbb X_i})(\bar{\mathbb F}_p)\ ,
\]
we see that projection to the diagonal defines a map
\[
\mathrm{Aut}(\widetilde{\mathbb X}_b)\to \underline{J_b(\mathbb Q_p)}\ .
\]
The structure of the fibres now follows from Corollary~\ref{the hom p-divisible group} (3) and Proposition~\ref{properties of p-divisible groups} (3). To check that $d=\langle 2\rho,\nu_b\rangle$, we count dimensions. More precisely, for $i>j$, $\widetilde{\cH}_{\mathbb{X}_i,\mathbb{X}_j}$ is representable by $\Spf\ \bar{\mathbb{F}}_p[[x^{1/p^\infty}_1,\dots,x^{1/p^\infty}_{d_{i,j}}]]$, where $d_{i,j}$ is the dimension of $\cH_{\mathbb{X}_i,\mathbb{X}_j}$. If the height of $\mathbb{X}_i$ is $m_i$, then Lemma~\ref{dieudonne module} implies that the slope of $\cH_{\mathbb{X}_i,\mathbb{X}_j}$ is $\lambda_j-\lambda_i$ and its dimension is $d_{i,j}=m_im_j\left(\lambda_j-\lambda_i\right)$.

On the other hand, by making the root data of $GL_n$ explicit, cf.~\cite[Appendix A]{hamacher}, we can compute the contribution of the slopes $\lambda_i,\lambda_j$ to $\langle 2\rho,\nu_b\rangle$. The positive roots of $GL_n$ (corresponding to the Borel subgroup given by the upper triangular matrices) are 
\[R^+=\{e_k-e_l |k,l\in \{1,\dots,n\},k<l\}.\] We also have 
\[\nu_b=(\underbrace{\lambda_1,\dots,\lambda_1}_{m_1},\dots,\underbrace{\lambda_r,\dots,\lambda_r}_{m_r}).\]
The contribution coming from $\lambda_i,\lambda_j$ to $\langle 2\rho,\nu_b\rangle$ is precisely $m_im_j\left(\lambda_j-\lambda_i\right)=d_{i,j}$.

The case of a general unramified extension $F/\mathbb{Q}_p$ follows in the same way, by working in the category of $p$-divisible groups with $\cO_F$-action instead. Let $d=[F:\mathbb Q_p]$. The theory developed in Section 4.1 can be extended to define an internal homomorphism in the category of $p$-divisible groups with $\cO_F$-action. If $\cG$ is a $p$-divisible group with $\cO_F$-action, its rational Dieudonn\'e module $\mathbb D(\cG)$ decomposes as $\mathbb{D}(\cG)=\oplus_{\tau: \cO_F\hookrightarrow W(\mathbb{\bar F}_p)} \mathbb{D}(\cG)_\tau$. Choose an embedding $\tau_0: \cO_F\hookrightarrow W(\mathbb{\bar F}_p)$ and let $\mathbb{D}_F(\cG):=\mathbb{D}(\cG)_{\tau_0}$. The analogue of Lemma~\ref{dieudonne module} holds for $\mathbb{D}_F$ and homomorphisms of $p$-divisible groups with $\cO_F$-action, with the same proof (but replacing $\varphi$ by $\varphi^d$ and embedding $F$ into $B^+_{\mathrm{cris}}$ via $\tau_0$).\footnote{For $p$-divisible groups with $\cO_F$-action, there is a more restricted notion of $p$-divisible $\cO_F$-module; the requirement is that the two actions of $\cO_F$ on the Lie algebra agree. This condition cannot be formulated for $p$-divisible groups with $\cO_F$-action up to quasi-isogeny, and in fact for $p$-divisible groups with $\cO_F$-action up to quasi-isogeny, everything works very similarly to the case of $p$-divisible $\cO_F$-modules. For example, note that $B(F,GL_n) = B(\mathbb Q_p,\mathrm{Res}_{F/\mathbb Q_p} GL_n)$.} The structure of $\mathrm{Aut}_F(\widetilde{\mathbb X}_b)$ can now be deduced in the same way. The dimension computation is also analogous to the one above. Let $\mathbb{X}_b=\oplus_{i=1}^r\mathbb{X}_i$, with the slope of the $F$-isocrystal attached to $\mathbb{X}_i$ being equal to $\lambda_i$ (here, $0\leq\lambda_i\leq d$, and $\lambda_i/d$ is the slope of $\mathbb X_i$ as a $p$-divisible group) and $\mathbb{X}_i$ having height $m_i$ as a $p$-divisible group with $\cO_F$-action, i.e. height $dm_i$ as $p$-divisible group. The dimension of the $p$-divisible group with $\cO_F$-action corresponding to the $\cO_F$-linear homomorphisms between $\mathbb{X}_i$ and $\mathbb{X}_j$ is $d_{i,j}=m_im_j\left(\lambda_j-\lambda_i\right)$. On the other hand, the positive roots of $\mathrm{Res}_{F/\mathbb{Q}_p}GL_n$ are 
\[R^+=\{e_{\tau,k}-e_{\tau,l} |k,l\in \{1,\dots,n\},k<l,\tau:F\hookrightarrow \mathbb{\bar{Q}}_p\}\] 
and,
\[\nu_b=(\underbrace{\frac{\lambda_1}{d},\dots,\frac{\lambda_1}{d}}_{m_1},\dots,\underbrace{\frac{\lambda_r}{d},\dots,\frac{\lambda_r}{d}}_{m_r}).\] 
The contribution from slopes $\lambda_i,\lambda_j$ is again $d_{i,j}=m_im_j\left(\lambda_j-\lambda_i\right)$.

We now consider the case when $\mathbb{X}_b$ has an unramified PEL structure. Recall that we are assuming that the PEL datum is of type (AC). By similar Morita-theoretic arguments as above, cf.~\cite[Corollary 4.5]{hamacher}, we can write $(B,\cO_B,*)=\prod_{i}(B_i,\cO_{B_i},*)$ as a product of simple PEL data. On the level of quasi-self-isogenies we get
\[\mathrm{Aut}_{G}(\widetilde{\mathbb{X}}_{b})=\left(\prod_i\mathrm{Aut}_{G_i}(\widetilde{\mathbb{X}}_{b,i})\right)^1\hookrightarrow \prod_i\mathrm{Aut}_{G_i}(\widetilde{\mathbb{X}}_{b,i}),\]
where $\left(\prod_i\mathrm{Aut}_{G_i}(\widetilde{\mathbb{X}}_{b,i})\right)^1$ is a closed subfunctor of the product, defined by the condition that the similitude factors on each term are the same. The group $G$ is defined similarly, as the closed subgroup $(\prod_i G_i)^1\hookrightarrow \prod_i G_i$. The similitude factor on $\mathrm{Aut}_{G_i}(\widetilde{\mathbb X}_{b_i})$ defines a map
\[
\mathrm{Aut}_{G_i}(\widetilde{\mathbb X}_{b_i})\to \underline{\mathbb Q_p^\times}
\]
which will factor as
\[
\mathrm{Aut}_{G_i}(\widetilde{\mathbb X}_{b_i})\to \underline{J_{b_i}(\mathbb Q_p)}\to \underline{\mathbb Q_p^\times}\ ,
\]
where the latter map is the natural similitude morphism on $J_{b_i}$. We see that the result for all $G_i$ implies the result for $G$, so we can assume that $G$ is simple.

We reduce to one of the following three cases. 
\begin{enumerate}
\item $\mathbb{X}_b$ is a $p$-divisible group with $(F,\cO_F)$-EL structure, where $F/\mathbb{Q}_p$ is unramified.
\item $\mathbb{X}_b$ is a $p$-divisible group with $(F,\cO_F,*)$-PEL structure, where $*$ is the identity on $F$.
\item $\mathbb{X}_b$ is a $p$-divisible group with $(F,\cO_F, *)$-PEL structure, with $\mathbb{Q}_p\subset F^+\subset F$ unramified extensions, $*$ an automorphism of order $2$ and $F^+=F^{*=1}$.
\end{enumerate}
\noindent The first case was already dealt with above. The second case corresponds to $G=GSp_{n}/\cO_F$ with $n$ even, while the third to $G=GU_n/\cO_{F^+}$. 

We explain the computation of $\mathrm{Aut}_G(\widetilde{\mathbb{X}}_b)$ in the case of $G=GSp_{n}/\cO_F$. As before $d=[F:\mathbb Q_p]$, and we write $\mathbb{X}_b=\oplus_{i=1}^r\mathbb{X}_i$, with each $\mathbb{X}_i$ isoclinic of slope $\lambda_i\in [0,d]$ as $p$-divisible group with $\cO_F$-action, and the $\lambda_i$ in strictly decreasing order. The fact that $\mathbb{X}_b$ is equipped with a symmetric polarization means that $d-\lambda_i$ is also a slope of $\mathbb{X}_b$, occuring corresponding to the same height $m_i$ as $\lambda_i$. As before, the restriction of an automorphism of $\widetilde{\mathbb X}_b$ to the graded pieces $\widetilde{\mathbb X}_i$ of the slope filtration defines the map
\[
\mathrm{Aut}_G(\widetilde{\mathbb X}_b)\to \underline{J_b(\mathbb Q_p)}\ .
\]

The fibres of this map can be computed at the same time as the dimension, and we concentrate on the dimension in the following. We can write
\[\nu_b=(\underbrace{\frac{\lambda_1}{d},\dots,\frac{\lambda_1}{d}}_{m_1},\dots,\underbrace{\frac{\lambda_{r}}{d},\dots,\frac{\lambda_r}{d}}_{m_r}),\] with $\lambda_i+\lambda_{r+1-i}=d$, $m_i=m_{r+1-i}$. 
Using the same choices as in~\cite[Appendix A]{hamacher} and recalling that $c:G\to \mathbb{G}_m$ is the multiplier character, the positive roots of $G=GSp_{n}/\cO_F$ are 
\[R^+=\{e_{\tau,k}-e_{\tau,l}|k<l\in \{1,\dots,n/2\},\tau:F\hookrightarrow \mathbb{\bar{Q}}_p\}\]
\[\cup\{e_{\tau,k}+e_{\tau,l}-c | k\not= l\in \{1,\dots,n/2\}, \tau:F\hookrightarrow \mathbb{\bar{Q}}_p\}\]
\[\cup\{2e_{\tau,k}-c | k\in\{1,\dots,n/2\}, \tau:F\hookrightarrow \mathbb{\bar{Q}}_p\}.\] 
We compute the contributions coming from slopes $\lambda_i,\lambda_j$ to both the dimension of $\mathrm{Aut}_G(\widetilde{\mathbb{X}}_b)$ and to $\langle 2\rho,\nu_b\rangle$ and check that they are the same. 
\begin{enumerate}
\item If $\lambda_j>\lambda_i\geq \frac{d}{2}$, then the contribution to the dimension of $\mathrm{Aut}_G(\widetilde{\mathbb{X}}_b)$ is, just like in the EL case, $d_{i,j}=m_im_j(\lambda_j-\lambda_i)$ and it matches the contribution from $\frac{\lambda_j}{d},\frac{\lambda_i}{d}$ to $\langle 2\rho,\nu_b\rangle$ by the same argument. Using the polarization, this also takes care of all cases with $\frac d2\geq \lambda_j>\lambda_i$.
 
\item If $\lambda_j\geq \frac{d}{2}\geq d-\lambda_i$, with $i\neq j$, then the contribution to the dimension of $\mathrm{Aut}_G(\widetilde{\mathbb{X}}_b)$ is $m_im_j(\lambda_i+\lambda_j-d)$. This is given by the dimension of the internal Hom $\cO_F$-module between $\mathbb{X}^\vee_i$ and $\mathbb{X}_j$ if $j<i$, computed as in the EL case, which by the compatibility with the polarization also pins down the quasi-isogeny between $\mathbb{X}^\vee_j$ and $\mathbb{X}_i$. This matches the contribution from $\frac{\lambda_j}{d}, 1-\frac{\lambda_i}{d}$ and $\frac{\lambda_i}{d}, 1-\frac{\lambda_j}{d}$ to $\langle 2\rho,\nu_b\rangle$, using the fact that $\langle c,\nu_b\rangle=1$.

\item If $\lambda_i > \frac{d}{2}$, the contribution to $\langle 2\rho,\nu_b\rangle$ from $\frac{\lambda_i}{d},1-\frac{\lambda_i}{d}$ is $\frac{m_i(m_i+1)}{2}(2\lambda_i-d)$. This is also the dimension of the part of $\mathscr{H}om_{\cO_F}(\mathbb{X}^\vee_i,\mathbb{X}_i)[1/p]$ which is compatible with the polarization. Indeed, the polarization induces an involution on $\mathscr{H}om_{\cO_F}(\mathbb{X}^\vee_i,\mathbb{X}_i)[1/p]$ and we can compute the dimension of the part fixed under the polarization using Lemma~\ref{dieudonne module}: the slope is $2\frac{\lambda_i}{d}-1$ and the height of the fixed part as a $p$-divisible $\cO_F$-module is $\frac{m_i(m_i+1)}{2}$.
\end{enumerate}

The case $G=GU_n$ is similar and left as an exercise. 
\end{proof}

\begin{remark} In view of the theory developed in Subsection~\ref{a product formula} and Corollary~\ref{faithfully flat} in particular, the dimension of $\mathrm{Aut}_G(\widetilde{\mathbb X}_b)$ should match the dimension of central leaves inside the Newton stratum corresponding to $b$ on the special fiber of a corresponding Shimura variety. This indeed agrees with the dimension of central leaves as computed by~\cite[Corollary 7.8]{hamacher}.
\end{remark}

Note that there is an action of $\mathrm{Aut}_G(\widetilde{\mathbb X}_b)$ on $\mathfrak{M}_{\cD_{\mathrm{int}}}$. We let $\mathrm{Aut}_G(\widetilde{\mathbb{X}}_b)^{\mathrm{ad}}_{\eta}$ be its adic generic fiber over $\mathrm{Spa}(L,\cO_L)$. Then the action of $\mathrm{Aut}_G(\widetilde{\mathbb{X}}_b)^{\mathrm{ad}}_{\eta}$ on $\cM_{\cD_{\mathrm{int}}}$ extends to an action on $\cM_{\cD,\infty}$. The map $\pi_{HT}^b: \cM_{\cD,\infty}\to \Fl^b_{G,\mu}$ is equivariant for this action with respect to the trivial action on the target. We would like to say that $\pi_{HT}^b: \cM_{\cD,\infty}\to \Fl^b_{G,\mu}$ is an $\mathrm{Aut}_G(\widetilde{\mathbb X}_b)^{\mathrm{ad}}_\eta$-torsor. However, we have only defined the target as a locally closed subspace of $\Fl_{G,\mu}$. Also, the condition of being a torsor includes the condition that the map is surjective locally in some specified topology. It is probably necessary to use some of the fine topologies from~\cite{scholzelectures} here. Thus, we content ourselves with some more basic information. Recall that $\cM_{\cD,\infty}$ is preperfectoid and lives over the perfectoid field $E(\zeta_{p^\infty})^\wedge$; thus, one can form a perfectoid space $\widehat{\cM}_{\cD,\infty}$ as in~\cite[Proposition 2.3.6]{scholzeweinstein}. The product
\[
\widehat{\cM}_{\cD,\infty}\times_{\mathrm{Spa}(L,\cO_L)} \mathrm{Aut}_G(\widetilde{\mathbb{X}}_b)^\mathrm{ad}_\eta
\]
exists in the category of adic spaces, and is still a perfectoid space, by the local structure of the automorphism scheme. On the other hand, the space
\[
\cM_{\cD,\infty} \times_{\Fl_{G,\mu}} \cM_{\cD,\infty}\subset \cM_{\cD,\infty}\times_{\mathrm{Spa}(\breve E,\cO_{\breve E})} \cM_{\cD,\infty}
\]
is preperfectoid (as this condition passes to closed subsets, cf.~\cite[Proposition 2.3.7]{scholzeweinstein}), so again we can pass to a perfectoid space
\[
(\cM_{\cD,\infty} \times_{\Fl_{G,\mu}} \cM_{\cD,\infty})^\wedge\ .
\]

\begin{prop}\label{torsors for aut} The action map
\[
\widehat{\cM}_{\cD,\infty}\times_{\mathrm{Spa}(L,\cO_L)} \mathrm{Aut}_G(\widetilde{\mathbb{X}}_b)^\mathrm{ad}_\eta\to (\cM_{\cD,\infty} \times_{\Fl_{G,\mu}} \cM_{\cD,\infty})^\wedge
\]
is an isomorphism of perfectoid spaces.
\end{prop}

\begin{proof} Let $(R,R^+)$ be a perfectoid affinoid algebra over $\breve E$.\footnote{In the proof, we are really only using that $R^+\subset R$ is bounded, and that this property passes to rational subsets.} We have to construct an inverse map
\[
(\cM_{\cD,\infty} \times_{\Fl_{G,\mu}} \cM_{\cD,\infty})(R,R^+)\to (\cM_{\cD,\infty}\times_{\mathrm{Spa}(L,\cO_L)} \mathrm{Aut}_G(\widetilde{\mathbb{X}}_b)^{\mathrm{ad}}_{\eta})(R,R^+)\ .
\]
Given an element of the source, we have (after localization on $\mathrm{Spa}(R,R^+)$) two $p$-divisible groups $\cG_1$, $\cG_2$ over $R^+$,\footnote{Here, we use that $R^+\subset R$ is bounded.} equipped with quasi-isogenies to $\mathbb X_b$ over $R^+/p$, and trivializations of the Tate module on the generic fibre. In particular, we get an isomorphism of the $\mathbb Z_p$-local systems given by the Tate modules of $\cG_1$ and $\cG_2$ over $R$, in other words an isomorphism $\cG_{1,R}\cong \cG_{2,R}$. We need to check that this isomorphism extends to $R^+$, as one can then compose this isomorphism with the given quasi-isogenies to $\mathbb X_b$ over $R^+/p$ to get a self-quasi-isogeny of $\mathbb X_b$, as desired. In this regard, we observe the following lemma, which is a non-noetherian version of a result of Berthelot,~\cite{berthelotvalparfait}.

\begin{lemma}\label{extend map from generic fibre} Let $R^+$ be a $\mathbb Z_p$-algebra which is integrally closed in $R=R^+[1/p]$. Let $G$, $H$ be $p$-divisible groups over $R^+$. Assume that the Newton polygon of $G_s$ is independent of $s\in \mathrm{Spec}(R^+/p)$, and that the same holds true for $H$. Let $f_R: G_R\to H_R$ be a morphism of $p$-divisible groups over $R$. Then $f_R$ extends, necessarily uniquely, to a morphism $f: G\to H$ of $p$-divisible groups over $R^+$ if and only if for all geometric rank $1$ points $\mathrm{Spa}(C,\cO_C)$ of $\mathrm{Spa}(R,R^+)$, the base change $f_C: G_C\to H_C$ extends to a map $f_{\cO_C}: G_{\cO_C}\to H_{\cO_C}$.
\end{lemma}

\begin{proof} For each $n\geq 1$, we have to check that the map $G[p^n]_R\to H[p^n]_R$ extends to $R^+$. Both schemes $G[p^n]$, $H[p^n]$ in question are affine, and finite locally free over $R^+$. Thus, the question whether this morphism extends is the question whether a matrix with entries in $R$ has entries in $R^+$. As
\[
R^+ = \{f\in R\mid \forall x\in \mathrm{Spa}(R,R^+)\ :\ |f(x)|\leq 1\}\ ,
\]
we can reduce to the case of a point, i.e. $R=K$ is a complete nonarchimedean field, and $K^+\subset K$ is an open and bounded valuation subring. We may also assume that $K$ is algebraically closed, and rename $C=K$, $C^+=K^+$. By assumption, the map extends to $\cO_C$. Let $\mathfrak{m}_{\cO_C}\subset \cO_C$ be the maximal ideal; it is also contained in $C^+$. Then $C^+/\mathfrak{m}_{\cO_C}\subset \cO_C/\mathfrak{m}_{\cO_C}$ is a valuation subring. Finally, we are reduced to the following lemma.
\end{proof}

\begin{lemma}\label{valringfullyfaithful} Let $V$ be a valuation ring of characteristic $p$ with quotient field $K$. Let $G$, $H$ be $p$-divisible groups over $V$ with constant Newton polygon. Then the map

\[
\Hom(G,H)\to \Hom(G_K,H_K)
\]
is a bijection.
\end{lemma}

\begin{remark}\label{intdomfullyfaithful} Using this lemma, one can remove the noetherian hypothesis from the main result of~\cite{berthelotvalparfait}, i.e.~the same fully faithfulness result holds true for any integral domain $R$ in place of $V$. Indeed, to check whether a homomorphism over $K$ extends to $R$, one has to check whether certain matrices over $K$ have entries in $R$, which can be checked on valuation rings.
\end{remark}

\begin{proof} The map is clearly injective. For surjectivity, we have to check as above that certain matrices with coefficients in $K$ have entries in $V$. Thus, we may assume that $K$ is algebraically closed.

Observe that it is enough to prove the result up to quasi-isogeny. Indeed, if $f: G\to H$ becomes divisible by $p$ over $K$, then $G[p]_K\subset G_K$ is killed by $f$, whence its flat closure $G[p]\subset G$ is killed by $f$, which shows that $f$ is divisible by $p$.

Now, e.g.~by the Dieudonn\'e-Manin classification, both $G_K$ and $H_K$ admit a quasi-isogeny to a completely slope divisible $p$-divisible group $G_0$, $H_0$ (defined over $\bar{\mathbb F}_p\subset V$). We may assume that these quasi-isogenies are genuine isogenies; then we may take their flat closures over $V$ and divide $G$, resp. $H$, by them; thus, we may assume that $G_K$ and $H_K$ are completely slope divisible. Then by~\cite[Proposition 2.3]{oort-zink}, $G$ and $H$ are themselves completely slope divisible. As $V$ is perfect, both $G$ and $H$ decompose as a direct sum of their isoclinic pieces, cf.~\cite[Proposition 1.3]{oort-zink}; thus $G\cong G_0\times_{\bar{\mathbb F}_p} V$, $H\cong H_0\times_{\bar{\mathbb F}_p} V$.

Finally, we use that the Dieudonn\'e module functor on $V$ is fully faithful, cf.~\cite{berthelotvalparfait}. Thus, as $G$ and $H$ come via base extension from $\bar{\mathbb F}_p$, it remains to show that if $(D,\varphi)$ is any isocrystal over $\bar{\mathbb F_p}$, then
\[
(D\otimes_{W(\bar{\mathbb F}_p)[1/p]} W(V)[1/p])^{\varphi = 1} = (D\otimes_{W(\bar{\mathbb F}_p)[1/p]} W(K)[1/p])^{\varphi = 1}\ .
\]
We may assume that $D=D_\lambda$ is simple of slope $\lambda=s/r$. In that case, we have to prove
\[
W(V)[1/p]^{\varphi^r = p^s} = W(K)[1/p]^{\varphi^r = p^s}\ .
\]
Clearly, the left-hand side is contained in the right-hand side. If $s\neq 0$, then the right-hand is $0$, as follows by looking at the $p$-adic valuation of any nonzero element. We are left with the case $s=0$, where $r=1$. But $W(K)[1/p]^{\varphi = 1} = \mathbb Q_p\subset W(V)[1/p]^{\varphi = 1}$, finishing the proof.
\end{proof}

Using Lemma~\ref{extend map from generic fibre}, we only have to check the result on geometric rank $1$ points. But now, by~\cite[Theorem B]{scholzeweinstein}, $p$-divisible groups over $\cO_C$ are equivalent to pairs $(T,W)$, where $T$ is a finite free $\mathbb Z_p$-module, and $W\subset T\otimes_{\mathbb Z_p} C$ is the Hodge-Tate filtration. Thus, it remains to check that the Hodge-Tate filtration is preserved, but this is true as we started with an element of the fibre product
\[
(\cM_{\cD,\infty} \times_{\Fl_{G,\mu}} \cM_{\cD,\infty})(R,R^+)\ .
\]
\end{proof}

We also have the following surjectivity result.

\begin{lemma}\label{surjectivitypiHT} Let $C/\breve{E}(\zeta_{p^\infty})$ be a complete algebraically closed extension with ring of integers $\cO_C$. Then the map
\[
\pi_{HT}^b: \cM_{\cD,\infty}(C,\cO_C)\to \Fl^b_{G,\mu}(C,\cO_C)
\]
is surjective.
\end{lemma}

\begin{proof} Given $x\in \Fl^b_{G,\mu}(C,\cO_C)$, we get (corresponding to the representation $G\to GL(V)$, and using~\cite[Theorem B]{scholzeweinstein}) a $p$-divisible group $\cG/\cO_C$ with trivialized Tate module, which by functoriality comes equipped with an action of $\cO_B$ and a principal polarization. To give a point of $\cM_{\cD,\infty}(C,\cO_C)$, it remains to construct a quasi-isogeny $\rho$ over $\cO_C/p$. For this, note that the proof of Proposition~\ref{local Hodge-Tate period map} gives an identification between the $G$-bundle $\cE_{\cG}$ corresponding to $\cG$, and the $G$-bundle $\cE_x$ corresponding to the point $x$. By assumption, $x\in \Fl^b_{G,\mu}(C,\cO_C)$, so there is an isomorphism of $G$-bundles $\cE_x\cong \cE_b$, which gives an isomorphism of $G$-bundles $\cE_{\cG}\cong \cE_b$. Using Theorem~\ref{p-divisible groups and vector bundles}, this gives the desired quasi-isogeny.
\end{proof}

Using these results, we can compute the dimension of the strata $\Fl_{G,\mu}^b\subset \Fl_{G,\mu}$. Here, we define the dimension as the Krull dimension, i.e.~the length of the longest chain of specializations.

\begin{prop}\label{transcendence dim} Let $K$ be a complete nonarchimedean field with ring of integers $\cO_K$ and residue field $k$. Let $X$ be a partially proper adic space over $\mathrm{Spa}(K,\cO_K)$. Then the dimension of $X$ is equal to the maximal transcendence degree of $k(x)$ for $x\in X$, where $k(x)$ is the residue field of the ring of integers $\cO_{K(x)}$ in the completed residue field $K(x)$ at $x$.
\end{prop}

\begin{remark} Recall that a map $f: X\to Y$ of analytic adic spaces is \emph{partially proper} if for any complete nonarchimedean field $K$ with ring of integers $\cO_K\subset K$ and open and bounded valuation subring $K^+\subset K$ (so $K^+\subset \cO_K$), the map
\[
X(K,K^+)\to X(K,\cO_K)\times_{Y(K,\cO_K)} Y(K,K^+)
\]
is a bijection. This is the analogue of the valuative criterion for properness in this setup.
\end{remark}

\begin{proof} As $X$ lives over $\mathrm{Spa}(K,\cO_K)$, it is analytic, and thus any point generalizes to a rank $1$ point. It is thus enough to prove the more precise assertion that for any rank $1$ point $x$, the dimension of the closure $\overline{\{x\}}$ is equal to the transcendence degree of $k(x)$. But the closure $\overline{\{x\}}$ gets identified with the Zariski-Riemann space for $k(x)/k$ (using partial properness), whose dimension is equal to the transcendence degree of $k(x)/k$.
\end{proof}

\begin{prop}\label{dim add} Let $K$ be a complete nonarchimedean field with ring of integers $\cO_K$ and residue field $k$. Let $f: X\to Y$ be map of partially proper adic spaces over $\mathrm{Spa}(K,\cO_K)$, and fix a rank $1$ point $x\in X$, with image $y\in Y$. Let $X_y = X\times_Y \{y\}$ be the fibre of $f$ over $y$. Let $\overline{\{x\}}^X\subset X$, $\overline{\{y\}}^Y\subset Y$ and $\overline{\{x\}}^{X_y}\subset X_y$ be the respective closure. Then
\[
\dim \overline{\{x\}}^X = \dim \overline{\{y\}}^Y + \dim \overline{\{x\}}^{X_y}\ .
\]
\end{prop}

\begin{proof} Let $k(x)$ and $k(y)$ have the same meaning as in Proposition~\ref{transcendence dim}. Then the statement translates into the additivity of transcendence degrees for the extensions $k(x)/k(y)/k$.
\end{proof}

\begin{prop}\label{dim aut} For any complete nonarchimedean field $K/\cO_{\breve{E}}$, the space
\[
\mathrm{Aut}_G(\widetilde{\mathbb X}_b)^{\mathrm{ad}}\times_{\mathrm{Spa}(\cO_{\breve{E}},\cO_{\breve{E}})} \mathrm{Spa}(K,\cO_K)
\]
is partially proper over $\mathrm{Spa}(K,\cO_K)$, of dimension $\langle 2\rho,\nu_b\rangle$.
\end{prop}

\begin{proof} The adic generic fiber is partially proper by Lemma~\ref{extend map from generic fibre}. (A quasi-self-isogeny respecting extra structures over $\mathrm{Spa}(C,\cO_C)$ will also respect the extra structures when it extends to $\mathrm{Spa}(C,C^+)$ by the injectivity of the map in Lemma~\ref{valringfullyfaithful}.) For the claim about the dimension of $\mathrm{Aut}_G(\widetilde{\mathbb{X}}_b)^{\mathrm{ad}}\times_{\mathrm{Spa}(\cO_{\breve{E}},\cO_{\breve{E}})} \mathrm{Spa}(K,\cO_K)$, it is enough to consider a connected component, all of which are by Proposition~\ref{description quasi-aut} given by
\[
\mathrm{Spa}(\cO_{\breve{E}}[[x_1^{1/p^\infty},\ldots,x_d^{1/p^\infty}]],\cO_{\breve{E}}[[x_1^{1/p^\infty},\ldots,x_d^{1/p^\infty}]])\times_{\mathrm{Spa}(\cO_{\breve{E}},\cO_{\breve{E}})} \mathrm{Spa}(K,\cO_K)\ .
\]
To compute the dimension, we may assume that $K$ is algebraically closed. Then $K$ is perfectoid, and by tilting we can assume that $K$ is of characteristic $p$. In that case, the space is topologically the same as
\[
\mathrm{Spa}(\cO_{\breve{E}}[[x_1,\ldots,x_d]],\cO_{\breve{E}}[[x_1,\ldots,x_d]])\times_{\mathrm{Spa}(\cO_{\breve{E}},\cO_{\breve{E}})} \mathrm{Spa}(K,\cO_K)\ .
\]
But this is the $d$-dimensional open unit disc over $K$.
\end{proof}

\begin{prop} The dimension of $\Fl_{G,\mu}^b$ is equal to $\langle 2\rho,\mu\rangle -\langle 2\rho,\nu_b\rangle$.
\end{prop}

\begin{proof} Both $\Fl_{G,\mu}$ and $\cM_{\cD,\infty}$ are partially proper adic spaces over $\mathrm{Spa}(\breve{E},\cO_{\breve{E}})$ of dimension $\langle 2\rho,\mu\rangle$. Pick any rank $1$ point $x\in \cM_{\cD,\infty}$ such that the dimension of $\overline{\{x\}}$ is $\langle 2\rho,\mu\rangle$, and let $y\in \Fl_{G,\mu}^b$ be its image. Let $\bar{y}$ be a geometric point above $y$, corresponding to a completed algebraic closure $C$ of $K(y)$, and pick a lift of $\bar{y}$ to $\cM_{\cD,\infty}$, using Lemma~\ref{surjectivitypiHT}. Then Proposition~\ref{dim add} shows that
\[
\langle 2\rho,\mu\rangle\leq \dim \overline{\{y\}} + \dim \cM_{\cD,\infty,y}\ .
\]
But $\dim \cM_{\cD,\infty,y} = \dim \cM_{\cD,\infty,\bar{y}}$, and using Proposition~\ref{torsors for aut} and the choice of $\bar{y}$, one has
\[
\dim \cM_{\cD,\infty,\bar{y}} = \dim \mathrm{Aut}_G(\widetilde{\mathbb X}_b)^{\mathrm{ad}}\times_{\mathrm{Spa}(\cO_{\breve{E}},\cO_{\breve{E}})} \mathrm{Spa}(C,\cO_C)\ .
\]
The latter has been computed in Proposition~\ref{dim aut}, showing the inequality
\[
\dim \Fl_{G,\mu}^b\geq \dim \overline{\{y\}}\geq \langle 2\rho,\mu\rangle -\langle 2\rho,\nu_b\rangle\ .
\]
For the converse, pick any rank $1$ point $y\in \Fl_{G,\mu}$. As before, one sees that $\dim \cM_{\cD,\infty,y} = \langle 2\rho,\nu_b\rangle$, so pick a rank $1$ point $x\in \cM_{\cD,\infty,y}$ whose closure is of dimension $\langle 2\rho,\nu_b\rangle$. Applying Proposition~\ref{dim add}, we see that the dimension of the closure of $x$ in $\cM_{\cD,\infty}$ is at least $\dim \overline{\{y\}} + \langle 2\rho,\nu_b\rangle$. On the other hand, the dimension of the closure of $x$ is bounded by $\dim \cM_{\cD,\infty} = \langle 2\rho,\mu\rangle$. This shows that
\[
\dim \overline{\{y\}}\leq \langle 2\rho,\mu\rangle -\langle 2\rho,\nu_b\rangle\ ,
\]
which (as $y$ was arbitrary) proves the other inequality.
\end{proof}

\subsection{A product formula}\label{a product formula}
We now return to our global setting, where we want to study the Hodge-Tate period map $\pi_{HT}:\cS_{K^p}\to \Fl_{G}$. Recall that we are restricting to the case when the Shimura datum $(G,X)$ is of PEL type.

More precisely, we fix global PEL data as follows, cf.~\cite[\S 5]{kottwitzpoints}. Let $B$ be a finite-dimensional simple $\mathbb Q$-algebra with center $F$, and let $V$ be a faithful finitely generated $B$-representation. Let $\ast$ be a positive involution on $B$, and $F^+ = F^{\ast = 1}$. On $V$, we fix a nondegenerate $\mathbb Q$-valued alternating form $(\cdot,\cdot)$ such that $(bv,w) = (v,b^\ast w)$ for all $v,w\in V$ and $b\in B$. Let $G/\mathbb Q$ be the algebraic group whose $R$-valued points are
\[
G(R) = \{x\in \mathrm{End}_{B\otimes R}(V\otimes R)\mid xx^\ast\in R^\times\}\ .
\]
We assume that $G$ is connected; under the classification of~\cite{kottwitzpoints}, this amounts to excluding type D. Finally, we fix a $\ast$-homomorphism $h: \mathbb{C}\to \mathrm{End}_{B\otimes \mathbb R}(V\otimes \mathbb R)$ such that the symmetric real-valued bilinear form $(v,h(i)w)$ on $V\otimes \mathbb R$ is positive-definite. Note that $h$ induces a map, denoted in the same way, $h: \mathrm{Res}_{\mathbb C/\mathbb R}\to G_{\mathbb{R}}$, and in particular a Shimura datum.

We need to assume that these data are ``unramified'' at $p$. More precisely, we assume that $B_{\mathbb Q_p}$ is a product of matrix algebras over unramified extensions of $\mathbb Q_p$, and fix a maximal $\mathbb{Z}_{(p)}$-order $\cO_B\subset B$; we assume that $\ast$ preserves $\cO_B$. Finally, we assume that there exists a $\mathbb Z_{(p)}$-lattice $\Lambda\subset V$ that is self-dual under $(\cdot,\cdot)$ and stable under $\cO_B$, and we fix such a $\Lambda$. Using these data, we can define a connected reductive group $G_{\mathbb Z_{(p)}}$ over $\mathbb Z_{(p)}$ with generic fibre $G$ as
\[
G_{\mathbb Z_{(p)}}(R) = \{x\in \mathrm{End}_{\cO_B\otimes R}(\Lambda\otimes R)\mid xx^\ast\in R^\times\}\ .
\]
We fix the hyperspecial maximal compact open subgroup $K_p = G_{\mathbb Z_{(p)}}(\mathbb Z_p)\subset G(\mathbb Q_p)$.

Let $K^p\subset G(\mathbb{A}_f^p)$ be a compact open subgroup, and fix a place $\mathfrak p|p$ of $E$. As in~\cite{kottwitzpoints}, one can define a moduli space of abelian varieties with extra structures $\mathscr{S}_{K_pK^p}$ over $\cO_{E,\mathfrak p}\subset E$. In most cases, the generic fibre $S_{K_pK^p} / E$ of $\mathscr{S}_{K_pK^p}$ is the Shimura variety corresponding to $(G,\{h\})$; in general, however, the Hasse principle for the group $G$ fails, and it consists of $|\mathrm{ker}^1(\mathbb Q,G)|$ copies of this Shimura variety. Thus, the notation of this section conflicts slightly with the previous notation for Shimura varieties of Hodge type.

Let $\mathbb F_q$ be the residue field of $\cO_{E,\mathfrak p}$. The special fiber $\mathscr{S}_{K_pK^p}\times_{\cO_{E,\mathfrak p}} \mathbb F_q$ admits a Newton stratification by locally closed strata $\mathscr{S}_{K_pK^p}^b$ indexed by $b\in B(G, \mu^{-1})$, cf.~\cite{rapoport-richartz}: A point $x\in \mathscr{S}_{K_pK^p}\times_{\cO_{E,\mathfrak p}} \mathbb F_q$ gives rise to a $p$-divisible group with extra structure, which can be translated into an isocrystal with $G$-structure, and is classified by an element $b\in B(G)$. By~\cite{rapoport-richartz}, this element actually lies in $B(G,\mu^{-1})$.

One of the main results of~\cite{mantovan} is a decomposition of the Newton stratum $\mathscr{S}_{K_pK^p}^b$ into the Rapoport-Zink space $\mathfrak{M}^b$ and the Igusa variety $\mathrm{Ig}^b$ corresponding to $b$. Thus, we first recall these two objects. From the last section, we already know the Rapoport-Zink space:

For $b\in B(G,\mu^{-1})$, choose a completely slope divisible $p$-divisible group $\mathbb{X}_b$ over $\bar{\mathbb{F}}_q$ with extra structures giving rise to the $\sigma$-conjugacy class $b$, as in~\cite[\S 3]{mantovan}. Let $\cD_{\mathrm{int},b}$ be the integral data corresponding to the base extension of $B,V,\cO_B,\Lambda$ to $\mathbb Z_p$, and $(\mu, b)$. Then $\cD_{\mathrm{int},b}$ is of PEL type, and we consider the corresponding Rapoport-Zink space $\mathfrak{M}^b := \mathfrak{M}_{\cD_{\mathrm{int},b}}$, which lives over $\cO_{\breve{E}}$, where $\breve{E}$ is the completion of the maximal unramified extension of $E_{\mathfrak{p}}$.

Next, we want to introduce the Igusa variety.

\begin{defn}\label{new Igusa variety}
We let $\mathrm{Ig}^b/\mathrm{Spec}\ \bar{\mathbb F}_q$ be the functor sending an $\bar{\mathbb F}_q$-algebra $R$ to the set of isomorphism classes of pairs
\[
\{(A, \rho) \mid A\in \mathscr{S}_{K_pK^p}(R)\ ,\ \rho: A[p^\infty]\toisom \mathbb{X}_b\times_{\bar{\mathbb{F}}_p} R\}\ ,
\]
where $A\in \mathscr{S}_{K_pK^p}(R)$ is an abelian variety equipped with extra structures (and satisfying the determinant condition) and the isomorphism $\rho$ is compatible with the extra structures; as usual, it is only supposed to preserve the polarization up to a scalar, i.e. an automorphism of $\mu_{p^\infty,R}$.
\end{defn}

\begin{remark} This definition is different from the Igusa varieties defined in~\cite{mantovan}, and we will explain their relation below.
\end{remark}

\begin{prop}\label{Ig repr} The functor $\mathrm{Ig}^b$ is representable by a scheme.
\end{prop}

\begin{proof} It is enough to prove that the map $\mathrm{Ig}^b\to \mathscr{S}_{K_pK^p}\times_{\cO_{E,\mathfrak p}} \bar{\mathbb F}_q$ is relatively representable. Let $\cA$ be the universal abelian variety over $\mathscr{S}_{K_pK^p}$. Then we are considering the inverse limit of the schemes parametrizing isomorphisms $\cA[p^n]\cong \mathbb X_b[p^n]$ compatible with extra structures, each of which is representable.
\end{proof}

From the definition of $\mathrm{Ig}^b$, it is evident that the group of automorphisms of $\mathbb X_b$ respecting the extra structures acts on it. However, we give next an alternative description of $\mathrm{Ig}^b$ which shows that the larger group $\mathrm{Aut}_G(\widetilde{\mathbb X}_b)$ acts on $\mathrm{Ig}^b$.

\begin{lemma}\label{alternate description of Ig^b} For an $\bar{\mathbb F}_q$-algebra $R$, $\mathrm{Ig}^b(R)$ can be identified with the set of isomorphism classes of pairs $(A,\tilde \rho)$, where $A\in \mathscr{S}_{K_pK^p}(R)$ is an abelian variety \emph{considered up to $p$-power isogeny (respecting the extra structures)} and \[\rho:A[p^\infty] \toisom \mathbb{X}_b\times_{\mathbb{\bar F}_p}R\] is a quasi-isogeny (respecting the extra structures).
\end{lemma}

\begin{proof} Each element $(A, \rho)$ of $\mathrm{Ig}^b(R)$ determines a pair $(A,\rho)$ as in the statement of the lemma.

Conversely, given $A\in \mathscr{S}_{K_pK^p}(R)$ with a quasi-isogeny
\[\rho: A[p^\infty] \toisom \mathbb{X}_b\times_{\mathbb{\bar F}_p}R\ ,\]
we can find a unique abelian variety $A^\prime$ with extra structures equipped with a $p$-power isogeny to $A$, such that $A[p^\infty]$ gets identified with $\mathbb X_b$, i.e.~the induced quasi-isogeny
\[\rho^\prime: A^\prime[p^\infty] \toisom \mathbb{X}_b\times_{\mathbb{\bar F}_p}R
\]
is an isomorphism. Then $(A^\prime,\rho^\prime)$ defines a point of $\mathrm{Ig}^b(R)$, as desired.
\end{proof}

\begin{cor}\label{Ig perfect} The formal group scheme $\mathrm{Aut}_G(\widetilde{\mathbb X}_b)$ acts canonically on $\mathrm{Ig}^b$. Moreover, $\mathrm{Ig}^b$ is perfect, i.e.~the Frobenius map is an automorphism.
\end{cor}

\begin{proof} The first part follows from Lemma~\ref{alternate description of Ig^b} by acting on $\rho$ (noting that quasi-isogenies of $\mathbb X_b$ are the same as automorphisms of $\widetilde{\mathbb X}_b$).

For the second part, we have to see that for any $\bar{\mathbb F}_q$-algebra $R$, the Frobenius of $R$ induces an automorphism of $\mathrm{Ig}^b(R)$. But pulling back under Frobenius induces an equivalence on the category of abelian varieties up to $p$-power isogeny (as Verschiebung gives an inverse up to multiplication by $p$). Similarly, pull back under Frobenius induces an equivalence on the category of $p$-divisible groups up to quasi-isogeny, showing that the datum of $\rho$ is preserved. 
\end{proof}

Now we recall the more classical objects; for more details, see~\cite{mantovan}. The leaf $\mathscr{C}^b$ corresponding to $\mathbb{X}_b$ is the subset of the locally closed stratum $\mathscr{S}^b_{K_pK^p}\times_{\mathbb F_q} \bar{\mathbb F}_q$ where the fibers of the $p$-divisible group $\cA[p^\infty]$ at all geometric points are isomorphic to $\mathbb{X}_b$:
\[
\mathscr C^b:=\left\{x\in \mathscr{S}^b_K\ |\ \cA_x[p^\infty]\times_{\kappa(x)}\overline{\kappa(x)}\simeq \mathbb{X}_b\times_{\bar{\mathbb F}_p} \overline{\kappa(x)} \right\}\ .
\]
This is a priori defined only as a subset of $\mathscr{S}^b_{K_pK^p}\times_{\mathbb F_q} \bar{\mathbb F}_q$, but Proposition 1 of~\cite{mantovan} shows that $\mathscr C^b$ is a closed subset and defines a smooth subscheme of $\mathscr{S}^b_{K_pK^p}\times_{\mathbb F_q} \bar{\mathbb F}_q$ when endowed with the induced reduced structure. We note that contrary to the objects defined so far, $\mathscr{C}^b$ depends on the choice of $\mathbb X_b$ within its isogeny class.

Recall that \[\mathbb{X}_b=\oplus_{i=1}^r\mathbb{X}_i,\] where the $\mathbb{X}_i$ are isoclinic $p$-divisible groups of strictly decreasing slopes $\lambda_i\in [0,1]$. Let $\cG_b$ be the $p$-divisible group of the universal abelian variety $\mathscr{A}/\mathscr{S}_{K_pK^p}$ restricted to $\mathscr C^b$. Then $\cG_b$ is completely slope divisible, with slope filtration \[0\subset \cG_{b,1}\subset \dots \subset \cG_{b,r}=\cG_b,\] with $\cG^{i}_b:=\cG_{b,i}/\cG_{b,i-1}$ isoclinic of slope $\lambda_{i}$. The $\cO_B$-action on $\cG$ and the polarization respect this filtration, so that each $\cG_b^i$ is endowed with an $\cO_B$-action and there are induced polarizations $\cG_b^i\to (\cG_b^j)^\vee$ for all $i,j$ with $\lambda_i+\lambda_j=1$. 

\begin{defn}\label{classical Igusa variety} The \emph{(pro-)Igusa variety} is the map
\[
\mathscr{I}^b_{\mathrm{Mant}}\to \mathscr{C}^b
\]
which over a $\mathscr{C}^b$-scheme $S$ parametrizes tuples $(\rho_i)_{i=1}^r$ of isomorphisms
\[
\rho_i: \cG^i_b\times_{\mathscr{C}^b} S\toisom\mathbb{X}_i\times_{\mathrm{Spec}\ \bar{\mathbb F}_p}S
\]
which are compatible with the $\cO_B$-actions on $\cG^i_b$ and $\mathbb{X}_i$, and commute with the polarizations on $\cG$ and $\mathbb{X}_b$, up to an automorphism of $\mu_{p^\infty,S}$.
\end{defn}

\begin{remark}\label{comparison with finite level} A version of these Igusa varieties is considered in~\cite{mantovan} (see also Section II of~\cite{harris-taylor} for the case of one-dimensional $p$-divisible groups). Rather than trivializing the whole isoclinic $p$-divisible group $\cG^i_b$, one trivializes the $\cG^i_b[p^m]$ for some positive integer $m$. More precisely, let $\mathscr{I}^{b}_{\mathrm{Mant},m}$ be the moduli space of isomorphisms on $\mathscr{C}^b$-schemes $S$ \[\rho_{i,m}:\cG^i_b[p^m]\toisom \mathbb{X}_i[p^m]\times_{\mathbb{\bar F}_p}S,\] which (fppf locally) lift to arbitrary $m'\geq m$ and which respect the extra structures. Proposition 4 of~\cite{mantovan} shows that the underlying reduced subscheme of $\mathscr{I}^{b}_{\mathrm{Mant},m}$ is a finite \'etale and Galois cover of $\mathscr{C}^b$. 

In view of the theory developed in Section~\ref{$p$-divisible groups}, we can identify the set of endomorphisms of $\mathbb{X}_i[p^m]$, which lift to arbitrary $m'\geq m$, with the $p^m$-torsion in the \'etale $p$-divisible group $\cH_{\mathbb{X}_i,\mathbb{X}_i}$. Now consider the intersection of the scheme-theoretic images of the automorphisms of $\mathbb{X}_i[p^{m+k}]$ inside the automorphisms of $\mathbb{X}_i[p^m]$ (under the natural restriction map). By Lemma~\ref{precise chai-oort}, the images of $\mathscr{A}ut(\mathbb{X}_i[p^{m+k}])\hookrightarrow \mathscr{A}ut(\mathbb{X}_i[p^{m}])$ will stabilize for large enough $k$, giving rise to an open and closed subscheme of the finite \'etale scheme $\cH_{\mathbb{X}_i,\mathbb{X}_i}[p^m]$. This shows that $\mathscr{I}^b_{\mathrm{Mant},m}\to \mathscr{C}^b$ is a quasitorsor under an \'etale group scheme. From~\cite[Proposition 4]{mantovan} (which produces sections over a finite \'etale cover), it follows that they are actually torsors. In particular, we see that $\mathscr{I}^b_{\mathrm{Mant},m}$ is actually already reduced.

Thus, the scheme
\[
\mathscr{I}^b_{\mathrm{Mant}} = \varprojlim_m \mathscr{I}^b_{\mathrm{Mant},m}
\]
is a pro-\'etale cover of $\mathscr{C}^b$.
\end{remark}

Note that, as $\mathrm{Ig}^b$ is reduced, the natural map $\mathrm{Ig}^b\to \mathscr{S}_{K_pK^p}$ factors over $\mathrm{Ig}^b\to \mathscr{C}^b$. Moreover, as any homomorphism between $p$-divisible groups preserves the slope filtration by Corollary~\ref{the hom p-divisible group}, we see that any isomorphism $\cG_b\cong \mathbb X_b$ induces isomorphisms $\cG^i_b\cong \mathbb X_i$, and thus there is a natural map $\mathrm{Ig}^b\to \mathscr{I}^b_{\mathrm{Mant}}$.

\begin{prop}\label{perfection} The perfect scheme $\mathrm{Ig}^b$ is the perfection of $\mathscr{I}^b_{\mathrm{Mant}}$, via the natural map $\mathrm{Ig}^b\to \mathscr{I}^b_{\mathrm{Mant}}$.
\end{prop}

\begin{proof} Let $(\mathscr{I}^b_{\mathrm{Mant}})^\mathrm{perf}$ be the perfection of $\mathscr{I}^b_{\mathrm{Mant}}$. Then we claim that the $p$-divisible group $\cG_b$ over $\mathscr{C}^b$ becomes canonically isomorphic to $\mathbb{X}_b$ when pulled back to $(\mathscr{I}^b_{\mathrm{Mant}})^\mathrm{perf}$. Recall that $\cG_b$ has a slope filtration \[0\subset \cG_{b,1}\subset \dots \subset \cG_{b,r}=\cG_b,\] with $\cG^{i}_b:=\cG_{b,i}/\cG_{b,i-1}$ isoclinic of slope $\lambda_{i}$. Moreover, when pulled back along $\mathscr{I}^b_{\mathrm{Mant}}\to \mathscr{C}^b$, each $\cG^i_{b}$ becomes trivialized to $\mathbb{X}_i$.

The existence of the slope filtration on $\cG_b$ means that we have integers $0\leq t_r<\dots<t_2<t_1\leq s$, such that for $i=1,\dots,r$:
\begin{enumerate}
\item the slope $\lambda_i=\frac{t_i}{s}$;
\item the quasi-isogenies \[\frac{F^{s}}{p^{t_i}}:\cG_{b,i}\to (\cG_{b,i})^{(p^s)},\] where $F$ is the Frobenius isogeny, are genuine isogenies.
\item the induced maps \[\frac{F^{s}}{p^{t_i}}: \cG^i_b\to (\cG_b^i)^{(p^s)}\] are isomorphisms.  
\end{enumerate}
The inequalities between the $t_i$ imply that $\frac{F^{s}}{p^{t_i}}$ acts nilpotently on $\cG_{b,i-1}$. Repeated iterations of \[\frac{F^{s}}{p^{t_i}}:(\cG_{b,i})^{(p^{-s})}\to \cG_{b,i}\] can be used to construct canonical splittings $\cG^i_b\hookrightarrow \cG_{b,i}$ over $(\mathscr{I}^b_{\mathrm{Mant}})^\mathrm{perf}$.

Thus, $\cG$ decomposes canonically into $\cG_1\times\dots\times \cG_r$ over $(\mathscr{I}^b_{\mathrm{Mant}})^\mathrm{perf}$, and this is trivialized to $\mathbb X_1\times\dots \times \mathbb X_r = \mathbb X_b$, as desired.
\end{proof}

We remark that $J_b(\mathbb Q_p)\subset \mathrm{Aut}_G(\widetilde{\mathbb X}_b)$ acts on $\mathrm{Ig}^b$. However, only a certain submonoid of $J_b(\mathbb Q_p)$ acts on $\mathscr{I}^b_{\mathrm{Mant}}$; Mantovan, \cite{mantovan}, does however construct a canonical action of $J_b(\mathbb Q_p)$ on the \'etale cohomology of $\mathscr{I}^b_{\mathrm{Mant}}$. From Proposition~\ref{perfection}, it follows that the \'etale cohomology of $\mathscr{I}^b_{\mathrm{Mant}}$ is also the \'etale cohomology of $\mathrm{Ig}^b$, on which we have a natural action of $J_b(\mathbb Q_p)$. We leave it to the reader to verify that this is the same action as the one constructed by Mantovan.

\begin{cor}\label{faithfully flat} The map $\mathrm{Ig}^b\to \mathscr{C}^b$ is faithfully flat.
\end{cor}

As the map is obviously a quasitorsor under the automorphisms of $\mathbb X_b$ respecting the extra structure, this implies that it is in fact a torsor under this group. Note that $\mathscr{C}^b$ is smooth, while the scheme of automorphisms of $\mathbb X_b$ is a highly nonreduced object like $\mathrm{Spec} \bar{\mathbb F}_p[[X_1^{1/p^\infty},\ldots,X_d^{1/p^\infty}]]/(X_1,\ldots,X_d)$. The fact that a torsor under this group over something smooth is a perfect scheme forces the smooth directions of the base to match with the nonreduced directions of the group, so that one can deduce that the dimension of $\mathscr{C}^b$ is $d=\langle 2\rho,\nu_b\rangle$, for example by looking at the transitivity triangle for the cotangent complex.

\begin{proof} As $\mathscr{I}^b_{\mathrm{Mant}}$ is a cofiltered limit of smooth schemes along affine transition maps, its Frobenius morphism is (faithfully) flat, and thus $\mathrm{Ig}^b\to \mathscr{I}^b_{\mathrm{Mant}}$ is faithfully flat. We have already seen that $\mathscr{I}^b_{\mathrm{Mant}}\to \mathscr{C}^b$ is faithfully flat, so we get the result.
\end{proof}

As $\mathrm{Ig}^b$ is a perfect scheme, it lifts uniquely to a flat $p$-adic formal scheme over $W(\bar{\mathbb F}_q) = \cO_{\breve{E}}$, which we denote by $\mathrm{Ig}^b_{\cO_{\breve{E}}}$. As a moduli problem on $\mathrm{Nilp}^{\mathrm{op}}_{\cO_{\breve{E}}}$, it parametrizes abelian varieties up to $p$-power isogeny in $\mathscr{S}_{K_pK^p}$, equipped with an isomorphism of $\widetilde{A[p^\infty]}$ with (the canonical lift of) $\widetilde{\mathbb X}_b$, respecting all extra structures.

One can also describe this deformation of $\mathrm{Ig}^b$ to mixed characteristic differently. For this, fix a lift $(\mathbb X_b)_{\cO_K}$ of $\mathbb X_b$ up to quasi-isogeny (with its extra structures) to $\cO_K$, where $\cO_K$ is the ring of integers of some complete nonarchimedean field $K/\breve{E}$; in other words, pick a point $(\mathbb X_b)_{\cO_K}\in \mathfrak{M}^b(\cO_K)$. This is possible (with $K=\breve{E}$), as $\mathfrak{M}^b$ is formally smooth. One gets the following lemma.

\begin{lemma}\label{Ig mixed char} The points of the formal scheme $\mathrm{Ig}^b_{\cO_K} = \mathrm{Ig}^b_{\cO_{\breve{E}}}\times_{\cO_{\breve{E}}} \cO_K$ over $R\in \mathrm{Nilp}^{\mathrm{op}}_{\cO_K}$ are given by the pairs $(A,\rho)$, where $A\in \mathscr{S}_{K_pK^p}(R)$ is an abelian variety with extra structure, and
\[
\rho: A[p^\infty]\toisom (\mathbb X_b)_{\cO_K}\times_{\cO_K} R
\]
is an isomorphism compatible with the extra structure.$\hfill \Box$
\end{lemma}

We will also need a variant of Igusa varieties, where one trivializes $A[p^\infty]$ only up to quasi-isogeny.

\begin{defn}\label{product formal scheme} Let $\mathfrak{X}^b$ be the functor sending $R\in \mathrm{Nilp}_{\cO_{E,\mathfrak p}}^{\mathrm{op}}$ to the set of pairs $(A,\rho)$, where $A\in \mathscr{S}_{K_pK^p}(R)$ is an abelian variety with extra structure, and
\[
\rho: A[p^\infty]\times_R R/p\to \mathbb{X}_b\times_{\bar{\mathbb{F}}_q} R/p
\]
is a quasi-isogeny compatible with the $\cO_B$-action and the polarization, up to an automorphism of $\widetilde{\mu}_{p^\infty,R/p}$.
\end{defn}

Fix a lift $(\mathbb X_b)_{\cO_K}$ of $\mathbb X_b$ to $\cO_K$ as above. We define a map of formal schemes over $\cO_K$,
\[
\mathrm{Ig}^b_{\cO_K} \times_{\cO_{\breve{E}}} \mathfrak{M}^b\to \mathfrak{X}^b_{\cO_K}\ .
\]
For $R\in \mathrm{Nilp}^{\mathrm{op}}_{\cO_K}$, let
\[
(A,\rho),(\cG,\rho')\in (\mathrm{Ig}^b_{\cO_K}\times \mathfrak{M}^b)(R)\ .
\]
Thus, $A\in \mathscr{S}_{K_pK^p}(R)$ is an abelian variety with extra structure, equipped with an isomorphism
\[
\rho: A[p^\infty]\cong (\mathbb X_b)_{\cO_K}\times_{\cO_K} R\ .
\]
On the other hand, $\cG$ is a $p$-divisible group with extra structure over $R$, equipped with a quasi-isogeny $\rho^\prime$ to $\mathbb X_b$ over $R/p$, which lifts uniquely to a quasi-isogeny (denoted in the same way)
\[
\rho^\prime: \cG\to (\mathbb X_b)_{\cO_K}\times_{\cO_K} R\ .
\]
We get the composite quasi-isogeny $\cG\to A[p^\infty]$. It follows that there is a unique quasi-isogeny of $p$-power order $A^\prime\to A$ such that $A^\prime[p^\infty]\to A[p^\infty]$ gets identified with $\cG\to A[p^\infty]$. This defines a new point $A^\prime\in \mathscr{S}_{K_pK^p}(R)$, which comes equipped with a quasi-isogeny
\[
\rho^\prime: A^\prime[p^\infty] = \cG\to (\mathbb X_b)_{\cO_K}\times_{\cO_K} R\ ,
\]
and in particular a quasi-isogeny to $\mathbb X_b$ over $R/p$.

\begin{lemma}\label{integral product} The map constructed above induces an isomorphism, and fits into a commutative diagram
\[
\xymatrix{\mathrm{Ig}^b_{\cO_K} \times_{\cO_{\breve{E}}} \mathfrak{M}^b\ar[d]\ar[rr]^{\cong} && \mathfrak{X}^b_{\cO_K}\ar[d]\\ \mathfrak{M}^b\ar@{=}[rr]&&\mathfrak{M}^b.
}\]
Here, the first vertical map is projection onto the second factor, and the second vertical map sends $(A,\rho)\in \mathfrak{X}^b$ to $(A[p^\infty],\rho)\in \mathfrak{M}^b$.

In particular, choosing $K=\breve{E}$ above, $\mathfrak{X}^b$ is representable by a formal scheme.
\end{lemma}

\begin{proof} The diagram commutes by construction.

We now define the inverse of the top horizontal map: suppose we are given a pair $(A',\rho')\in \mathfrak{X}^b(R)$. In order to define $(\cG,\rho')\in \mathfrak{M}^b(R)$ we just take $(A'[p^\infty],\rho')$. From the quasi-isogeny
\[
\rho': A'[p^\infty]\to (\mathbb X_b)_{\cO_K}\times_{\cO_K} R\ ,
\]
we find a quasi-isogeny of $p$-power degree $A'\to A$ such that the induced quasi-isogeny
\[
\rho: A[p^\infty]\to (\mathbb X_b)_{\cO_K}\times_{\cO_K} R
\]
is an isomorphism, so we get $(A,\rho)\in \mathrm{Ig}^b_{\cO_K}(R)$. It is easy to verify that this construction is inverse to the horizontal map.
\end{proof}

We would like to say that $\mathfrak{X}^b$ is an $\mathrm{Aut}_G(\widetilde{\mathbb X}_b)$-torsor over the completion of $\mathscr{S}_{K_pK^p}$ along $\mathscr{S}^b_{K_pK^p}$. It is clear that it is a quasitorsor, and it remains to show that the map is locally surjective in some topology, the naive choice of course being the fpqc topology.

If this were true, then one could take the pushout along $\mathrm{Aut}_G(\widetilde{\mathbb X}_b)\to J_b(\mathbb Q_p)$ to get a $J_b(\mathbb Q_p)$-torsor over $\mathscr{S}_{K_pK^p}^b$. This $J_b(\mathbb Q_p)$-torsor can in fact be constructed, as in the following proposition (which will not be used in the sequel, but is included as it fits the current discussion).

\begin{prop} Let $S$ be a scheme over $\bar{\mathbb F}_p$, and let $\mathbb X$ be a $p$-divisible group with extra structure over $S$. Assume that there is some $b\in B(G)$ such that all fibres of $\mathbb X$ are quasi-isogenous to $\mathbb X_b$ (compatibly with extra structures). Then there is a natural $J_b(\mathbb Q_p)$-torsor over $S$ which above any geometric point $\bar{x}\in S$ parametrizes quasi-isogenies between $\mathbb X_{\bar{x}}$ and $\mathbb X_b$ (compatible with extra structures).
\end{prop}

\begin{remark} The $J_b(\mathbb Q_p)$-torsor is to be understood as in~\cite{bhattscholze}; more precisely, there is a sheaf of (abstract) groups on $S_{\mathrm{pro\acute{e}t}}$ corresponding to the topological group $J_b(\mathbb Q_p)$, and we are considering a torsor on $S_{\mathrm{pro\acute{e}t}}$ under this sheaf of groups. If $S$ is connected and locally topologically noetherian and $\bar{x}\in S$ is a geometric base point, this corresponds to a map
\[
\pi_1^{\mathrm{pro\acute{e}t}}(S,\bar{x})\to J_b(\mathbb Q_p)\ .
\]
This map, and the $J_b(\mathbb Q_p)$-torsor, only depend on $\mathbb X$ up to isogeny. We remark that the displayed map may have noncompact image in general, but the image is compact in case $\mathbb X$ admits a slope decomposition (or is isogenous to such an $\mathbb X$); this explains~\cite[Example 4.2]{oort-zink}, where a $p$-divisible group over a non-normal base is constructed which is not isogenous to one admitting a slope filtration. We remark that most Newton strata, e.g.~the basic one, give such examples: For the basic Newton stratum, the image of the displayed homomorphism is a discrete cocompact subgroup of $J_b(\mathbb Q_p)$ related to $p$-adic uniformization.
\end{remark}

\begin{proof} We may assume that $S$ is perfect. In that case, we consider the functor sending any $T\in S_{\mathrm{pro\acute{e}t}}$ to the set of quasi-isogenies between $\mathbb X_T$ and $(\mathbb X_b)_T$, respecting extra structures. This is a $J_b(\mathbb Q_p)$-quasitorsor, and we want to prove that it is a torsor.

First, we check this when $S$ is strictly local, so assume $S=\mathrm{Spec}\ R$ is the spectrum of a strictly henselian perfect ring $R$. In that case, we need to show that there is a quasi-isogeny between $\mathbb X$ and $\mathbb X_b$, compatible with extra structures. As there is such a quasi-isogeny over the special point, the result follows from the following lemma.

\begin{lemma}\label{pdivstrictlylocal} Let $R$ be a strictly henselian perfect ring with residue field $k$. Then the functor $G\mapsto G_k$ from the category of $p$-divisible groups over $R$ with constant Newton polygon, up to isogeny, to $p$-divisible groups over $k$ up to isogeny is an equivalence of categories.
\end{lemma}

\begin{remark}\label{lifthomstrictlylocal} In fact, the proof will show that if $G$ and $H$ are $p$-divisible groups with constant Newton polygon over $R$, then there is a constant $c$ depending only on the heights of $G$ and $H$ such that for any homomorphism $\psi_k: G_k\to H_k$ over $k$, $p^c \psi_k$ lifts to a (necessarily unique) homomorphism $G\to H$. (Cf.~\cite[Corollary 3.4]{oort-zink}.)
\end{remark}

\begin{proof} Choose an embedding $\bar{\mathbb F}_p\hookrightarrow R$. Assume for the moment that we know that any $p$-divisible group $G$ over $R$ with constant Newton polygon is isogenous to $G_{0,R} := G_0\times_{\bar{\mathbb F}_p} R$ for some $p$-divisible group $G_0$ over $\bar{\mathbb F}_p$. By the Dieudonn\'e-Manin classification, the functor in the lemma is essentially surjective. To check fully faithfulness of the functor, we may restrict to calculating $\Hom_R(G,H)[1/p]$ where $G=G_{0,R}$, $H=H_{0,R}$. By fully faithfulness of the Dieudonn\'e module functor over perfect rings (first deduced by Gabber from results of Berthelot, \cite{berthelotvalparfait}, cf.~also~\cite[Theorem D]{lau}), it is then enough to check that for any isocrystal $(D,\varphi)$ over $\bar{\mathbb F}_p$,
\[
(D\otimes W(R)[1/p])^{\varphi = 1} = (D\otimes W(k)[1/p])^{\varphi = 1}\ .
\]
We may assume that $D=D_\lambda$ is simple of slope $\lambda$; if $\lambda\neq 0$, then there are no $\varphi$-invariants, and if $\lambda = 0$, then both sides are equal to $\mathbb Q_p$.

It remains to see that any $p$-divisible group $G$ over $R$ with constant Newton polygon is isogenous to a constant $p$-divisible group.\footnote{Cf.~\cite[Corollary 3.6]{oort-zink} in the case where $R$ is the perfection of a noetherian strictly henselian ring $R^\prime$ and $G$ is defined over $R^\prime$.} More precisely, choose a completely slope divisible $G_0/\bar{\mathbb F}_p$ with an isogeny $\psi_k: G_k\to G_{0,k}$ which one can assume to be of degree bounded only in terms of the height $h$ of $G$. Then we claim that there is a (necessarily unique) quasi-isogeny $\psi: G\to G_{0,R}$ lifting $\psi_k$, and whose degree is bounded only in terms of $h$; i.e.~there is a constant $c=c(h)$ such that $p^c \psi: G\to G_{0,R}$ is an isogeny.

For this, assume first that $R$ is an integral domain, with quotient field $K$. By Lemma~\ref{valringfullyfaithful} (cf.~Remark~\ref{intdomfullyfaithful}), the functor from $p$-divisible groups over $R$ to $p$-divisible groups over $K$ is fully faithful. We can find an isogeny $\psi^\prime_K: G_K\to G_{0,K}$ of degree bounded only in terms of $h$, which then extends to a map $\psi^\prime: G\to G_{0,R}$ of degree bounded only in terms of $h$. Over $k$, $\psi_k$ and $\psi^\prime_k$ differ by a quasi-isogeny of $G_0$ of bounded degree; correcting $\psi^\prime$ by this quasi-isogeny gives the desired quasi-isogeny $\psi: G\to G_{0,R}$ lifting $\psi_k$, which is of bounded degree.

In general, let $\{\p_i\}$ be the minimal prime ideals of $R$ (which may be infinitely many);\footnote{If there are only finitely many, e.g.~if $S$ is the perfection of a noetherian scheme, one can argue as in~\cite[end of proof of Proposition 3.3]{oort-zink}.} then the result holds true over each $R/\p_i$, which is still a strictly henselian perfect ring. Let $\tilde{R}\subset \prod_i R/\p_i$ be the subring of those elements $f=\{f_i\in R/\p_i\}$ for which $\bar{f}:=\bar{f}_i\in k$ is independent of $i$. Then $\tilde{R}$ is another strictly henselian perfect ring, $R\hookrightarrow \tilde{R}$, and there is an isogeny
\[
\psi_{\tilde{R}}: G_{\tilde{R}}\to G_{0,\tilde{R}}\ .
\]
Indeed, $p^c \psi_{\tilde{R}}$ will be an actual isogeny, and then to write down this isogeny, one has to write down many matrices with entries in $\tilde{R}$; but one has these matrices with entries in $R/\p_i$ for each $i$, reducing to the same matrix over $k$. It remains to see that $\psi_{\tilde{R}}$ is defined over $R$, i.e.~that some matrices with coefficients in $\tilde{R}$ have coefficients in $R$. For each $i$, $\tilde{R}/\p_i\tilde{R}$ is a strictly henselian perfect ring, so $\psi_{\tilde{R}/\p_i\tilde{R}}$ is uniquely determined; by uniqueness, it must be given by the base extension of $\psi_{R/\p_i}$, which is already known to exist. Thus, we finish by observing that
\[
R = \{f\in \tilde{R}\mid \forall i: f\mod \p_i\in R/\p_i\subset \tilde{R}/\p_i\tilde{R}\}\ .
\]
To verify the displayed equation, we observe that $R\to \tilde{R}$ is a $v$-cover in the sense of~\cite{bhattscholzeperf}, so that by~\cite[Theorem 4.1 (i)]{bhattscholzeperf} (applied to $\mathcal{E} = \cO_X$)
\[
R = \{f\in \tilde{R}\mid f\otimes 1 = 1\otimes f\in \tilde{R}\otimes_R \tilde{R}\}\ .
\]
As everything is reduced, the latter equality can be checked as a system of equalities in
\[
(\tilde{R}\otimes_R \tilde{R})/\p_i(\tilde{R}\otimes_R \tilde{R}) = \tilde{R}/\p_i\tilde{R}\otimes_{R/\p_i} \tilde{R}/\p_i\tilde{R}\ ,
\]
as desired.
\end{proof}

Now we go back to a general perfect base scheme $S$. We need to find a quasi-isogeny between $\mathbb X$ and $\mathbb X_b$ (compatible with extra structures) locally on $S_{\mathrm{pro\acute{e}t}}$. For any geometric point $\bar{x}\in S$, we can find such a quasi-isogeny over $S_{\bar{x}}$. Thus, fixing any $n$, after replacing $S$ by an \'etale neighborhood of $\bar{x}$ and $\mathbb X$ by a quasi-isogenous $p$-divisible group, we can assume that there is an isomorphism $\mathbb X[p^n]\cong \mathbb X_b[p^n]$ compatible with extra structure.

In that case, we can look at the $K_b$-quasitorsor $\tilde{S}\to S$ of isomorphisms $\mathbb X_T\cong (\mathbb X_b)_T$ compatible with extra structures on the category of perfect $S$-schemes $T$, where $K_b\subset J_b(\mathbb Q_p)$ is the compact open subgroup of automorphisms of $\mathbb X_b$, compatible with extra structures. Note that $\tilde{S}$ is representable by a perfect scheme. We claim that if $n$ was chosen large enough (depending only on $\mathbb X_b$), then this quasitorsor is a torsor, i.e.~$\tilde{S}\to S$ is faithfully flat. This will then give the desired quasi-isogeny locally on $S_{\mathrm{pro\acute{e}t}}$ (namely over the pro-\'etale cover $\tilde{S}\to S$).

To show that $\tilde{S}$ is a torsor, we need to see that it is faithfully flat, so we can assume that $S=\mathrm{Spec}\ R$ is strictly local. We need to show that there is an isomorphism $\mathbb X\cong (\mathbb X_b)_R$ compatible with extra structures, assuming that such an isomorphism exists on $p^n$-torsion for $n$ big enough.

As before, let $k$ be the residue field of $R$. Then $\mathbb X_k$ and $\mathbb X_b$ have isomorphic $p^n$-torsion; from~\cite[Lemma 4.4]{scholzeLKgen} one deduces that there is an isomorphism $\psi_x: \mathbb X_b\cong \mathbb X_k$ compatible with extra structures, if $n$ was chosen large enough; moreover, one can assume that this isomorphism reduces to the given one $\mathbb X_b[p^n]\cong \mathbb X_k[p^n]$ on $p^{n/2}$-torsion (say, $n=2m$ is even). From Lemma~\ref{pdivstrictlylocal} and Remark~\ref{lifthomstrictlylocal}, we see that $\psi_x$ lifts to a quasi-isogeny $\psi: (\mathbb X_b)_R\to \mathbb X$, such that $p^c \psi: (\mathbb X_b)_R\to \mathbb X$ and $p^c \psi^{-1}: \mathbb X\to (\mathbb X_b)_R$ are actual isogenies, where $c$ is a constant depending only on $\mathbb X_b$. Then the kernel $G\subset (\mathbb X_b)_R$ of $p^c \psi$ is contained in the $p^{2c}$-torsion; thus, it is the kernel of $p^c \psi: (\mathbb X_b)_R[p^{2c}]\to \mathbb X[p^{2c}]\cong (\mathbb X_b)_R[p^{2c}]$ (if $m\geq 2c$, which we may assume). By choosing $m$ large enough and using Lemma~\ref{precise chai-oort}, we may arrange that $p^c \psi$ lies in $\cH_{\mathbb X_b,\mathbb X_b}[p^{2c}](R)$. But as $R$ is strictly henselian perfect,
\[
\cH_{\mathbb X_b,\mathbb X_b}[p^{2c}](R) = \cH_{\mathbb X_b,\mathbb X_b}[p^{2c}](\bar{\mathbb F}_p)\ .
\]
It follows that $G\subset (\mathbb X_b)_R$ is constant, $G=G_{0,R}$, for $G_0\subset \mathbb X_b$, with $\mathbb X_b/G_0\cong \mathbb X_b$ compatibly with extra structures (as this is true over $k$). But then $p^c\psi$ factors over an isomorphism
\[
(\mathbb X_b/G_0)_R\cong \mathbb X\ ,
\]
where the left-hand side is isomorphic to $(\mathbb X_b)_R$. This gives the desired isomorphism $\mathbb X\cong (\mathbb X_b)_R$ compatible with extra structures.
\end{proof}

Now we go back to the study of Igusa varieties. Let $\cX^b:=(\mathfrak{X}^b)^\mathrm{ad}_\eta$ be the adic generic fiber of the formal scheme $\mathfrak{X}^b$.

\begin{defn} Let $\cX^b_\infty$ be the functor on complete affinoid $(\breve{E}(\zeta_{p^\infty}),\cO_{\breve{E}(\zeta_{p^\infty})})$-algebras sending $(R,R^+)$ to the set of triples $(\cA, \rho, \alpha)$, where $(\cA,\rho)\in \cX^b(R,R^+)$ and \[\alpha:\Lambda \to T_p\cA \] is a morphism of $\cO_B$-modules such that
\begin{enumerate}
\item the pairing $(\cdot,\cdot)$ on $\Lambda$ matches the pairing on $T_p\cA$ induced by the polarization and the fixed choice of $p$-power roots of unity, and
\item the induced maps \[\Lambda\to T_p\cA^\mathrm{ad}_{\eta}(C,C^+),\] on all geometric points $\mathrm{Spa}(C,C^+)$ of $\mathrm{Spa}(R,R^+)$ are isomorphisms.
\end{enumerate}
\end{defn}

\begin{remark}\label{cartesian diagram 1} There are natural maps $\mathfrak{X}^b\to \mathfrak{M}^b$ and $\cX^b_\infty \to \cM^b_\infty$, defined by sending an abelian variety to its $p$-divisible group. We can check on the level of moduli problems that $\cX^b_\infty$ fits into the Cartesian diagram \[\xymatrix{\cX^b_\infty\ar[r]\ar[d]&\cM^b_\infty\ar[d]\\ \mathcal{X}^b\ar[r]&\cM^b\ ,}\] therefore it is representable by an adic space.
\end{remark}

We let $(\mathrm{Ig}^b_{\cO_K})^\mathrm{ad}_\eta$ be the generic fiber of the formal scheme $\mathrm{Ig}^b_{\cO_K}$.

\begin{cor}\label{adic product} We have an isomorphism \[(\mathrm{Ig}^b_{\cO_K})^\mathrm{ad}_\eta\times_{\mathrm{Spa}(\breve{E},\cO_{\breve{E}})} \cM^b_\infty\toisom \cX^b_{\infty,K}.\]

In particular, $\cX^b_\infty$ is preperfectoid.
\end{cor}

\begin{proof} The first part follows from the decomposition of $\mathfrak{X}^b$ in Lemma~\ref{integral product}, and the cartesian diagram of Remark~\ref{cartesian diagram 1}. The final assertion follows formally from the facts that $\cM^b_\infty$ is preperfectoid, and that $\mathrm{Ig}^b_{\cO_K}$ is locally of the form $W(R)\otimes_{\cO_{\breve{E}}} \cO_K$ for a perfect ring $R$, so that (if $K$ is perfectoid) its generic fibre is a perfectoid space.
\end{proof}

We let $\widehat{\cX}^b_\infty$ be the perfectoid space associated with $\cX^b_\infty$ as in~\cite[Proposition 2.3.6]{scholzeweinstein}. Let $\cS_{K^p}$ be the perfectoid infinite-level Shimura variety over $E_{\mathfrak p}$. Let $\cS_{K^p}^b\subset \cS_{K^p}$ be the locus of those points $\mathrm{Spa}(K,K^+)\to \cS_{K^p}$ over which the universal abelian variety over $K$ extends to $K^+$, and defines a point of $\mathscr{S}_{K_pK^p}^b$ over the residue field of $K^+$. This is the preimage under the continuous specialization map of the locally closed subset $\mathscr{S}_{K_pK^p}^b\subset \mathscr{S}_{K_pK^p}\times_{\cO_{E,\mathfrak p}} \mathbb F_q$, and thus $\cS_{K^p}^b\subset \cS_{K^p}$ is a locally closed subset.

\begin{lemma}\label{cartesian diagram 2} The perfectoid space $\widehat{\cX}^b$ maps to $\cS_{K^p}^b$ by forgetting the quasi-isogeny $\rho$ and to $\cM^b_\infty$ by sending $(\cA,\rho)$ to $(\cA[p^\infty],\rho)$. The induced map
\[
\widehat{\cX}^b_\infty\to (\cM_\infty^b\times_{\Fl_{G,\mu}} \cS_{K^p}^b)^\wedge
\]
is an isomorphism of perfectoid spaces.
\end{lemma}

In other words,
\[\xymatrix{
\cX^b_\infty\ar[r]\ar[d] & \cM_\infty^b\ar[d]^{\pi^b_{HT}} \\
\cS_{K^p}^b\ar[r]^{\pi_{HT}} & \Fl_{G,\mu}
}\]
becomes a Cartesian diagram when one takes points over a perfectoid space.

\begin{proof} Note that the diagram commutes by Remark~\ref{local and global HT}. Therefore, the map in the lemma is well-defined. We first check the fact that the diagram is Cartesian on $(C,\cO_C)$-points, where $C/\breve{E}(\zeta_{p^\infty})$ is complete and algebraically closed with ring of integers $\cO_C$. A $(C,\cO_C)$-point of $\cS_{K^p}^b$ gives rise to a couple $(\cA,\alpha)$, where $\cA/\cO_C$ is an abelian variety with extra structures and $\alpha:\Lambda \to T_p\cA(C,\cO_C)$ is an isomorphism compatible with extra structures. A $(C, \cO_C)$-point of $\cM^b_\infty$ gives us a triple $(\cG,\beta, \rho)$, where $\cG/\cO_C$ is a $p$-divisible group with extra structures, $\beta: \Lambda \toisom T_p\cG(C,\cO_C)$ is a trivialization of its integral Tate module and $\cG\times_{\cO_C} \cO_C/p\to \mathbb{X}_b \times_{\mathbb{\bar F}_p} \cO_C/p$ is a quasi-isogeny. 

The fact that $(\cA,\alpha)$ and $(\cG,\beta, \rho)$ are mapped to the same point of $\Fl_{G,\mu}$ under $\pi_{HT}$ and $\pi^b_{HT}$ means that the Hodge-Tate filtrations on $T_p\cA\otimes C$ and $T_p\cG\otimes C$ are identified under the isomorphism $\beta\circ \alpha^{-1}$. Now \cite[Theorem B]{scholzeweinstein} gives an isomorphism $\cA[p^\infty]\cong \cG$ extending the given isomorphism on the generic fibre. Thus, the given data assemble into a point of $\cX^b_\infty$, and one checks that these constructions are inverse.

Now, if $(R,R^+)$ is any perfectoid affinoid $\breve{E}(\zeta_{p^\infty})$-algebra, one gets similar data $(\cA,\alpha)$, $(\cG,\beta,\rho)$ over $R^+$. One has to check that the isomorphism $\beta\circ \alpha^{-1}$ between $\cA[p^\infty]_R$ and $\cG_R$ extends to $R^+$. This follows from Lemma~\ref{extend map from generic fibre} above.
\end{proof}

Putting together Remark~\ref{cartesian diagram 1} and Lemma~\ref{cartesian diagram 2}, we get a diagram with Cartesian squares (the right one when evaluated on perfectoid spaces)
\[
\xymatrix{
\mathfrak{X}^b\ar[d] & \cX^b_\infty\ar[l]\ar[r]\ar[d] & \cS_{K^p}^b\ar[d]^{\pi_{HT}^b}\\
\mathfrak{M}^b & \cM_\infty^b\ar[l]\ar[r]_{\pi_{HT}} & \Fl_{G,\mu}\ .
}\]

\subsection{\'Etale cohomology}

Fix a prime $\ell\neq p$, and consider the map
\[
\pi_{HT}: \cS_{K^p}\to \Fl_{G,\mu}\ .
\]
In this final subsection, we use the geometric results established so far to identify the fibres of $\mathcal{F} = R\pi_{HT*}\mathbb Z/\ell^n \mathbb Z$ with the cohomology of Igusa varieties. In this section, we make the additional assumption that $\mathscr{S}_{K_pK^p}$ is proper over $\cO_{E,\mathfrak p}$. It is known that this is equivalent to asking that $G$ is anisotropic over $\mathbb Q$, cf.~\cite{lan}.

Let $C$ be a complete algebraically closed extension of $\breve{E}(\zeta_{p^\infty})$, with an open and bounded valuation subring $C^+\subset C$, and fix a point $x\in \Fl_{G,\mu}(C,C^+)$; we assume that $C$ is the completed algebraic closure of the residue field of $\Fl_{G,\mu}$ at the underlying (topological) point. We are interested in understanding the stalk $\mathcal{F}_x = (R^i\pi_{HT*} \mathbb Z/\ell^n \mathbb Z)_x$. In this respect, we have the following general base change lemma.

\begin{lemma}\label{base change adic} Let $f: Y\to X$ be a quasicompact and quasiseparated map of analytic adic spaces, and for definiteness assume that $X$ is either a locally strongly noetherian adic space or a perfectoid space over $\mathrm{Spa}(\mathbb Z_p,\mathbb Z_p)$, and $Y$ is perfectoid.\footnote{We only need to know that they have well-defined \'etale sites, and that the same holds for all fibres of $f$ over geometric points. For example, the lemma is also true when one asssumes instead that both $X$ and $Y$ are perfectoid.} Let $x\in X$ be a point with residue field $K(x)$ and open and bounded valuation subring $K(x)^+$. Let $C(\bar{x})$ be a completed algebraic closure of $K(x)$ with an open and bounded valuation subring $C(\bar{x})^+\subset C(\bar{x})$ lifting $K(x)^+$, giving rise to a geometric point $\bar{x}=\mathrm{Spa}(C(\bar{x}),C(\bar{x})^+)\to X$. Let
\[
Y_{\bar{x}} = (Y\times_X \mathrm{Spa}(C(\bar{x}),C(\bar{x})^+))^\wedge
\]
be the fibre of $Y$ over $\bar{x}$, which is a perfectoid space over $C(\bar{x})$. For any sheaf $\mathcal G$ of abelian groups on $Y_{\mathrm{\acute{e}t}}$ and all $i\geq 0$, the natural map
\[
(R^i f_\ast \mathcal G)_{\bar{x}}\to H^i(Y_{\bar{x}},\mathcal G)
\]
is an isomorphism.
\end{lemma}

Here, and in the following, these statements will also be true for sheaves of groups and $i=0,1$, and sheaves of sets and $i=0$. We will not spell this out.

\begin{proof} Let $U_j = \mathrm{Spa}(R_j,R_j^+)\to X$ be a cofinal system of affinoid \'etale neighborhoods of $\bar{x}$; then
\[
\mathrm{Spa}(C(x),C(x)^+)\sim \varprojlim_j \mathrm{Spa}(R_j,R_j^+)\ ,
\]
and one has
\[
(R^i f_\ast \mathcal G)_{\bar{x}} = \varinjlim_j H^i(Y\times_X \mathrm{Spa}(R_j,R_j^+),\mathcal{G})\ .
\]
It remains to see that
\[
\varinjlim_j H^i(Y\times_X \mathrm{Spa}(R_j,R_j^+),\mathcal{G}) = H^i(Y_{\bar{x}},\mathcal G)\ .
\]
But this follows from $Y_{\bar{x}}\sim \varprojlim_j Y\times_X \mathrm{Spa}(R_j,R_j^+)$ (cf.~\cite[Proposition 2.4.3]{scholzeweinstein}), where all terms are quasicompact and quasiseparated, and the resulting consequence for \'etale cohomology, cf.~\cite[Corollary 7.18]{scholzeperfectoid}.\footnote{In the discussion around \cite[Corollary 7.18]{scholzeperfectoid}, the $X_i$ are assumed to be strongly noetherian; the discussion is also valid if all $X_i$ are perfectoid.}
\end{proof}

In particular, the fibre
\[
(R^i\pi_{HT\ast} \mathbb Z/\ell^n\mathbb Z)_x = H^i(\cS_{K^p,x},\mathbb Z/\ell^n \mathbb Z)\ .
\]
Next, we reduce to the case of rank $1$ points. For this, we use the following lemma.

\begin{lemma}\label{only rank 1 matters} Let $X$ be a quasicompact and quasiseparated analytic adic space, and for definiteness assume that $X$ is a perfectoid space.\footnote{Again, the lemma also holds true when $X$ is a strongly noetherian adic space, or whenever $X$ has a well-behaved \'etale site.} Let $U\subset X$ be a quasicompact open subset which contains all rank $1$ points of $X$. Then, for any locally constant sheaf $\mathcal G$ of abelian groups on $X_{\mathrm{\acute{e}t}}$ and all $i\geq 0$, the natural map
\[
H^i(X,\mathcal{G})\to H^i(U,\mathcal G)
\]
is an isomorphism.
\end{lemma}

\begin{proof} Let $j: U\hookrightarrow X$ be the inclusion. It is enough to prove that $\mathcal G\to Rj_\ast \mathcal G$ is an isomorphism. This can be checked on geometric points, which, using Lemma~\ref{base change adic}, reduces us to the case $X=\mathrm{Spa}(C,C^+)$ for some complete algebraically closed field $C$ with an open and bounded valuation subring $C^+\subset C$. Then $U=\mathrm{Spa}(C,D^+)$ for a different open and bounded valuation subring $D^+\subset C$, containing $C^+$. As $X$ is strictly local, the sheaf $\mathcal G$ is the constant sheaf associated with some abelian group $G$. But as any \'etale cover of $X$ splits, one has $R\Gamma(X,G) = G$, and similarly for $U$, giving the result.
\end{proof}

Applying Lemma~\ref{only rank 1 matters} to the inclusion $\cS_{K^p,\tilde{x}}\subset \cS_{K^p,x}$ shows that
\[
(R^i\pi_{HT\ast} \mathbb Z/\ell^n\mathbb Z)_x = (R^i\pi_{HT\ast} \mathbb Z/\ell^n\mathbb Z)_{\tilde{x}} = H^i(\cS_{K^p,\tilde{x}},\mathbb Z/\ell^n \mathbb Z)\ .
\]
Thus, we will from now on assume that $x=\tilde{x}$ is a rank $1$ point, and write $C^+ = \cO_C$. Now choose $b\in B(G,\mu^{-1})$ such that $x\in \Fl_{G,\mu}^b$. If $y\in \cS_{K^p,x}$ is any (geometric) rank $1$ point, the argument of Lemma~\ref{surjectivitypiHT} shows that $y\in \cS_{K^p}^b$. Thus, $\cS_{K^p,x}^b=\cS_{K^p,x}\times_{\cS_{K^p}} \cS_{K^p}^b\subset \cS_{K^p,x}$ is a quasicompact open subset with the same rank $1$ points, so applying Lemma~\ref{only rank 1 matters} once more, we see that
\[
(R^i\pi_{HT\ast} \mathbb Z/\ell^n\mathbb Z)_x = H^i(\cS_{K^p,x}^b,\mathbb Z/\ell^n \mathbb Z)\ .
\]

Now we apply Lemma~\ref{surjectivitypiHT} to lift $x\in \Fl_{G,\mu}^b(C,\cO_C)$ to a point $z\in \cM^b_\infty(C,\cO_C)$, giving rise in particular to a $p$-divisible group $(\mathbb X_b)_{\cO_C}$ (with extra structures) lifting $\mathbb X_b$. Then Lemma~\ref{cartesian diagram 2} identifies the fibre $\cS_{K^p,x}^b$ with the fibre $\cX^b_{\infty,z}$. This, in turn, gets identified with $(\mathrm{Ig}^b_{\cO_C})^{\mathrm{ad}}_\eta$ by Corollary~\ref{adic product}. Combining the discussion so far, we see that
\[
(R^i\pi_{HT\ast} \mathbb Z/\ell^n \mathbb Z)_x = H^i((\mathrm{Ig}^b_{\cO_C})^{\mathrm{ad}}_\eta,\mathbb Z/\ell^n \mathbb Z)\ .
\]
Next, we pass to the special fibre.

\begin{lemma}\label{comparison to special fibre} Let $X/\bar{\mathbb F}_p$ be a perfect scheme and let $C$ be a complete algebraically closed nonarchimedean field whose residue field contains $\bar{\mathbb F}_p$. Let $\mathfrak{X}_{\cO_C}$ be the flat formal scheme over $\mathrm{Spf}\ \cO_C$ which is the unique lifting of $X\times_{\bar{\mathbb F}_p} \cO_C/p$, and let $\cX_C = (\mathfrak{X}_{\cO_C})^{\mathrm{ad}}_\eta$ be its generic fibre, which is a perfectoid space. For all $i$, the canonical maps
\[
H^i(X,\mathbb Z/\ell^n \mathbb Z)\leftarrow H^i(\mathfrak{X}_{\cO_C},\mathbb Z/\ell^n\mathbb Z)\to H^i(\cX_C,\mathbb Z/\ell^n \mathbb Z)
\]
are isomorphisms.
\end{lemma}

\begin{proof} The question is local on $X$, so we can assume that $X$ is affine. Then we can write $X=\varprojlim X_j$ as a cofiltered inverse limit of affine schemes $X_j$ which are perfections of schemes of finite type over $\bar{\mathbb F}_p$. One also gets $\cX_C\sim\varprojlim_j \cX_{j,C}$, so all cohomology groups in question become a filtered colimit over $j$; thus, we can assume that $X$ is the perfection of an affine scheme $X_0$ of finite type. Then the cohomology of $X$ agrees with the cohomology of $X_0$.

Moreover, the cohomology of $\mathfrak{X}_{\cO_C}$ is the same as the cohomology of its special fibre $X\times_{\bar{\mathbb F}_p} k$, where $k$ is the residue field of $\cO_C$, which in turn agrees with the cohomology of $X_0\times_{\bar{\mathbb F}_p} k$. Thus, the first map is an isomorphism by invariance of \'etale cohomology under change of algebraically closed base field.

Also, under tilting, the \'etale cohomology of $\cX_C$ agrees with the \'etale cohomology of $\cX_{C^\flat}$. We may thus assume that $C$ is of characteristic $p$. In that case, one can also form $\mathfrak{X}_{0,\cO_C} = X_0\times_{\mathrm{Spec}\ \bar{\mathbb F}_p} \mathrm{Spf}\ \cO_C$ and its generic fibre $\cX_{0,C}$, which is a rigid-analytic variety over $C$, with $\cX_C\sim\varprojlim_{\mathrm{Frob}} \cX_{0,C}$. Thus, the cohomology of $\cX_C$ agrees with the cohomology of $\cX_{0,C}$. Finally, we are reduced to proving that the map
\[
H^i(\mathfrak{X}_{0,\cO_C},\mathbb Z/\ell^n \mathbb Z)\to H^i(\cX_{0,C},\mathbb Z/\ell^n \mathbb Z)
\]
is an isomorphism. The right hand side can be computed, by~\cite[Corollary 3.5.17]{huber}, in terms of $H^{i-j}(X_0\times_{\bar{\mathbb F}_p}k,R^{j}\psi\mathbb Z/\ell^n \mathbb Z)$. 

It is enough to see that, if $X_0$ is a scheme of finite type over $\bar{\mathbb{F}}_p$, then the complex of nearby cycles of $X_{0,C}=X_0\times_{\bar{\mathbb{F}}_p}C$ is quasi-isomorphic to the constant sheaf $\mathbb Z/\ell^n \mathbb Z$. By~\cite[XIII 2.1.4]{SGA7.2}, we can compute the stalk of $R^{j}\psi\mathbb Z/\ell^n \mathbb Z$ at a geometric point $\bar x$ as $H^{j}((X_{0,\cO_C})_{\bar x}\times C,\mathbb Z/\ell^n \mathbb Z)$, with $(X_{0,\cO_C})_{\bar x}$ the strict henselization of $X_{0,\cO_C}=X_0\times_{\bar{\mathbb F}_p} \cO_C$ at $\bar x$. We conclude, since the map $X_{0,C}\to \mathrm{Spec}\ \cO_C$ is the base change along the map $\mathrm{Spec}\ \cO_C\to\mathrm{Spec}\ \bar{\mathbb F}_p$ of the universally locally acyclic map $X_0\to \bar{\mathbb{F}}_p$. (For universal local acyclicity, we use the definition of~\cite{SGA41/2}. Every scheme of finite type is universally locally acyclic over a point, cf.~\cite[Th. finitude, Th\'eor\`eme 2.13]{SGA41/2}.)
\end{proof}

Thus, we get
\[
(R^i\pi_{HT\ast} \mathbb Z/\ell^n \mathbb Z)_x = H^i(\mathrm{Ig}^b,\mathbb Z/\ell^n \mathbb Z)\ ,
\]
where $\mathrm{Ig}^b / \bar{\mathbb F}_p$ is the perfect scheme introduced in Definition~\ref{new Igusa variety}. Using Proposition~\ref{perfection}, we finally arrive at the following formula.

\begin{thm}\label{comp fiber} For any geometric point $\bar{x}$ of $\Fl_{G,\mu}$ contained in $\Fl_{G,\mu}^b$, there is an isomorphism
\[
(R^i\pi_{HT\ast} \mathbb Z/\ell^n \mathbb Z)_{\bar{x}} = H^i(\mathrm{Ig}^b,\mathbb Z/\ell^n \mathbb Z) = \varinjlim_m H^i(\mathscr{I}^b_{\mathrm{Mant},m},\mathbb Z/\ell^n\mathbb Z)\ .
\]
It (only) depends on the choice of a lift of $\bar{x}$ to $\cM^b_\infty$, and is compatible with the Hecke action of $G(\mathbb A_f^p)$.$\hfill \Box$
\end{thm}

One can formulate a version of this result where one replaces $\mathbb Z/\ell^n \mathbb Z$ by the local system corresponding to an algebraic representation $\xi$.
\newpage

\section{The cohomology of Igusa varieties}\label{Igusa varieties}

The goal of this section is to compute the alternating sum of cohomology groups $[H(\mathfrak{Ig}^b,\bar{\mathbb{Q}}_\ell)]$ as a virtual representation of $G(\mathbb{A}_f^p)\times J_b(\mathbb{Q}_p)$. We will work with (the Igusa varieties corresponding to) unitary Shimura varieties. Our setup is similar to that of~\cite{scholze-shin} (see Section~\ref{setup} for more detail) and we intend to prove a version of Theorem 6.1 of~\cite{shin-galois} in this situation.

By Proposition~\ref{perfection} and since perfection does not change the \'etale topos, it is enough to work with the classical objects $\mathscr{I}^b_{\mathrm{Mant}}$. By Poincar\'e duality, it is enough to compute the alternating sum of the compactly supported cohomology groups. Sug Woo Shin has derived a formula for the alternating sum $[H_c(\mathscr{I}^b_{\mathrm{Mant}},\bar{\mathbb{Q}}_\ell)]$ as a sum of stable orbital integrals for $G$ and its elliptic endoscopic groups (see Theorem~\ref{stable trace formula for Igusa}). We reinterpret this formula as the geometric side of the twisted trace formula and compare it to the spectral side. 

\subsection{Setup}\label{setup} We assume that $F=F^{+}\cdot \cK$ is the composition of a totally real field $F^+$ and an imaginary quadratic field $\cK$. Let $c\in \Gal(F/F^+)$ be the non-trivial element. Let $G/\mathbb{Q}$ be a unitary similitude group preserving an alternating hermitian form $\langle\ ,\ \rangle$ on an $F$-vector space $V$ of dimension $n$. Let $\Spl_{F/F^+}$ denote the set of rational primes $v$ such that every prime of $F^+$ above $v$ splits in $F$. We make the following further assumptions on $F$ and $G$. 

\begin{enumerate}
\item $F^+\neq \mathbb{Q}$;
\item the set of rational primes which are ramified in $F$ is contained in $\Spl_{F/F^+}$;
\item $G$ is quasi-split at all finite places.
\end{enumerate}

\noindent See Section 10 of~\cite{scholze-shin} for a discussion of these conditions. The first two are imposed to avoid issues with $L$-packets and base change for unitary groups.\footnote{Actually, (2) implies (1).} The third condition implies that endoscopic representations will contribute to $[H_c(\mathscr{I}^b_{\mathrm{Mant}},\bar{\mathbb{Q}}_\ell)]$, and is thus in some sense the hardest case.

Let $h:\mathbb{C}\to \mathrm{End}_F(V)_{\mathbb{R}}$ be an $\mathbb{R}$-algebra homomorphism such that $h(z^c)=h(z)^c$ for all $z\in \mathbb{C}$ and such that the bilinear pairing $(v,w)\mapsto \langle v,h(i)w\rangle$ is symmetric and positive definite. Then $(F,c,V,\langle\ ,\ \rangle,h)$ is a Shimura datum of PEL type. The fact that $c$ is an involution of the second kind implies that the PEL datum is of type $(A)$, according to the classification on page 375 of~\cite{kottwitz}. 

The $\mathbb{R}$-algebra homomorphism $h$ induces a homomorphism of algebraic groups $h:\mathrm{Res}_{\mathbb{C}/\mathbb{R}}\mathbb{G}_m\to G_\mathbb{R}$. Then $(G,\{h\})$ is a Shimura datum as in Section~\ref{recollections for Hodge type}. For $K\subset G(\mathbb{A}_f)$ an open compact subgroup, we can define the Shimura variety $S_K$, which has a model over the reflex field $E$. Let $\mu$ be the Hodge cocharacter corresponding to $h$. We follow the slight abuse of notation in denoting by $S_K$ not the actual Shimura variety, but the PEL moduli problem, which is the disjoint union of $|\mathrm{ker}^1(G,\mathbb Q)|$ copies of the actual Shimura variety. This factor $|\mathrm{ker}^1(\mathbb Q,G)|$ will thus appear in many formulas below.

Also assume that the prime $p$ is unramified in $F$ and splits in $\cK$ (so, in particular, it lies in $\Spl_{F/F^+}$).

Let $\p$ be a prime in the reflex field $E$ of the Shimura datum above the rational prime $p$. Let $K\subset G(\mathbb{A}_f)$ be a compact open subgroup which is sufficiently small and has the form $K^pK_p$, such that $K_p\subset G(\mathbb{Q}_p)$ is hyperspecial. The fact that $p$ is unramified in $F$ means that good integral models $\mathscr{S}_K$ of $S_K$ exist over $\cO_{E_\p}$.

We fix a field isomorphism $\iota_\ell:\mathbb{\bar Q}_\ell \toisom \mathbb{C}$ throughout. If $\mathrm{G}$ is a topological group, such that every neighborhood of the identity contains a compact-open subgroup and $\Omega$ is an algebraically closed field of characteristic $0$, we let $C_c^\infty(\mathrm{G})$ be the space of smooth, compactly supported, $\Omega$-valued functions on $\mathrm{G}$ (usually they will be $\mathbb{C}$-valued; if they are valued in $\mathbb{\bar Q}_\ell$, then by smooth we mean locally-constant). We let $\mathrm{Irr}(\mathrm{G})$ denote the set of isomorphism classes of irreducible admissible representations of $\mathrm{G}$ over $\Omega$ and $\mathrm{Groth}(\mathrm{G})$ be the corresponding Grothendieck group. For all the groups we consider, we choose Haar measures and transfer factors as in~\cite{shin-stable, shin-galois}. 

In particular, if $\mathrm{G}$ is an unramified group over a non-archimedean field $\mathrm{F}$, we choose a hyperspecial maximal compact subgroup $\mathrm{K}$ and a Haar measure such that $\mathrm{K}$ has volume $1$. We let $\cH^{\mathrm{ur}}(\mathrm{G}(\mathrm{F}))$ be the subspace of $C_c^\infty(\mathrm{G}(\mathrm{F}))$ consisting of bi-$\mathrm{K}$-invariant functions, which is an algebra with respect to convolution. 

\subsection{A stable trace formula}\label{stable formula} In this section, we recall the main constructions and results of~\cite{shin-stable}. For any open compact subgroup $K\subset G(\mathbb{A}_f)$ which is hyperspecial at $p$ we have an integral model $\mathscr{S}_K/\cO_{E_\p}$. As described in Section~\ref{a product formula}, the special fiber of $\mathscr{S}_{K}$ has a Newton polygon stratification, in terms of elements $b\in B(G, \mu^{-1})$. Fix $b$ and also a $p$-divisible group with extra structures $\mathbb{X}_b/\mathbb{\bar F}_p$ as in Section~\ref{a product formula}. Recall that $J_b(\mathbb{Q}_p)$ is the group of quasi-self-isogenies of $\mathbb{X}_b$ which respect all the extra structures. 

By \emph{the Igusa variety} $\mathscr{I}^b_{\mathrm{Mant}}$ we mean the projective system of $\mathbb{\bar F}_p$-schemes $\mathscr{I}^b_{\mathrm{Mant},K^p,m}$, where $K^p\subset G(\mathbb{A}_f^p)$ runs over sufficiently small open compact subgroups and $m$ runs over positive integers. Each of these schemes is a finite Galois cover of the leaf $\mathscr{C}_b$ inside $\mathscr{S}_K^b$. Define 
\[[H_c(\mathscr{I}^b_{\mathrm{Mant}},\mathbb{\bar Q}_\ell)]:=\bigoplus_k (-1)^k\varinjlim_{K^p,m}H^k_c(\mathscr{I}^b_{\mathrm{Mant},K^p,m},\mathbb{\bar Q}_\ell).\] 
Since each of the summands is an admissible representation of $G(\mathbb{A}_{f}^p)\times J_b(\mathbb{Q}_p)$, we think of $[H_c(\mathscr{I}^b_{\mathrm{Mant}},\mathbb{\bar Q}_\ell)]$ as a virtual representation in $\mathrm{Groth}(G(\mathbb{A}_{f}^p)\times J_b(\mathbb{Q}_p))$.

Often, we will fix a finite set $S$ of places of $\mathbb Q$ including $p,\infty$ and all places at which $F$ ramifies. If we fix a compact open subgroup $K^S\subset G(\mathbb A^S)$ which is a product of hyperspecial maximal compact open subgroups $K_q\subset G(\mathbb Q_q)$, we let
\[
[H_c(\mathscr{I}^b_{\mathrm{Mant}},\bar{\mathbb Q}_\ell)]^{S\ur}
\]
be the summand of $[H_c(\mathscr{I}^b_{\mathrm{Mant}},\bar{\mathbb Q}_\ell)]$ of those representations which are unramified outside $S$. More precisely, any element $\pi\in \mathrm{Groth}(G(\mathbb{A}_{f}^p)\times J_b(\mathbb{Q}_p))$ can be written as a (possibly infinite) sum
\[
\pi = \sum_i n_i \pi_i\ ,
\]
where $\pi_i$ runs through the irreducible representations of $G(\mathbb A_f^p)\times J_b(\mathbb Q_p)$ (all of which decompose into a tensor product), $n_i\in \mathbb Z$, and for each compact open subgroup $K\subset G(\mathbb A_f^p)\times J_b(\mathbb Q_p)$, there are only finitely many $i$ for which $n_i\neq 0$ and $\pi_i^K\neq 0$. Then we define
\[
\pi^{S\ur} = \sum_{i: \pi_i^{K^S}\neq 0} n_i \pi_i\ .
\]
Let $\mathrm{Groth}(G(\mathbb{A}_{f}^p)\times J_b(\mathbb{Q}_p))^{S\ur}$ denote the subgroup of $\mathrm{Groth}(G(\mathbb{A}_{f}^p)\times J_b(\mathbb{Q}_p))$ consisting of those $\pi$ for which $\pi=\pi^{S\ur}$. Then there are nondegenerate trace pairings
\[
\mathrm{Groth}(G(\mathbb{A}_{f}^p)\times J_b(\mathbb{Q}_p))\times C_c^\infty(G(\mathbb{A}_{f}^p)\times J_b(\mathbb{Q}_p))\to \mathbb C\ ,
\]
and
\[
\mathrm{Groth}(G(\mathbb{A}_{f}^p)\times J_b(\mathbb{Q}_p))^{S\ur}\times \left(\cH^\ur(G(\mathbb A^S))\otimes C_c^\infty(G(\mathbb A_{S_{\mathrm{fin}}\setminus \{p\}})\times J_b(\mathbb Q_p))\right)\to \mathbb C\ .
\]

Let $\phi \in C^\infty_c(G(\mathbb{A}_{f}^p)\times J_b(\mathbb{Q}_p))$. We say that $\phi$ is \emph{acceptable} if it satisfies the conditions of Definition 6.2 of~\cite{shin-igusa}. The main condition is that $\phi$ is a linear combination of functions of the form $\phi^p\times \phi_p$, where $\phi_p$ is supported on $\nu_b$-acceptable elements of $J_b(\mathbb{Q}_p)$. These are those elements $\delta \in J_b(\mathbb{Q}_p)$, $\delta=(\delta_i)\in \prod_{i=1}^r\mathscr{A}ut^0(\mathbb{X}_i)$, such that any eigenvalues $e_i$ of $\delta_i$ satisfy \[v_p(e_i)<v_p(e_j)\ \mathrm{whenever}\ \lambda_i>\lambda_j\] (Definition 6.1 of~\cite{shin-igusa}). 
\begin{remark} This condition will separate components of $J_b(\mathbb{Q}_p)$ corresponding to different slopes in terms of their $p$-adic valuation, which in turn is needed in order to transfer functions from $J_b(\mathbb{Q}_p)$ to $G(\mathbb{Q}_p)$. See Lemma 3.9 of~\cite{shin-stable} and Lemma V.5.2 of~\cite{harris-taylor} for more details. 
\end{remark}
\noindent Lemma 6.3 of~\cite{shin-igusa} guarantees that the twist of any $\phi$ by a sufficiently high power of Frobenius is acceptable.  

We recall the set $\cE^\mathrm{ell}(G)$ of elliptic endoscopic triples for $G$. In fact, we work more generally: let $\mathrm{F}$ be a local or global field of characteristic $0$ and let $\mathrm{G}$ be a connected reductive group over $\mathrm{F}$. An \emph{endoscopic triple} for $\mathrm{G}$ is a triple $(\mathrm{H},s,\eta)$, where $\mathrm{H}$ is a quasi-split connected reductive group over $\mathrm{F}$, $s$ is an element of $Z(\hat{\mathrm{H}})$ and $\eta: \hat{\mathrm{H}}\to \hat{\mathrm{G}}$ is an embedding of complex Lie groups. The triple has to satisfy certain conditions, as in 7.4 of~\cite{kottwitz-cuspidal}. Let $\Gamma:= \Gal (\mathbb{\bar Q}/\mathbb{Q})$. An endoscopic triple is called \emph{elliptic} if $(Z(\hat{\mathrm{H}})^{\Gamma})^\circ \subset Z(\mathrm{G})$. We will use the notion of isomorphism of endoscopic triples in Section 2.1 of~\cite{shin-stable}, which is stronger than the one in~\cite{kottwitz-cuspidal}. We let $\cE^\mathrm{ell}(\mathrm{G})$ be the set of isomorphism classes of elliptic endoscopic triples for $\mathrm{G}$. 

Assume that $\mathrm{G}^{\mathrm{der}}$ is simply-connected (this will be the case for $\mathrm{G}:=G$, our unitary similitude group). We use Weil groups to construct $L$-groups; then we can choose an extension of $\eta$ to an $L$-group morphism $\tilde\eta: {}^L\mathrm{H}\to {}^L\mathrm{G}$ by Proposition 1 of~\cite{langlands}.

Assume that $\mathrm{F}$ is a local field. Given $\tilde \eta$, Langlands and Shelstad (see~\cite{langlands-shelstad}) define a transfer factor \[\Delta:\mathrm{H}(\mathrm{F})_{\mathrm{ss},(\mathrm{G},\mathrm{H})\mathrm{-reg}}\times\mathrm{G}(\mathrm{F})_{\mathrm{ss}}\to \mathbb{C},\] which is canonical up to a non-zero constant. 

The fundamental lemma and the transfer conjecture, which are now theorems due to Ngo, Waldspurger and others (see~\cite{ngo, waldspurger}), assert that for each function $\phi\in C_c^\infty(\mathrm{G}(\mathrm{F}))$, there exists $\phi^\mathrm{H}\in C^\infty_c(\mathrm{H}(\mathrm{F}))$ satisfying an identity about the transfer of orbital integrals 
\[SO^{\mathrm{H}(\mathrm{F})}_{\gamma_{\mathrm{H}}}(\phi^{\mathrm{H}})=\sum_{\gamma\in \mathrm{G}(\mathrm{F})_{\mathrm{ss}}/\sim}\Delta(\gamma_{\mathrm{H}},\gamma)e(\mathrm{G}_{\gamma})O^{\mathrm{G}(\mathrm{F})}_{\gamma}(\phi)\] 
(see Theorem 3.1 of~\cite{scholze-shin} for an explanation of the notation). If $\mathrm{H},\mathrm{G}$ and $\tilde\eta$ are unramified and if $\phi\in \cH^{\mathrm{ur}}(\mathrm{G}(\mathrm{F}))$, then $\Delta$ can be normalized such that $\phi^\mathrm{H}$ can be taken to be $\tilde\eta^*(\phi)$, where $\tilde\eta^*:\cH^{\mathrm{ur}}(\mathrm{G}(\mathrm{F}))\to \cH^{\mathrm{ur}}(\mathrm{H}(\mathrm{F}))$ is the morphism of unramified Hecke algebras induced from $\tilde\eta$ via the Satake isomorphism. In particular, if $\phi$ is the idempotent associated to a hyperspecial maximal compact subgroup, then $\phi^{\mathrm{H}}$ can also be taken to be the idempotent of a hyperspecial maximal compact subgroup.  

Let $\phi\in  C^\infty_c(G(\mathbb{A}_{f}^p)\times J_b(\mathbb{Q}_p))$ be an acceptable function of the form \[\phi=\prod_{v\not=\infty}\phi_v,\ \mathrm{with}\ \phi_v\in C_c^\infty(G(\mathbb{Q}_v)), v\neq p,\ \phi_p\in C_c^\infty(J_b(\mathbb{Q}_p)).\] Let $(H,s,\eta)\in \cE^{\mathrm{ell}}(G)$. 

\begin{defn}\label{definition of Igusa transfer} Let $\phi^H:=\phi^{H,p}\phi^H_p\phi^H_\infty\in C_c^\infty(H(\mathbb{A}))$, where:
\begin{itemize}
\item $\phi^{H,p}$ is the Langlands-Shelstad transfer of $\phi^p$ (as described above);
\item $\phi^H_\infty$ is constructed by Kottwitz in Section 7 of~\cite{kottwitz-lambda-adic}, where we take the trivial algebraic representation of $G$ as an input (this corresponds to the fact that our local system on $\mathscr{I}^b_{\mathrm{Mant}}$ is $\bar{\mathbb Q}_\ell$.) We give more details in the case when $G$ is a unitary similitude group below. 
\item $\phi^H_p$ is constructed in Section 6 of~\cite{shin-stable}. The function $\phi^H_p$ is the key construction of~\cite{shin-stable}; we give more details in Section~\ref{transfer at p} below. 
\end{itemize} 
\end{defn}
\noindent The following is the main result of~\cite{shin-stable}, Theorem 7.2 of loc.~cit.

\begin{thm}\label{stable trace formula for Igusa} Let $\phi$ and $\phi^H$ be as above, with $(H,s,\eta)\in \cE^{\mathrm{ell}}(G)$. Then 
\[\mathrm{tr}(\phi|\iota_\ell H_c(\mathscr{I}^b_{\mathrm{Mant}},\mathbb{\bar Q}_\ell))=|\ker^1(\mathbb{Q},G)|\sum_{(H,s,\eta)}\iota(G,H)ST^H_e(\phi^H).\]
\end{thm}

\begin{remark} Shin's result is in fact valid for any PEL Shimura variety of type $(A)$ or $(C)$. We recall that \[\ker^1(\mathbb{Q},G):=\ker\left(H^1(\mathbb{Q},G)\to \prod_v H^1(\mathbb{Q}_v,G)\right)\ ,\]
and that $S_K$ is the disjoint union of $|\ker^1(\mathbb Q,G)|$ copies of the Shimura variety for $G$. Also, $\iota(G,H):=\tau(G)\tau(H)^{-1}|\mathrm{Out}(H,s,\eta)|^{-1}$. The term $ST^H_e(\phi^H)$ is a sum of stable orbital integrals over (representatives of) $\mathbb{Q}$-elliptic semisimple stable conjugacy classes in $H(\mathbb{Q})$.  
\end{remark}

In the case of our unitary similitude group $G$, the set $\cE^{\mathrm{ell}}(G)$ only depends on the quasi-split inner form $G_n$ of $G$ and in~\cite{shin-galois}, Shin gives a concrete description of a set of representatives for the isomorphism classes in $\cE^{\mathrm{ell}}(G_n)$. If $\vec{n}=(n_i)_{i=1}^s$ is a vector with entries positive integers, one can define a quasi-split group $G_{\vec{n}}$ over $\mathbb{Q}$ as in Section 3.1 of~\cite{shin-galois}. Define $GL_{\vec{n}}:=\prod_{i=1}^s GL_{n_i}$ and let $i_{\vec{n}}: GL_{\vec{n}}\to GL_{(\sum_i n_i)}$ be the natural map. Let \[\Phi_{\vec{n}}:=i_{\vec{n}}(\Phi_{n_1},\dots,\Phi_{n_s}),\] where $\Phi_n$ is the matrix in $GL_n$ with entries $(\Phi_n)_{ij}=(-1)^{i+1}\delta_{i,n+1-j}$. Then $G_{\vec{n}}$ is the algebraic group over $\mathbb{Q}$ sending a $\mathbb{Q}$-algebra $R$ to 
\[G_{\vec{n}}(R)=\{(\lambda, g)\in R^\times \times GL_{\vec{n}}(F\otimes_{\mathbb{Q}}R)| g\cdot\Phi_{\vec{n}}\cdot {}^tg^c=\lambda\Phi_{\vec{n}}\}.\] 
Since $G$ is quasi-split at all finite places, we have \[G\times_{\mathbb{Q}}\mathbb{A}_f\simeq G_n\times_{\mathbb{Q}}\mathbb{A}_f\] and we fix such an isomorphism.    

The representatives for $\cE^{\mathrm{ell}}(G_n)$ can be taken to be \[\left\{(G_n,s_n,\eta_n)\right\}\cup \left\{(G_{n_1,n_2},s_{n_1,n_2},\eta_{n_1,n_2})|n_1+n_2=n,n_1\geq n_2\geq 0\right\},\]  where $(n_1,n_2)$ may be excluded if both $n_1$ and $n_2$ are odd numbers (see condition 7.4.3 of~\cite{kottwitz-cuspidal}). Here, $s_n=1\in \hat{G}_n, s_{n_1,n_2}=(1,(I_{n_1},-I_{n_2}))\in \hat{G}_{n_1,n_2}$, $\eta_n$ is the identity map and $\eta_{n_1,n_2}:\hat{G}_{n_1,n_2}\to \hat{G}_n$ is the natural embedding induced by $GL_{n_1}\times GL_{n_2}\hookrightarrow GL_n$.

If we choose a Hecke character $\psi:\mathbb{A}^\times_{\cK}/\cK^\times\to \mathbb{C}^\times$ such that $\psi|_{\mathbb{A}^\times/\mathbb{Q}^\times}$ corresponds via class field theory to the quadratic character associated to $\cK/\mathbb{Q}$, we can extend $\eta_{n_1,n_2}$ to an $L$-morphism \[\tilde{\eta}_{n_1,n_2}:{}^L G_{n_1,n_2}\to {}^L G_n.\] (See Section 3.2 of~\cite{shin-galois} for the precise formula.) By Proposition 7.1 of~\cite{shin-galois}, $\psi$ can be chosen such that the set of primes where $\psi$ is ramified is contained in $\Spl_{F/F^+}$. As a consequence, we can use the explicit transfer factors described in Section 3.4 of~\cite{shin-galois} at all places not equal to $p,\infty$. These are compatible with the Langlands-Shelstad transfer described above: at unramified places $v$, we take \[\tilde \eta_{n_1,n_2}^*: \cH^{\mathrm{ur}}(G_n(\mathbb{Q}_v))\to \cH^{\mathrm{ur}}(G_{n_1,n_2}(\mathbb{Q}_v)),\] 
making use of the fundamental lemma~\cite{ngo}. Since we have fixed an isomorphism $G\times_{\mathbb{Q}}\mathbb{A}_f\toisom G_n\times_{\mathbb{Q}}\mathbb{A}_f$, we can also think of this as a transfer from $G$ to $G_{n_1,n_2}$ at places away from $p,\infty$.  

We also describe the explicit transfer at the place $\infty$. The transfer is as in Section 7 of~\cite{kottwitz-lambda-adic} and uses Shelstad's theory of real endoscopy and the Langlands correspondence for real reductive groups; see also Section 3.5 of~\cite{shin-galois} for any unfamiliar notation. Recall that over $\mathbb{R}$, $G$ is an inner form of the quasi-split unitary similitude group $G_n$. For any discrete $L$-parameter $\varphi_{G_{\vec{n}}}$ for $G_{\vec{n}}$, with $L$-packet $\Pi(\varphi_{G_{\vec{n}}})$, define \[\phi_{\varphi_{G_{\vec{n}}}}:=\frac{1}{|\Pi(\varphi_{G_{\vec{n}}})|}\sum_{\pi \in \Pi(\varphi_{G_{\vec{n}}})}\phi_{\pi},\] where $\phi_{\pi}$ is a pseudo-coefficient for $\pi$. When $\varphi_{G_{\vec{n}}}\sim \varphi_{\xi}$ corresponds to an $L$-packet $\Pi_{\mathrm{disc}}(G_{\vec{n}}(\mathbb{R}),\xi^\vee)$ for some irreducible algebraic representation $\xi$ of $G_{\vec{n}}$, the function $\phi_{\varphi_{G_{\vec{n}}}}$ is called an \emph{Euler-Poincar\'e function}; we denote it also by $\phi_{G_{\vec{n}},\xi}$. The desired function $\phi_\infty^{\vec{n}}$ will be a precise linear combination of the Euler-Poincar\'e functions for $L$-parameters $\varphi_{G_{\vec{n}}}$ for which $\tilde \eta \circ \varphi_{G_{\vec{n}}}$ corresponds to the trivial algebraic representation of $G_{\mathbb{C}}$ (see 5.11 of~\cite{shin-galois} for the precise formula).   

For further use, we record a version of Theorem~\ref{stable trace formula for Igusa} for the group $G$.    

\begin{cor}\label{explicit stable trace formula} If $\phi^{p}\cdot \phi_p\in C^\infty_c(G(\mathbb{A}_f^p)\times J_b(\mathbb{Q}_p))$ is acceptable, then 
\[\mathrm{tr}(\phi|\iota_\ell H_c(\mathscr{I}^b_{\mathrm{Mant}},\mathbb{\bar Q}_\ell))=|\ker^1(\mathbb{Q},G)|\sum_{G_{\vec{n}}} \iota(G,G_{\vec{n}})ST_e^{G_{\vec{n}}}(\phi^{\vec{n}}),\] where $G_{\vec{n}}$ runs over the set described above and $\phi^{\vec{n}}$ is obtained from $\phi$ as in Definition~\ref{definition of Igusa transfer}. 
\end{cor}

\begin{remark}\label{constants} The constants $\iota(G,G_{\vec{n}})$ can be computed explicitly: \[\iota(G,G_{\vec{n}})=\begin{cases}\frac{1}{2}\tau(G)\tau(G_{\vec{n}})^{-1}&\mathrm{if}\ \vec{n}=(\frac{n}{2},\frac{n}{2})\\ \tau(G)\tau(G_{\vec{n}})^{-1}&\mathrm{otherwise.}\end{cases}\]
\end{remark}

\subsection{Base change and the twisted trace formula}\label{base change} Let $\mathbb{G}_{\vec{n}}:=\mathrm{Res}_{\cK/\mathbb{Q}}(G_{\vec{n}}\times_{\mathbb{Q}}\cK)$. One can define $L$-morphisms $BC_{\vec{n}}:{}^L G_{\vec{n}}\to {}^L\mathbb{G}_{\vec{n}}$ and $\tilde{\zeta}_{n_1,n_2}: {}^L\mathbb{G}_{n_1,n_2}\to {}^L\mathbb{G}_n$ and there is a commutative diagram of $L$-morphisms 
\begin{equation}\label{commutative diagram of L-morphisms}\xymatrix{^LG_{n_1,n_2}\ar[d]^{BC_{n_1,n_2}}\ar[r]^{\tilde \eta}&^LG_n\ar[d]^{BC_n}\\^L\mathbb{G}_{n_1,n_2}\ar[r]^{\tilde\zeta}&^L\mathbb{G}_n}.\end{equation}

In this section, we review the associated base change for the groups $G_{\vec{n}}$ and $\mathbb{G}_{\vec{n}}$ as well as the twisted trace formula. Let $S$ be a finite set of primes containing $\infty,p$ and all the primes where either the CM field $F$ or the Hecke character $\psi$ are ramified. Recall that, by the assumptions in Section~\ref{setup}, we can and will arrange that $S_{\mathrm{fin}}\subset \Spl_{F/F^+}$.

We can define a notion of $BC$-transfer of functions as in Section 4 of~\cite{shin-galois}. If $v$ is a finite place of $\mathbb{Q}$ such that $v\not\in S$, then the dual map to the $L$-morphism $BC_{\vec{n}}$ defines the transfer \[BC_{\vec{n}}^*:\cH^{\mathrm{ur}}(\mathbb{G}_{\vec{n}}(\mathbb{Q}_v))\to \cH^{\mathrm{ur}}(G_{\vec{n}}(\mathbb{Q}_v)),\] (Case 1) of Section 4.2 of~\cite{shin-galois}. Otherwise, if $v\in S_{\mathrm{fin}}\subset \Spl_{F/F^+}$ then Section 4.2 of loc.~cit., (Case 2), constructs a $BC$-transfer $\phi_v\in C_c^\infty(G_{\vec{n}}(\mathbb{Q}_v))$ of $f_v\in C_c^\infty(\mathbb{G}_{\vec{n}}(\mathbb{Q}_v))$. We remark that, if $v$ splits in $\cK$ (e.g. if $v=p$), one can check directly that $BC^*_{\vec{n}}$ is surjective. It is also possible to define a transfer $\tilde\zeta_{\vec{n}}^*$, as in Section 4 of loc.~cit., making the obvious diagram commutative. 

At $\infty$, the transfer is defined in Section 4.3 of loc. cit. Let $\xi$ be an irreducible algebraic representation of $(G_{\vec{n}})_{\mathbb{C}}$, giving rise to the representation $\Xi$ of $(\mathbb{G}_{\vec{n}})_{\mathbb{C}}$ which is just $\Xi:=\xi\otimes \xi$. Recall that $\phi_{G_{\vec{n}},\xi}$ is the Euler-Poincar\'e function for $\xi$. Associated to $\Xi$, Labesse defined a twisted analogue of the Euler-Poincar\'e function, a Lefschetz function $f_{\mathbb{G}_{\vec{n}},\Xi}$~\cite{labesse-lefschetz}. The discussion on page 24 of~\cite{shin-galois} implies that $f_{\mathbb{G}_{\vec{n}},\Xi}$ and $\phi_{G_{\vec{n}},\xi}$ are $BC$-matching functions. 

Define the group \[\mathbb{G}^+_{\vec{n}}:=(\mathrm{Res}_{\cK/\mathbb{Q}}GL_1\times \mathrm{Res}_{F/\mathbb{Q}}GL_{\vec{n}})\rtimes\{1,\theta\},\] where $\theta(\lambda,g)\theta^{-1}=(\lambda^c, \lambda^cg^\sharp)$ and $g^\sharp=\Phi_{\vec{n}}^tg^c\Phi_{\vec{n}}^{-1}$. If we denote by $\mathbb{G}_{\vec{n}}^\circ$ and $\mathbb{G}_{\vec{n}}^\circ\theta$ the cosets of $\{1\}$ and $\{\theta\}$ in $\mathbb{G}_{\vec{n}}^+$, then $\mathbb{G}_{\vec{n}}^+=\mathbb{G}_{\vec{n}}^\circ\sqcup \mathbb{G}_{\vec{n}}^\circ \theta$. There is a natural $\mathbb{Q}$-isomorphism $\mathbb{G}_{\vec{n}}\toisom \mathbb{G}_{\vec{n}}^\circ$, which extends to an isomorphism 
\[\mathbb{G}_{\vec{n}}\rtimes \Gal(\cK/\mathbb{Q})\toisom \mathbb{G}_{\vec{n}}^+\] 
so that $c\in \Gal(\cK/\mathbb{Q})$ maps to $\theta$. Using this isomorphism, we write $\mathbb{G}_{\vec{n}}$ and $\mathbb{G}_{\vec{n}}\theta$ for the two cosets. 

If $f\in C^\infty_c(\mathbb{G}_{\vec{n}}(\mathbb{A}))$ (with trivial character on $A^\circ_{\mathbb{G}_n,\infty}$), then we define $f\theta$ to be the function on $\mathbb{G}_{\vec{n}}\theta(\mathbb{A})$ obtained via translation by $\theta$. The (invariant) twisted trace formula (see~\cite{arthura, arthurb}) gives an equality 
\begin{equation} \label{twisted} I^{\mathbb{G}_{\vec{n}}\theta}_{\mathrm{geom}}(f\theta)=I^{\mathbb{G}_{\vec{n}}\theta}_{\mathrm{spec}}(f\theta).\end{equation} 
The left hand side of the equation is defined in Section 3 of~\cite{arthurb}, while the right hand side is defined in Section 4 of loc. cit.

Let $f_{\mathbb{G}_{\vec{n}},\Xi}$ and $\phi_{G_{\vec{n}},\xi}$ be as defined above. The following is Corollary 4.7 of~\cite{shin-galois}. 
\begin{prop}\label{comparing geometric sides} We have the following equality:
\begin{equation}I^{\mathbb{G}_{\vec{n}}\theta}_{\mathrm{geom}}(f\theta)=\tau(G_{\vec{n}})^{-1}ST^{G_{\vec{n}}}_e(\phi), \end{equation}  when $\phi$ and $f$ satisfy \[\phi =\phi^S\cdot \phi_{S_{\mathrm{fin}}}\cdot \phi_{G_{\vec{n}},\xi}\ \mathrm{and}\ f=f^S\cdot f_{S_{\mathrm{fin}}}\cdot f_{\mathbb{G}_{\vec{n}},\Xi}\] with $\phi^S$ a $BC$-transfer of $f^S$, $\phi_{S_{\mathrm{fin}}}$ a $BC$-transfer of $f_{S_{\mathrm{fin}}}$. 
\end{prop}

\begin{proof}We sketch the proof here: first, use Theorem 4.3.4 of~\cite{labesse} to rewrite the sum of stable orbital integrals on the right as the elliptic part of the twisted trace formula for $\mathbb{G}_{\vec{n}}\theta$. Then the geometric side of the twisted trace formula for $\mathbb{G}_{\vec{n}}\theta$ is simplified using similar techniques to those in Chapter 7 of~\cite{arthurb}: the key facts are that the Lefschetz function $f_{\mathbb{G}_{\vec{n}},\Xi}$ is cuspidal, so only $\theta$-elliptic elements contribute, and that $[F^+:\mathbb{Q}]\geq 2$, so that the only Levi subgroup that contributes to the geometric side is $\mathbb{G}_{\vec{n}}\theta$ itself. 
\end{proof}

We now explain how to construct our test functions, which is exactly as in the proof of Theorem 6.1 of~\cite{shin-galois}. We let $(f^n)^S$ be any function in $\cH^{\mathrm{ur}}(\mathbb{G}_n(\mathbb{A}^S))$ and $f^n_{S_{\mathrm{fin}}\setminus \{p\}}$ be any function in $C^\infty_c(\mathbb{G}_n(\mathbb{A}_{S_{\mathrm{fin}\setminus \{p\}}}))$. We let $\phi^S,\phi_{S_{\mathrm{fin}\setminus\{p\}}}$ be their $BC$-transfers, as described above. We take $\phi_p\in C_c^\infty(J_b(\mathbb{Q}_p))$ be any acceptable function and set \[\phi:= \phi^S\cdot\phi_{S_{\mathrm{fin}\setminus\{p\}}}\cdot \phi_p.\]  
From these test functions, we construct all the other test functions we will need. First, for each elliptic endoscopic group $G_{\vec{n}}$ we let $\phi^{\vec{n}}$ be constructed from $\phi$ as in Definition~\ref{definition of Igusa transfer}. Let $(f^{n_1,n_2})^S$ and $(f^{n_1,n_2})_{S_{\mathrm{fin}\setminus\{p\}}}$ be obtained from  $(f^{n})^S$ and $f^{n}_{S_{\mathrm{fin}\setminus\{p\}}}$ by transfer along the $L$-morphism $\tilde \zeta$. We choose $f^{\vec{n}}_p$ so that $BC^*_{\vec{n}}(f^{\vec{n}}_p)=\phi^{\vec{n}}_p$ (recall that $BC^*_{\vec{n}}$ is surjective at $p$). We can define $f^{\vec{n}}_\infty$ explicitly, as a linear combination of Lefschetz functions for representations $\Xi(\varphi_{\vec{n}})$ of $\mathbb{G}_{\vec{n}}$ for which $\tilde \eta \circ \varphi_{\vec{n}}$ corresponds to the trivial representation of $G$ (see (6.7) of~\cite{shin-galois} for the precise formula). Finally, we set \[f^{\vec{n}}:= (f^{\vec{n}})^S\cdot(f^{\vec{n}})_{S_{\mathrm{fin}\setminus\{p\}}}\cdot f^{\vec{n}}_p\cdot f^{\vec{n}}_{\infty}.\] By the commutative diagram of $L$-morphisms~\eqref{commutative diagram of L-morphisms}, we can apply Proposition~\ref{comparing geometric sides} to $f^{\vec{n}}$ and $\phi^{\vec{n}}$. To check the compatibility, see (4.18) of~\cite{shin-galois} for primes away from $S$, (4.19) of loc. cit. for primes in $S_{\mathrm{fin}}\setminus \{p\}$ and compare the precise formulas for $\phi^{\vec{n}}_\infty$ and $f_\infty^{\vec{n}}$. We mention that the formulas for $\phi^{\vec{n}}_\infty$ and $f_\infty^{\vec{n}}$ use as input an inner form $G$ of $G_{n}$ over $\mathbb{R}$; in loc. cit. this inner form has a specific signature (a group of so-called Harris-Taylor type), but here we work more generally. In particular, the integer $q(G)$ appearing there is defined as $\frac{1}{2}\mathrm{dim}(G(\mathbb{R})/K_\infty A_\infty^\circ)$. 

\begin{thm}\label{spectral side} We have an equality 
\[\mathrm{tr}(\phi|\iota_\ell [H_c(\mathscr{I}^b_{\mathrm{Mant}},\mathbb{\bar Q}_\ell)])=|\ker^1(\mathbb{Q},G)|\tau(G)\sum_{G_{\vec{n}}}\epsilon_{\vec{n}}\cdot I^{\mathbb{G}_{\vec{n}}\theta}_{\mathrm{spec}}(f^{\vec{n}}\theta),\] where $\epsilon_{\vec{n}}=\frac{1}{2}$ if $\vec{n}=(\frac{n}{2},\frac{n}{2})$ and $\epsilon_{\vec{n}} = 1$ otherwise. 
\end{thm}
\begin{proof} This follows by combining Corollary~\ref{explicit stable trace formula}, Remark~\ref{constants}, Proposition~\ref{comparing geometric sides} and equation~\eqref{twisted}. 
\end{proof}

Fix $G_{\vec{n}}$. We now proceed to simplify the spectral side $I^{\mathbb{G}_{\vec{n}}\theta}_{\mathrm{spec}}(f^{\vec{n}}\theta)$. We need the following notation from~\cite{shin-galois}: let $M_0$ be a minimal Levi subgroup of $\mathbb{G}_{\vec{n}}$. For $M$ a rational Levi of $\mathbb{G}_{\vec{n}}$ containing $M_0$, choose a parabolic subgroup $Q$ containing $M$ as a Levi. The group $W^{\mathbb{G}_{\vec{n}}\theta}(\mathfrak{a}_M)_\mathrm{reg}$ defined in~\cite{arthurb} acts on the set of parabolic subgroups which have $M$ as a Levi component. The automorphism $\Phi_{\vec{n}}^{-1}\theta$ of $\mathbb{G}_{\vec{n}}$ preserves $M$ and acts on $W^{\mathbb{G}_{\vec{n}}\theta}(\mathfrak{a}_M)_\mathrm{reg}$. By combining Proposition 4.8 and Corollary 4.14 of~\cite{shin-galois}, we have the following expression for the summands on the right hand side of Theorem~\ref{spectral side}.

\begin{prop}\label{expansion on spectral side} There is an equality 
\[I^{\mathbb{G}_{\vec{n}}\theta}_{\mathrm{spec}}(f^{\vec{n}}\theta)=\sum_{M}\frac{|W_M|}{|W_G|}|\det(\Phi_{\vec{n}}^{-1}\theta - 1)_{\mathfrak{a}^{\mathbb{G}_{\vec{n}}\theta}_M}|^{-1}\sum_{\Pi_M}\mathrm{tr}\left(\mathrm{n-Ind}^{\mathbb{G}_{\vec{n}}}_Q(\Pi_M)_{\xi}(f^{\vec{n}})\circ A'\right) \] where $M$ runs over the Levi subgroups of $\mathbb{G}_{\vec{n}}$ containing $M_0$ and $\Pi_M$ runs over the irreducible $\Phi_{\vec{n}}^{-1}\theta$-stable subrepresentations of the discrete spectrum $R_{M,\mathrm{disc}}$.  
\end{prop}
\begin{remark}\label{explanation of expansion} The subscript $\xi$ indicates a possible twist by a character of $A^\circ_{\mathbb{G}_{\vec{n}},\infty}$ corresponding to an irreducible algebraic representation $\xi$ of $G_{\vec{n}}$ and $A'$ is a normalized intertwiner on $\mathrm{n-Ind}^{\mathbb{G}_{\vec{n}}}_Q(\Pi_M)_{\xi}$. We do not make this precise, as we will not need these details. We do note that, as $\Pi_M$ is $\Phi_{\vec{n}}^{-1}\theta$-stable, $\mathrm{n-Ind}^{\mathbb{G}_{\vec{n}}}_Q(\Pi_M)_{\xi}$ is $\theta$-stable. 
\end{remark}

\subsection{The transfer at $p$}\label{transfer at p} We recall the construction of the function $\phi^{\vec{n}}_p$, starting from an acceptable function $\phi_p\in C_c^\infty(J_b(\mathbb{Q}_p))$, as well as the representation-theoretic counterpart to this construction, $\mathrm{Red}^b_{\vec{n}}$. 

The group $J_b(\mathbb{Q}_p)$ is an inner form of a Levi subgroup $M_b(\mathbb{Q}_p)$ of $G(\mathbb{Q}_p)$; for further reference, we recall their precise definitions, following Chapter 1 of~\cite{rapoport-zink}. According to Definition 1.8 of loc. cit., an element $\tilde b$ of $G(L)$ is called \emph{decent} if there exists a positive integer $s$ such that
\[
(\tilde b\sigma)^s=s\nu_{\tilde b}(p)\sigma^s,
\]
where $s\nu_{\tilde b}$ factors through a morphism $\mathbb{G}_m\to G$. By Section 4.3 of~\cite{kottwitz}, any $\sigma$-conjugacy class $b\in B(G)$ admits a decent representative $\tilde b$; as $G$ is quasisplit, one can moreover arrange that $\nu_{\tilde b}$ is defined over $\mathbb Q_p$, cf.~\cite[p. 219]{kottwitz}. Let $M_b$ be the centralizer of $\nu$ in $G$, which is a $\mathbb Q_p$-rational Levi subgroup. Then $b$ is a basic element of $M_b$, and $J_b$ is an inner form of $M_b$.

Fix $G_{\vec{n}}$ an elliptic endoscopic group for $G$. The set $\cE_p^{\mathrm{eff}}(J_b, G,G_{\vec{n}})$ is defined in Section 6.2 of~\cite{shin-igusa}; it consists of certain isomorphism classes $(M_{G_{\vec{n}}},s_{\vec{n}},\eta_{\vec{n}})$ of $G_{\vec{n}}$-endoscopic triples for $J_b(\mathbb{Q}_p)$. The function $\phi^{\vec{n}}_p$ is constructed via transfer from $\phi_p$ on $J_b(\mathbb{Q}_p)$ to $M_{G_{\vec{n}}}(\mathbb{Q}_p)$, followed by a version of transfer from $M_{G_{\vec{n}}}(\mathbb{Q}_p)$ to $G_{\vec{n}}(\mathbb{Q}_p)$. We remark that the latter step makes crucial use of the acceptability of $\phi_p$, cf. Lemma 3.9 of~\cite{shin-stable}.   

There is a representation-theoretic counterpart to this construction. This is a group morphism \[\mathrm{Red}_{\vec{n}}^b: \mathrm{Groth}(G_{\vec{n}}(\mathbb{Q}_p))\to \mathrm{Groth}(J_b(\mathbb{Q}_p)).\] $\mathrm{Red}_{\vec{n}}^b$ will be defined as the composition of the following maps:\begin{enumerate}
\item \[\mathrm{Groth}(G_{\vec{n}}(\mathbb{Q}_p))\to \bigoplus_{(M_{G_{\vec{n}}},s_{\vec{n}},\eta_{\vec{n}})} \mathrm{Groth}(M_{G_{\vec{n}}}(\mathbb{Q}_p)),\] where the sum runs over $G_{\vec{n}}$-endoscopic triples in $\cE_p^{\mathrm{eff}}(J_b, G,G_{\vec{n}})$ and the map on each term is a linear combination of normalized Jacquet functors (indexed over a finite set of allowed Levi embeddings $M_{G_{\vec{n}}}\hookrightarrow G_{\vec{n}}$);
\item \[\mathrm{Groth}(M_{G_{\vec{n}}}(\mathbb{Q}_p))\to \mathrm{Groth}(M_b(\mathbb{Q}_p)),\] which is the functorial transfer with respect to the $L$-morphism $\tilde \eta_{\vec{n}}$.  Both $M_{G_{\vec{n}}}$ and $M_b$ are (restrictions of scalars of) products of general linear groups and the transfer ends up being a normalized parabolic induction. 
\item \[\mathrm{Groth}(M_b(\mathbb{Q}_p))\to \mathrm{Groth}(J_b(\mathbb{Q}_p)),\] which is the Langlands-Jacquet map on Grothendieck groups, defined by~\cite{badulescu}.
\end{enumerate} 
(See Section 5.5 of~\cite{shin-galois} for the precise definition of these three maps; even though the case we are considering is slightly more general, the formulas will be exact analogues.)
\begin{remark} When $\vec{n}=(n)$, $\cE_p^{\mathrm{eff}}(J_b, G,G_{\vec{n}})$ has only one element, namely $(M_b,1,\mathrm{id})$. The morphism $\mathrm{Red}^b_n$ consists of a normalized Jacquet functor followed by the Langlands-Jacquet map. 
\end{remark}
\noindent We record the relationship between $\mathrm{Red}_{\vec{n}}^b$ and $\phi^{\vec{n}}_p$ in the following lemma.

\begin{lemma}\label{phi and red} For any $\pi_p\in \mathrm{Groth}(G_{\vec{n}}(\mathbb{Q}_p))$, \[\mathrm{tr}\ \pi_p(\phi^{\vec{n}}_p)=\mathrm{tr}\left(\mathrm{Red}_{\vec{n}}^b(\pi_p)\right)(\phi_p).\]
\end{lemma}
\begin{proof} The statement follows in the same way as Lemma 5.10 of~\cite{shin-galois} (see also Lemmas 6.3 and 6.4 of~\cite{caraiani} for a unitary group with a slightly different signature). The idea is that the constructions of both $\mathrm{Red}^b_{\vec{n}}$ and $\phi^{\vec{n}}$ can be broken down into the three steps outlined above and the constructions in each of these steps are dual to each other. One of the key points is that the transfer of $\phi_p$ from $J_b(\mathbb{Q}_p)$ to $M_{G_{\vec{n}}}(\mathbb{Q}_p)$ can be broken down into transfer from $J_b(\mathbb{Q}_p)$ to the quasi-split form $M_b(\mathbb{Q}_p)$ followed by transfer from $M_b(\mathbb{Q}_p)$ to $M_{G_{\vec{n}}}(\mathbb{Q}_p)$. The other key point is the slightly non-standard transfer between $M_{G_{\vec{n}}}$ and $G_{\vec{n}}$, where the desired compatibility follows from Lemma 3.9 of~\cite{shin-stable}. 
\end{proof}

We note that the whole situation decomposes into a product. Namely, let $\p_1,\ldots,\p_m$ be the primes of $F^+$ above $p$, and fix a decomposition $p=uu^c$ in $\cK$. We denote by $\p_i$ also the place of $F$ lying over $\p_i$ in $F^+$, and $u$ in $\cK$, and by $\p_i^c$ the complex conjugate place of $F$. With these choices, we get a decomposition
\[
G_{\mathbb Q_p} = \prod_i \mathrm{Res}_{F_{\p_i}/\mathbb Q_p} GL_n\times \mathbb G_m\ .
\]
Here, the projection to the $\mathbb G_m$-factor is the unitary similitude factor, and the projection to the general linear groups is via the projection
\[
V\otimes_{\mathbb{Q}} \mathbb Q_p = \bigoplus_i (V\otimes_F F_{\p_i}\oplus V\otimes_F F_{\p_i^c})\to \bigoplus_i V\otimes_F F_{\p_i}\ .
\]
The resulting constructions above admit similar decompositions. In particular,
\[
b=((b_i)_{i=1,\ldots,m},b_0)\in B(G) = \prod_{i=1}^m B(\mathrm{Res}_{F_{\p_i}/\mathbb Q_p} GL_n)\times B(\mathbb G_m)\ ,
\]
and $J_b = \prod_{i=1}^m J_{b_i}\times \mathbb G_m$. Also, any irreducible representation $\pi_p$ of $G(\mathbb Q_p)$ decomposes into a tensor product
\[
\pi_p = \bigotimes_{i=1}^m \pi_{\p_i}\otimes \pi_0\ ,
\]
where $\pi_{\p_i}$ is an irreducible representation of $GL_n(F_{\p_i})$, and $\pi_0$ is a character of $\mathbb Q_p^\times$. A similar discussion applies to representations of
\[
G_{\vec{n}}(\mathbb Q_p) = \prod_{i=1}^m GL_{\vec{n}}(F_{\p_i})\times \mathbb Q_p^\times\ .
\]

\begin{lemma}\label{red vanishes on generic} Let $\pi^{\vec{n}}_p\in \mathrm{Irr}(G_{\vec{n}}(\mathbb{Q}_p))$ be decomposed as
\[
\pi^{\vec{n}}_p = \bigotimes_{i=1}^m \pi^{\vec{n}}_{\p_i}\otimes \pi^{\vec{n}}_0\ .
\]
Assume that there is some $i$ such that $\pi^{\vec{n}}_{\p_i}$ transfers to a generic principal series representation of $GL_n(F_{\p_i})$ and $J_{b_i}$ is a non-quasi-split inner form of $M_{b_i}$. Then \[\mathrm{Red}^b_{\vec{n}}(\pi_p)=0.\] 
\end{lemma}

\begin{proof} This follows from the explicit description of $\mathrm{Red}^b_{\vec{n}}$ above, which includes the Langlands-Jacquet map. If $\pi^{\vec{n}}_p$ satisfies the above condition, then its image $\rho$ in $\mathrm{Groth}(M_b(\mathbb{Q}_p))$ will have as $M_{b_i}(F_{\p_i})$-components only generic principal series representations. Indeed, to see this, note that by the definition in Section 2 of~\cite{shin-stable}, for a $G_{\vec{n}}$-endoscopic triple, the $L$-morphism $^L M_{G_{\vec{n}}}\to {}^L M_b$ is the restriction of the $L$-morphism $\tilde \zeta_{n_1,n_2}: {}^L G_{n_1,n_2}\to {}^L G_n$. The condition of being a generic principal series representation can be interpreted on the dual side, and is then easily deduced from this diagram. But if $\rho\in \mathrm{Groth}(M_b(\mathbb Q_p))$ has only generic principal series representations as $M_{b_i}(F_{\p_i})$-components, then it lies in the kernel of the Langlands-Jacquet map whenever $J_{b_i}$ is a non-quasi-split inner form, by the construction of this map following Theorem 3.1 and Proposition 3.3 of~\cite{badulescu}. 
\end{proof} 

\subsection{Generic principal series}\label{generic ps}

Fix test functions $f^S\in \cH^{\mathrm{ur}}(\mathbb{G}_n(\mathbb{A}^S)), f_{S_{\mathrm{fin}}\setminus\{p\}}\in C^\infty_c(\mathbb{G}_n(\mathbb{A}_{S_{\mathrm{fin}}\setminus \{p\}}))$ and let $\phi^S,\phi_{S_{\mathrm{fin}}\setminus\{p\}}$ be their base change transfers to $G_n(\mathbb{A}^S)$ and $G_n(\mathbb{A}_{S_{\mathrm{fin}}\setminus\{p\}})$ as defined in Section~\ref{base change}.

\begin{lemma}\label{expansion of igusa cohomology}  For any test function $f_p\in C^\infty_c(\mathbb{G}(\mathbb{Q}_p))$, let $\phi_p\in C^\infty_c(G(\mathbb{Q}_p))$ be its base change transfer. The trace
\[
\mathrm{tr}\left(\phi^S\phi_{S_{\mathrm{fin}}\setminus \{p\}}\phi_p|\iota_\ell([H_c(\mathscr{I}^b_{\mathrm{Mant}},\mathbb{\bar Q}_\ell)])\right)
\]
can be written as a linear combination of terms of the form
\[
\mathrm{tr}\left((f^{\vec{n}})^S\mid (\Pi^{\vec{n}})^S\right) \mathrm{tr}\left((f^{\vec{n}})_{S_{\mathrm{fin}}\setminus \{p\}}\mid (\Pi^{\vec{n}})_{S_{\mathrm{fin}}\setminus \{p\}}\right) \mathrm{tr}\left(\phi_p\mid\mathrm{Red}^{b}_{\vec{n}}(\pi^{\vec{n}}_p)\right)
\]
for $\pi^{\vec{n}}_p\in \mathrm{Irr}(G_{\vec{n}}(\mathbb{Q}_p))$, $\Pi^{\vec{n}}$ satisfying the following condition. The representation $\pi^{\vec{n}}_p$ base changes to the component $\Pi^{\vec{n}}_p\in \mathrm{Irr}(\mathbb G_{\vec{n}}(\mathbb Q_p))$ at $p$ of a $\theta$-stable isobaric automorphic representation $\Pi^{\vec{n}}$ of $\mathbb G_{\vec{n}}$ of the form
\[
\Pi^{\vec{n}} = (\mathrm{n-Ind}_Q^{\mathbb{G}_{\vec{n}}}\Pi_M)_\xi\ ,
\]
where $\Pi_M$ occurs in the (relatively) discrete part of the automorphic spectrum of the Levi subgroup $M$ of a parabolic $Q\subset \mathbb{G}_{\vec{n}}$. Moreover, $\Pi^{\vec{n}}_\infty$ is cohomological (with respect to the trivial algebraic representation).
\end{lemma}

\begin{proof} We follow the proof of Proposition 6.1 of~\cite{shin-galois}, in a more general situation, but without keeping track of endoscopic signs and constants. First, assume that $f_p$ is chosen such that $\phi_p$ is acceptable. Let $\phi$ be the product of test functions $\phi^S\phi_{S_{\mathrm{fin}}\setminus \{p\}}\phi_p$. By combining Theorem~\ref{spectral side} and Proposition~\ref{expansion on spectral side}, we can write $\mathrm{tr}(\phi|\iota_\ell[H_c(\mathscr{I}^b_{\mathrm{Mant}},\mathbb{\bar Q}_\ell)])$ as a finite linear combination on terms of the form $\mathrm{tr}\left(\Pi^{\vec{n}}(f^{\vec{n}})\circ A'\right)$, where $\Pi^{\vec{n}}$ is a $\theta$-stable irreducible automorphic representation of $\mathbb{G}_{\vec{n}}$. 

Recall that each $\Pi^{\vec{n}}$ is of the form $(\mathrm{n-Ind}_Q^{\mathbb{G}_{\vec{n}}}\Pi_M)_\xi$, where $\Pi_M$ occurs in the (relatively) discrete part of the automorphic spectrum of the Levi subgroup $M$ of $\mathbb{G}_{\vec{n}}$. The fact that $\Pi^{\vec{n}}$ is $\theta$-stable follows from Remark~\ref{explanation of expansion} and the irreducibility of $\Pi^{\vec{n}}$ follows from the fact $\Pi_M$ is unitary and that, for general linear groups, any parabolic induction of a unitary representation is irreducible. Moreover, the representation $\Pi_M$ must be~\emph{isobaric}, since it contributes to the discrete spectrum of $M$. (This follows from the classification of automorphic representations occurring in the discrete spectrum of general linear groups due to Moeglin and Waldspurger, \cite{moeglin-waldspurger}. See, for example, Theorem 1.3.3 of~\cite{arthur-endoscopic-classification} and the discussion below it.) Now the strong multiplicity one result due to Jacquet and Shalika (the main result of~\cite{jacquet-shalika}, see also Theorem 1.3.2 of~\cite{arthur-endoscopic-classification}) implies that the string of Satake parameters outside the finite set $S$ determines $\Pi_M$. The parabolic induction $\Pi^{\vec{n}}$ is also isobaric, because it is irreducible, and therefore it is determined by $(\Pi^{\vec{n}})^S$. To check that $\Pi^{\vec{n}}_\infty$ is cohomological (for the trivial representation), it is enough to determine the infinitesimal character of $\Pi^{\vec{n}}_\infty$, which can be done using the definition of the test functions at $\infty$.

Decompose the intertwiner $A'$ as $(A')^p\cdot A'_p$. Using the fact that $\Pi^{\vec{n}}$ is $\theta$-stable and that the base change map at $p$ is injective (since $p$ splits in the quadratic field $\cK$), we can rewrite $\mathrm{tr}\left(\Pi_p^{\vec{n}}(f^{\vec{n}}_p)\circ A'_p\right)$ as $\mathrm{tr}\ \pi^{\vec{n}}_p(\phi^{\vec{n}}_p)$, for some representation $\pi_p^{\vec{n}}$ in $\mathrm{Irr}(G_{\vec{n}}(\mathbb{Q}_p))$ (at least up to a sign). Now, using Lemma~\ref{phi and red}, we can rewrite the latter as $\mathrm{tr}\ \mathrm{Red}^{b}_{\vec{n}}(\pi^{\vec{n}}_p)(\phi_p)$.

Keeping $\phi^p$ fixed, we have a formula for $\mathrm{tr}(\phi|\iota_\ell[H_c(\mathscr{I}^b_{\mathrm{Mant}},\mathbb{\bar Q}_\ell)])$ as a finite linear combination of traces of $\phi_p$ against irreducible representations of $J_b(\mathbb{Q}_p)$. At this stage, we can take $\phi_p$ to be any smooth, compactly-supported function on $J_b(\mathbb{Q}_p)$, not necessarily an acceptable one. Indeed, recall that the twist of any such $\phi^p$ by any sufficiently high power of Frobenius is acceptable, so the equality above holds for $\phi_p^{(N)}$ for sufficiently large $N$. The argument in the proof of Lemma 6.4 of~\cite{shin-igusa} now proves that the desired equality holds for every integer $N$ and, in particular, for $N=0$. 
\end{proof}

\begin{remark} As a consequence, we see that the $\mathbb G_n(\mathbb{A}_f^p)$-representation
\[
BC^p([H_c(\mathscr{I}^b_{\mathrm{Mant}},\mathbb{\bar Q}_\ell)]^{S\ur})
\]
can be written in terms of the transfer to $\mathbb{G}_n(\mathbb{A}_f^p)$ of representations $(\Pi^{\vec{n}})_f^p$, where the $\Pi^{\vec{n}}$ are $\theta$-stable automorphic representations of $\mathbb{G}_{\vec{n}}$ as in the statement of the lemma, which are unramified outside $S$.
\end{remark}

Moreover, recall the existence of Galois representations in the conjugate self-dual case.

\begin{thm}\label{ex gal repr} Let $\Pi^{\vec{n}}$ be an isobaric automorphic representation of $\mathbb G_{\vec{n}}$, unramified outside $S$, of the form
\[
\Pi^{\vec{n}} = (\mathrm{n-Ind}_Q^{\mathbb{G}_{\vec{n}}}\Pi_M)_\xi\ ,
\]
where $\Pi_M$ occurs in the (relatively) discrete part of the automorphic spectrum of the Levi subgroup $M$ of a parabolic $Q\subset \mathbb{G}_{\vec{n}}$ and is $\Phi_{\vec{n}}^{-1}\theta$-stable, with $\Pi^{\vec{n}}_\infty$ cohomological. Then there exists a conjugate self-dual (up to twist) continuous semisimple Galois representation
\[
r_{\Pi^{\vec{n}},\ell}: \mathrm{Gal}(\bar{F}/F)\to GL_{\vec{n}}(\bar{\mathbb Q}_\ell)
\]
which is unramified outside $S\cup \{\ell\}$, and such that for all primes $q\neq \ell$ of $\mathbb Q$, in the decomposition
\[
\Pi^{\vec{n}}_q = (\bigotimes_{\mathfrak q|q} \Pi^{\vec{n}}_\q)\otimes \chi_q
\]
corresponding to
\[
\mathbb G_{\vec{n}}(\mathbb Q_p) = (\prod_{\mathfrak q|q} GL_{\vec{n}}(F_{\mathfrak q}))\times \cK_q^\times\ ,
\]
the representations $\Pi^{\vec{n}}_{\mathfrak q}$ and $r_{\Pi^{\vec{n}},\ell}|_{\mathrm{Gal}(\bar{F}_{\mathfrak q}/F_{\mathfrak q})}$ correspond under the local Langlands correspondence (up to Frobenius semisimplification).\footnote{For our purposes, it is enough to know the compatibility up to semisimplification, i.e.~without identification of the monodromy operator, which is the most subtle part of the local-global-compatibility.}
\end{thm}

\begin{proof} Using the classification of the discrete spectrum of general linear groups due to Moeglin and Waldspurger, \cite{moeglin-waldspurger}, this follows from the main theorems of \cite{shin-galois}, \cite{chenevier-harris} and \cite{caraiani}. We remark that the conjugate self-dual, regular algebraic case suffices here because $\Pi_M$ is $\Phi_{\vec{n}}^{-1}\theta$-stable and has regular infinitesimal character (since $\Pi^{\vec{n}}$ has regular infinitesimal character). 
\end{proof}

In the following, we fix a Galois representation
\[
r: \mathrm{Gal}(\bar{F}/F)\to GL_n(\bar{\mathbb Q}_\ell)
\]
which is unramified outside $S\cup \{\ell\}$, and restrict attention to the summand 
\[
BC^p([H_c(\mathscr{I}^b_{\mathrm{Mant}},\mathbb{\bar Q}_\ell)]^{S\ur})^{S\ur}_r\ \mathrm{of}\ 
BC^S([H_c(\mathscr{I}^b_{\mathrm{Mant}},\mathbb{\bar Q}_\ell)]^{S\ur})
\]
coming from representations $\Pi^{\vec{n}}$ as above, with $r_{\Pi^{\vec{n}},\ell}\cong r$ (under the embedding $GL_{\vec{n}}(\bar{\mathbb Q}_\ell)\hookrightarrow GL_n(\bar{\mathbb Q}_\ell)$).

The following theorem is the key result of this section. Recall that we have fixed a prime $\p|p$ of the reflex field $E$, so that we have embeddings $E\hookrightarrow \mathbb C$, $E\hookrightarrow E_\p\hookrightarrow \bar{\mathbb Q}_p$. For convenience, let us fix an isomorphism $\iota_p: \bar{\mathbb Q}_p\cong \mathbb C$ compatible with the embedding of $E$.

\begin{thm}\label{igusa vanishing} For each prime $\p_i$ of $F$, let
\[
S_i=\{\tau: F\hookrightarrow \mathbb{C}|\iota_p\circ\tau\ \mathrm{induces}\ \mathfrak{p_i}\}\ .
\]
Assume that for each $i$, $S_i$ contains at most one $\tau$ for which $p_\tau q_\tau$ is nonzero, where $G$ has signature $(p_\tau,q_\tau)$ at $\tau: F\hookrightarrow \mathbb C$. Moreover, for each $i$ for which $S_i$ contains some $\tau$ for which $p_\tau q_\tau$ is nonzero, assume that
\[
r_{\mathrm{Gal}(\bar{F}_{\p_i}/F_{\p_i})} = \chi_{i,1}\oplus \ldots\oplus \chi_{i,n}
\]
decomposes as a direct sum of characters, such that for all $a\neq b$, $\chi_{i,a} \chi_{i,b}^{-1}$ is not the cyclotomic character.

Then, if $b\in B(G,\mu^{-1})$ is not $\mu$-ordinary,\footnote{It might be more accurate to write $\mu^{-1}$-ordinary.}
\[
BC^p([H_c(\mathscr{I}^b_{\mathrm{Mant}},\mathbb{\bar Q}_\ell)]^{S\ur})_r = 0\ .
\]
\end{thm}

\begin{proof} Assume the contrary. Then there is some $\theta$-stable isobaric automorphic representation $\Pi^{\vec{n}}$ of $\mathbb G_{\vec{n}}$ as above, with $r_{\Pi^{\vec{n}},\ell}\cong r$, which contributes to $BC^S([H_c(\mathscr{I}^b_{\mathrm{Mant}},\mathbb{\bar Q}_\ell)]^{S\ur})$. The component $\Pi^{\vec{n}}_p$ of $\Pi^{\vec{n}}$ at $p$ comes from a unique representation $\pi_p^{\vec{n}}\in \mathrm{Irr}(G_{\vec{n}}(\mathbb Q_p))$ via base change. We may decompose
\[
\pi_p^{\vec{n}} = \bigotimes_{i=1}^m \pi_{\p_i}^{\vec{n}}\otimes \pi_0^{\vec{n}}
\]
according to
\[
G_{\vec{n}}(\mathbb Q_p) = \prod_{i=1}^m GL_{\vec{n}}(F_{\p_i})\times \mathbb Q_p^\times\ .
\]
By the assumption on $r$ and local-global compatibility, we know that $\pi_{\p_i}^{\vec{n}}$ transfers to a generic principal series representation of $GL_n(F_{\p_i})$ for all $i$ for which $S_i$ contains some $\tau$ with $p_\tau q_\tau\neq 0$. By Lemma~\ref{red vanishes on generic}, $\mathrm{Red}_{\vec{n}}^b(\pi^{\vec{n}}_p)=0$ as soon as $J_{b_i}$ is not quasisplit for some such $i$, so that in this case there is no contribution by Lemma~\ref{expansion of igusa cohomology}.

It remains to see that if $b\in B(G,\mu^{-1})$ is not $\mu$-ordinary, then there is some $i$ for which $S_i$ contains some $\tau$ with $p_\tau q_\tau\neq 0$, such that $J_{b_i}$ is not quasisplit.

We can decompose
\[
\mu=((\mu_i)_{i=1,\ldots,m},\mu_0): \mathbb G_m\to G_{\bar{\mathbb Q}_p} = \prod_{i=1}^m (\prod_{F_{\p_i}\hookrightarrow \bar{\mathbb Q}_p} GL_{n,\bar{\mathbb Q}_p})\times \mathbb G_{m,\bar{\mathbb Q}_p}\ ;
\]
let $G_i = \mathrm{Res}_{F_{\p_i}/\mathbb Q_p} GL_n$. Then $\mu_i$ is a conjugacy class of minuscule cocharacters of $G_i$, and we have a decomposition
\[
B(G,\mu^{-1}) = \prod_{i=1}^m B(G_i,\mu_i^{-1})\ ,
\]
as the $\mathbb G_m$ factor plays no role here. In each factor $GL_{n,\bar{\mathbb Q}_p}$, $\mu$ has the form
\[
t\mapsto \mathrm{diag}(t,\ldots,t,1,\ldots,1)
\]
with $t$ occuring $p_\tau$ times, and $1$ occuring $q_\tau$ times, where $\tau: F\to \bar{\mathbb Q}_p\cong \mathbb C$ is the corresponding complex place. In particular, for each $i$ for which $S_i$ does not contain any $\tau$ with $p_\tau q_\tau\neq 0$, $\mu_i$ is central, which implies that $B(G_i,\mu_i^{-1})$ has precisely one element. If there is exactly one such $\tau$, then denoting by $\mu_{i,\tau}$ the corresponding component of $\mu_i$, one sees that
\[
B(G_i,\mu_i^{-1}) = B(GL_n/F_{\p_i},\mu_{i,\tau}^{-1})\ ,
\]
using the relative $B(H/L)=B(L,H)$ for a reductive group $H$ over a $p$-adic field $L$.\footnote{So far, we were only using the case $L=\mathbb Q_p$, and did not include this in the notation.} Now the result follows from the next lemma.
\end{proof}

\begin{lemma} Let $L$ be any $p$-adic field, and let
\[
\mu: \mathbb G_m\to GL_n\ :\ t\mapsto \mathrm{diag}(t,\ldots,t,1,\ldots,1)
\]
be a minuscule cocharacter with $n-q$ occurences of $t$ and $q$ occurences of $1$. Then there is exactly one element $b\in B(GL_n/L,\mu^{-1})$ for which $J_b$ is quasisplit, namely the $\mu$-ordinary element represented by $\mathrm{diag}(\varpi^{-1},\ldots,\varpi^{-1},1,\ldots,1)$, with $n-q$ occurences of the uniformizer $\varpi$ of $L$, and $q$ occurences of $1$.
\end{lemma}

\begin{proof} By the choice of $\mu$, we know that for any $b\in B(GL_n/L,\mu^{-1})$, the slopes $\lambda_i$ satisfy $-1\leq \lambda_i\leq 0$. If some slope $\lambda$ is nonintegral, then $J_b$ is not quasisplit, as it contains a factor which is a general linear group over the division algebra of invariant $\lambda\pmod 1$ over $L$. Thus, if $J_b$ is quasisplit, then all slopes are equal to $0$ or $-1$; from the equality $\kappa(b) = -\mu$ one deduces that slope $-1$ occurs with multiplicity $n-q$, and slope $0$ with multiplicity $q$, which corresponds to the $\mu$-ordinary element $b=\mathrm{diag}(\varpi^{-1},\ldots,\varpi^{-1},1,\ldots,1)$. For this $b$, $J_b\cong GL_{n-q}\times GL_q$ is quasisplit.
\end{proof}

\subsection{Simple Shimura varieties}\label{trace-uniformization}

In this section, we sketch how to adapt the arguments above for Kottwitz' simple Shimura varieties as in \cite{Kottwitz-lambda}. This includes the case of Shimura varieties which admit $q$-adic uniformization, for some rational prime $q$ distinct from $p$ and $\ell$. In that case, our main result is related to level-raising results, as shown in~\cite{thorne}. 

Recall that $F=F^+\cdot \cK$. Assume that we have a PEL datum of the form $(B,*, V,\langle\ ,\ \rangle,h)$, where $B$ is a division algebra with center $F$, $V$ is a simple $B$-module, and $\ast$ is an involution of the second kind. Then the corresponding Shimura varieties $S_{K}$ are proper and the group $G$ has no endoscopy. Assume that $B$ is split at all places over $p$, in which case the constructions and results of Section~\ref{transfer at p} carry over. However, Theorem~\ref{stable trace formula for Igusa} simplifies considerably. We follow Section 6 of~\cite{shin-rapoport-zink}, where it is assumed that $p$ is inert in $F^+$; this assumption is not necessary for our purposes. As above, let $G_n$ be a quasi-split inner form of $G$ over $\mathbb{Q}$ and fix an isomorphism $G_n\simeq G$ over $\mathbb{Q}_p$. 

\begin{prop}\label{simplified stable trace formula} Let $\phi=\phi^p\phi_p\in \cC^\infty_c(G(\mathbb{A}^p_f)\times J_b(\mathbb{Q}_p))$, with $\phi_p$ an acceptable function. Then 
\[\mathrm{tr}(\phi|\iota_\ell H_c(\mathscr{I}^b_{\mathrm{Mant}},\mathbb{\bar Q}_\ell))=|\ker^1(\mathbb{Q},G)|\iota(G,G_n)ST^{G_n}_e(\phi^{G_n}).\]
\end{prop}
\begin{proof} The other terms in the stable trace formula vanish by Lemma 7.1 of~\cite{shin-stable}.
\end{proof}

We can now combine this with the stable trace formula for the $S_K$, which is Theorem 6.1 of~\cite{arthur-lefschetz} and which is simplified in our situation as in Proposition 6.3 of~\cite{shin-rapoport-zink}, also making use of Lemma~\ref{phi and red} for $G_n(\mathbb{Q}_p)\simeq G(\mathbb{Q}_p)$. We get
\[
\mathrm{Red}^b_{n}\left([H(S_K,\mathbb{\bar Q}_\ell)]\right)=\epsilon_G\cdot d(G_{\mathbb{R}})\cdot [H_c(\mathscr{I}^b_{\mathrm{Mant}},\mathbb{\bar Q}_\ell)]\ ,
\]
where $\epsilon_G,d(G_{\mathbb{R}})$ are certain non-zero constants. Again, we appeal to Lemma 6.4 of~\cite{shin-galois} to extend a trace identity from acceptable $\phi_p$ to all $\phi_p\in \cC_c^\infty(J_b(\mathbb{Q}_p))$. We combine this with Matsushima's formula, which gives a description of $[\iota_\ell H(S_K,\mathbb{\bar Q}_\ell)]$ in terms of automorphic representations of $G$. We get an analogue of Corollary 6.12 of~\cite{shin-rapoport-zink}.

\begin{cor}\label{simplified igusa cohomology} We have the following equality in $\mathrm{Groth}(G(\mathbb{A}_f^p)\times J_b(\mathbb{Q}_p))$:
\[
[\iota_\ell H_c(\mathscr{I}^b_{\mathrm{Mant}},\mathbb{\bar Q}_\ell)]=(-1)^{q(G)}\sum_{\pi_f} c(\pi_f) [\pi^p_f][\mathrm{Red}^b_n(\pi_p)]\ .
\]
\end{cor}
\noindent The sum runs over admissible representations $\pi_f$ of $G(\mathbb{A}_f)$ such that $\pi_f\pi_\infty$ is an automorphic representation of $G$, for some representation $\pi_\infty$ of $G(\mathbb{R})$ which is cohomological for the trivial algebraic representation. The coefficients $c(\pi_f)$ are related to the automorphic multiplicity of $\pi_f\pi_\infty$.

In this case, the existence of Galois representations is also known, as the stable base change of such $\pi$ to $GL_n$ has been established by Shin in the appendix to~\cite{shin-appendix}. As before, for a Galois representation
\[
r: \mathrm{Gal}(\bar{F}/F)\to GL_n(\bar{\mathbb Q}_\ell)\ ,
\]
we restrict attention to the summand $[H_c(\mathscr{I}^b_{\mathrm{Mant}},\mathbb{\bar Q}_\ell)]_r$ of
\[
[H_c(\mathscr{I}^b_{\mathrm{Mant}},\mathbb{\bar Q}_\ell)]
\]
coming from representations $\pi$ as above, with $r_{\pi,\ell}\cong r$.

We get the following analogue of Theorem~\ref{igusa vanishing}, which is proved in the same way.

\begin{cor}\label{igusa vanishing 2} For each prime $\p_i$ of $F$ above $p$, let
\[
S_i=\{\tau: F\hookrightarrow \mathbb{C}|\iota_p\circ\tau\ \mathrm{induces}\ \mathfrak{p_i}\}\ .
\]
Assume that for each $i$, $S_i$ contains at most one $\tau$ for which $p_\tau q_\tau$ is nonzero, where $G$ has signature $(p_\tau,q_\tau)$ at $\tau: F\hookrightarrow \mathbb C$. Moreover, for each $i$ for which $S_i$ contains some $\tau$ for which $p_\tau q_\tau$ is nonzero, assume that
\[
r_{\mathrm{Gal}(\bar{F}_{\p_i}/F_{\p_i})} = \chi_{i,1}\oplus \ldots\oplus \chi_{i,n}
\]
decomposes as a direct sum of characters, such that for all $a\neq b$, $\chi_{i,a} \chi_{i,b}^{-1}$ is not the cyclotomic character.

Then, if $b\in B(G,\mu^{-1})$ is not $\mu$-ordinary,
\[
[H_c(\mathscr{I}^b_{\mathrm{Mant}},\mathbb{\bar Q}_\ell)]_r = 0\ .
\]
\end{cor}
\newpage

\section{Torsion in the cohomology of unitary Shimura varieties}\label{final section}

In this final section, we give a precise formulation and proof of our main result. We start by formulating and proving the critical perversity result.

\subsection{Perverse sheaves on the flag variety}

Consider the Hodge-Tate period map
\[
\pi_{HT}: \cS_{K^p}\to \Fl_{G,\mu}
\]
for a compact Hodge type Shimura variety. In this section, we would like to make precise in which sense $R\pi_{HT\ast} \mathbb F_\ell$ is perverse.\footnote{As we are far from a finite type situation, we avoid talking about $\mathbb Q_\ell$-sheaves. We could talk about $\mathbb Z/\ell^n\mathbb Z$-sheaves, but in that case the notion of perversity is slightly subtle as $\mathbb Z/\ell^n\mathbb Z$ is not a field. For our applications, the $\mathbb F_\ell$-case is enough.}

Recall the following result on preservation of perversity under nearby cycles.

\begin{thm}[{\cite[Corollaire 4.5]{illusie-autour}}]\label{nearby perverse} Let $K$ be a complete nonarchimedean field with ring of integers $\cO_K$ and completed algebraic closure $C$ with $\cO_C\subset C$, and let $\ell$ be a prime which is invertible in $\cO_K$. Let $X$ be a scheme of finite type over $\cO_K$. Let $X_{\cO_C}$ be the base-change to $\cO_C$, with geometric generic fibre $j: X_{\bar{\eta}} = X_{\cO_K}\otimes_{\cO_K} C\hookrightarrow X_{\cO_C}$ and geometric special fibre $i: X_{\bar{s}}\hookrightarrow X_{\cO_C}$. Let $\mathscr{F}$ be a perverse $\mathbb F_\ell$-sheaf on $X_\eta = X\times_{\cO_K} K$. Then $R\psi \mathscr{F} = i^\ast Rj_\ast \mathscr{F}|_{X_{\bar{\eta}}}$ is a perverse sheaf on $X_{\bar{s}}$.
\end{thm}

Moreover, nearby cycles in the scheme setting agree with nearby cycles in the formal/rigid setting. More precisely, we have the following result.

\begin{thm}[{\cite[Theorem 3.5.13]{huber}}]\label{nearby adic} Let the situation be as in Theorem~\ref{nearby perverse}. Let $\cX_\eta$ be the associated rigid-analytic variety over $K$, considered as an adic space, with base change $\cX_{\bar{\eta}}$ to $C$. There is a natural morphism of sites $\lambda: \cX_{\bar{\eta},\et}\to (X_{\bar{s}})_\et$, given by lifting an \'etale map $Y\to X_{\bar{s}}$ to an \'etale map of formal schemes over $\cO_C$, and then taking the generic fibre.

Let $\mathscr{F}^{\mathrm{ad}}$ be the pullback of $\mathscr{F}$ under $\cX_\et\to X_\et$. Then
\[
R\lambda_\ast (\mathscr{F}^{\mathrm{ad}}|_{\cX_{\bar{\eta}}})\cong R\psi \mathscr{F}\ .
\]
\end{thm}

In our situation, it is hard to give a direct definition of perversity of $R\pi_{HT\ast} \mathbb F_\ell$. However, the above properties suggest that at least, for every formal model $X$ of the flag variety $\Fl_{G,\mu}$, the nearby cycles $R\psi_X R\pi_{HT\ast} \mathbb F_\ell$ should be a perverse sheaf on the special fibre $X_{\bar{s}}$ of $X$. This is still not true, as $G(\mathbb Q_p)$ acts on $R\pi_{HT\ast} \mathbb F_\ell$; one can only hope for the $K_p$-invariants to be perverse, for any sufficiently small $K_p\subset G(\mathbb Q_p)$. Thus, we work with the equivariant sites introduced in~\cite[\S 2]{scholze-lubintate}.

First, note that $R\pi_{HT\ast} \mathbb F_\ell$ is a canonically a complex of sheaves on the equivariant site $(\Fl_{G,\mu}/G(\mathbb Q_p))_\et$. More precisely, one has the map of equivariant sites
\[
\pi_{HT}/G(\mathbb Q_p): (\cS_{K^p}/G(\mathbb Q_p))_\et\to (\Fl_{G,\mu}/G(\mathbb Q_p))_\et\ ,
\]
and one can look at $R(\pi_{HT}/G(\mathbb Q_p))_{\et\ast} \mathbb F_\ell$, and this pulls back to $R\pi_{HT\ast} \mathbb F_\ell$ under the projection $(\Fl_{G,\mu})_\et\to (\Fl_{G,\mu}/G(\mathbb Q_p))_\et$. To check the latter statement, note first that by passing to slice categories, using~\cite[Proposition 2.9]{scholze-lubintate}, one may replace $G(\mathbb Q_p)$ by any compact open subgroup $K_p\subset G(\mathbb Q_p)$, and then one can pass to the limit using~\cite[Proposition 2.8]{scholze-lubintate}.

Now take any \'etale $U=\mathrm{Spa}(A,A^\circ)\to \Fl_{G,\mu}$. By~\cite[Corollary 2.5]{scholze-lubintate}, the action of $K_p$ extends to a continuous action on $U$ if $K_p$ is sufficiently small. Let $\mathfrak{U} = \mathrm{Spf}(A^\circ)$ with special fibre $\mathfrak{U}_s = \mathrm{Spec}(A^\circ/p)$. Then $K_p$ acts trivially on $\mathfrak{U}_s$ if $K_p$ is sufficiently small, by continuity of the $K_p$-action and finite generation of $A^\circ/p$. It follows that any \'etale map to $\mathfrak{U}_{\overline{s}}$ lifts to a $K_p$-equivariant \'etale map to $\mathfrak{U}_{\cO_C}$ (where $C=\mathbb C_p$), giving a natural morphism of sites
\[
\lambda_{U/K_p}: (U_{\bar{\eta}}/K_p)_\et\to \mathfrak{U}_{\bar{s},\et}\ .
\]

\begin{prop}\label{perverse sheaf} Let
\[
\pi_{HT}: \cS_{K^p}\to \Fl_{G,\mu}
\]
be the Hodge-Tate period map for a compact Shimura variety of Hodge type and any sufficiently small compact open subgroup $K^p\subset G(\mathbb A_f^p)$. Let $\bar{x}\in \Fl_{G,\mu}$ be a geometric point. Then there exists a neighborhood basis of affinoid \'etale neighborhoods $U=\mathrm{Spa}(A,A^\circ)$ of $x$ in $\Fl_{G,\mu}$ such that, denoting $\mathfrak{U} = \mathrm{Spf}(A^\circ)$,
\[
R\lambda_{U/K_p \ast} (R(\pi_{HT}/G(\mathbb Q_p))_\ast \mathbb F_\ell)|_{U_{\bar{\eta}}/K_p}[\langle 2\rho,\mu\rangle]
\]
is a perverse sheaf on $\mathfrak{U}_{\bar{s}}$ for any sufficiently small pro-$p$ compact open subgroup $K_p\subset G(\mathbb Q_p)$.
\end{prop}

\begin{proof} By~\cite[Theorem IV.1.1 (i)]{scholze}, one can find some affinoid \'etale (in fact, open) neighborhood $U$ of $x$ such that $\cS_{K^p,U} = \cS_{K^p}\times_{\Fl_{G,\mu}} U$ is affinoid perfectoid, and equal to the preimage of an affinoid \'etale $\cS_{K_pK^p,U}\to \cS_{K_pK^p}$ for any sufficiently small $K_p$. These properties will then also be true for any \'etale $V\to U$ that factors as a composite of finite \'etale maps and rational embeddings, and such $V$ are cofinal. Thus, fix any $U$ with the stated properties.

Let
\[
\pi_{HT,U}: \cS_{K^p,U}\to U = \mathrm{Spa}(A,A^\circ)
\]
be the restriction of $\pi_{HT}$. As $\pi_{HT}$ is partially proper, so is $\pi_{HT,U}$. If $K_p$ is sufficiently small, $\pi_{HT,U}$ is $K_p$-equivariant, and induces a map
\[
\pi_{HT,U/K_p}: \cS_{K^p,U}/K_p\to U/K_p\ .
\]
Also
\[
(R(\pi_{HT}/G(\mathbb Q_p))_\ast \mathbb F_\ell)|_{U/K_p} = R\pi_{HT,U/K_p \ast} \mathbb F_\ell\ ,
\]
and by~\cite[Proposition 2.12]{scholze-lubintate}, there is an equivalence of sites $(\cS_{K^p,U}/K_p)_\et\cong \cS_{K_pK^p,U,\et}$.

Now any $\cS_{K_pK^p,U} = \mathrm{Spa}(R_{K_pK^p,U},R_{K_pK^p,U}^\circ)$ has its natural integral model $\mathfrak{S}_{K_pK^p,U} = \mathrm{Spf}(R_{K_pK^p,U}^\circ)$, with inverse limit $\mathfrak{S}_{K^p,U} = \mathrm{Spf}(R_{K^p,U}^\circ)$, where $\cS_{K^p,U} = \mathrm{Spa}(R_{K^p,U},R_{K^p,U}^\circ)$. We get a map of formal schemes
\[
\pi_{HT,\mathfrak{U}}: \mathfrak{S}_{K^p,U}\to \mathfrak{U}\ .
\]
Modulo $p$, we get a map of schemes
\[
\pi_{HT,\mathfrak{U}_s}: \mathfrak{S}_{K^p,U,s}\to \mathfrak{U}_s\ ,
\]
with $\mathfrak{S}_{K^p,U,s} = \mathrm{Spec}(R_{K^p,U}^\circ/p)$, and $\mathfrak{U}_s = \mathrm{Spec}(A^\circ/p)$. But $\mathfrak{U}_s$ is of finite type over $\mathbb F_p$, and $\mathfrak{S}_{K^p,U,s} = \varprojlim_{K_p} \mathfrak{S}_{K_pK^p,U,s}$ in the category of (affine) schemes. It follows that $\pi_{HT,\mathfrak{U}_s}$ factors over a map
\[
\pi_{HT,K_p,\mathfrak{U}_s}: \mathfrak{S}_{K_pK^p,U,s}\to \mathfrak{U}_s
\]
(of affine schemes of finite type over $\mathbb F_p$) for any sufficiently small $K_p$. We claim that $\pi_{HT,K_p,\mathfrak{U}_s}$ satisfies the valuative criterion of properness. If $K$ is an algebraically closed field with a rank-$1$-valuation ring $V\subset K$, and we are given a $V$-point of $\mathfrak{U}_s$ together with a lift of the corresponding $K$-valued point to a $K$-valued point of $\mathfrak{S}_{K_pK^p,U,s}$, we need to show that this $K$-valued point is in fact $V$-valued. We may lift the $K$-valued point of $\mathfrak{S}_{K_pK^p,U,s}$ to $\mathfrak{S}_{K^p,U,s}$ (as all transition maps are finite and surjective). We may then find a complete algebraically closed extension $C/\mathbb Q_p$ with residue field $K$ and a $(C,\cO_C)$-valued point of $\cS_{K^p,U}$ specializing to this $K$-valued point of $\mathfrak{S}_{K^p,U,s}$. Let $C^+\subset \cO_C$ be the preimage of $V\subset K$. Then the image of the $(C,\cO_C)$-valued point of $\cS_{K^p,U}$ under $\pi_{HT,U}$ is a $(C,\cO_C)$-valued point of $U$ which extends to a $(C,C^+)$-valued point. As $\pi_{HT,U}$ is partially proper, it follows that we get a $(C,C^+)$-valued point of $\cS_{K^p,U}$, which specializes to a $V$-valued point of $\mathfrak{S}_{K^p,U,s}$ and thus of $\mathfrak{S}_{K_pK^p,U,s}$, as desired.

Thus, $\pi_{HT,K_p,\mathfrak{U}_s}$ is a map of affine schemes of finite type over $\mathbb F_p$ which satisfies the valuative criterion of properness, i.e., it is finite.\footnote{Thus, we are in the somewhat curious situation that $\pi_{HT,\mathfrak{U}_s}$ is ind-finite, but $\pi_{HT,U}$ has fibres of positive dimension.} Now consider the following diagram, where we have base-changed some spaces and maps to algebraically closed fields.
\[\xymatrix{
(\cS_{K^p,U,\bar{\eta}}/K_p)_\et\ar[d]_\cong \ar[rrr]^{\pi_{HT,U_{\bar{\eta}}/K_p}} &&& (U_{\bar{\eta}}/K_p)_\et\ar[dd]^{\lambda_{U/K_p}}\\
\cS_{K_pK^p,U,\bar{\eta},\et}\ar[d]_{\lambda_{\mathfrak{S}_{K_pK^p,U}}} &&& \\
\mathfrak{S}_{K_pK^p,U,\bar{s},\et}\ar[rrr]_{\pi_{HT,K_p,\mathfrak{U}_{\bar{s}}}} &&& \mathfrak{U}_{\bar{s},\et}
}\]
We are interested in the pushforward of $\mathbb F_\ell$ from the upper left to the lower right corner, computed via the upper right corner. We may equivalently compute it via the lower left corner. In that case, the first pushforward is perverse by Theorem~\ref{nearby perverse} and Theorem~\ref{nearby adic}, up to the shift $\langle 2\rho,\mu\rangle = \dim \cS_{K_pK^p,U}$. But $\pi_{HT,K_p,\mathfrak{U}_{\bar{s}}}$ is finite, so it also preserves perversity under pushforward.
\end{proof}

We will need the following consequence, which is a statement purely about the cohomology of Igusa varieties. For the statement, let $S$ be a finite set of primes such that $K^p = K_S^p K^S$, where $K^S\subset G(\mathbb A_f^S)$ is a product of hyperspecial maximal compact open subgroups, and $K_S^p\subset G(\mathbb A_S^p)$. Let
\[
\mathbb T^S = Z[G(\mathbb A_f^S) // K^S]
\]
be the abstract (commutative) Hecke algebra of $K^S$-biinvariant compactly supported functions on $G(\mathbb A_f^S)$.

\begin{cor}\label{conc in one degree} Fix a maximal ideal $\mathfrak m\subset \mathbb T^S$, and among all $b\in B(G,\mu^{-1})$ with the property that the $\mathfrak m$-torsion
\[
H^i(\mathrm{Ig}^b,\mathbb F_\ell)[\mathfrak m]\neq 0
\]
for some $i\in \mathbb Z$, take some $b$ with $d=\langle 2\rho,\nu_b\rangle$ minimal. Then $H^i(\mathrm{Ig}^b,\mathbb F_\ell)[\mathfrak m]$ is nonzero only for $i=d$.
\end{cor}

The idea is that the sheaf $(R\pi_{HT\ast} \mathbb F_\ell)_{\mathfrak m}$ is concentrated on a subset of dimension $\langle 2\rho,\mu\rangle - d$ by assumption. Thus, $\Fl_{G,\mu}^b$ is one of the largest strata where $(R\pi_{HT\ast} \mathbb F_\ell)_{\mathfrak m}$ is nonzero. But as this sheaf is (up to shift) perverse, one concludes by observing that on the largest stratum where a perverse sheaf is nonzero, it is concentrated in one degree. However, as the notion of perversity is defined via nearby cycles, we need to rewrite this argument slightly.

\begin{proof} As
\[
H^i(\mathrm{Ig}^b,\mathbb F_\ell) = \varinjlim_m H^i(\mathscr{I}^b_{\mathrm{Mant},m},\mathbb F_\ell)\ ,
\]
where the transition maps are split injective (namely, projections are given by averaging operators over compact open subgroups of $J_b(\mathbb Q_p)$), and the terms on the right are finite-dimensional, we see that the $\mathfrak m$-torsion is nonzero precisely when the $\mathfrak m$-localization is nonzero. Thus, we may work with the localization at $\mathfrak m$ instead.

Now $R\pi_{HT\ast} \mathbb F_\ell$ (in fact, the $G(\mathbb Q_p)$-equivariant version $R(\pi_{HT}/G(\mathbb Q_p))_\et \mathbb F_\ell$) is a sheaf of $\mathbb T_S$-modules, as the Hecke operators away from $p$ act trivially on $\Fl_{G,\mu}$. We may thus form the localization $(R\pi_{HT\ast} \mathbb F_\ell)_\mathfrak m$. We claim that $(R\pi_{HT\ast} \mathbb F_\ell)_{\mathfrak m}$ is concentrated on the union $\Fl_{G,\mu}^{\geq d}$ of $\Fl_{G,\mu}^{b^\prime}$ over all $b^\prime$ with $\langle 2\rho,\nu_{b^\prime}\rangle\geq d$ (which is a closed subset of $\Fl_{G,\mu}$).

Indeed, if $y\in \Fl_{G,\mu}$ does not lie in $\Fl_{G,\mu}^{\geq d}$, then it lies in $\Fl_{G,\mu}^{b^\prime}$ for some $b^\prime$ with $\langle 2\rho,\nu_{b^\prime}\rangle<d$. Now Theorem~\ref{comp fiber} computes the fibre of $R\pi_{HT\ast} \mathbb F_\ell$ at any geometric point above $y$ as $R\Gamma(\mathrm{Ig}^{b^\prime},\mathbb F_\ell)$. We may pass to localizations at $\mathfrak m$ in this statement, and thus the assumption of the corollary shows that the localization of $(R\pi_{HT\ast} \mathbb F_\ell)_{\mathfrak m}$ at $y$ vanishes.

Next, we claim that for any affinoid {\'e}tale $U\to \Fl_{G,\mu}$ with formal model $\mathfrak{U}$, equivariant under $K_p$, with trivial action on $\mathfrak{U}_s$, the nearby cycles
\[
R\lambda_{U/K_p\ast}\left( (R(\pi_{HT}/G(\mathbb Q_p)_\ast \mathbb F_\ell)_{\mathfrak m}\right)|_{U_{\bar{\eta}}/K_p}
\]
are supported on a closed subset of $\mathfrak{U}_{\bar{s}}$ of dimension $\langle 2\rho,\mu\rangle - d$. Indeed, the sheaf is supported on the closure in $\mathfrak{U}$ of the preimage $U^{\geq d}\subset U$ of $\Fl_{G,\mu}^{\geq d}\subset \Fl_{G,\mu}$. But $U^{\geq d}\subset U$ is a closed subset of dimension $\leq \langle 2\rho,\mu \rangle - d$, and then the same is true for its closure in $\mathfrak{U}$: If $x\in \mathfrak{U}_s$ is a point whose closure is of dimension $e$, then the closure in $U$ of any lift $\tilde{x}\in U$ of $x$ will have at least dimension $e$ (as the specialization map is specializing).

Recall that $R\lambda_{U/K_p \ast} (R(\pi_{HT}/G(\mathbb Q_p))_\ast \mathbb F_\ell)|_{U_{\bar{\eta}}/K_p}[\langle 2\rho,\mu\rangle]$ is perverse. It follows that the same is true for its localization
\[\begin{aligned}
&(R\lambda_{U/K_p \ast} (R(\pi_{HT}/G(\mathbb Q_p))_\ast \mathbb F_\ell)|_{U_{\bar{\eta}}/K_p}[\langle 2\rho,\mu\rangle])_\mathfrak{m}\\
 = &R\lambda_{U/K_p\ast}\left( (R(\pi_{HT}/G(\mathbb Q_p))_\ast \mathbb F_\ell)_{\mathfrak m}\right)|_{U_{\bar{\eta}}/K_p}[\langle 2\rho,\mu\rangle]
\end{aligned}\]
at $\mathfrak m$. This sheaf is supported on a scheme of finite type of dimension $\langle 2\rho,\mu\rangle - d$. It follows that the localization
\[
\left(R\lambda_{U/K_p\ast}\left( (R(\pi_{HT}/G(\mathbb Q_p))_\ast \mathbb F_\ell)_{\mathfrak m}\right)|_{U_{\bar{\eta}}/K_p}\right)|_{\bar{x}}
\]
at any geometric point $\bar{x}\in \mathfrak{U}_{\bar{s}}$ whose closure is of dimension $\langle 2\rho,\mu\rangle - d$ is concentrated in degree $d$.

Now pick $b$ as in the statement, and choose a rank $1$ point $y\in \Fl_{G,\mu}^b$ with $\dim \overline{\{y\}} = \langle 2\rho,\mu\rangle - d$, and a geometric point $\bar{y}$ above $y$. One has an identification
\[
(R\pi_{HT\ast} \mathbb F_\ell)_{\mathfrak m,\bar{y}} = R\Gamma(\mathrm{Ig}^b,\mathbb F_\ell)_{\mathfrak m}\ .
\]
On the other hand, choose a cofinal system of affinoid \'etale neighborhoods $U_i=\mathrm{Spa}(R_i,R_i^\circ)\to \Fl_{G,\mu}$ of $\bar{y}$ as in Proposition~\ref{perverse sheaf}, with formal models $\mathfrak{U}_i$. Let $\bar{x}_i\in \mathfrak{U}_{i,s}$ be the specialization of $\bar{y}$, which is a geometric point of $\mathfrak{U}_{i,s}$. If $i$ is large enough, the dimension of the closure of $\bar{x}_i$ will be equal to $\langle 2\rho,\mu\rangle - d$: One needs to arrange that the image of $R_i^\circ\to \cO_{K(\bar{y})}\to k(\bar{y})$, where $K(\bar{y})$ is the completed residue field at $\bar{y}$, with ring of integers $\cO_{K(\bar{y})}$ and residue field $k(\bar{y})$, contains a transcendence basis. Also, choose compact open subgroups $K_{p,i}\subset G(\mathbb Q_p)$ that act on $U_i$ and trivially on $\mathfrak{U}_{i,s}$, such that the $K_{p,i}$ shrink to $1$.

In that situation, we know that for all large enough $i$
\[
\left(R\lambda_{U_i/K_{p_i}\ast}\left( (R(\pi_{HT}/G(\mathbb Q_p))_\ast \mathbb F_\ell)_{\mathfrak m}\right)|_{U_{i,\bar{\eta}}/K_{p,i}}\right)|_{\bar{x}_i}
\]
is concentrated in degree $d$. Finally, we conclude by observing that
\[
(R\pi_{HT\ast} \mathbb F_\ell)_{\mathfrak m,\bar{y}} = \varinjlim_i \left(R\lambda_{U_i/K_{p,i}\ast}\left( (R(\pi_{HT}/G(\mathbb Q_p))_\ast \mathbb F_\ell)_{\mathfrak m}\right)|_{U_{i,\bar{\eta}}/K_{p,i}}\right)|_{\bar{x}_i}\ .
\]
\end{proof}

\subsection{A genericity assumption}

In our main theorem, we impose a genericity assumption at some auxiliary prime. In this section, we briefly study this genericity condition.

\begin{defn} Let $L$ be a $p$-adic field, and let
\[
\overline{\rho}: \mathrm{Gal}(\bar{L}/L)\to GL_n(\bar{\mathbb F}_\ell)
\]
be an unramified, continuous representation, with $\ell\neq p$. Then $\overline{\rho}$ is \emph{decomposed generic} if the eigenvalues $\lambda_1,\ldots,\lambda_n$ of $\overline{\rho}(\mathrm{Frob})$ satisfy $\lambda_a/\lambda_b\not\in \{1,q\}$ for all $a\neq b$, where $\mathrm{Frob}$ is an arithmetic Frobenius, and $q$ is the cardinality of the residue field of $L$.
\end{defn}

We note that this condition actually only depends on the semisimplification of $\overline{\rho}$, but also implies that $\overline{\rho}$ is semisimple. In particular, if
\[
\rho: \mathrm{Gal}(\bar{L}/L)\to GL_n(\bar{\mathbb Q}_\ell)
\]
is a continuous representation, the condition that the reduction $\overline{\rho}$ be decomposed generic is unambiguous.

\begin{lemma}\label{generic lifts} Assume that
\[
\rho: \mathrm{Gal}(\bar{L}/L)\to GL_n(\bar{\mathbb Q}_\ell)
\]
is a continuous representation such that the reduction $\overline{\rho}$ is decomposed generic. Then $\rho$ decomposes as a sum $\rho = \bigoplus_{i=1}^n \chi_i$ of characters, and $\chi_a/\chi_b$ is not the cyclotomic character for any $a\neq b$.
\end{lemma}

In particular, the representation of $GL_n(L)$ corresponding to $\rho$ is a generic principal series representation.

\begin{proof} We may conjugate $\rho$ into $GL_n(\cO_K)$ for some finite extension $K\subset \bar{\mathbb Q}_\ell$. Writing $\overline{\rho} = \bigoplus_{i=1}^n \overline{\chi}_i$, we may further conjugate $\rho$ into the matrices in $GL_n(\cO_K)$ which are diagonal modulo a uniformizer $\varpi$ of $\cO_K$. Now we try to conjugate $\rho$ into the matrices which are diagonal modulo higher powers of $\varpi$. By standard calculations in deformation theory, the relevant obstruction groups are given by
\[
H^1(\mathrm{Gal}(\bar{L}/L),\overline{\chi}_a/\overline{\chi}_b)
\]
for $a\neq b$. But if
\[
\overline{\chi}_\lambda: \mathrm{Gal}(\bar{L}/L)\to \bar{\mathbb F}_\ell^\times
\]
denotes the unramified character sending $\mathrm{Frob}$ to $\lambda$, then it is well-known that
\[
H^1(\mathrm{Gal}(\bar{L}/L),\overline{\chi}_\lambda) = 0
\]
if $\lambda\not\in \{1,q\}$. By assumption, it follows that all relevant obstruction groups vanish.

The final statement follows because $\overline{\chi}_a/\overline{\chi}_b$ is not the cyclotomic character.
\end{proof}

\subsection{Conclusion}

Finally, we can tie everything together and prove our main theorem.

Let us recall the relevant Shimura varieties. We fix a compact Shimura variety of PEL type, associated with PEL data $(B,\ast,V,(\cdot,\cdot))$ of type A satisfying one of the following assumptions. In both cases, $F=F^+\cdot \cK$ is a CM field with totally real subfield $F^+$ containing an imaginary-quadratic field $\cK$.

\textbf{Case 1.} Assume that $B$ is a central division algebra over $F$, and $V\cong B$ is a simple $B$-module.

\textbf{Case 2.} Assume that $B=F$, $F^+\neq \mathbb Q$, the corresponding group $G$ is quasi-split at all finite places, and if a rational prime $q$ is ramified in $F$, then $F/F^+$ is split at all places above $q$.

In both cases, let $\Spl_{F/F^+}$ denote the set of rational primes $q$ such that every place of $F^+$ above $q$ splits in $F$. Moreover, fix a finite set $S$ of primes such that $F$ and $G$ are unramified outside $S$, and pick a sufficiently small compact open subgroup $K=K_S K^S\subset G(\mathbb A_f)=G(\mathbb A_S)\times G(\mathbb A_f^S)$ such that $K^S$ is a product of hyperspecial maximal compact open subgroups $K_q\subset G(\mathbb Q_q)$. In Case 2, we assume that $S\subset \Spl_{F/F^+}$. Finally, take some rational prime $\ell$. We will consider the following abstract Hecke algebra
\[
\mathbb T^S = \bigotimes_{q\in \Spl_{\cK/\mathbb Q}\setminus (S\cup \{\ell\})} \mathbb Z[G(\mathbb Q_q)//K_q]\ .
\]

\begin{thm}\label{main theorem} Let $\mathfrak m\subset \mathbb T^S$ be a maximal ideal such that
\[
H^i(S_K,\mathbb F_\ell)_{\mathfrak m}\neq 0
\]
for some $i\in \mathbb Z$.
\begin{enumerate}
\item There is a (unique) semisimple continuous Galois representation
\[
\rho_{\mathfrak m}: \mathrm{Gal}(\overline{F}/F)\to GL_n(\bar{\mathbb F}_\ell)
\]
unramified outside the places above $S\cup \{\ell\}$, such that for all finite places $v$ lying above a prime $q\in \Spl_{\cK/\mathbb Q}\setminus (S\cup \{\ell\})$, the characteristic polynomial of $\rho_{\mathfrak m}(\mathrm{Frob}_v)$ is given by the image of
\[
X^n - T_{1,v} X^{n-1} \pm \ldots + (-1)^i q_v^{i(i-1)/2} T_{i,v} X^{n-i} + \ldots + (-1)^n q_v^{n(n-1)/2} T_{n,v}
\]
under a fixed embedding $\mathbb T^S/\mathfrak m\hookrightarrow \bar{\mathbb F}_\ell$, where $q_v$ is the cardinality of the residue field at $v$, and
\[
T_{i,v}\in \mathbb Z[G(\mathbb Q_q)//K_q]
\]
is the characteristic function of
\[
GL_n(\cO_{F_v})\mathrm{diag}(\underbrace{\varpi_v,\ldots,\varpi_v}_i,1,\ldots,1)GL_n(\cO_{F_v})\times \prod_{w\neq v} GL_n(\cO_{F_w})\times \mathbb Z_p^\times
\]
inside
\[
G(\mathbb Q_q) = \prod_w GL_n(F_w)\times \mathbb Q_p^\times\ ,
\]
where $w$ runs over all places of $F$ lying over the same place of $\cK$ as $v$.

\item Assume that there is some rational prime $p\in \Spl_{\cK/\mathbb Q}\setminus (S\cup \{\ell\})$, split as $p=uu^c$ in $\cK$, and a prime $\mathfrak p|p$ of $E$ such that the following condition involving the primes $\p_i|u$ of $F$, $i=1,\ldots,m$, and the sets $S_i$ from Theorem~\ref{igusa vanishing} holds true. For any $i$, there is at most one $\tau\in S_i$ such that $p_\tau q_\tau\neq 0$; if there is such a $\tau\in S_i$, then $\rho_{\mathfrak m}$ is decomposed generic at $\p_i$. Then
\[
H^i(S_K,\mathbb F_\ell)_{\mathfrak m}\neq 0
\]
only for $i=\dim S_K$.
\end{enumerate}
\end{thm}

Before giving the proof, let us explain in two examples how the condition (2) can be ensured, thus connecting it with the conditions stated in the introduction.

\begin{remark} Assume that there is a prime $p$ which is completely decomposed in $F$ and such that $\rho_{\mathfrak m}$ is unramified and decomposed generic at all places above $p$. Using Chebotarev, there are then many such $p$, and we can assume that $p\not\in S\cup \{\ell\}$. In that case, all sets $S_i$ in (2) have just one element, and we see that the desired condition is satisfied.
\end{remark}

\begin{remark} Assume that the signature of $G$ is $(0,n)$ at all except for one infinite place. Moreover, assume that there is some finite prime $v$ of $F$ such that $\rho_{\mathfrak m}$ is unramified and decomposed generic at $v$. By Chebotarev, there are then many such $v$ which are moreover decomposed over the rational prime $p$ of $\mathbb Q$, with $p\not\in S\cup\{\ell\}$.\footnote{In Chebotarev's theorem, only places with residue field $\mathbb F_p$ contribute to the Dirichlet density.} In particular, $p$ needs to be split in $\cK$. There is just one $\tau$ for which $p_\tau q_\tau\neq 0$, and by choosing the prime $\mathfrak p$ of the reflex field correctly, one can arrange that this $\tau$ appears in $S_i$ for $\p_i = v$. We see that condition (2) applies.
\end{remark}

\begin{proof} We write out the argument in the more involved Case 2.

For part (1), pick any $p\in \Spl_{\cK/\mathbb Q}\setminus (S\cup \{\ell\})$. Then $K = K_p K^p$ is decomposed. There is a Hochschild-Serre spectral sequence relating
\[
H^i(\cS_{K^p},\mathbb F_\ell)
\]
and $H^i(S_K,\mathbb F_\ell)$.\footnote{Here and in the following, all cohomology groups are \'etale cohomology groups after base change to an algebraically closed field.} In particular, it follows that if $i$ is minimal with $H^i(S_K,\mathbb F_\ell)_{\mathfrak m}\neq 0$, then
\[
H^i(\cS_{K^p},\mathbb F_\ell)_{\mathfrak m}\neq 0\ .
\]
Thus, there is some $b\in B(G,\mu^{-1})$ such that
\[
H^i(\mathrm{Ig}^b,\mathbb F_\ell)_{\mathfrak m}\neq 0
\]
for some $i\in \mathbb Z$; otherwise we would have
\[
(R\pi_{HT\ast} \mathbb F_\ell)_{\mathfrak m} = 0\ ,
\]
and hence
\[
R\Gamma(\cS_{K^p},\mathbb F_\ell)_{\mathfrak m} = 0
\]
by the Leray spectral sequence for $\pi_{HT}: \cS_{K^p}\to \Fl_{G,\mu}$. Now pick some $b\in B(G,\mu^{-1})$ with $d=\langle 2\rho,\nu_b\rangle$ minimal such that for some $i\in \mathbb Z$
\[
H^i(\mathrm{Ig}^b,\mathbb F_\ell)_{\mathfrak m}\neq 0\ .
\]
In that case, this group is nonzero exactly for $i=d$ by Corollary~\ref{conc in one degree}. Taking invariants under a pro-$p$-compact open subgroup of $J_b(\mathbb Q_p)$ (which is an exact operation), this implies that 
\[
H^i(\mathscr{I}^b_{\mathrm{Mant},m},\mathbb F_\ell)_{\mathfrak m}
\]
is nonzero at most for $i=d$; if $m$ is large enough, it is nonzero if $i=d$. It follows that the cohomology with $\mathbb Z_\ell$-coefficients is concentrated in the middle degree and flat, and thus the $\bar{\mathbb Q}_\ell$-cohomology
\[
H^i(\mathscr{I}^b_{\mathrm{Mant},m},\mathbb Z_\ell)_{\mathfrak m}\otimes \bar{\mathbb Q}_\ell
\]
is nonzero for $i=d$. By Poincar\'e duality (and applying the same discussion with the ``dual'' set of Hecke eigenvalues), the same holds true for compactly supported cohomology. We have a decomposition
\[
[H_c(\mathscr{I}^b_{\mathrm{Mant}},\bar{\mathbb Q}_\ell)]^{S\ur} = [H_c(\mathscr{I}^b_{\mathrm{Mant}},\bar{\mathbb Q}_\ell)]^{S\ur}_{\mathfrak m} + [H_c(\mathscr{I}^b_{\mathrm{Mant}},\bar{\mathbb Q}_\ell)]^{S\ur,\mathfrak{m}}
\]
according to systems of Hecke eigenvalues lifting $\mathfrak m$, or a different set of Hecke eigenvalues modulo $\ell$, and by concentration in one degree, the first summand is nonzero in the Grothendieck group, and its base change $BC^p$ is still nonzero. It follows that there is some $\Pi^{\vec{n}}$ as in Lemma~\ref{expansion of igusa cohomology} whose Hecke eigenvalues lift $\mathfrak m$. Then Theorem~\ref{ex gal repr} implies that there is a Galois representation $r_{\Pi^{\vec{n}},\ell}$, whose reduction is the desired Galois representation $\rho_{\mathfrak{m}}$.

Now, we deal with part (2). We choose $p$ and $\mathfrak p$ as guaranteed in the statement. It is enough to prove that $H^i(S_K,\mathbb F_\ell)_{\mathfrak m}$ is nonzero only for $i\geq \dim S_K$; the other bound follows by Poincar\'e duality (and the result for the ``dual'' ideal, which satisfies the same hypothesis). Now a Hochschild-Serre spectral sequence shows that it is enough to prove that
\[
H^i(\cS_{K^p},\mathbb F_\ell)_{\mathfrak m}=0
\]
for $i<\dim S_K$. As above, we take some $b\in B(G,\mu^{-1})$ with $d=\langle 2\rho,\nu_b\rangle$ minimal such that
\[
H^i(\mathrm{Ig}^b,\mathbb F_\ell)_{\mathfrak m}\neq 0
\]
for some $i\in \mathbb Z$. We get concentration in middle degree in this case, and hence the argument above shows that there is some Galois representation $r$ lifting $\rho_{\mathfrak m}$ with
\[
BC^p([H_c(\mathscr{I}^b_{\mathrm{Mant}},\bar{\mathbb Q}_\ell)]^{S\ur})_r\neq 0\ .
\]
But by Lemma~\ref{generic lifts} and the assumptions on $p$, $\mathfrak p$ and $\rho_{\mathfrak m}$, the hypothesis of Theorem~\ref{igusa vanishing} are satisfied. Thus, if $b$ is not $\mu$-ordinary, we arrive at a contradiction. It follows that $b$ is $\mu$-ordinary.

In that case, $\langle 2\rho,\mu\rangle = \langle 2\rho,\nu_b\rangle = \dim S_K$, so Corollary~\ref{conc in one degree} shows that
\[
H^i(\mathrm{Ig}^b,\mathbb F_\ell)_{\mathfrak m}
\]
vanishes for $i<\dim S_K$, for all $b\in B(G,\mu^{-1})$. Thus, $(R^i\pi_{HT\ast} \mathbb F_\ell)_{\mathfrak m}$ vanishes for $i<\dim S_K$, and the result follows by applying the Leray spectral sequence for $\pi_{HT}: \cS_{K^p}\to \Fl_{G,\mu}$.
\end{proof}

\bibliographystyle{amsalpha} 
\bibliography{Torsionvanishing} 

\providecommand{\MR}[1]{}
\providecommand{\bysame}{\leavevmode\hbox to3em{\hrulefill}\thinspace}
\providecommand{\MR}{\relax\ifhmode\unskip\space\fi MR }
\providecommand{\MRhref}[2]{%
  \href{http://www.ams.org/mathscinet-getitem?mr=#1}{#2}
}
\providecommand{\href}[2]{#2}
\begin{thebibliography}{{Wei}14}

\bibitem[Art88a]{arthura}
James Arthur, \emph{The invariant trace formula. {I}. {L}ocal theory}, J. Amer.
  Math. Soc. \textbf{1} (1988), no.~2, 323--383. \MR{928262 (89e:22029)}

\bibitem[Art88b]{arthurb}
\bysame, \emph{The invariant trace formula. {II}. {G}lobal theory}, J. Amer.
  Math. Soc. \textbf{1} (1988), no.~3, 501--554. \MR{939691 (89j:22039)}

\bibitem[Art89]{arthur-lefschetz}
\bysame, \emph{The {$L^2$}-{L}efschetz numbers of {H}ecke operators}, Invent.
  Math. \textbf{97} (1989), no.~2, 257--290. \MR{1001841 (91i:22024)}

\bibitem[Art13]{arthur-endoscopic-classification}
\bysame, \emph{The endoscopic classification of representations}, American
  Mathematical Society Colloquium Publications, vol.~61, American Mathematical
  Society, Providence, RI, 2013, Orthogonal and symplectic groups. \MR{3135650}

\bibitem[Bad07]{badulescu}
Alexandru~Ioan Badulescu, \emph{Jacquet-{L}anglands et unitarisabilit\'e}, J.
  Inst. Math. Jussieu \textbf{6} (2007), no.~3, 349--379. \MR{2329758
  (2008h:20065)}

\bibitem[BBM82]{bbm}
Pierre Berthelot, Lawrence Breen, and William Messing, \emph{Th\'eorie de
  {D}ieudonn\'e cristalline. {II}}, Lecture Notes in Mathematics, vol. 930,
  Springer-Verlag, Berlin, 1982. \MR{667344 (85k:14023)}

\bibitem[Ber80]{berthelotvalparfait}
Pierre Berthelot, \emph{Th\'eorie de {D}ieudonn\'e sur un anneau de valuation
  parfait}, Ann. Sci. \'Ecole Norm. Sup. (4) \textbf{13} (1980), no.~2,
  225--268. \MR{584086 (82b:14026)}

\bibitem[Bla94]{blasius}
Don Blasius, \emph{A {$p$}-adic property of {H}odge classes on abelian
  varieties}, Motives ({S}eattle, {WA}, 1991), Proc. Sympos. Pure Math.,
  vol.~55, Amer. Math. Soc., Providence, RI, 1994, pp.~293--308. \MR{1265557
  (95j:14022)}

\bibitem[Boy15]{boyer}
Pascal Boyer, \emph{{S}ur la torsion dans la cohomologie des variétés de
  {S}himura de {K}ottwitz-{H}arris-{T}aylor},
  https://www.math.univ-paris13.fr/$\sim$boyer/recherche/torsion-localise.pdf.

\bibitem[BS15a]{bhattscholze}
Bhargav Bhatt and Peter Scholze, \emph{The pro\'etale topology for schemes},
  Ast\'erisque 369 (2015).

\bibitem[BS15b]{bhattscholzeperf}
\bysame, \emph{Projectivity of the {W}itt vector affine {G}rassmannian},
  arXiv:1507.06490.

\bibitem[BW80]{borelwallach}
Armand Borel and Nolan~R. Wallach, \emph{Continuous cohomology, discrete
  subgroups, and representations of reductive groups}, Annals of Mathematics
  Studies, vol.~94, Princeton University Press, Princeton, N.J.; University of
  Tokyo Press, Tokyo, 1980. \MR{554917 (83c:22018)}

\bibitem[Car12]{caraiani}
Ana Caraiani, \emph{Local-global compatibility and the action of monodromy on
  nearby cycles}, Duke Math. J. \textbf{161} (2012), no.~12, 2311--2413.
  \MR{2972460}

\bibitem[CG]{calegari-geraghty}
Frank Calegari and David Geraghty, \emph{Modularity lifting beyond the
  taylor-wiles method}, https://www2.bc.edu/david-geraghty/files/merge.pdf.

\bibitem[CH13]{chenevier-harris}
Ga{\"e}tan Chenevier and Michael Harris, \emph{Construction of automorphic
  {G}alois representations, {II}}, Camb. J. Math. \textbf{1} (2013), no.~1,
  53--73. \MR{3272052}

\bibitem[Del70]{delignevb}
Pierre Deligne, \emph{\'{E}quations diff\'erentielles \`a points singuliers
  r\'eguliers}, Lecture Notes in Mathematics, Vol. 163, Springer-Verlag,
  Berlin-New York, 1970. \MR{0417174 (54 \#5232)}

\bibitem[Del71]{DeligneVarSh}
\bysame, \emph{Travaux de {S}himura}, S\'eminaire {B}ourbaki, 23\`eme ann\'ee
  (1970/71), {E}xp. {N}o. 389, Springer, Berlin, 1971, pp.~123--165. Lecture
  Notes in Math., Vol. 244. \MR{0498581 (58 \#16675)}

\bibitem[Del77]{SGA41/2}
P.~Deligne, \emph{Cohomologie \'etale}, Lecture Notes in Mathematics, Vol. 569,
  Springer-Verlag, Berlin-New York, 1977, S{\'e}minaire de G{\'e}om{\'e}trie
  Alg{\'e}brique du Bois-Marie SGA 4$\ {1/2}$, Avec la collaboration de J. F.
  Boutot, A. Grothendieck, L. Illusie et J. L. Verdier. \MR{0463174 (57
  \#3132)}

\bibitem[Del79]{deligne-varietes}
Pierre Deligne, \emph{Vari\'et\'es de {S}himura: interpr\'etation modulaire, et
  techniques de construction de mod\`eles canoniques}, Automorphic forms,
  representations and {$L$}-functions ({P}roc. {S}ympos. {P}ure {M}ath.,
  {O}regon {S}tate {U}niv., {C}orvallis, {O}re., 1977), {P}art 2, Proc. Sympos.
  Pure Math., XXXIII, Amer. Math. Soc., Providence, R.I., 1979, pp.~247--289.
  \MR{546620 (81i:10032)}

\bibitem[Del82]{delignehodge}
\bysame, \emph{Hodge cycles on abelian varieties}, Hodge cycles, motives, and
  {S}himura varieties, Lecture Notes in Mathematics, vol. 900, Springer-Verlag,
  Berlin-New York, 1982, pp.~ii+414. \MR{654325 (84m:14046)}

\bibitem[EG15]{emerton-gee}
Matthew Emerton and Toby Gee, \emph{{$p$}-adic {H}odge-theoretic properties of
  \'etale cohomology with {${\rm mod}\, p$} coefficients, and the cohomology of
  {S}himura varieties}, Algebra Number Theory \textbf{9} (2015), no.~5,
  1035--1088. \MR{3365999}

\bibitem[Far]{fargues-geom-langlands}
Laurent Fargues, \emph{Geom\'etrisation de la correspondance de {L}anglands},
  in preparation.

\bibitem[Far15a]{farguesGbun}
\bysame, \emph{{$G$}-torseurs en th\'eorie de {H}odge {$p$}-adique},
  http://webusers.imj-prg.fr/$\sim$laurent.fargues/Gtorseurs.pdf.

\bibitem[Far15b]{fargues}
\bysame, \emph{Quelques r\'esultats et conjectures concernant la courbe},
  Ast\'erisque 369 (2015).

\bibitem[FF14]{farguesfontaine}
L.~{Fargues} and J.~M. {Fontaine}, \emph{Vector bundles on curves and
  {$p$}-adic {H}odge theory}, {A}utomorphic {F}orms and {G}alois
  representations, {V}olume {I}, London Math. Soc. Lec. Note., vol. Ser. 414,
  Cambrigde Univ. Press, 2014.

\bibitem[Fra98]{Franke}
Jens Franke, \emph{Harmonic analysis in weighted {$L_2$}-spaces}, Ann. Sci.
  \'Ecole Norm. Sup. (4) \textbf{31} (1998), no.~2, 181--279. \MR{1603257
  (2000f:11065)}

\bibitem[Gol14]{shin-appendix}
Wushi Goldring, \emph{Galois representations associated to holomorphic limits
  of discrete series}, Compos. Math. \textbf{150} (2014), no.~2, 191--228.
  \MR{3177267}

\bibitem[Ham]{hamacher}
P.~Hamacher, \emph{The geometry of {N}ewton strata in the reduction modulo $p$
  of {S}himura varieties of {PEL} type}, to appear in Duke Math. J.

\bibitem[HT01]{harris-taylor}
Michael Harris and Richard Taylor, \emph{The geometry and cohomology of some
  simple {S}himura varieties}, Annals of Mathematics Studies, vol. 151,
  Princeton University Press, Princeton, NJ, 2001, With an appendix by Vladimir
  G. Berkovich. \MR{1876802 (2002m:11050)}

\bibitem[Hub96]{huber}
Roland Huber, \emph{\'{E}tale cohomology of rigid analytic varieties and adic
  spaces}, Aspects of Mathematics, E30, Friedr. Vieweg \& Sohn, Braunschweig,
  1996. \MR{1734903 (2001c:14046)}

\bibitem[Ill94]{illusie-autour}
Luc Illusie, \emph{Autour du th\'eor\`eme de monodromie locale}, Ast\'erisque
  (1994), no.~223, 9--57, P{\'e}riodes $p$-adiques (Bures-sur-Yvette, 1988).
  \MR{1293970 (95k:14032)}

\bibitem[JS81]{jacquet-shalika}
H.~Jacquet and J.~A. Shalika, \emph{On {E}uler products and the classification
  of automorphic forms. {II}}, Amer. J. Math. \textbf{103} (1981), no.~4,
  777--815. \MR{623137 (82m:10050b)}

\bibitem[Kat79]{katz}
Nicholas~M. Katz, \emph{Slope filtration of {$F$}-crystals}, Journ\'ees de
  {G}\'eom\'etrie {A}lg\'ebrique de {R}ennes ({R}ennes, 1978), {V}ol. {I},
  Ast\'erisque, vol.~63, Soc. Math. France, Paris, 1979, pp.~113--163.
  \MR{563463 (81i:14014)}

\bibitem[{Kis}]{kisin}
M.~{Kisin}, \emph{{Mod $p$ points on Shimura varieties of abelian type}},
  preprint.

\bibitem[Kis10]{kisinintegral}
Mark Kisin, \emph{Integral models for {S}himura varieties of abelian type}, J.
  Amer. Math. Soc. \textbf{23} (2010), no.~4, 967--1012. \MR{2669706
  (2011j:11109)}

\bibitem[KL15]{kedlayaliu}
K.~S. {Kedlaya} and R.~{Liu}, \emph{Relative p-adic {H}odge theory, {I}:
  {F}oundations}, Ast\'erisque 371 (2015).

\bibitem[Kot84]{kottwitz-cuspidal}
Robert~E. Kottwitz, \emph{Stable trace formula: cuspidal tempered terms}, Duke
  Math. J. \textbf{51} (1984), no.~3, 611--650. \MR{757954 (85m:11080)}

\bibitem[Kot85]{kottwitz}
\bysame, \emph{Isocrystals with additional structure}, Compositio Math.
  \textbf{56} (1985), no.~2, 201--220. \MR{809866 (87i:14040)}

\bibitem[Kot90]{kottwitz-lambda-adic}
\bysame, \emph{Shimura varieties and {$\lambda$}-adic representations},
  Automorphic forms, {S}himura varieties, and {$L$}-functions, {V}ol.\ {I}
  ({A}nn {A}rbor, {MI}, 1988), Perspect. Math., vol.~10, Academic Press,
  Boston, MA, 1990, pp.~161--209. \MR{1044820 (92b:11038)}

\bibitem[Kot92a]{Kottwitz-lambda}
\bysame, \emph{On the {$\lambda$}-adic representations associated to some
  simple {S}himura varieties}, Invent. Math. \textbf{108} (1992), no.~3,
  653--665. \MR{1163241 (93f:11046)}

\bibitem[Kot92b]{kottwitzpoints}
\bysame, \emph{Points on some {S}himura varieties over finite fields}, J. Amer.
  Math. Soc. \textbf{5} (1992), no.~2, 373--444. \MR{1124982 (93a:11053)}

\bibitem[Kud94]{kudla}
Stephen~S. Kudla, \emph{The local {L}anglands correspondence: the
  non-{A}rchimedean case}, Motives ({S}eattle, {WA}, 1991), Proc. Sympos. Pure
  Math., vol.~55, Amer. Math. Soc., Providence, RI, 1994, pp.~365--391.
  \MR{1265559 (95d:11065)}

\bibitem[Lab91]{labesse-lefschetz}
J.-P. Labesse, \emph{Pseudo-coefficients tr\`es cuspidaux et {$K$}-th\'eorie},
  Math. Ann. \textbf{291} (1991), no.~4, 607--616. \MR{1135534 (93b:22026)}

\bibitem[Lab99]{labesse}
Jean-Pierre Labesse, \emph{Cohomologie, stabilisation et changement de base},
  Ast\'erisque (1999), no.~257, vi+161, Appendix A by Laurent Clozel and
  Labesse, and Appendix B by Lawrence Breen. \MR{1695940 (2001k:22040)}

\bibitem[Lan79]{langlands}
R.~P. Langlands, \emph{Stable conjugacy: definitions and lemmas}, Canad. J.
  Math. \textbf{31} (1979), no.~4, 700--725. \MR{540901 (82j:10054)}

\bibitem[Lan11]{lan}
Kai-Wen Lan, \emph{Elevators for degenerations of {PEL} structures}, Math. Res.
  Lett. \textbf{18} (2011), no.~5, 889--907. \MR{2875862}

\bibitem[Lau13]{lau}
Eike Lau, \emph{Smoothness of the truncated display functor}, J. Amer. Math.
  Soc. \textbf{26} (2013), no.~1, 129--165. \MR{2983008}

\bibitem[LS87]{langlands-shelstad}
R.~P. Langlands and D.~Shelstad, \emph{On the definition of transfer factors},
  Math. Ann. \textbf{278} (1987), no.~1-4, 219--271. \MR{909227 (89c:11172)}

\bibitem[LS12]{lan-suh}
Kai-Wen Lan and Junecue Suh, \emph{Vanishing theorems for torsion automorphic
  sheaves on compact {PEL}-type {S}himura varieties}, Duke Math. J.
  \textbf{161} (2012), no.~6, 1113--1170. \MR{2913102}

\bibitem[Man05]{mantovan}
Elena Mantovan, \emph{On the cohomology of certain {PEL}-type {S}himura
  varieties}, Duke Math. J. \textbf{129} (2005), no.~3, 573--610. \MR{2169874
  (2006g:11122)}

\bibitem[Mat67]{Matsushima}
Yoz{\^o} Matsushima, \emph{A formula for the {B}etti numbers of compact locally
  symmetric {R}iemannian manifolds}, J. Differential Geometry \textbf{1}
  (1967), 99--109. \MR{0222908 (36 \#5958)}

\bibitem[Mes72]{messing}
William Messing, \emph{The crystals associated to {B}arsotti-{T}ate groups:
  with applications to abelian schemes}, Lecture Notes in Mathematics, Vol.
  264, Springer-Verlag, Berlin-New York, 1972. \MR{0347836 (50 \#337)}

\bibitem[Mil90]{milne}
J.~S. Milne, \emph{Canonical models of (mixed) {S}himura varieties and
  automorphic vector bundles}, Automorphic forms, {S}himura varieties, and
  {$L$}-functions, {V}ol.\ {I} ({A}nn {A}rbor, {MI}, 1988), Perspect. Math.,
  vol.~10, Academic Press, Boston, MA, 1990, pp.~283--414. \MR{1044823
  (91a:11027)}

\bibitem[MW89]{moeglin-waldspurger}
C.~M{\oe}glin and J.-L. Waldspurger, \emph{Le spectre r\'esiduel de {${\rm
  GL}(n)$}}, Ann. Sci. \'Ecole Norm. Sup. (4) \textbf{22} (1989), no.~4,
  605--674. \MR{1026752 (91b:22028)}

\bibitem[Ng{\^o}10]{ngo}
Bao~Ch{\^a}u Ng{\^o}, \emph{Le lemme fondamental pour les alg\`ebres de {L}ie},
  Publ. Math. Inst. Hautes \'Etudes Sci. (2010), no.~111, 1--169. \MR{2653248
  (2011h:22011)}

\bibitem[NP01]{ngopolo}
B.~C. Ng{\^o} and P.~Polo, \emph{R\'esolutions de {D}emazure affines et formule
  de {C}asselman-{S}halika g\'eom\'etrique}, J. Algebraic Geom. \textbf{10}
  (2001), no.~3, 515--547. \MR{1832331 (2002f:14032)}

\bibitem[OZ02]{oort-zink}
Frans Oort and Thomas Zink, \emph{Families of {$p$}-divisible groups with
  constant {N}ewton polygon}, Doc. Math. \textbf{7} (2002), 183--201
  (electronic). \MR{1938119 (2003m:14066)}

\bibitem[Rap15]{rapoportappendix}
Michael Rapoport, \emph{Accessible and weakly accessible period domains},
  appendix to http://www.math.uni-bonn.de/people/scholze/pAdicCohomLT.pdf.

\bibitem[RR96]{rapoport-richartz}
M.~Rapoport and M.~Richartz, \emph{On the classification and specialization of
  {$F$}-isocrystals with additional structure}, Compositio Math. \textbf{103}
  (1996), no.~2, 153--181. \MR{1411570 (98c:14015)}

\bibitem[RZ96]{rapoport-zink}
M.~Rapoport and Th. Zink, \emph{Period spaces for {$p$}-divisible groups},
  Annals of Mathematics Studies, vol. 141, Princeton University Press,
  Princeton, NJ, 1996. \MR{1393439 (97f:14023)}

\bibitem[Sch11]{scholze-letter-dat}
Peter Scholze, \emph{{L}etter to {J}ean-{F}rancois {D}at}, April 28, 2011.

\bibitem[Sch12a]{scholzeperfectoid}
\bysame, \emph{Perfectoid spaces}, Publ. Math. Inst. Hautes \'Etudes Sci.
  \textbf{116} (2012), 245--313. \MR{3090258}

\bibitem[Sch12b]{scholzesurvey}
\bysame, \emph{Perfectoid {S}paces: {A} survey}, Current Developments in
  Mathematics (2012).

\bibitem[Sch13a]{scholzeLKgen}
\bysame, \emph{The {L}anglands-{K}ottwitz method and deformation spaces of
  {$p$}-divisible groups}, J. Amer. Math. Soc. \textbf{26} (2013), no.~1,
  227--259. \MR{2983011}

\bibitem[Sch13b]{scholzeLLC}
\bysame, \emph{The local {L}anglands correspondence for {$GL_n$} over
  {$p$}-adic fields}, Invent. Math. \textbf{192} (2013), no.~3, 663--715.
  \MR{3049932}

\bibitem[Sch13c]{scholzerigid}
\bysame, \emph{{$p$}-adic {H}odge theory for rigid-analytic varieties}, Forum
  Math. Pi \textbf{1} (2013), e1, 77. \MR{3090230}

\bibitem[Sch15a]{scholze-lubintate}
\bysame, \emph{{O}n the {$p$}-adic cohomology of the {L}ubin-{T}ate tower},
  http://www.math.uni-bonn.de/people/scholze/pAdicCohomLT.pdf.

\bibitem[Sch15b]{scholze}
\bysame, \emph{On torsion in the cohomology of locally symmetric varieties},
  Annals of Mathematics (2015).

\bibitem[SGA73]{SGA7.2}
\emph{Groupes de monodromie en g\'eom\'etrie alg\'ebrique. {II}}, Lecture Notes
  in Mathematics, Vol. 340, Springer-Verlag, Berlin-New York, 1973,
  S{\'e}minaire de G{\'e}om{\'e}trie Alg{\'e}brique du Bois-Marie 1967--1969
  (SGA 7 II), Dirig{\'e} par P. Deligne et N. Katz. \MR{0354657 (50 \#7135)}

\bibitem[Shi09]{shin-igusa}
Sug~Woo Shin, \emph{Counting points on {I}gusa varieties}, Duke Math. J.
  \textbf{146} (2009), no.~3, 509--568. \MR{2484281 (2011e:11107)}

\bibitem[Shi10]{shin-stable}
\bysame, \emph{A stable trace formula for {I}gusa varieties}, J. Inst. Math.
  Jussieu \textbf{9} (2010), no.~4, 847--895. \MR{2684263 (2012b:11095)}

\bibitem[Shi11]{shin-galois}
\bysame, \emph{Galois representations arising from some compact {S}himura
  varieties}, Ann. of Math. (2) \textbf{173} (2011), no.~3, 1645--1741.
  \MR{2800722}

\bibitem[Shi12]{shin-rapoport-zink}
\bysame, \emph{On the cohomology of {R}apoport-{Z}ink spaces of {EL}-type},
  Amer. J. Math. \textbf{134} (2012), no.~2, 407--452. \MR{2905002}

\bibitem[Shi15]{shin-torsion}
\bysame, \emph{Supercuspidal part of the mod {$l$} cohomology of
  {$\text{GU}(1,n-1)$}-{S}himura varieties}, J. Reine Angew. Math. \textbf{705}
  (2015), 1--21. \MR{3377388}

\bibitem[SS13]{scholze-shin}
Peter Scholze and Sug~Woo Shin, \emph{On the cohomology of compact unitary
  group {S}himura varieties at ramified split places}, J. Amer. Math. Soc.
  \textbf{26} (2013), no.~1, 261--294. \MR{2983012}

\bibitem[SW13]{scholzeweinstein}
Peter Scholze and Jared Weinstein, \emph{Moduli of {$p$}-divisible groups},
  Camb. J. Math. \textbf{1} (2013), no.~2, 145--237. \MR{3272049}

\bibitem[Tho14]{thorne}
Jack~A. Thorne, \emph{Raising the level for {${\rm GL}_{n}$}}, Forum Math.
  Sigma \textbf{2} (2014), e16, 35. \MR{3264255}

\bibitem[Wal97]{waldspurger}
J.-L. Waldspurger, \emph{Le lemme fondamental implique le transfert},
  Compositio Math. \textbf{105} (1997), no.~2, 153--236. \MR{1440722
  (98h:22023)}

\bibitem[Wed99]{wedhorn}
Torsten Wedhorn, \emph{Ordinariness in good reductions of {S}himura varieties
  of {PEL}-type}, Ann. Sci. \'Ecole Norm. Sup. (4) \textbf{32} (1999), no.~5,
  575--618. \MR{1710754 (2000g:11054)}

\bibitem[{Wei}14]{scholzelectures}
Jared {Weinstein}, \emph{Peter {S}cholze's lectures on {$p$}-adic geometry},
  Available online at http://math.berkeley.edu/~jared/Math274/index.html
  (2014).

\end{thebibliography}

\end{document}